\documentclass[12pt,a4paper,final]{book}
\usepackage[marginratio=1:1]{geometry}
\date{\today}

\usepackage{hyperref}
\usepackage[T1]{fontenc}
\usepackage{ae,aecompl}
\usepackage[final,tracking=true,kerning=true,spacing=true,stretch=10,shrink=10]{microtype}
\SetTracking{encoding=*, shape=sc}{50}
\microtypecontext{spacing=nonfrench}
\usepackage[utf8]{inputenc}
\usepackage[british]{babel}
\usepackage{csquotes}
\usepackage{imakeidx}
\usepackage{xcolor}
\usepackage{amsfonts}
\usepackage{amssymb}
\usepackage{amsmath}
\usepackage{amsthm}
\usepackage{tikz-cd}
\usepackage{float}
\usepackage{graphicx}
\usepackage{tikz}
\usepackage{thmtools}
\usepackage{thm-restate}
\usepackage{mathrsfs}
\usepackage{enumitem}
\usepackage{mathtools}
\usepackage{fancyhdr,textcase}
\usepackage{isomath}
\renewcommand{\chaptermark}[1]{\markboth{#1}{}}

\setenumerate{label=(\arabic*),topsep=.1\baselineskip,itemsep=-0.1\baselineskip}
\setitemize{topsep=.1\baselineskip,itemsep=-0.1\baselineskip}
\usepackage[backend=biber,
url=false,
isbn=false,
backref=true,
citestyle=alphabetic,
bibstyle=alphabetic,
autocite=inline,
maxnames=99,
minalphanames=4,
maxalphanames=4,
sorting=nyt,]{biblatex}
\hypersetup{
	colorlinks,
	linkcolor={red!50!black},
	citecolor={blue!50!black},
	urlcolor={blue!80!black}
}
\usetikzlibrary{arrows.meta}
\usetikzlibrary{matrix, arrows}
\makeindex[title=List of symbols and definitions]
\usepackage{hyperref}

\makeatletter
\tikzset{
	edge node/.code={%
		\expandafter\def\expandafter\tikz@tonodes\expandafter{\tikz@tonodes #1}}}
\makeatother
\tikzset{
	subseteq/.style={
		draw=none,
		edge node={node [sloped, allow upside down, auto=false]{$\subseteq$}}},
	Subseteq/.style={
		draw=none,
		every to/.append style={
			edge node={node [sloped, allow upside down, auto=false]{$\subseteq$}}}
	}
}

\newcommand{\fC}{{\mathfrak C}}
\newcommand{\cM}{{\mathcal M}}
\newcommand{\cN}{{\mathcal N}}
\newcommand{\cB}{{\mathcal B}}
\newcommand{\bN}{{\mathbf{N}}}
\newcommand{\bR}{{\mathbf{R}}}
\newcommand{\bZ}{{\mathbf{Z}}}
\newcommand{\bQ}{{\mathbf{Q}}}

\newcommand{\topo}{{\mathrm{top}}}
\newcommand\Lasc{{\mathrm{L}}}
\newcommand\KP{{\mathrm{KP}}}
\newcommand\Sh{{\mathrm{Sh}}}
\newcommand{\restr}{\mathord{\upharpoonright}}
\newcommand{\EZ}{\mathrel{ { {\mathbf E}_0 } } }
\newcommand{\Er}{\mathrel{E}}

\newcommand{\lang}{{\mathcal L}}
\newcommand{\catg}{{\mathcal C}}
\newcommand{\powerset}{{\mathcal P}}

\newcommand{\proves}{\vdash}
\newcommand{\Rr}{\mathrel{R}}
\DeclareMathOperator{\SO}{{SO}}
\DeclareMathOperator{\GL}{{GL}}
\DeclareMathOperator{\st}{{st}}
\DeclareMathOperator{\cl}{{cl}}
\DeclareMathOperator{\tp}{{tp}}
\DeclareMathOperator{\acl}{{acl}}
\DeclareMathOperator{\dcl}{{dcl}}
\DeclareMathOperator{\Th}{{Th}}
\DeclareMathOperator{\Gal}{{Gal}}

\DeclareMathOperator{\id}{{id}}
\DeclareMathOperator{\Aut}{{Aut}}
\DeclareMathOperator{\Homeo}{{Homeo}}
\DeclareMathOperator{\Autf}{{Aut\mkern 0.5\thinmuskip f}}

\DeclareMathOperator{\Core}{{Core}}
\DeclareMathOperator{\Stab}{{Stab}}
\DeclareMathOperator{\Souslin}{{\mathcal A}}
\let\unlhd\trianglelefteq
\let\leq\leqslant
\let\geq\geqslant

\newtheorem{mainthm}{Main Theorem}

\newtheorem{thm}{Theorem}[chapter]

\newtheorem{ques}[thm]{Question}

\newtheorem{lem}[thm]{Lemma}
\newtheorem{fct}[thm]{Fact}
\newtheorem{cor}[thm]{Corollary}
\newtheorem{prop}[thm]{Proposition}
\newtheorem{qu}[thm]{Question}

\theoremstyle{remark}
\newtheorem{rem}[thm]{Remark}
\theoremstyle{definition}
\newtheorem{dfn}[thm]{Definition}

\newtheorem*{clm*}{Claim}
\newtheorem{ex}[thm]{Example}
\newcounter{claimcounter}[thm]

\newenvironment{clmproof}[1][\proofname]{\proof[#1]}{\endproof}

\newcommand{\xqed}[1]{%
	\leavevmode\unskip\penalty9999 \hbox{}\nobreak\hfill
	\quad\hbox{\ensuremath{#1}}}

\title{Bounded Invariant Equivalence Relations}
\author{Tomasz Rzepecki}
\usepackage{datetime}
\date{\today}

\begin{document}
	\newpage
	\thispagestyle{empty}
	\microtypesetup{disable}
	\makeatletter
	\begin{center}
		\textbf{\large Uniwersytet Wrocławski\\
			Wydział Matematyki i Informatyki\\
			Instytut Matematyczny}\\
		\vspace{4cm}
		\textbf{\textit{\large \@author}\\
			\vspace{0.5cm}
			{\Large \@title}}\\
	\end{center}
	\vspace{3cm}
	{\large \hspace*{6.5cm}doctoral thesis\\
		\hspace*{6.5cm}supervised by\\
		\hspace*{6.5cm}prof. Krzysztof Krupiński}\\
	\vfill
	\begin{center}
		{\large Wrocław 2018}\\
	\end{center}
	\makeatother
	\newpage
	\thispagestyle{empty}
	\makeatletter
	\begin{center}
		\textbf{\large Uniwersytet Wrocławski\\
			Wydział Matematyki i Informatyki\\
			Instytut Matematyczny}\\
		\vspace{4cm}
		\textbf{\textit{\large \@author}\\
			\vspace{0.5cm}
			{\Large Ograniczone, niezmiennicze relacje równoważności}}\\
	\end{center}
	\vspace{3cm}
	{\large \hspace*{6.5cm}praca doktorska\\
		\hspace*{6.5cm}promotor:\\
		\hspace*{6.5cm}prof. dr hab. Krzysztof Krupiński}\\
	\vfill
	\begin{center}
		{\large Wrocław 2018}\\
	\end{center}
	\makeatother
	\newpage
	\microtypesetup{enable}

	\section*{Abstract}
	\thispagestyle{headings}
	\vspace{-.5\baselineskip}
	We study strong types and Galois groups in model theory from a topological and descriptive-set-theoretical point of view.
	
	The main results of the thesis are the following:
	\begin{itemize}
		\item
		we present the (Lascar) Galois group of an arbitrary countable first-order theory (as a topological group, and --- to a degree --- as a ``Borel quotient") as the quotient of a compact Polish group (which is a certain quotient of the Ellis group of a dynamical system associated with the automorphism group of a suitable countable model) by a normal $F_\sigma$ subgroup; we also show that all strong type spaces are ``locally" the quotient of the same group by a subgroup (which is not necessarily normal, but is Borel if the strong type is Borel);
		\item
		we show that a bounded invariant equivalence relation on the set of realisations of a single complete type is either relatively definable (and thus has finitely many classes), type-definable with at least continuum many classes, or (assuming that the theory is countable) non-smooth in the descriptive-set-theoretic sense (in which case, if it is analytic, it also has at least continuum many classes);
		\item
		we find a sufficient condition for a bounded invariant equivalence relation under which its type-definability is equivalent to type-definability of all of its classes; this is enough to show that (under this condition) smoothness is equivalent to type-definability.
	\end{itemize}
	The first result is joint with Krzysztof Krupiński, the second one is joint with Krzysztof Krupiński and Anand Pillay, while the third is mine alone.
	
	In this thesis, I consider the more abstract case of an equivalence relation invariant under a group action, satisfying various additional assumptions. This allows us to prove general principles which imply the results mentioned above, as well as similar results in several different contexts in model theory and beyond, e.g.\ related to model-theoretic group components and compact group actions.
	
	Thus we extend a previous result of Kaplan and Miller and (independently) of mine and Krupiński about equivalence of smoothness and type-definability for certain $F_\sigma$ strong types (solving some open problems from earlier papers), as well as the theorem of Krupiński and Pillay about presenting the quotient of a definable group by its model-theoretic connected component as the quotient of a compact group by a subgroup.
	
	Furthermore, the obtained results bring new perspective on several open problems related to Borel cardinalities of strong types in model theory, and the methods developed both exploit and highlight the connections between model theory, topological dynamics and Banach space theory, extending previously known results in that area.
	
	\newpage
	\thispagestyle{empty}
	\section*{Thanks}
	{\setlength{\parskip}{1em}\setlength{\parindent}{0em}
		To my advisor, Krzysztof Krupiński, for his patience and invaluable help in polishing this thesis, as well as his enormous support (moral and mathematical) throughout my doctoral studies.
		
		To both Krzysztof and Anand Pillay for all the insights and the work they have done on our papers.
		
		To Itay Kaplan and the anonymous referees for helpful comments about my papers.
		
		To my colleagues in Wrocław, for the stimulating seminars, interesting discussions and all the great time we have had over the years.
		
		To the model theory group at the Hebrew University in Jerusalem, for providing a supportive environment in which I finished the thesis.
		
		To my mathematics teachers, in particular to Barbara Obremska and Grzegorz Lichwa, for giving me the opportunity to broaden my horizons and for encouraging me to pursue mathematical career.
		
		To all the friends and colleagues I did not mention.
		
		To my parents, for always being there for me.
		
		To Karolina, for everything.
		
		Thanks to all of you.
	}

	\microtypesetup{protrusion=false}
	
	\tableofcontents
	\microtypesetup{protrusion=true}

	\chapter{Introduction}
	
	\section{Strong types in model theory}
	Strong types originally arose from the study of forking, which is one of the most important notions in modern model theory.
	
	In his classification theory (see \cite{Sh90} for the second edition), Shelah introduced the notion of a \emph{strong type} of a tuple $a$ over a set $A$ (which, for $A=\emptyset$ corresponds to a single class of the relation $\equiv_\Sh$ defined in Definition~\ref{dfn:class_stp}), which turned out to be a central notion in the study of stable theories, as these strong types correspond exactly to types which have unique global nonforking extensions (see \cite[Corollary 2.9]{Sh90}).
	
	In his paper \cite{Las82}, Lascar introduced the notion of a model-theoretic Galois group (see Definition~\ref{definition: Galois groups}), as well as what is now called the Lascar strong type (see $\equiv_\Lasc$ in Definition~\ref{dfn:class_stp}). Loosely speaking, they were used to recover the theory of some $\omega$-categorical structures from the categories of their models (with elementary maps as morphisms).
	
	In stable theories, the Lascar strong types and the Shelah strong types coincide. In the more general class of simple theories, the Lascar strong types coincide with so-called Kim-Pillay strong types. Like the Shelah strong types in stable theories, they turned out to be useful in the study of simple theories (particularly for the general formulation of the independence theorem, which is one of the most important fundamental results in simplicity theory; see \cite[Corollary 10.9]{Cas11}).
	
	Furthermore, Lascar strong types also appear in the study of forking in generalisations of stability and simplicity (especially in NIP and NTP$_2$ theories, see e.g.\ \cite{BC14} and \cite[Proposition 2.1]{HP11}).
	
	In the context of definable groups, there is a theory of model-theoretic connected components, largely parallel to strong types, and playing an important role in the study of stable and NIP groups. The main results concerning connected components are related to the celebrated Pillay's conjecture (see \cite{peterzil_survey}).
	
	\section{History of the problem}
	The main problem tackled in this thesis is understanding the Galois groups and strong type spaces in arbitrary theories, and in particular, estimating their Borel cardinalities, and exploring the connection between descriptive-set-theoretic smoothness and model-theoretic type-definability of a strong type.
	
	It is well-known that the type-definable strong type spaces can be well understood as compact Hausdorff topological spaces (see Fact~\ref{fct:logic_top_cpct_T2}). If, in a given theory, the Lascar strong type $\equiv_\Lasc$ is type-definable, then the same is true about the Galois group, namely, it is a compact Hausdorff topological group. However, in general, the corresponding topology need not be Hausdorff, and in particular, the topology on the Galois group may be trivial.
	
	The paper \cite{CLPZ01} essentially began this line of study. There, the authors gave the first example of a theory where the Lascar strong type $\equiv_\Lasc$ is not type-definable. They suggested that in such cases, it would be prudent to treat the Galois group (and, by extension, the class spaces of $\equiv_\Lasc$) as ``descriptive set theoretic'' objects, and they asked about the possible ``Borel cardinality'' one may obtain in this way (see Definition~\ref{dfn:bier_borelcard} for precise definition). They suggested that when $\equiv_\KP$ and $\equiv_\Lasc$ differ (i.e.\ when the latter is not type-definable), this ``Borel cardinality'' should be nontrivial, which would mean that the class space of $\equiv_\Lasc$ is very complex.
	
	In \cite{Ne03}, it was shown that if for some tuple $a$ we have $[a]_{\equiv_\Lasc}\neq [a]_{\equiv_\KP}$, then the $\equiv_\KP$-class of $a$ splits into at least $2^{\aleph_0}$ $\equiv_\Lasc$-classes (see Fact~\ref{fct:newelski}), which supported that conjecture.
	
	Later, in \cite{KPS13}, the authors described precisely in what sense the Borel cardinality of $\Gal(T)$ is a well-defined invariant of the theory (see Definition~\ref{dfn:bcard_galois} for the precise definition), and similarly for the Borel cardinality of $\equiv_\Lasc$ (even restricted to a single $\equiv_\KP$-class; see Fact~\ref{fct:cartdf}). They also made a more precise conjecture about the Borel cardinality: they conjectured that if a $\equiv_\KP$-class is not a single $\equiv_\Lasc$-class, then the Borel cardinality of $\equiv_\Lasc$ (restricted to that $\equiv_\KP$-class) is non-smooth (in the sense of Definition~\ref{dfn:smt}).
	
	In \cite{KMS14}, the authors proved that this is indeed true (see Fact~\ref{fct:KMS_theorem}), showing that the Lascar strong type $\equiv_\Lasc$ is smooth (in the sense of Borel cardinality) if and only if it is type-definable. In a later paper \cite{KM14} and, independently, in \cite{KR16} (which was based on my master's thesis), the result was extended to arbitrary ``orbital $F_\sigma$ strong types" (see Fact~\ref{fct:mainA}).
	
	All of the definitions and results mentioned in the previous three paragraphs have their counterparts in the context of the model-theoretic group components.
	
	The methods of \cite{KMS14}, \cite{KM14} and \cite{KR16} were similar, but there seems to be no hope to extend them to arbitrary strong types (which are not $F_\sigma$). Moreover, they do not seem to be capable of giving any precise estimates of the Borel cardinalities of the Galois groups or strong types. In this thesis, we use completely different methods, developing and taking advantage of a deep topological dynamical apparatus with roots in \cite{KP17}, paired with the so-called Bourgain-Fremlin-Talagrand dichotomy from the theory of Rosenthal compacta.
	\section{Results}
	The main results of the thesis are essentially contained in three papers: \cite{KPR15} (joint with Krzysztof Krupiński and Anand Pillay), \cite{Rz16} (which was my own) and \cite{KR18} (joint with Krzysztof Krupiński).

	The essential contribution of this thesis, which did not appear in these papers (and is of my own conception) is the introduction of weakly uniformly properly group-like equivalence relations on an ambit. Using that notion, we redevelop the topological dynamical machinery based on \cite{KP17} (which was later refined in \cite{KPR15} and \cite{KR18}) in a much more general and abstract context. This allows us to prove the following theorem.
	\begin{mainthm}
		\label{mainthm:abstract_card}
		Suppose $E$ is weakly uniformly properly group-like, analytic equivalence relation on an ambit $(G,X,x_0)$ (where $X$ is an arbitrary compact Hausdorff space).
		
		Then $X/E$ is the topological quotient of a compact Hausdorff group by an analytic subgroup.
		
		We conclude that $E$ is either clopen (as a subset of $X^2$), or it has $2^{\aleph_0}$ classes.
		
		Moreover, if $E$ is not closed, then for every closed and $E$-invariant $Y\subseteq X$, $E\restr_Y$ has at least $2^{\aleph_0}$ classes.
	\end{mainthm}
	(See Lemma~\ref{lem:weakly_grouplike}, Lemma~\ref{lem:new_preservation_E_to_H}, Theorem~\ref{thm:general_cardinality_intransitive}, and Theorem~\ref{thm:general_cardinality_transitive} for precise statements.)
	
	In the metrisable case, we can obtain a stronger conclusion.
	\begin{mainthm}
		\label{mainthm:abstract_smt}
		Suppose $E$ is weakly uniformly properly group-like equivalence relation on an ambit $(G,X,x_0)$, where $X$ is a compact Polish space.
		
		Then $X/E$ is the topological quotient of a compact Polish group by a subgroup.
		
		Moreover, exactly one of the following holds:
		\begin{enumerate}
			\item
			$E$ is clopen and has finitely many classes,
			\item
			$E$ is closed and has exactly $2^{\aleph_0}$ classes,
			\item
			$E$ is not closed and not smooth. In this case, if $E$ is analytic, then $E$ has exactly $2^{\aleph_0}$ classes.
		\end{enumerate}
		In particular, $E$ is smooth (according to Definition~\ref{dfn:smt}) if and only if $E$ is closed.
	\end{mainthm}
	(See Theorem~\ref{thm:main_abstract} and Corollary~\ref{cor:metr_smt_cls})
	
	The main results of the thesis (which are also the main theorems of \cite{KPR15} and \cite{KR18}) can be deduced from the Main~Theorems~\ref{mainthm:abstract_card} and \ref{mainthm:abstract_smt}. The main advantage of the abstract formulation is that we can obtain similar results in many distinct contexts, which previously required careful repetitions of similar, but complicated arguments. In contrast, to apply Main~Theorems~\ref{mainthm:abstract_card} and \ref{mainthm:abstract_smt}, it is enough to check that that several basic axioms are satisfied in each case, which is relatively straightforward. Besides the other main theorems listed below, this allows us to recover (or even improve) virtually all similar results in model theory, in addition to providing corollaries in other contexts, occurring naturally in model theory. In Section~\ref{sec:other_apps}, we briefly discuss some examples, including the topological connected components of \cite{KP16} and the relative Galois groups of \cite{DKL17}.
	
	The principal result in the thesis is the following theorem. It is essentially Theorem 7.13 in \cite{KR18} (joint with Krzysztof Krupiński). Here, we deduce it from Main~Theorem~\ref{mainthm:abstract_smt} (or rather, the more precise statement in Theorem~\ref{thm:main_abstract}).
	\begin{mainthm}
		\label{mainthm_group_types}
		Given a countable (complete, first order) theory $T$, there is a compact Polish group $\hat G$ such that the Galois group of $T$ is the quotient of $\hat G$ by an $F_\sigma$ normal subgroup, as a topological group, and if the theory has NIP, in terms of Borel cardinality.
		
		Moreover, the space of classes of a bounded invariant equivalence relation $E$ defined on single complete type over $\emptyset$ (in a countable theory) is also a quotient of $\hat G$ by some subgroup (which inherits the good descriptive set theoretic properties of $E$), topologically, and under NIP, also in terms of Borel cardinality.
	\end{mainthm}
	(For precise statements, see Theorem~\ref{thm:main_galois} and Corollary~\ref{cor:galois_quotient}. See also Theorem~\ref{thm:main_aut} for a related fact with relaxed NIP assumption for the second part.)
	
	As a corollary, we obtain the following theorem, which essentially supersedes the main results of both  \cite{KMS14} and \cite{KM14}/\cite{KR16} (see Fact~\ref{fct:KMS_theorem} and Fact~\ref{fct:mainA}). It originally appeared as Corollary~4.2 and Corollary~6.1 in \cite{KPR15}, and is basically the main result of that paper (joint with Krzysztof Krupiński and Anand Pillay).
	\begin{mainthm}
		\label{mainthm:smt}
		Suppose that the theory is countable, while $E$ is a strong type, and $Y$ is type-definable, $E$-saturated, and such that $\Aut(\fC/\{Y\})$ acts transitively on $Y$ (e.g.\ $Y$ is the set of realisations of a single complete type over $\emptyset$, or a single Shelah or Kim-Pillay strong type). Then exactly one of the following is true:
		\begin{enumerate}
			\item
			$E\restr_Y$ is relatively definable (as a subset of $Y^2$) and has finitely many classes,
			\item
			$E\restr_Y$ is type-definable and has exactly $2^{\aleph_0}$ classes,
			\item
			$E\restr_Y$ is not type-definable and not smooth. In this case, if $E\restr_Y$ is analytic, then $E\restr_Y$ has exactly $2^{\aleph_0}$ classes.
		\end{enumerate}
		In particular, $E\restr_Y$ is smooth if and only if $E\restr_Y$ is type-definable. (And this is true even if $\Aut(\fC/\{Y\})$ does not act transitively on $Y$.)
	\end{mainthm}
	(This is Corollary~\ref{cor:trich_plus} and Corollary~\ref{cor:smt_type}.)
	
	If we do not assume that the theory is countable, the relevant spaces of types are not metrisable, and so Main~Theorem~\ref{mainthm:abstract_smt} does not apply. However, we can still apply Main~Theorem~\ref{mainthm:abstract_card}, yielding the following theorem, which generalises the main theorem of \cite{Ne03} (Fact~\ref{fct:newelski}). It originally appeared as Theorem 5.1 in \cite{KPR15}.
	\begin{mainthm}
		\label{mainthm:nwg}
		Suppose $E$ is an analytic strong type defined on $[a]_{\equiv}$, while $Y\subseteq [a]_{\equiv}$ is type-definable and $E$-saturated, such that $\lvert Y/E\rvert<2^{\aleph_0}$.
		
		Then $E$ is type-definable, and if, in addition, $\Aut(\fC/\{Y\})$ acts transitively on $Y/E$, then $E\restr_Y$ is relatively definable (as a subset of $Y^2$) and it has finitely many classes.
	\end{mainthm}
	(This is Theorem~\ref{thm:nwg}.)
	
	Besides Main~Theorems~\ref{mainthm_group_types}, \ref{mainthm:smt} and \ref{mainthm:nwg}, we recover analogous results for type-definable group actions, which also significantly improve the previous results from \cite{KM14} and \cite{KR16}. One of them is the following trichotomy, which supersedes the corresponding statements from \cite{Ne03} and \cite{KM14} (Fact~\ref{fct:new_group} and Fact~\ref{fct:KM_about_groups}). It appeared originally in \cite{KPR15} in the case when $G$ is a type-definable subgroup of a definable group (as Corollaries 5.4 and 6.2); for type-definable groups, this appeared as \cite[Corollary 8.6]{KR18} (under the assumption that the language is countable).
	\begin{mainthm}
		\label{mainthm:tdgroup}
		Suppose $G$ is a type-definable group, while $H\leq G$ is an analytic subgroup, invariant over a small set. Then exactly one of the following holds:
		\begin{itemize}
			\item
			$[G:H]$ is finite and $H$ is relatively definable in $G$,
			\item
			$[G:H]\geq 2^{\aleph_0}$, but is bounded, and $H$ is not relatively definable.
			\item
			$[G:H]$ is unbounded (i.e.\ not small).
		\end{itemize}
		In particular, $[G:H]$ cannot be infinite and smaller than $2^{\aleph_0}$.
		
		Moreover, in the second case, if the language is countable, and $G$ consists of countable tuples, then either $H$ is type-definable, or $G/H$ is not smooth.
	\end{mainthm}
	(This is Corollary~\ref{cor:trich_tdgroups}.)

	The final series of results comes from my own paper \cite{Rz16}, and is represented by the following theorem (which was originally \cite[Corollary 4.10]{Rz16}).
	\begin{mainthm}
		\label{mainthm_worb}
		Suppose $E$ is a strong type whose domain is a $\emptyset$-type-definable set $X$. Suppose, moreover, that $E$ is orbital, or, more generally, weakly orbital by type-definable. Then the following are equivalent:
		\begin{itemize}
			\item
			$E$ is type-definable,
			\item
			each class of $E$ is type-definable (equivalently, for every $p\in S(\emptyset)$ such that $p\vdash X$, $E\restr_{p(\fC)}$ is type-definable),
			\item
			$E$ is smooth.
		\end{itemize}
	\end{mainthm}
	(See Corollary~\ref{cor:smt_aut}.)
	
	The essential part of Main~Theorem~\ref{mainthm_worb} is the implication from type-definability of classes (a ``local" property) to the ``global" type-definability of the relation itself. The other implications are straightforward or follow from Main~Theorem~\ref{mainthm:smt}. When $X=p(\fC)$ for $p\in S(\emptyset)$, this is a simple exercise (see Proposition~\ref{prop:type-definability_of_relations}), but in general, it is not true. We show that the hypotheses of Main~Theorem~\ref{mainthm_worb} provide a general context in which the implication holds.
	
	Also, just like Main~Theorems~\ref{mainthm_group_types}, \ref{mainthm:smt} and \ref{mainthm:nwg}, Main~Theorem~\ref{mainthm_worb} has counterparts in different contexts, including type-definable group actions (see e.g.\ Corollary~\ref{cor:smt_def}). It is also the most general known description of the (sufficient) conditions under which the smoothness of a strong type implies its type-definability.
	
	Main~Theorem~\ref{mainthm_group_types} (at least under NIP assumption) provides a way to identify the Galois group, along with its Borel cardinality. Section~\ref{sec:examples} (which is the appendix of \cite{KR18}, expanded to provide more details) contains precise examples of such calculation. Namely, we determine the Galois group in the standard example of a non-G-compact theory from \cite{CLPZ01} and its modification from \cite{KPS13} (in both cases, the group and the Borel cardinality were given in \cite{KPS13}, but with very few details of the proof, and using different methods). In order to do that, we compute the Ellis groups associated with certain dynamical systems.
	
	For virtually all the results mentioned above, we show or deduce analogues which apply in the context of continuous actions of compact Hausdorff groups.
	
	Besides the main theorems mentioned above, in Chapter~\ref{chap:nonmetrisable_card}, we discuss the analogues of Main~Theorem~\ref{mainthm:abstract_smt} which provide a degree of ``non-smoothness'' in the non-metrisable/uncountable language case (giving more information than Main~Theorem~\ref{mainthm:abstract_card}). This is based on \cite{KPR15}, but put into the general context introduced here (the corresponding results of \cite{KPR15} are recovered). In Section~\ref{section: semigroup operation} (which is the appendix of \cite{KPR15}), we show that the stability of any given theory is equivalent to the existence of a canonical semigroup operation on a certain type space, associated with a monster model of that theory.
	
	\section{Structure of the thesis}
	Chapter~\ref{chap:prelims} contains the preliminaries, including basic facts and conventions. It is divided into the following parts:
	\begin{itemize}
		\item
		topology,
		\item
		descriptive set theory,
		\item
		topological dynamics,
		\item
		Rosenthal compacta and tame dynamical systems,
		\item
		model theory, and
		\item
		a short section containing the formal statements of previous results which we improve in the thesis.
	\end{itemize}
	They contain mostly known facts (published or folklore) and their straightforward generalisations. Nevertheless, for convenience of the reader, we recall complete proofs for many of them.
	
	Chapter~\ref{chap:toy} (partly based on Section~3 of \cite{KR18}) contains some basic examples coming from compact Hausdorff groups and their continuous actions on compact Hausdorff spaces, and the relatively easy model-theoretic case of $\Gal_\KP(T)$ and strong types coarser than $\equiv_\KP$. It is supposed to show some of the major ideas of the proofs of all the main theorems, while avoiding the need to use the difficult topological dynamical machinery, and other technical difficulties which are treated in the later chapters.
	
	In Chapter~\ref{chap:toolbox} (almost entirely based on Sections~4 and 5 of \cite{KR18}), we develop new tools in topological dynamics, and on the intersection of model theory and topological dynamics. Some of them are folklore, but many seem to be completely new.
	
	In Chapter~\ref{chap:grouplike} (which is new, but borrows many ideas from \cite{KP17} and \cite{KPR15}), we introduce and study and the notion of a group-like equivalence relation and its variants. In particular, we prove Main~Theorem~\ref{mainthm:abstract_card} and Main~Theorem~\ref{mainthm:abstract_smt}.
	
	In Chapter~\ref{chap:applications}, we specialise the results of Chapter~\ref{chap:grouplike} in various situations. In particular, we prove Main~Theorems~\ref{mainthm_group_types}, \ref{mainthm:smt}, \ref{mainthm:nwg}, and \ref{mainthm:tdgroup}. In Section~\ref{sec:examples}, we compute the Galois groups in a couple of examples by applying Main~Theorem~\ref{mainthm_group_types} (and computing certain Ellis groups).
	
	In Chapter~\ref{chap:intransitive}, we develop the notions of orbitality and weak orbitality in an abstract framework, and then apply them to prove Main~Theorem~\ref{mainthm_worb} (along with several related statements in various contexts).
	
	In Chapter~\ref{chap:nonmetrisable_card}, we discuss possible extensions of Main~Theorem~\ref{mainthm:abstract_smt} and (by extension) \ref{mainthm:smt} to non-metrisable dynamical systems (corresponding to uncountable languages in model theory), with the aim to obtain the equivalence between closedness and some sort of ``smoothness'' of an equivalence relation in such context. In particular, we pose Question~\ref{qu:broad_nonmetrisable} (the positive answer to which would give such an extension), and we show provide some partial results around it.
	
	In Appendix~\ref{app:topdyn}, we prove facts related to elementary topological dynamics which, while folklore, apparently cannot be found in the literature (in sufficient generality).
	
	Appendix~\ref{app:side} contains some tangential results which appeared in the course of the study. In particular, we give the criteria for the type space $S_{\bar c}(\fC)$ to have a natural left topological semigroup structure (namely, its existence is equivalent to stability), and --- using non-standard analysis --- we show that a closed group-like equivalence relation is always properly group-like (see Definitions~\ref{dfn:glike} and \ref{dfn:prop_glike}).

	\chapter{Preliminaries}
	\label{chap:prelims}
	Most facts in this chapter are classical or folklore. The few (apparent) exceptions are, for the most part, straightforward generalisations of well-known facts, some of which originate from \cite{KPR15} (joint with Krzysztof Krupiński and Anand Pillay).
	\section{Topology}
	
	\subsection*{Compact spaces and analytic sets; Baire property}
	In this thesis, compact spaces are not Hausdorff by definition, so we will add the adjective ``Hausdorff" whenever it is needed.
	
	\begin{fct}
		For a compact Hausdorff space $X$ the following conditions are equivalent:
		\begin{itemize}
			\item $X$ is second countable,
			\item $X$ is is metrisable,
			\item $X$ is Polish (i.e.\ separable and completely metrisable).
		\end{itemize}
	\end{fct}
	\begin{proof}
		It follows from \cite[Theorem 5.3]{Kec95}.
	\end{proof}
	
	\begin{fct}\label{fct: preservation of metrizability}
		Metrisability is preserved by continuous surjections between compact, Hausdorff spaces.
	\end{fct}
	\begin{proof}
		This follows from \cite[Theorem 4.4.15]{Eng89}.
	\end{proof}
	The  notion of a quotient map is one of the fundamental topological notions in this thesis.
	\begin{dfn}
		\index{quotient map}
		A surjection $f \colon X \to Y$ between topological spaces is said to be a {\em topological quotient map} if it has the property that a subset $A$ of $Y$ is closed if an only if $f^{-1}[A]$ is closed. (This is equivalent to saying that the induced bijection $X/E \to Y$ is a homeomorphism, where $E$ in the equivalence relation of lying in the same fibre of $f$ and $X/E$ is equipped with the quotient topology.)
		\xqed{\lozenge}
	\end{dfn}
	
	\begin{rem}
		In the definition of a quotient map, we can replace both instances of ``closed'' by ``open''. It is also easy to see that continuous open surjections and continuous closed surjections are always quotient maps, but in general, a quotient map need not be open nor closed.\xqed{\lozenge}
	\end{rem}
	
	The following simple observation will be rather useful.
	\begin{rem}
		\label{rem:commu_quot}
		Suppose we have a commutative triangle:
		\begin{center}
			\begin{tikzcd}
			A \ar[r] \ar[dr] & B\ar[d] \\
			& C
			\end{tikzcd}
		\end{center}
		where $A,B,C$ are topological spaces, and the horizontal arrow is a quotient map. Then if one of the other two arrows is a continuous or a quotient map, then so is the other one (respectively).\xqed{\lozenge}
	\end{rem}
	
	\begin{rem}\label{rem: continuous surjection is closed}
		A continuous map from a compact space to a Hausdorff space is closed. In particular, if it is onto, it is a quotient topological map.
		\xqed{\lozenge}
	\end{rem}
	
	\begin{fct}
		\label{fct:quot_T2_iff_closed}
		If $X$ is a compact Hausdorff topological space and $E$ is an equivalence relation on $X$, then $E$ is closed (as a subset of $X^2$) if and only if $X/E$ is a Hausdorff space, and $E$ is open if and only if $X/E$ is discrete (and in this case, $X/E$ is finite).
	\end{fct}
	\begin{proof}
		For the first part, this is \cite[Theorem 3.2.11]{Eng89}. The second part is easy by compactness: if $E$ is open, it has open classes, so points in $X/E$ are open. On the other hand, if $X/E$ is discrete, then it must be finite (as a discrete compact space), so $E$ is open (as a finite union of open rectangles).
	\end{proof}
	
	\begin{dfn}
		\index{Souslin scheme}
		\index{Souslin operation}
		Recall that a {\em Souslin scheme} is a family $(P_s)_{s \in \omega^{<\omega}}$ of subsets of a given set, indexed by finite sequences of natural numbers. The {\em Souslin operation} $\Souslin$ applied to such a scheme produces the set
		\[
		\Souslin_s P_s:=\bigcup_{s \in \omega^\omega} \bigcap_n P_{s\restr_n}.
		\]
		
		We say that a Souslin scheme $(P_s)_{s \in \omega^{<\omega}}$ is {\em regular} if $s \subseteq t$ implies $P_s \supseteq P_t$.\xqed{\lozenge}
	\end{dfn}
	
	There seems to be no established notion of an ``analytic set" in an arbitrary topological space. The following one will be the most convenient for us.
	
	\begin{dfn}\label{definition: analytic sets}
		\index{analytic set}
		In a topological space $X$, we call a subset of $X$ \emph{analytic} if it can be obtained via the $\Souslin$ operation applied to a Souslin scheme of closed sets.\xqed{\lozenge}
	\end{dfn}
	
	\begin{rem}
		We will mostly consider analytic sets in compact Hausdorff spaces. There, the definition above coincides with the classical notion of a $K$-analytic set, see \cite[Théorème 1]{Cho59}.\xqed{\lozenge}
	\end{rem}
	
	\begin{rem}
		What we really need of the class of ``analytic sets" is the following:
		\begin{itemize}
			\item
			if $A$ is analytic in $X$ and $Y\subseteq X$ is a closed subspace, then $A\cap Y$ is analytic in $Y$,
			\item
			if $A$ is analytic in $Y$ and $Y$ is a closed subspace of $X$, then $A$ is analytic in $X$,
			\item
			if $A$ is analytic (in a compact Hausdorff space), it has the Baire property (see below),
			\item
			if $f\colon X\to Y$ is a continuous surjection and $X,Y$ are compact Hausdorff, then for every $A\subseteq Y$, we have that $A$ is analytic if and only if $f^{-1}[A]$ is analytic.
		\end{itemize}
		Any notion of an ``analytic set" with these properties will also work.\xqed{\lozenge}
	\end{rem}
	
	\begin{rem}
		It is easy to check that if
		$(P_s)_{s \in \omega^{<\omega}}$ is a Souslin scheme and $Q_s:= \bigcap_{s \subseteq t} P_t$, then $(Q_s)_{s \in \omega^{<\omega}}$ is regular and $\Souslin_s P_s =\Souslin_s Q_s$.

		In particular, in the definition of an analytic set, we can consider only regular Souslin schemes.\xqed{\lozenge}
	\end{rem}
	
	\begin{rem}
		If $X$ is a Polish space, then this definition coincides with the standard definition of analytic sets as continuous images of Borel sets (see \cite[Theorem 25.7]{Kec95}). In particular, all Borel sets are analytic.\xqed{\lozenge}
	\end{rem}

	\begin{dfn}
		Suppose $X$ is a topological space and $B\subseteq X$.
		
		\index{Baire!property}
		\index{Baire!property!strict}
		We say that $B$ has the \emph{Baire property} (BP) or that it is \emph{Baire} if there is an open set $U$ and a meagre set $M$ such that $B$ is the symmetric difference of $U$ and $M$.
		
		We say that $B$ has the \emph{strict Baire property} or that it is \emph{strictly Baire} if for every closed $F\subseteq X$, $F\cap B$ has BP in $F$. (This is equivalent to saying that the same holds for all $F$, not necessarily closed, see \cite[§11 VI.]{Ku}.)\xqed{\lozenge}
	\end{dfn}
	
	\begin{fct}
		The sets with the Baire property form a $\sigma$-algebra closed under the $\Souslin$ operation. In particular, every Borel set and every analytic set is strictly Baire.
	\end{fct}
	\begin{proof}
		See \cite[Theorem 25.3]{Arh}.
	\end{proof}
	
	\begin{dfn}
		\index{totally non-meagre}
		We say that a topological space $X$ is totally non-meagre if no closed subset of $X$ is meagre in itself.\xqed{\lozenge}
	\end{dfn}
	
	\begin{rem}
		It is easy to see that every compact Hausdorff space and every Polish space is totally nonmeagre, by the Baire category theorem. \xqed{\lozenge}
	\end{rem}

	\begin{prop}\label{prop: image of intersection}
		Assume that $X$ is a compact (not necessarily Hausdorff) space and that $Y$ is a $T_1$-space. Let $f\colon X \to Y$ be a continuous map. Suppose $(F_n)_{n\in \omega}$ is descending sequence of closed subsets of $X$. Then $f[\bigcap_n F_n]=\bigcap_n f[F_n]$.
	\end{prop}
	\begin{proof}
		The inclusion $(\subseteq)$ is always true. For the opposite inclusion, consider any $y \in \bigcap_n f[F_n]$. Then $f^{-1}(y) \cap F_n \ne \emptyset$ for all $n$. Since $(F_n)_{n\in \omega}$ is descending, we get that the family $\{f^{-1}(y) \cap F_n\mid n \in \omega\}$ has the finite intersection property. On the other hand, since $\{y\}$ is closed in $Y$ (as $Y$ is $T_1$) and $f$ is continuous, we have that each set $f^{-1}(y) \cap F_n$ is closed. So compactness of $X$ implies that $f^{-1}(y) \cap \bigcap_n F_n=\bigcap_n f^{-1}(y) \cap F_n \ne \emptyset$. Thus $y \in f[\bigcap_n F_n]$.
	\end{proof}
	
	\begin{prop}\label{prop: preservation of analyticity by images and preimages}
		Let $f\colon X \to Y$ be a continuous map between topological spaces. Then:
		\begin{enumerate}
			\item The preimage by $f$ of any analytic subset of $Y$ is an analytic subset of $X$.
			\item Assume that $X$ is compact (not necessarily Hausdorff) and that $Y$ is Hausdorff. Then the image by $f$ of any analytic subset of $X$ is an analytic subset of $Y$.
		\end{enumerate}
	\end{prop}
	\begin{proof}
		\ref{it:prop:dyn_BFT:untame} is clear by continuity of $f$ and the fact that preimages preserve unions and intersections.
		
		To show (2), consider any analytic subset $A$ of $X$. Then $A=\bigcup_{s \in \omega^\omega} \bigcap_n F_{s\restr_n}$ for some regular Souslin scheme $(F_s)_{s \in \omega^{<\omega}}$ of closed subsets of $X$. Because $X$ is compact, $Y$ is Hausdorff and $f$ is continuous, we see that each set $f[F_s]$ is closed. By Proposition~\ref{prop: image of intersection},
		\[
			f[X] = \bigcup_{s \in \omega^\omega} \bigcap_n f[F_{s\restr_n}].
		\]
		Hence, $f[X]$ is analytic.
	\end{proof}
	The following proposition summarises various preservation properties of continuous surjections between compact Hausdorff spaces.
	\begin{prop}
		\label{prop:preservation_properties}
		Suppose $f\colon X\to Y$ is a continuous surjection between compact Hausdorff spaces. Then:
		\begin{itemize}
			\item
			preimages and images of closed sets by $f$ are closed,
			\item
			preimages and images of $F_\sigma$ sets by $f$ are $F_\sigma$
			\item
			preimages and images of analytic sets by $f$ are analytic,
			\item
			preimages of Borel sets by $f$ are Borel, and sets with Borel preimage are Borel.
		\end{itemize}
		
		Furthermore, for every $Y_0\subseteq Y_1\subseteq Y$, $Y_0$ is open or closed in $Y_1$ if and only if $X_0:=f^{-1}[Y_0]$ is open or closed (respectively) in $X_1:=f^{-1}[Y_1]$.
	\end{prop}
	\begin{proof}
		For analytic sets, this follows from Proposition~\ref{prop: preservation of analyticity by images and preimages}. For closed sets, this follows from Remark~\ref{rem: continuous surjection is closed}. For $F_\sigma$ sets, this follows from Remark~\ref{rem: continuous surjection is closed} and Proposition~\ref{prop: image of intersection}. For Borel sets, it follows from Fact~\ref{fct:borel_preimage}.
		
		For the ``furthermore'' part, just note that since $f$ is continuous and closed and $X_1=f^{-1}[Y_1]$, the restriction $f\restr_{X_1}\colon X_1\to Y_1$ is also continuous and closed (and hence a quotient map), which completes the proof.
	\end{proof}

	\begin{prop}[Mycielski's theorem]
		\label{prop:mycielski}
		Suppose $E$ is a meagre equivalence relation on a locally compact, Hausdorff space $X$. Then $\lvert X/E\rvert \geq 2^{\aleph_0}$.
	\end{prop}
	\begin{proof}
		The proof mimics that of the classical theorem for Polish spaces (for example see \cite[Theorem 5.3.1]{SuG}), except we use compactness instead of metric completeness to obtain a nonempty intersection.
		
		Firstly, we can assume without loss of generality that $X$ is compact. This is because we can restrict our attention to the closure $\overline U$ of a small open set $U$: $E$ restricted to $\overline U$ is still meagre, and if we show that $\overline U/E$ has the cardinality of at least the continuum, clearly the same will hold for $X/E$.

		Suppose $E\subseteq \bigcup_{n\in \omega} F_n$ with $F_n\subseteq X^2$ closed, nowhere dense. We can assume
		that the sets $F_n$ form an increasing sequence. We will define a family of nonempty open sets $U_s$ with $s\in 2^{<\omega}$, recursively with respect to the length of $s$, such that:
		\begin{itemize}
			\item
			$\overline{U_{s0}},\overline{U_{s1}}\subseteq U_s$,
			\item
			if $s\neq t$ and $s,t\in 2^{n+1}$, then $(U_s\times U_t)\cap F_n=\emptyset$.
		\end{itemize}
		
		Then, by compactness, for each $\eta\in 2^\omega$ we will find a point $x_\eta\in \bigcap_n U_{\eta\restr n}$. It is easy to see that this will yield a map from $2^\omega$ into $X$ such that any two distinct points are mapped to $E$-unrelated points.
		
		The construction can be performed as follows:
		\begin{enumerate}
			\item
			For $s=\emptyset$, we put $U_\emptyset=X$.
			\item
			Suppose we already have $U_s$ for all $\lvert s\rvert \leq n$, satisfying the assumptions.
			\item
			By compactness (more precisely, regularity), for each $s\in 2^n$ and $i\in\{0,1\}$ we can find a nonempty open set $U_{si}'$ such that $\overline{U_{si}'}\subseteq U_s$.
			\item
			For each (ordered) pair of distinct $\sigma,\tau\in 2^{n+1}$, the set $(U'_\sigma\times U'_\tau)\setminus F_n$ is a nonempty open set (because $F_n$ is closed, nowhere dense), so in particular, $U'_\sigma\times U'_\tau$ contains a smaller (nonempty, open) rectangle $U''_\sigma\times U''_\tau$ which is disjoint from $F_n$.
			\item
			Repeating the procedure from the previous point recursively, for each ordered pair $(\sigma,\tau)$, we obtain for each $\sigma\in 2^{n+1}$ a nonempty open set $U_\sigma\subseteq U_\sigma'$ such that for $\sigma\neq \tau$ we have $(U_\sigma\times U_\tau)\cap F_n=\emptyset$. It is easy to see that the sets $U_\sigma$ satisfy the inductive step for $n+1$.\qedhere
		\end{enumerate}
	\end{proof}

	\subsection*{Topological groups and continuous group actions}
	
	\begin{dfn}
		Given a group $G$ acting on a set $X$, an \emph{orbit map} is a map $G\to X$ of the form $g\mapsto g\cdot x$ for some $x\in X$.\xqed{\lozenge}
	\end{dfn}
	
	\begin{dfn}
		If $G$ acts on sets $X$ and $Y$, then a function $f\colon X\to Y$ is called \emph{$G$-equivariant} (or a \emph{$G$-map}) if for every $x\in X$ we have $gf(x)=f(gx)$.\xqed{\lozenge}
	\end{dfn}

	\begin{fct}[Pettis-Pickard theorem]\label{fct:pettis}
		Let $G$ be a topological group. If $A\subseteq G$ has the Baire property (e.g.\ it is analytic) and is non-meagre, the set $A^{-1}A:=\{a^{-1}b\mid a,b \in A\}$ contains an open neighbourhood of the identity. In particular, if $A$ is a subgroup of $G$, then $A$ is open.
	\end{fct}
	\begin{proof}
		This is \cite[Theorem 9.9]{Kec95}.
	\end{proof}

	\begin{fct}
		\label{fct:multiplication_open}
		Suppose $G$ is a topological group. Then the multiplication map $\mu\colon G\times G\to G$, $\mu(g_1,g_2)=g_1g_2$ and the map $\mu'\colon G\times G\to G$, $\mu'(g_1,g_2)=g_1^{-1}g_2$ are both continuous and open. In particular, they are topological quotient maps.
	\end{fct}
	\begin{proof}
		Continuity is immediate. For openness, note that if $A\subseteq G\times G$ is open, then $\mu[A]=\bigcup_{g\in G} gA_g$, where $A_g$ is the section of $A$ at $g$. Since open sets have open sections, the conclusion follows. Openness of $\mu'$ is analogous.
	\end{proof}
	
	\begin{fct}\label{fct:from_mycielski}
		Suppose $G$ is a locally compact, Hausdorff group and $H$ is a subgroup which has the Baire property, but is not open. Then $[G:H]\geq 2^{\aleph_0}$.
	\end{fct}
	\begin{proof}
		It follows from Fact~\ref{fct:pettis} that a non-meagre Baire subgroup of a topological group is open, so, in our case, $H$ is meagre. By Fact~\ref{fct:multiplication_open}, we have that the orbit equivalence relation of $H$ acting by left translations on $G$ is meagre (the preimage of a meagre set by an open continuous map is meagre). But then by Proposition~\ref{prop:mycielski}, it follows that $\lvert G/H\rvert\geq 2^{\aleph_0}$.
	\end{proof}

	In the thesis, coset equivalence relations appear very often, so the simple observation made in the following remark is very useful.
	\begin{rem}
		\label{rem:group_to_cosets}
		Note that if $G$ is a group and $H\leq G$, then the left coset equivalence relation $E_H$ of $H$ on $G$ is the preimage of $H$ by the map $\mu'$ from Fact~\ref{fct:multiplication_open}. In particular, if $G$ is a compact Hausdorff topological group, we can apply Proposition~\ref{prop:preservation_properties} and Fact~\ref{fct:multiplication_open} to show that $H$ and $E_H$ share good topological properties.
		\xqed{\lozenge}
	\end{rem}
	
	\begin{fct}
		\label{fct:quotient_by_closed_subgroup}
		Suppose $G$ is a (possibly non-Hausdorff) topological group and $H\leq G$ is a subgroup. Then $G/H$ is Hausdorff (with the quotient topology) if and only if $H$ is closed, and $G/H$ is discrete if and only if $H$ is open.
	\end{fct}
	\begin{proof}
		For the closed-Hausdorff correspondence, see \cite[III.2.5, Proposition 13]{NB66}.
		
		For open-discrete, just note that $H$ is open if and only if all of its cosets are open, which is the same as $G/H$ being discrete.
	\end{proof}

	While topological groups need not be Hausdorff, it is well-known that a $T_0$ topological group is completely regular Hausdorff. A related fact is that they (and their quotients) are $R_1$ spaces. First, recall the notion of a Kolmogorov quotient of a topological space.
	
	\begin{dfn}
		\label{dfn:kolmogorov}
		\index{Kolmogorov quotient}
		If $X$ is a topological space, then the \emph{Kolmogorov quotient} of $X$ is the quotient obtained by identifying topologically indistinguishable points, i.e.\ $x_1$ and $x_2$ are identified if the closures of $\{x_1\}$ and $\{x_2\}$ are equal.\xqed{\lozenge}
	\end{dfn}

	\begin{dfn}
		\label{dfn:R0_R_1}
		\index{space!R0@$R_0$}
		We say that a topological space $X$ is an \emph{$R_0$ space} if for every $x_1,x_2\in X$ we have that $x_1\in \overline{x_2}$ if and only if $x_2\in \overline{x_1}$, or equivalently, the Kolmogorov quotient of $X$ is a $T_1$ space.
		
		\index{space!R1@$R_1$}
		A topological space $X$ is an \emph{$R_1$ space} if its Kolmogorov quotient is Hausdorff.\xqed{\lozenge}
	\end{dfn}
	(Note that in particular, every $R_1$ space is $R_0$.)
	
	\begin{prop}
		\label{prop:top_gp_R1}
		Suppose $G$ is a topological group and $H\leq G$. Then $G/H$ is an $R_1$ space.
	\end{prop}
	\begin{proof}
		Note that $\overline H$ is a subgroup of $G$ (as the closure of a subgroup of a topological group). Since $G$ acts on itself by homeomorphisms, the closure of every $gH$ is $g\overline H=g\overline{H}H$, so $g'H\in \overline{\{gH\}}\subseteq G/H$ if and only if $g\overline H=g'\overline H$, and hence, $\overline{\{g'H\}}=\overline{\{gH\}}$ if and only if $g\overline H=g'\overline H$.
		
		It follows that the Kolmogorov quotient of $G/H$ is naturally homeomorphic to $G/\overline H$, which is Hausdorff by Fact~\ref{fct:quotient_by_closed_subgroup}.
	\end{proof}

	\begin{fct}
		\label{fct:cpct_action}
		Suppose $G$ is a compact Hausdorff group acting continuously on a Hausdorff space $X$. Then:
		\begin{itemize}
			\item
			$X/G$ is Hausdorff
			\item
			$G\times X\to X$ is a closed map (i.e.\ images of closed sets are closed).
			\item
			$X\to X/G$ is a closed map.
		\end{itemize}
	\end{fct}
	\begin{proof}
		See \cite[Theorems 1.2 and 3.1]{GB72}.
	\end{proof}

	There are several different notions of a ``proper map". One that will be convenient for us is the one used in \cite{NB66}.
	
	\begin{dfn}
		\index{proper map}
		A continuous map $f\colon X\to Y$ is said to be \emph{proper} if for every $Z$, the map $f\times \id_Z\colon X\times Z\to Y\times Z$ is closed.\xqed{\lozenge}
	\end{dfn}

	\begin{fct}
		\label{fct:proper}
		Suppose $f\colon X\to Y$ is continuous. Then the following are equivalent:
		\begin{itemize}
			\item
			$f$ is proper,
			\item
			$f$ is closed and its fibres (preimages of points) are compact.
		\end{itemize}
	\end{fct}
	\begin{proof}
		See \cite[Theorem 1 in \S10.2 of Chapter I]{NB66}.
	\end{proof}
	
	An important fact in this context is that continuous actions of compact groups are proper.
	\begin{fct}
		\label{fct:cpct_proper}
		If $G$ is a compact Hausdorff group acting continuously on a Hausdorff space $X$, then it is also acting properly, that is, the function $G\times X\to X\times X$, defined by the formula $(g,x)\mapsto (x,g\cdot x)$ is proper.
	\end{fct}
	\begin{proof}
		See \cite[Proposition 2 in \S4.1 of Chapter III]{NB66}.
	\end{proof}
	Recall that homeomorphisms of a compact Polish space form a Polish group with the uniform convergence topology (which is unique in a compact space).
	\begin{fct}
		\label{fct:homeo_is_Polish}
		If $X$ is a compact Polish space, then the group $\Homeo(X)$ of homeomorphisms of $X$, equipped with the uniform convergence topology, is a Polish group (but not compact in general).
	\end{fct}
	\begin{proof}
		See \cite[9.B(8)]{Kec95}.
	\end{proof}

	\begin{prop}
		\label{prop:cont_action_factors_through_Polish}
		Suppose $G$ is a compact Hausdorff group acting transitively and (jointly) continuously on a Polish space $X$. Then for any $x_0\in X$, if we denote by $H$ the stabiliser of $x_0$, then $G/\Core(H)$ is a compact Polish group (where $\Core(H)$ is the normal core of $H$, i.e.\ intersection of all of its conjugates) and the action of $G$ on $X$ factors through $G/\Core(H)$.
		
		In particular, if $G$ is a compact Hausdorff topological group and $H\leq G$ is such that $G/H$ is metrisable, then $G/\Core(H)$ is a compact Polish group (since $H$ is the stabiliser of $H$ for the natural left action of $G$ on $G/H$).
	\end{prop}
	\begin{proof}
		First, note that since $G$ acts transitively on $X$, the orbit map $g\mapsto g\cdot x_0$ is onto. By continuity of the action, it is also continuous, so $X$ is a compact Polish space.
		
		Let $\varphi \colon G\to \Homeo(X)$ be the homomorphism induced by the action, where $\Homeo(X)$ is the group of homeomorphisms of $X$.
		
		Since $G\times X$ is compact, continuity of the action of $G$ on $X$ implies uniform continuity (see \cite[Theorem II in \S4.2 of Chapter II]{NB66}). Therefore, if $(g_i)_i$ is a convergent net, then $(g_i\cdot x)_i$ converges uniformly in $x \in X$. This yields continuity of $\varphi$ with respect to the uniform convergence topology on $\Homeo(X)$.
		
		It is easy to check that $\ker(\varphi)=\Core(H)$ (which implies that the action of $G$ on $X$ factors through $G/\Core(H)$), and since $X$ is a compact Polish space, $\Homeo(X)$ is a Polish group (by Fact~\ref{fct:homeo_is_Polish}). By compactness of $G$, it follows that $\varphi[G]$ is a Polish group, and hence --- by Remark \ref{rem: continuous surjection is closed} --- so is $G/\Core(H)$.
	\end{proof}
	
	\section{Descriptive set theory}
	
	\begin{dfn}
		\label{dfn:borel_cardinality}
		\index{Borel!reduction}
		\index{Borel!reducible}
		Suppose $E$ and $F$ are equivalence relations on Polish spaces $X$ and $Y$. We say that \emph{$E$ is Borel reducible to $F$} --- written $E\leq_B F$ --- if there is a Borel reduction of $E$ to $F$, i.e.\ a Borel function $f\colon X\to Y$ such that $x_1\Er x_2$ if and only if $f(x_1)\mathrel{F} f(x_2)$.
		
		\index{Borel!equivalence}
		\index{Borel!bireducibility|see {Borel equivalence}}
		\index{Borel!cardinality}
		If $E\leq_B F$ and $F\leq_B E$, we say that $E$ and $F$ are \emph{Borel equivalent} or \emph{Borel bireducible}, written $E\sim_B F$. In this case, we also say that $E$ and $F$, or, abusing the notation, $X/E$ and $Y/F$, have the same {\em Borel cardinality}; informally speaking, the {\em Borel cardinality} of $E$ is its $\sim_B$-equivalence class.
		\xqed{\lozenge}
	\end{dfn}
	
	\begin{rem}
		Sometimes, we slightly abuse the notation and write e.g.\ $X/E\leq_B Y/F$ for $E\leq_B F$. In particular, if $G$ is a Polish group and $H\leq G$, then by $G/H\leq_B Y/F$ we mean that there is a Borel reduction from the left coset equivalence relation of $H$ on $G$ to the relation $F$.
		\xqed{\lozenge}
	\end{rem}
	
	\begin{fct}
		\label{fct:borel_isom}
		If $X$ and $Y$ are Polish spaces and $\lvert X\rvert \leq \lvert Y\rvert $, then there is a Borel embedding of $X$ into $Y$. In particular, the equality on $X$ is Borel reducible to equality on $Y$.
	\end{fct}
	\begin{proof}
		This is \cite[Theorem 15.6]{Kec95}.
	\end{proof}

	\begin{rem}
		Note that even if $E\sim_B F$, it may not be true that there is a Borel isomorphism of $X$ and $Y$ which is a reduction in both direction. For example, if $E$ and $F$ are total, then trivially $E\sim_B F$, but $X$ and $Y$ may have different cardinalities.
		\xqed{\lozenge}
	\end{rem}

	\begin{dfn}
		\index{smoothness}
		\index{smooth equivalence relation|see {smoothness}}
		\label{dfn:smt}
		We say that an equivalence relation $E$ on a Polish space $X$ (or the quotient $X/E$) is {\em smooth} if $E$ is Borel reducible to equality on $2^\bN$ (or, by Fact~\ref{fct:borel_isom}, equivalently, if it is Borel reducible to equality on some Polish space).
		\xqed{\lozenge}
	\end{dfn}

	Note that a Borel reduction of $E$ to $F$ yields an injection $X/E\to Y/F$, so if $X/E\leq_B Y/F$, then in particular, $\lvert X/E\rvert\leq \lvert Y/F\rvert$. On the other hand, if $E$ and $F$ are Borel and have countably many classes, it is easy to see that the converse is also true, i.e.\ $E\leq_B F$ if and only if $\lvert X/E\rvert\leq \lvert Y/F\rvert$. This, together with Fact~\ref{fct:borel_isom}, justifies the term ``Borel cardinality".
	
	Informally, we can think of $E\leq F$ as a statement that we can classify elements of $X$ up to $E$ using classes of $F$ as parameters. In particular, smooth equivalence relations can be classified by real numbers, which is why they are sometimes called ``classifiable equivalence relations''.

	\begin{rem}
		\label{rem:smooth_implies_borel}
		Note that in the definition of a Borel reduction and the Borel cardinality, we do not require the relations to be Borel. On the other hand, it is not hard to see that if $E\leq_B F$ and $F$ is Borel, then so is $E$, by simply considering the map $X\times X\to Y\times Y$ which is the square of the reduction of $E$ to $F$.
		
		In particular, smooth equivalence relations (in the sense just given) are all Borel.
		\xqed{\lozenge}
	\end{rem}
	
	\begin{fct}
		\label{fct:clsd_smth}
		Every closed (or, more generally, every $G_\delta$) equivalence relation on a Polish space is smooth.
	\end{fct}
	\begin{proof}
		See \cite[Corollary 1.2]{HKL90}.
	\end{proof}

	\begin{fct}
		\label{fct:borel_section}
		If $X,Y$ are compact Polish spaces and $f\colon X\to Y$ is a continuous surjection, then $f$ has a Borel section $g$.
		
		In particular, if $f$ is a reduction from $E$ on $X$ to $F$ on $Y$, then $g$ is a Borel reduction from $F$ to $E$, whence $E\sim_B F$.
	\end{fct}
	\begin{proof}
		The first part is \cite[Exercise 24.20]{Kec95}. The second is an immediate consequence of the first and the definition.
	\end{proof}
	
	Fact~\ref{fct:borel_section} immediately implies that a set $A$ in $Y$ with Borel preimage $B$ in $X$ is Borel (because $A$ is the preimage of $B$ by any section of $f$). However, a more general fact is also true.
	
	\begin{fct}
		\label{fct:borel_preimage}
		Suppose $X$ and $Y$ are compact Hausdorff topological spaces, and $f\colon X\to Y$ is a closed surjection, while $B\subseteq Y$. Then $B$ is Borel in $Y$ if and only if $f^{-1}[B]$ is Borel in $X$.
	\end{fct}
	\begin{proof}
		It is clear that if $B$ is Borel, then so is $f^{-1}[B]$
		
		The converse follows from \cite[Theorem 10]{HS03} (because Borel sets are exactly the sets obtained from open and closed sets by a sequence of operations consisting of taking complements and countable intersections, which are descriptive operations in the sense of \cite{HS03}).
	\end{proof}

	The following two dichotomies are fundamental in the theory of Borel equivalence relations.
	
	\begin{fct}[Silver dichotomy]
		\label{fct:silver}
		For every Borel (even coanalytic) equivalence relation $E$ on a Polish space either $E \leq_B \Delta_{\bN}$, or $\Delta_{2^{\bN}} \leq_B E$. (So in particular, every non-smooth Borel equivalence relation on a Polish space has $2^{\aleph_0}$ classes)
	\end{fct}
	\begin{proof}
		See \cite[Theorem 10.1.1]{kanovei}.
	\end{proof}
	
	By $\EZ$ we denote the equivalence relation of eventual equality on $2^{\bN}$ (i.e.\ for $\eta,\eta'\in 2^{\bN}$ we have $\eta\EZ \eta'$ when $\eta(n)=\eta'(n)$ for all but finitely many $n$).
	
	\begin{fct}[Harrington-Kechris-Louveau dichotomy]\label{fct:Harrington-Kechris-Louveau dichotomy}
		For every Borel equivalence relation $E$ on a Polish space $X$ either $E \leq_B \Delta_{2^{\bN}}$ (i.e.\ $E$ is smooth), or ${\EZ} \leq_B E$. In the latter case, the reduction is realised by a homeomorphic embedding of $2^{\bN}$ into $X$.
	\end{fct}
	\begin{proof}
		See \cite[Theorem 1.1]{HKL90} or \cite[Theorem 10.4.1]{kanovei}.
	\end{proof}
	
	\index{E-saturated@$E$-saturated set}
	\index{E-invariant@$E$-invariant set|see{$E$-saturated set}}
	Recall that for an equivalence relation $E$ on a set $X$, a subset $Y$ of $X$ is said to be {\em $E$-saturated} if it is a union of some classes of $E$. In this thesis, we will say that a family $\{B_i \mid i \in \omega\}$ of subsets of $X$ {\em separates classes} of $E$ if for every $x \in X$, $[x]_E= \bigcap \{ B_i \mid x \in B_i\}$. Note that this implies that all $B_i$ are $E$-saturated. Thus, a family $\{B_i \mid i \in \omega\}$ of subsets of $X$ separates classes of $E$ if and only if each $B_i$ is $E$-saturated and each class of $E$ is the intersection of those sets $B_i$ which contain it.
	The following characterisation of smoothness is folklore.
	
	\begin{fct}\label{fct:separating_family}
		Let $X$ be an equivalence relation on a Polish space $X$. Then, $E$ is smooth if and only if there is a countable family $\{B_i\mid i \in \omega\}$ of Borel ($E$-saturated) subsets of $X$ separating classes of $E$.
	\end{fct}
	\begin{proof}
		Let $f$ be a Borel reduction of $E$ to $\Delta_{2^{\bN}}$. Let $\{C_i\mid i \in \omega\}$ be a countable open basis of the space $2^{\bN}$. Then $\{f^{-1}[C_i]\mid i \in \omega \}$ is a countable family consisting of Borel ($E$-saturated) subsets of $X$ separating classes of $E$.
		
		For the converse, consider a family $\{ B_i \mid i \in \bN\}$ satisfying the hypothesis. Define $f \colon X \to 2^{\bN}$ by $f(x) = \chi_{\{i \in \bN \mid x \in B_i\}}$ (i.e.\ the characteristic function of ${\{i \in \bN \mid x \in B_i\}}$). It is easy to see that this is a Borel reduction of $E$ to $\Delta_{2^{\bN}}$.
	\end{proof}

	The following Fact provides a useful criterion of closedness for subgroups of totally nonmeagre topological groups (including compact Hausdorff groups).
	\begin{fct}\label{fct:miller}
		Assume $G$ is a totally non-meagre topological group (e.g.\ $G$ is Polish or locally compact). Suppose $H$ is a subgroup of $G$ and $\{E_i \mid i \in \omega\}$ is a collection of right $H$-invariant (i.e.\ such that $E_iH=E_i$), strictly Baire sets which separates left $H$-cosets (i.e.\ for each $g \in G$, $gH= \bigcap \{ E_i \mid g \in E_i\}$). Then $H$ is closed in $G$.
	\end{fct}
	\begin{proof}
		This is \cite[Theorem 1]{Mil77}.
	\end{proof}
	
	Using Fact~\ref{fct:miller}, we can prove the following proposition, which is one of the more important tools in the proofs of Main~Theorem~\ref{mainthm:smt} and Main~Theorem~\ref{mainthm:tdgroup}.
	\begin{prop}
		\label{prop:trichotomy_for_groups}
		Suppose $G$ is a compact Polish group and $H\leq G$. Then exactly one of the following holds:
		\begin{itemize}
			\item
			$[G:H]$ is finite and $H$ is open,
			\item
			$[G:H]=2^{\aleph_0}$ and $H$ is closed,
			\item
			$H$ has the property of Baire, $G/H$ is non-smooth (in the sense that the coset equivalence relation is non-smooth) and $[G:H]=2^{\aleph_0}$,
			\item
			$H$ does not have the property of Baire (and $G/H$ is non-smooth).
		\end{itemize}
	\end{prop}
	\begin{proof}
		If $H$ does not have the Baire property, then it is not Borel, and (by Remark~\ref{rem:group_to_cosets} and Proposition~\ref{prop:preservation_properties}) neither is the coset equivalence relation, so $G/H$ cannot be smooth (see Remark~\ref{rem:smooth_implies_borel}). Thus, in the following, we may assume that $H$ is Baire.
		
		If $H$ is not meagre, then by Fact~\ref{fct:pettis}, it is open. By compactness of $G$, it follows that $G/H$ is finite.
		
		If $H$ is meagre, then by Fact~\ref{fct:multiplication_open}, the coset equivalence relation of $H$ is also meagre (as the preimage of a meagre set by a continuous open map, cf.\ Remark~\ref{rem:group_to_cosets}), and thus, by Proposition~\ref{prop:mycielski}, $\lvert G/H\rvert=2^{\aleph_0}$.
		
		If $G/H$ is smooth, then by Remark~\ref{rem:smooth_implies_borel}, it is Borel, and hence (by Remark~\ref{rem:group_to_cosets}) so is $H$. Thus, $H$ has the strict Baire property, and by Fact~\ref{fct:separating_family} we have a countable Borel separating family for cosets of $H$, so by Fact~\ref{fct:miller}, $H$ is closed.
	\end{proof}
	
	The following example is essentially \cite[Example 3.39]{KM14}.
	\begin{ex}
		\label{ex:trichotomy_counterexample}
		In Proposition~\ref{prop:trichotomy_for_groups}, we cannot expect $G/H$ to be large for an arbitrary non-closed $H$. For example, if $G=\mathbf F_q^{\aleph_0}$ is a vector space over the finite field $\mathbf F_q$, then every nonzero linear functional $\eta$ on $G$ gives us a distinct subspace $H_\eta=\ker \eta$ of codimension $1$, which is then a subgroup of index $q$. But (by finite index) if $H$ is closed, it must be clopen. But as a compact Polish space, $G$ has only countably many clopen subsets, so (because there are $2^{\aleph_0}$ linear functionals) some $H_\eta$ is not closed (even though it has finite index).
		
		(In fact, a compact Hausdorff group has a non-open (equivalently, non-closed) subgroup of finite index if and only if it has uncountably many subgroups of finite index, see \cite[Theorem 2]{SW03}.)\xqed{\lozenge}
	\end{ex}
	
	\begin{rem}
		Note that Proposition~\ref{prop:trichotomy_for_groups} shows that a subgroup of a Polish group with the Baire property has index finite or $2^{\aleph_0}$, so one can ask if the same is true for arbitrary subgroups. In \cite[Theorem 2.3]{HHM16}, the authors show that if $G$ is compact Hausdorff but not profinite, then it has a subgroup of index $\aleph_0$. The case of profinite groups seems to remain open.\xqed{\lozenge}
	\end{rem}

	\section{Topological dynamics}
	\label{sec:prel_topdyn}
	All the relevant definitions and facts in topological dynamics (apart from the ones related to tame dynamical systems, which are in Section~\ref{section: tame systems}) can be found in Appendix~\ref{app:topdyn}, along with complete proofs of most of them. They have been deferred, because they have little expository value at this point. Here, we only list some of the most important ones.
	
	\begin{dfn}
		\label{dfn:dyn_system_pre}
		\index{dynamical system}
		By a \emph{dynamical system}, in this paper, we mean a pair $(G,X)$, where $G$ is an abstract group acting by homeomorphisms on a compact Hausdorff space $X$. Sometimes, $G$ is left implicit and we just say that $X$ is a dynamical system.
		
		\index{ambit}
		If $x_0\in X$ has orbit dense in $X$, then we call the triple $(G,X,x_0)$ a \emph{$G$-ambit}, or just an \emph{ambit}. Sometimes, when $G$ is clear from the context, we also write simply $(X,x_0)$.
		\xqed{\lozenge}
	\end{dfn}

	\begin{dfn}
		\label{dfn:ellis_group_pre}
		\index{piXg@$\pi_{X,g}$}
		\index{pig@$\pi_g$}
		\index{E(G,X)@$E(G,X)$|see {Ellis semigroup}}
		\index{EL@$EL$|see {Ellis semigroup}}
		\index{Ellis!semigroup}
		\index{Enveloping semigroup|see {Ellis semigroup}}
		If $(G,X)$ is a dynamical system, then its \emph{Ellis (or enveloping) semigroup} $EL=E(G,X)$ is the (pointwise) closure in $X^X$ of the set of functions $\pi_{X,g}\colon x\mapsto g\cdot x$ for $g\in G$. (When there is no risk of confusion, we write simply $\pi_g$, or --- abusing the notation --- just $g$ for $\pi_{X,g}$. When $(G,X)$ is clear from the context, we also write $EL$ for $E(G,X)$.)
		\xqed{\lozenge}
	\end{dfn}
	
	\begin{fct}
		\label{fct:ellis_clst_pre}
		If $(G,X)$ is a dynamical system, then $EL$ is a compact left topological semigroup (i.e.\ it is a semigroup with the composition as its semigroup operation, and the composition is continuous on the left). It is also a $G$-flow with $g\cdot f:= \pi_gf$ (i.e.\ $\pi_g$ composed with $f$).
	\end{fct}
	\begin{proof}
		Straightforward ($X^X$ itself is already a compact left topological semigroup, and it is easy to check that $EL$ is a closed subsemigroup).
	\end{proof}
	
	\begin{dfn}
		\index{ideal}
		\index{left ideal}
		A (left) ideal $I\unlhd S$ in a semigroup $S$ is a subset such that $IS\subseteq I$.\xqed{\lozenge}
	\end{dfn}
	
	\begin{rem}
		There is a corresponding notion of a right ideal in a semigroup (satisfying $SI\subseteq I$), as well as that of a two-sided ideal, but we will never use either of those in this thesis. Thus (for brevity), we often write just ``ideal" for ``left ideal''.\xqed{\lozenge}
	\end{rem}

	\begin{fct}[minimal ideals and the Ellis group]
		\label{fct:ideals_ellis_pre}
		\index{M@$\cM$}
		\index{minimal left ideal}
		\index{idempotent}
		\index{u@$u$}
		\index{Ellis!group}
		\index{uM@$u\cM$}
		Suppose $S$ is a compact Hausdorff left topological semigroup (e.g.\ the enveloping semigroup of a dynamical system). Then $S$ has a minimal (left) ideal $\cM$ (in the sense of inclusion). Furthermore, for any such ideal $\cM$:
		\begin{enumerate}
			\item
			$\cM$ is closed,
			\item
			for any element $a\in \cM$, $\cM=Sa=\cM a$,
			\item
			$\cM=\bigsqcup_u u\cM$, where $u$ runs over all idempotents in $\cM$ (i.e.\ elements such that $u\cdot u=u$) --- in particular, $\cM$ contains idempotents,
			\item
			for any idempotent $u\in \cM$, the set $u\cM$ is a subgroup of $S$ with the identity element $u$ (note that $u$ is usually \emph{not} the identity element of $S$ --- indeed, $S$ need not have an identity at all).
		\end{enumerate}
		Moreover, all the groups $u\cM$ (where $\cM$ ranges over all minimal left ideals and $u$ over idempotents in $\cM$) are isomorphic. The isomorphism type of all these groups is called the {\em ideal} group of $S$; if $S=E(G,X)$, we call this group the {\em Ellis group} of the flow $(G,X)$.
	\end{fct}
	\begin{proof}
		See Fact~\ref{fct:minimal_ideals_idempotents}.
	\end{proof}
	Throughout the thesis, we denote minimal ideals by $\cM$ or $\cN$, and we denote idempotents in minimal ideals by $u$ or $v$.
	Below, we summarise the basic facts related to the so-called $\tau$ topology of the Ellis groups of $(G,X)$.
	
	\begin{fct}
		\label{fct:tau_top_pre}
		Consider the Ellis semigroup $EL$ of a dynamical system $(G,X)$. Fix any minimal left ideal $\cM$ of $EL$ and an idempotent $u\in \cM$.
		\begin{enumerate}
			\item
			\index{o@$\circ$}
			For each $a\in EL$, $B\subseteq EL$, we write $a\circ B$ for the set of all limits of nets $(g_ib_i)_i$, where $g_i\in G$ are such that $\pi_{g_i}\to a$, and $b_i\in B$.
			\item
			For any $p,q\in EL$ and $A\subseteq EL$, we have:
			\begin{itemize}
				\item
				$p\circ(q\circ A)\subseteq (pq)\circ A$,
				\item
				$pA\subseteq p\circ A$,
				\item
				$p\circ A=p\circ \overline A$,
				\item
				$p\circ A$ is closed,
				\item
				if $A\subseteq \cM$, then $p\circ A\subseteq \cM$.
			\end{itemize}
			\item
			\index{clt@$\cl_{\tau}$}
			\index{topology!t@$\tau$}
			The formula $\cl_\tau(A):=(u\cM)\cap (u\circ A)$ defines a closure operator on $u\cM$. It can also be (equivalently) defined as $\cl_\tau(A)=u(u\circ A)$. We call the topology on $u\cM$ induced by this operator the {\em $\tau$ topology}.
			\item
			If $(f_i)_i$ (a net in $u\cM$) converges to $f\in \overline{u\cM}$ (the closure of $u\cM$ in $EL$), then $(f_i)_i$ converges to $uf$ in the $\tau$-topology.
			\item
			The $\tau$-topology on $u\cM$ is refined by the subspace topology inherited from $EL$.
			\item
			$u\cM$ with the $\tau$ topology is a compact $T_1$ semitopological group (i.e.\ with separately continuous multiplication).
			\item
			All the ideal groups $u\cM$ are isomorphic as semitopological groups, as we vary $\cM$ and $u$. We call them \emph{Ellis groups} of $(G,X)$.
			\item
			\index{H(uM)@$H(u\cM)$}
			$H(u\cM)=\bigcap_V \overline{V}$, where $V$ runs over the ($\tau$-)closures of all the ($\tau$-)neighbourhoods of the identity $u\in u\cM$, is a $\tau$-closed normal subgroup of $u\cM$, and $u\cM/H(u\cM)$ is a compact Hausdorff topological group.
		\end{enumerate}
	\end{fct}
	\begin{proof}
		See Facts~\ref{fct:circ_calculations}, \ref{fct:circ_with_closure} \ref{fct:circ_closed} \ref{fct:circ_stays_in_ideal} for (2).
		
		See Facts~\ref{fct:tau_closure}, \ref{fct:taucl_alt} for the (3).
		
		See Fact~\ref{fct:ulimit}, Fact~\ref{fct:tau_coarser}, Fact~\ref{fct:tau_T1}, Fact~\ref{fct:ellis_isomorphic} and Fact~\ref{fct:H(uM)} for the remaining points.
	\end{proof}
	
	The following technical observation comes from \cite{KPR15} (joint with Krzysztof Krupiński and Anand Pillay) and is essential there, as well as in \cite{KR18} and large parts of this thesis.
	\begin{prop}
		\label{prop:strange_cont_pre}
		The function $\xi\colon \overline {u\cM}\to u\cM$ (where $\overline{u\cM}$ is the closure of $u\cM$ in the topology of $EL$) defined by the formula $f\mapsto uf$ has the property that for any continuous function $\zeta\colon u\cM\to X$, where $X$ is a regular topological space and $u\cM$ is equipped with the $\tau$-topology, the composition $\zeta\circ \xi \colon \overline {u\cM}\to X$ is continuous, where $\overline{u\cM}$ is equipped with subspace topology from $EL$. In particular, the map $\overline{u\cM}\to u\cM/H(u\cM)$ given by $f \mapsto uf/H(u\cM)$ is continuous.
	\end{prop}
	\begin{proof}
		See Proposition~\ref{prop:strange_cont}.
	\end{proof}

	\section[Rosenthal compacta and tame dynamical systems]{Rosenthal compacta and tame dynamical systems\sectionmark{Rosenthal compacta and tameness}}
	\sectionmark{Rosenthal compacta and tameness}
	\subsection*{Rosenthal compacta, independent sets, and \texorpdfstring{$\ell^1$}{l1} sequences}\label{section: Rosenthal compacta}
	Here, we will discuss selected properties of Rosenthal compacta. For a broader exposition, refer to \cite{Debs14}.

	\begin{dfn}
		\index{Baire!class 1 function}
		\index{B_1(X)@$\cB_1(X)$|see{Baire class 1 function}}
		Given a topological space $X$, we say that a function $X\to \bR$ is of \emph{Baire class 1} if it is the pointwise limit of a sequence of continuous real-valued functions.
		We denote by $\cB_1(X)$ the set of all such functions.\xqed{\lozenge}
	\end{dfn}

	\begin{dfn}
		\index{Rosenthal compactum}
		A compact, Hausdorff space $K$ is a \emph{Rosenthal compactum} if it embeds homeomorphically into $\cB_1(X)$ for some Polish space $X$, where $\cB_1(X)$ is equipped with the pointwise convergence topology.
		\xqed{\lozenge}
	\end{dfn}
	\begin{dfn}
		\index{space!Fréchet}
		\index{space!Fréchet-Urysohn|see {Fréchet space}}
		A {\em Fréchet} (or {\em Fréchet-Urysohn}) space is a topological space in which any point in the closure of a given set $A$ is the limit of a sequence of elements of $A$.
		\xqed{\lozenge}
	\end{dfn}
	
	\begin{fct}\label{fct: Rosnthal implies Frechet}
		Rosenthal compacta are Fréchet.
	\end{fct}
	\begin{proof}
		\cite[Theorem 4.1]{Debs14}.
	\end{proof}

	\begin{fct}
		\label{fct:bft}
		Suppose $X$ is a compact metric space and $A\subseteq C(X)$ is a family of $0-1$ valued functions (i.e.\ characteristic functions of clopen subsets of $X$). Put $\mathcal A:=\{U\subseteq X\mid \chi_U\in A \}$. The following are equivalent:
		\begin{itemize}
			\item
			$\overline A\subseteq {\bR}^X$ is Fréchet (equivalently, Rosenthal),
			\item
			$\mathcal A$ contains no infinite independent family, i.e.\ $\mathcal A$ contains no family $(A_i)_{i \in \bN}$ such that for every $I \subseteq \bN$ the intersection $\bigcap_{i \in I} A_i \cap \bigcap_{i \in \bN \setminus I} A_i^c$ is nonempty.
		\end{itemize}
	\end{fct}
	\begin{proof}
		$A$ is clearly pointwise bounded, so by \cite[Corollary 4G]{BFT78}, $A$ is relatively compact in $\cB_1(X)$ (which is equivalent to the first condition) if and only if it satisfies the condition (vi) from \cite[Theorem 2F]{BFT78}, which for $0-1$ functions on a compact space reduces to the statement that for each sequence $(a_n)$ of elements of $A$ there is some $I\subseteq \bN$ for which there is no $x\in X$ such that $a_n(x)=1$ if and only if $n\in I$.
		This is clearly equivalent to the second condition.
	\end{proof}
	
	The next definition is classical and can be found for example in \cite[Section 5]{Koh95}.
	\begin{dfn}
		\index{l1 sequence@$\ell^1$ sequence}
		If $(f_n)_{n\in {\bN}}$ is a sequence of elements in a Banach space, we say that it is an {\em $\ell^1$ sequence} if it is bounded and there is a constant $\theta>0$ such that for any scalars $c_0,\ldots,c_n$ we have the inequality
		\[
		\theta\cdot \sum_{i=0}^n \lvert c_i\rvert \leq \left \lVert \sum_{i=0}^n c_i f_i\right\rVert.
		\]
		(This is equivalent to saying that $e_n\mapsto f_n$ extends to a topological vector space isomorphism of $\ell^1$ and the closed span of $(f_n)_n$ (in the norm topology), where $e_n$ are the standard basis vectors.)
		\xqed{\lozenge}
	\end{dfn}
	
	In fact, $\ell^1$ sequences are very intimately related to ``independent sequences" (via the Rosenthal's dichotomy).
	The following is a simple example of this relationship:
	
	\begin{fct}
		\label{fct:ind_untame}
		Suppose $X$ is a compact, Hausdorff space and $(A_n)_n$ is an independent sequence of clopen subsets of $X$. Then $(\chi_{A_n})_n$ is an $\ell^1$ sequence in the Banach space $C(X)$ (with the supremum norm).
	\end{fct}
	\begin{proof}
		Fix any sequence $c_0,\ldots,c_n$ of real numbers. Write $[n]$ for $\{0,\ldots,n\}$ and put $f:=\sum_{i\in [n]} c_i\chi_{A_i}$. Let $I:=\{i\in [n]\mid c_i\geq 0 \}$. Assume without loss of generality that $\sum_{i\in I} c_i\geq -\sum_{i\in [n]\setminus I} c_i$ (the other case is analogous). Then for any $x\in \bigcap_{i\in I} A_i\cap \bigcap_{i\in [n]\setminus I} A_i^c$ we have $f(x)=\sum_{i\in I} c_i\geq \frac{1}{2} \sum_{i\in [n]} \lvert c_i\rvert$.
	\end{proof}

	\subsection*{Tame dynamical systems}\label{section: tame systems}

	\begin{dfn}
		\index{tame!function}
		\index{tame!dynamical system}
		\label{dfn:tame_function_system}
		If $(G,X)$ is a dynamical system and $f\in C(X)$, then we say that $f$ is a \emph{tame function} if for every sequence $(g_n)_n$ of elements of $G$, $(f\circ g_n)_n$ is not an $\ell^1$ sequence.
		
		We say that $(G,X)$ is a \emph{tame dynamical system} if every $f\in C(X)$ is tame.
		\xqed{\lozenge}
	\end{dfn}
	
	\begin{rem}
		\label{rem:dfn_equiv}
		The notion of tame dynamical system is due to Kohler \cite{Koh95}. She used the adjective ``regular" instead of (now established) ``tame", and formulated it for actions of $\bN$ on metric compacta, but we can apply the same definition to arbitrary group actions on compact spaces.
		
		Some authors use different (but equivalent) definitions of tame function or tame dynamical system. For example, \cite[Fact 4.3 and Proposition 5.6]{GM12} yields several equivalent conditions for tameness of a function (including the definition given above and \cite[Definition 5.5]{GM12}). By this and \cite[Corollary 5.8]{GM12}, we obtain equivalence between our definition of tame dynamical system and \cite[Definition 5.2]{GM12}.\xqed{\lozenge}
	\end{rem}

	The condition in the following fact can be used as a definition of tameness for metric dynamical systems.
	\begin{fct}\label{fct: metric tameness}
		If $(G,X)$ is a metric dynamical system and $f\in C(X)$, then $f$ is tame if and only if the pointwise closure $\overline{\{f\circ g\mid g\in G \} }\subseteq {\bR}^X$ consists of Baire class 1 functions (note that it is true if and only if the closure is a Rosenthal compactum: in one direction, it is clear, while the other follows from Fact~\ref{fct: Rosnthal implies Frechet}).
	\end{fct}
	\begin{proof}
		It follows immediately from \cite[Fact 4.3 and Proposition 4.6]{GM12}.
	\end{proof}

	\begin{fct}
		\label{fct:tame_closed}
		For any dynamical system, the tame functions form a closed subalgebra of $C(X)$ (with pointwise multiplication and norm topology).
	\end{fct}
	\begin{proof}
		First, by Remark~\ref{rem:dfn_equiv}, tame functions in $(G,X)$ satisfy \cite[Definition 5.5]{GM12}, i.e.\ for every $f$ tame in $X$ there is a tame dynamical system $(G,Y_f)$ and an epimorphism $\phi_f\colon X\to Y_f$ such that $f=\phi_f^*(f'):=f'\circ \phi_f$ for some $f'\in C(Y_f)$.
		
		Since tame dynamical systems are closed under subsystems and under arbitrary products (\cite[Lemma 5.4]{GM12}), there is a universal $Y$ for all tame functions $f$, i.e.\ such that the set of all tame functions in $(G,X)$ is exactly the image of $\phi^*\colon C(Y)\to C(X)$, where $\phi\colon X\to Y$ is an epimorphism and $Y$ is tame (just take $\phi\colon X\to \prod_f Y_f$ to be the diagonal of $\phi_f$, and take $Y:=\phi[X] \subseteq \prod_f Y_f$).
		
		Since $C(Y)$ is a Banach algebra and $\phi^*$ is a homomorphism and an isometric embedding (as $\phi$ is onto), the fact follows.
	\end{proof}
	
	\begin{cor}
		\label{cor:tame_dense}
		If $(G,X)$ is a dynamical system and $\mathcal A\subseteq C(X)$ is a family of functions separating points, then $(G,X)$ is tame if and only if every $f\in \mathcal A$ is tame.
	\end{cor}
	\begin{proof}
		The implication $(\leftarrow)$ is obvious.
		
		$(\rightarrow)$.
		Since constant functions are trivially tame, by the assumption and the Stone-Weierstrass theorem, it follows that tame functions are dense in $C(X)$, and thus the conclusion follows immediately from Fact~\ref{fct:tame_closed}.
	\end{proof}

	\begin{fct}
		\label{fct:tame_preserved}
		Suppose $(G,X)$ is a tame dynamical system. Then the following dynamical systems are tame:
		\begin{itemize}
			\item
			$(H,X)$, where $H\leq G$,
			\item
			$(G,X_0)$, where $X_0\subseteq X$ is a closed invariant subspace,
			\item
			$(G,Y)$, where $Y$ is a $G$-equivariant quotient of $X$.
		\end{itemize}
	\end{fct}
	\begin{proof}
		The first bullet is trivial. The second follows from the Tietze extension theorem. For the third, note that any potentially untame function on $Y$ can be pulled back to $X$.
	\end{proof}

	The following is a dynamical variant of the so-called Bourgain-Fremlin-Talagrand dichotomy, slightly modified for our needs from \cite[Theorem 3.2]{GM06}.
	\begin{prop}
		\label{prop:dyn_BFT}
		Suppose $X$ is a totally disconnected metric compactum. Let $G$ act on $X$ by homeomorphisms. Then the following are equivalent:
		\begin{enumerate}
			\item
			\label{it:prop:dyn_BFT:untame}
			$(G,X)$ is untame,
			\item
			\label{it:prop:dyn_BFT:clopen}
			there is a clopen set $U$ and a sequence $(g_n)_{n\in \bN}$ of elements of $G$ such that the sets $g_n U$ are independent,
			\item
			\label{it:prop:dyn_BFT:betaN}
			$EL:=E(G,X)$ contains a homeomorphic copy of $\beta \bN$,
			\item
			\label{it:prop:dyn_BFT:large}
			$\lvert EL\rvert=2^{2^{\aleph_0}}$,
			\item
			\label{it:prop:dyn_BFT:Frechet}
			$EL$ is not Fréchet,
			\item
			\label{it:prop:dyn_BFT:Rosenthal}
			$EL$ is not a Rosenthal compactum.
		\end{enumerate}
		If $X$ is not necessarily totally disconnected, all conditions but \ref{it:prop:dyn_BFT:clopen} are equivalent.
	\end{prop}
	\begin{proof}
		The equivalence of all conditions but \ref{it:prop:dyn_BFT:clopen}
		is proved in \cite[Theorem 3.2]{GM06} (based on the Bourgain-Fremlin-Talagrand dichotomy).
		For the reader's convenience, we will prove here that all conditions (including \ref{it:prop:dyn_BFT:clopen}) are equivalent in the totally disconnected case (the case which appears in our model-theoretic applications).
		
		$\ref{it:prop:dyn_BFT:untame} \rightarrow \ref{it:prop:dyn_BFT:clopen}$. Since the characteristic functions of clopen subsets of $X$ are continuous and separate points in $X$, by \ref{it:prop:dyn_BFT:untame} and Corollary~\ref{cor:tame_dense}, the characteristic function $\chi_U$ is not tame for some clopen $U\subseteq X$. By Fact \ref{fct: metric tameness}, this is equivalent to the fact that $\overline{\{\chi_{gU} \mid g \in G\}}$ is not a Rosenthal compactum. Hence, Fact \ref{fct:bft} implies that some family $\{g_n U: n \in \bN\}$ (with $g_n \in G$) is independent.
		
		$\ref{it:prop:dyn_BFT:clopen} \rightarrow \ref{it:prop:dyn_BFT:untame}$. The reversed argument works. Alternatively, it follows immediately from Fact \ref{fct:ind_untame}.
		
		$\ref{it:prop:dyn_BFT:clopen} \rightarrow \ref{it:prop:dyn_BFT:betaN}$. Let $(g_n)$ be a sequence of elements of $G$ such that the sets $g_n U$ are independent. By the universal property of $\beta \bN$, we have the continuous function $\beta \colon \beta \bN\to EL$ given by $\mathcal F\mapsto \lim_{n\to \mathcal F} g_n^{-1}$. It remains to check that $\beta$ is injective. Consider two distinct ultrafilters $\mathcal F_1$ an $\mathcal F_2$ on $\bN$. Choose $F \in \mathcal F_1 \setminus \mathcal F_2$. By the independence of the $g_n U$, we can find $x \in \bigcap_{n \in F} g_nU \cap \bigcap_{n \in \bN \setminus F} g_nU^c$. It suffices to show that $\beta(\mathcal F_1)(x) \ne \beta(\mathcal F_2)(x)$. Note that $\{n \in \bN \mid g_n^{-1}x \in U^c\} = \bN \setminus F \notin \mathcal F_1$ and $U^c$ is open, so $\beta(\mathcal F_1)(x) \in U$. Similarly, $\beta(\mathcal F_2)(x) \in U^c$, and we are done.
		
		$\ref{it:prop:dyn_BFT:betaN} \rightarrow \ref{it:prop:dyn_BFT:large}$. The group $\{\pi_g \mid g \in G\}$ is contained in the Polish group $\Homeo(X)$ of all homeomorphisms of $X$ equipped with the uniform convergence topology (cf.\ Fact~\ref{fct:homeo_is_Polish}). So $\{\pi_g \mid g \in G\}$ is separable in the subspace topology (as a subspace of a separable metrisable space), and hence also in the pointwise convergence topology (which is weaker). Therefore, $EL =\overline{ \{\pi_g \mid g \in G\}}$ is of cardinality at most $2^{2^{\aleph_0}}$. On the other hand, $|\beta\bN| = 2^{2^{\aleph_0}}$. Hence, $|EL|= 2^{2^{\aleph_0}}$.
		
		$\ref{it:prop:dyn_BFT:large} \rightarrow \ref{it:prop:dyn_BFT:Frechet}$. If $EL$ is Fréchet, then, using the above observation that $\{\pi_g \mid g \in G\}$ is separable, we get that $|EL| = 2^{\aleph_0}$.
		
		$\ref{it:prop:dyn_BFT:Frechet} \rightarrow \ref{it:prop:dyn_BFT:Rosenthal}$. This is Fact \ref{fct: Rosnthal implies Frechet}.
		
		$\ref{it:prop:dyn_BFT:Rosenthal} \rightarrow \ref{it:prop:dyn_BFT:untame}$. Embed homeomorphically $X$ in $\bR^\bN$. Then $EL$ embeds homeomorphically in $\bR^{X \times \bN}$ via the map $\Phi$ given by $\Phi(f)(x,i):=f(x)(i)$. Take $f \in EL$, and let $\pi_i \colon X \to \bR$ be the projection to the $i$-th coordinate, i.e.\ $\pi_i(x):=x(i)$. Suppose $(G,X)$ is tame. Then $\pi_i \circ f \in \cB_1(X)$ by Fact \ref{fct: metric tameness}, so for every $i \in \bN$ there is a sequence of continuous functions $f^i_n \colon X \to \bR$ such that $\lim_n f^i_n = \pi_i \circ f$. Define $f_n \in \bR^{X \times \bN}$ by $f_n(x,i) := f^i_n(x)$. Then all $f_n$'s are continuous and $\Phi(f)= \lim_n f_n$.
		So $\Phi[EL]$ is a compact subset of $\cB_1(X \times \bN)$, i.e.\ $EL$ is Rosenthal.
	\end{proof}

	\begin{fct}
		\label{fct:tame_borel}
		If $(G,X)$ is a metric dynamical system, then $(G,X)$ is tame if and only if all functions in $E(G,X)$ are Borel measurable.
	\end{fct}
	\begin{proof}
		By Proposition~\ref{prop:dyn_BFT}, if $(G,X)$ is tame, $E(G,X)$ is Fréchet. Since the pointwise limit of a sequence of continuous functions between Polish spaces is always Borel (cf.\ \cite[11.2]{Kec95}), it follows that $E(G,X)$ consists of Borel functions.
		
		In the other direction, since $X$ is Polish, there are at most $2^{\aleph_0}$ many Borel functions $X\to X$ (because a Borel function is determined by the preimages of the basic open sets, and these preimages are among at most $2^{\aleph_0}$ many Borel sets in $X$). In particular, if $E(G,X)$ consists of Borel functions, $\lvert E(G,X)\rvert\leq 2^{\aleph_0}<2^{2^{\aleph_0}}$, which implies tameness by Proposition~\ref{prop:dyn_BFT}.
	\end{proof}

	\section{Model theory}
	In this section, we recall briefly some basic facts, definitions and conventions which will be applied in the model-theoretical parts of this paper. This is not comprehensive, and is only meant to remind the most important notions; for more in-depth explanation, see e.g.\ \cite{Hod93}, \cite{marker}, \cite{TZ12} for the elementary and \cite{CLPZ01}, \cite{GiNe08}, \cite{KP97}, \cite{KPS13}, \cite{LaPi}, \cite{Zie02} for the more advanced topics.
	\subsection*{Elementary matters}
	
	\index{T@$T$}
	Throughout, $T$ will denote the ambient (first order, complete, often countable) theory.
	
	\begin{dfn}
		\index{space!of types}
		\index{Sx(A)@$S_x(A)$|see {space of types}}
		Given a model $M\models T$, a set $A\subseteq M$, and a (possibly infinite) tuple $x$ of variables, by $S_x(A)$ we mean the space of complete types over $A$ in variables $x$, i.e.\ the Stone space of the Lindenbaum-Tarski algebra of formulas with free variables $x$ and parameters from $A$ modulo equivalence under $T$ (i.e.\ we identify formulas $\varphi_1(x),\varphi_2(x)$ when $T\vdash \varphi_1(x)\leftrightarrow \varphi_2(x)$). (Note that this implies that each $S_x(A)$ is a compact Hausdorff topological space, of weight at most $\lvert A\rvert+\lvert x\rvert+\lvert T\rvert$.)\xqed{\lozenge}
	\end{dfn}
	
	We fix a strong limit cardinal $\kappa$ larger than $\lvert T\rvert$ and ``all the objects we are interested in". We also fix a monster model $\fC$, satisfying the following definition. (Note that if $\kappa$ is strongly inaccessible, then we can choose $\fC$ as simply a saturated model of cardinality $\kappa$. If $\kappa$ is not strongly inaccessible, a saturated model of cardinality $\kappa$ may not exist.)
	
	\begin{dfn}
		\index{model!monster}
		\index{C@$\fC$|see {monster model}}
		A \emph{monster model} is a model $\fC$ of $T$ which is $\kappa$-saturated (i.e.\ each type over an arbitrary set of parameters from $\fC$ of size less than $\kappa$ is realized in $\fC$) and strongly $\kappa$-homogeneous (i.e.\ any elementary map between subsets of $\fC$ of cardinality less than $\kappa$ extends to an automorphism of $\fC$).
		\xqed{\lozenge}
	\end{dfn}
	
	Sometimes we refer to the $\kappa$ fixed above as ``the degree of saturation of $\fC$", even though this may not be strictly true (as $\fC$ may be saturated in a higher cardinality).
	
	\begin{fct}
		For every $\kappa$ as above and every complete theory $T$, there is a monster model.
	\end{fct}
	\begin{proof}
		See \cite[Theorem 10.2.1]{Hod93} (note that what we call a monster model is referred to as a ``$\kappa$-big model").
	\end{proof}
	
	\begin{dfn}
		\index{small}
		We call an object \emph{small} if it has cardinality smaller than the degree of saturation of $\fC$.\xqed{\lozenge}
	\end{dfn}
	
	The following remark highlights one of the most important features of the monster model, which provides a very strong link between syntax and semantics.
	\begin{rem}
		\label{rem:types_orbits}
		If $\fC$ is a monster model, then for every small $A\subseteq \fC$, and and any small (but possibly infinite) tuple $x$ of variables, if $p(x)$ is a complete type over $A$, then the set $p(\fC)$ of realisations of $p(x)$ is nonempty, and it is a single orbit of $\Aut(\fC/A)$, the group of automorphisms of $\fC$ fixing $A$ pointwise.
		
		Consequently, there is a natural bijection between the space $S_x(A)$ of types over $A$ in variables $x$ and the orbits of $\Aut(\fC/A)$ in the product of sorts corresponding to the variable $x$.
		\xqed{\lozenge}
	\end{rem}
	
	By convention, whenever we mention a small model $M$ of $T$, we assume that $M$ is an elementary substructure of $\fC$ (i.e.\ for every formula $\varphi(x)$ with parameters in $M$, $M\models \exists x \varphi(x)$ if and only if $\fC\models \exists x \varphi(x)$). Every small model of $T$ can be embedded this way (this follows from $\kappa$-saturation of $\fC$).
	
	\begin{dfn}
		\index{phi(M)@$\varphi(M)$}
		Given a formula $\varphi(x)$ and a model $M\models T$ (including $\fC$), by $\varphi(M)$ we mean the set of all realisations of $\varphi$ in $M$, i.e.\ tuples $a$ in $M$ such that $M\models \varphi(a)$.
		
		Likewise, if $\pi$ is a partial type, by $\pi(M)$ we mean the set of all realisations of $\pi$ in $M$.
		\xqed{\lozenge}
	\end{dfn}
	
	\begin{dfn}
		Let $A\subseteq \fC$ be a small set, and let $X$ be a subset of a fixed product of sorts of $X$.
		\begin{itemize}
			\item
			\index{invariant set}
			\index{XA@$X_A$}
			\index{SX(A)@$S_X(A)$}
			We say that $X$ is {\em $A$-invariant} if it is setwise invariant under $\Aut(\fC/A)$ (note that by Remark~\ref{rem:types_orbits}, such a set is a union of sets of realizations of some number of complete types over $A$). In this case, we denote by $X_A$ (or by $S_X(A)$) the set of types over $A$ of elements of $X$. We say that $X$ is simply \emph{invariant} if it is invariant over $\emptyset$, i.e.\ under $\Aut(\fC)$.
			\item
			\index{definable set}
			We say that $X$ is \emph{$A$-definable} or \emph{definable over $A$} if for some formula $\varphi(x)$ with parameters in $A$, we have $X=\varphi(\fC)$. We say that $X$ is simply \emph{definable} if it is definable over some $A$.
			\item
			\index{type-definable set}
			We say that $X$ is \emph{type-definable over $A$} or \emph{$A$-type-definable} if it is the set of realisations of a partial type over $A$, or, equivalently, it is the intersection of a family of $A$-definable sets. We say that $X$ is simply \emph{type-definable} if it is $A$-type-definable for some small $A$ (equivalently, if it is the intersection of a small family of definable sets).
			\item
			\index{analytic set!in model theory}
			\index{Fs set (in model theory) @ $F_\sigma$ set (in model theory)}
			\index{Borel!set (in model theory)}
			We say that \emph{$X$ is analytic, $F_\sigma$, or Borel over $A$} (respectively) if it is $A$-invariant and $X_A$ is analytic, $F_\sigma$ or Borel (respectively) in the the compact space $S_x(A)$, where $x$ is the tuple of variables corresponding to the product of sorts containing $X$. We say that $X$ is simply \emph{analytic, $F_\sigma$, or Borel} (respectively) if it is such over some small $A$.
			\item
			\index{definable set!relatively}
			If $Y\subseteq \fC$ is arbitrary (usually, type-definable), while $X\subseteq Y$, then we say that $X$ is \emph{relatively definable [over $A$] in $Y$} if there is some [$A$-]definable $X'$ such that $X=X'\cap Y$.
			\xqed{\lozenge}
		\end{itemize}
	\end{dfn}
	
	It should be stressed that by ``type-definable" we mean ``type-definable with parameters" whereas ``invariant" means ``invariant over $\emptyset$" (unless specified otherwise).

	\begin{dfn}
		\index{Sa(M)@$S_a(M)$}
		For a tuple $\bar a$ from $\fC$ and a set of parameters $A$, by $S_{\bar a}(A)$ we denote the space of all types $\tp(\bar b/A)$ with $\bar b \equiv \bar a$ (i.e.\ the space of types over $A$ extending $\tp(a/\emptyset)$). (In particular, $S_a(A)=([a]_{\equiv})_A$.)\xqed{\lozenge}
	\end{dfn}

	\begin{fct}
		\label{fct:dtafb}
		If $X$ is as a subset of a fixed product of sorts of $\fC$ which is $A$-invariant, then it is definable, type-definable, analytic, Borel or $F_\sigma$ if and only if it is such over $A$, and if and only if $X_A$ is clopen, closed, analytic or $F_\sigma$ in $S_x(A)$ (for the relevant tuple $x$ of variables).
		
		Likewise, if $X\subseteq Y$ and $X,Y$ are $A$-invariant, then $X$ is relatively definable in $Y$ if and only if it is relatively definable over $A$, if and only if $X_A$ is clopen in $Y_A$.
	\end{fct}
	\begin{proof}
		Note that (immediately or almost immediately by definition) $X$ is definable, type-definable, Borel, analytic, or $F_\sigma$ over $A$ if and only if $X_A$ is clopen, closed, Borel, analytic or $F_\sigma$ (respectively) in $S_x(A)$, and it is relatively definable in $Y$ over $A$ if and only if it is clopen in $Y_A$, so it is enough to show that this happens for $A$ if and only if it happens for every $B$ over which $X,Y$ are invariant.
		
		We may assume without loss of generality that $A\subseteq B$. Then we have a continuous surjection $\pi \colon S_x(B)\to S_x(A)$ (given by restriction), $X_B=\pi^{-1}[X_A]$ and $Y_B=\pi^{-1}[Y_A]$. The fact follows from Proposition~\ref{prop:preservation_properties}.
	\end{proof}
	
	\begin{dfn}
		\index{equiv@$\equiv$}
		By $\equiv$, we denote the equivalence relation (on all small tuples of elements of $\fC$) of having the same type over $\emptyset$, or equivalently, of lying in the same $\Aut(\fC)$ orbit. Likewise, for a small $A\subseteq \fC$, $\equiv_A$ denotes the relation of having the same type over $A$.\xqed{\lozenge}
	\end{dfn}
	
	\begin{dfn}
		\index{indiscernible sequence}
		Suppose $(a_i)_{i\in I}$ is a sequence of elements of $\fC$, indexed by a totally ordered set $(I,\leq)$. We say that $(a_i)$ is \emph{(order) indiscernible over $A$} (where $A\subseteq \fC$ is a small set) if for each $n$, for all increasing sequences $i_1< i_2< \ldots < i_n$ and $j_1<j_2<\ldots<j_n$, there is some $\sigma\in \Aut(\fC/A)$ such that $\sigma(a_{i_1}\ldots a_{i_n})=a_{j_1}\ldots a_{j_n}$, or, equivalently, $a_{i_1}\ldots a_{i_n}\equiv_A a_{j_1}\ldots a_{j_n}$.
		
		We say that $(a_i)_i$ is simply \emph{indiscernible} if it is indiscernible over $\emptyset$.\xqed{\lozenge}
	\end{dfn}

	\subsection*{Bounded invariant equivalence relations and strong types}
	
	\begin{dfn}
		\index{equivalence relation!bounded invariant}
		\label{dfn:stype}
		We say that an [$A$-]invariant equivalence relation $E$ on an [$A$-]invariant set $X\subseteq \fC$ (in a small product of sorts of $\fC$) is \emph{bounded} if it has a small number of classes. (Remark~\ref{rem:bounded_noofclasses} implies that the number of classes is at most $2^{\lvert T\rvert+\lambda[+\lvert A\rvert]}$ when $X$ is contained in a product of $\lambda$ sorts of $\fC$).
		
		\index{strong type}
		In a slight abuse of the terminology, we say that $E$ is a \emph{strong type} if $E$ is a bounded invariant equivalence relation and $E\subseteq {\equiv}$. The classes of $E$ are also called \emph{strong types}.\xqed{\lozenge}
	\end{dfn}

	\begin{dfn}
		\label{dfn:logic_topology}
		\index{topology!logic}
		Let $E$ be a bounded, invariant equivalence relation on a $\emptyset$-type-definable set $X$. We define the {\em logic topology} on $X/E$ by saying that a subset $D \subseteq X/E$ is closed if its preimage in $X$ is type-definable.\xqed{\lozenge}
	\end{dfn}

	\begin{rem}
		\label{rem:logic_top_larger_sets}
		Note that if $X$ and $E$ are only invariant over some small $A\subseteq \fC$, we can think of them as invariant in the language expanded by constants for elements for $A$. Thus, if $X$ is $A$-type-definable and $E$ is bounded, we can still talk about the logic topology on $X/E$ and it will have properties analogous to the ones listed below.
		\xqed{\lozenge}
	\end{rem}

	\begin{fct}
		\label{fct:logic_by_type_space}
		Fix $M\preceq \fC$, a $\emptyset$-type-definable set $X$ and a bounded invariant equivalence relation $E$ on $X$. Recall that $X_M=\{\tp(a/M)\mid a\in X \}$.
		
		Then $E$ is refined by $\equiv_M$, and thus the map $X\to X/E$ factors through a map $X_M\to X/E$, given by $\tp(a/M)\mapsto [a]_E$. Moreover, the latter map is a topological quotient map.
	\end{fct}
	\begin{proof}
		The inclusion ${\equiv_M}\subseteq E$ follows from Fact~\ref{fct:model_to_indiscernible}.
		
		Note that this implies that for every $A\subseteq X/E$, the preimage $A'$ in $X$ (via the quotient map) is $M$-invariant. In particular, if $A'$ is type-definable, it is type-definable over $M$. It follows that $A'_M$, the preimage of $A$ via the induced map $X_M\to X/E$, is also closed. Conversely, if $A'_M$ is closed, then $A'$ is type-definable, so $A$ is closed.
	\end{proof}
	
	\begin{dfn}
		\label{dfn:EM}
		\index{EM@$E^M$}
		Given a bounded invariant equivalence relation $E$ on $X$ and a model $M\preceq \fC$, by $E^M$ we denote the equivalence relation on $X_M$ given by $p\equiv q$ when for some (equivalently, every) $a\models p$ and $\models q$ we have $a\Er b$. (Note that $E^M\subseteq (X_M)^2$ and it is distinct from $E_M\subseteq (X^2)_M$, which is its preimage by the restriction map $(X^2)_M\to (X_M)^2$.)
		\xqed{\lozenge}
	\end{dfn}
	
	\begin{rem}
		Fact~\ref{fct:logic_by_type_space} allows us to identify $X/E$ and $X_M/E^M$ (which we do freely).\xqed{\lozenge}
	\end{rem}

	\begin{fct}\label{fct: Borel in various senses}
		Suppose $E$ is a bounded invariant equivalence relation on a $\emptyset$-type-definable set $X$. Then $E$ is relatively definable (in $X^2$), type-definable, $F_\sigma$, Borel or analytic if and only if for some (equivalently, every) small model $M$, the equivalence relation $E^M$ is clopen, closed, $F_\sigma$, Borel or analytic (respectively) as a subset of $X_M^2$.
		
		Furthermore, we have the same conclusion for $E\restr_Y$ and $(E^M)\restr_{Y_M}=(E\restr_Y)^M$, where $Y$ is an arbitrary type-definable, $E$-invariant set.
	\end{fct}
	\begin{proof}
		Note that since $E$ is bounded invariant, it is also $M$-invariant (by Fact~\ref{fct:logic_by_type_space}). Thus, by Fact~\ref{fct:dtafb}, $E$ is relatively definable, type-definable, $F_\sigma$, analytic or Borel if and only if $E_M\subseteq (X^2)_M$ is clopen, closed, $F_\sigma$, analytic or Borel (respectively).
		
		On the other hand, by Fact~\ref{fct:logic_by_type_space}, it follows that $E_M$ is the preimage of $E^M$ by the restriction map $(X^2)_M\to (X_M)^2$, so the fact follows by Proposition~\ref{prop:preservation_properties}.
		
		The ``furthermore'' part is analogous.
	\end{proof}
	
	\begin{fct}
		\label{fct:logic_top_cpct_T2}
		For any $\emptyset$-type-definable $X$ and bounded invariant equivalence relation $E$ on $X$, $X/E$ is a compact space, and it is Hausdorff if and only if $E$ is type-definable, and discrete if and only if $E$ is relatively definable (as a subset of $X^2$).
	\end{fct}
	\begin{proof}
		It follows easily from Fact~\ref{fct: Borel in various senses}, Fact~\ref{fct:logic_by_type_space} and Fact~\ref{fct:quot_T2_iff_closed}.
	\end{proof}
	
	The following proposition (which appeared in \cite{KPR15}, joint with Krzysztof Krupiński and Anand Pillay) gives an equivalent condition for type-definability of an equivalence relation defined on some $p(\fC)$ for $p\in S(\emptyset)$.
	The idea that having a type-definable class is equivalent to being type-definable was one of the main motivations for Chapter~\ref{chap:intransitive}, where we find a more general context where type-definability of classes implies type-definability of a bounded invariant equivalence relation (see in particular Theorem~\ref{thm:worb_aut} and Corollary~\ref{cor:smt_aut}).
	\begin{prop}\label{prop:type-definability_of_relations}
		If $E$ is an invariant equivalence relation defined on
		a single complete type $[a]_{\equiv}$ over $\emptyset$, then $E$ has a type-definable [resp. relatively definable] class if and only if $E$ is type-definable [resp. relatively definable].
	\end{prop}
	\begin{proof}
		We prove the type-definable version; the relatively definable version is similar.
		The implication $(\Leftarrow)$ is obvious. For the other implication,
		without loss of generality $[a]_E$ is type-definable. Since $[a]_E$ is $a$-invariant, we get that it is type-definable over $a$, i.e.\ $[a]_E=\pi(\fC,a)$ for some partial type $\pi(x,y)$ over $\emptyset$. Then, since $E$ is invariant, for any $b \equiv a$ we have $[b]_E=\pi(\fC,b)$. Thus, $\pi(x,y)$ defines $E$.
	\end{proof}
	
	\begin{rem}
		\label{rem:tdf_iff_restr}
		Proposition~\ref{prop:type-definability_of_relations} immediately implies that if $Y\subseteq [a]_{\equiv}$ is type-definable and $E$-saturated, then $E\restr_Y$ is type-definable if and only if $E$ is type-definable.\xqed{\lozenge}
	\end{rem}

	\begin{fct}\label{fct:cartdf}
		Assume that the language is countable. For any $E$ which is a bounded, invariant equivalence relation on some $\emptyset$-type-definable and countably supported set $X$, and for any $Y\subseteq X$ which is type-definable and $E$-saturated, the Borel cardinality of the restriction of $E^M$ to $Y_M$ does not depend on the choice of the countable model $M$. In particular, if $X=Y$, the Borel cardinality of $E^M$ does not depend on the choice of the countable model $M$.
		
		Analogously, if $E$ and $X$ are invariant over a countable set $A$, then the Borel cardinality of $E^M\restr_{Y_M}$ does not depend on the choice of a countable model $M\supseteq A$.
	\end{fct}
	\begin{proof}
		This follows from Fact~\ref{fct:borel_section}. See \cite[Proposition 2.12]{KR16} for details.
	\end{proof}
	
	This justifies the following definition.
	
	\begin{dfn}[$T$ countable]
		\label{dfn:bier_borelcard}
		\index{Borel!cardinality!in model theory}
		\index{smoothness}
		If $E$ is a bounded invariant equivalence relation on a $\emptyset$-type-definable set $X$ (in countably many variables), then by the \emph{Borel cardinality of $E$} we mean the Borel cardinality of $E^M$ for a countable model $M$. In particular, we say that $E$ is {\em smooth} if $E^M$ is smooth for a countable model $M$.
		
		Similarly, if $Y$ is type-definable and $E$-saturated, the Borel cardinality of $E\restr_Y$ is the Borel cardinality of $E^M\restr_{Y_M}$ for a countable model $M$.\xqed{\lozenge}
	\end{dfn}
	
	\begin{rem}
		\label{rem:tdf_implies_smooth}
		Note that if $E$ is as in Definition~\ref{dfn:bier_borelcard}, and type-definable, then it is smooth by Fact~\ref{fct:clsd_smth}.\xqed{\lozenge}
	\end{rem}

	\subsection*{Classical strong types and strong automorphism groups}
	We provide proofs of various facts related to the classical strong types, strong automorphism groups and the corresponding Galois groups. They are all well-known, but the proofs are rather scattered, and different authors use different (but equivalent) definitions of some notions, and so most have been collected here for the convenience of the reader. Most (if not all) of them can be found in \cite{CLPZ01}, \cite{KP97}, \cite{LaPi} and \cite{Zie02}.
	
	\begin{fct}
		\label{fct:indisc_bdd}
		Suppose $E$ is a bounded invariant equivalence relation and $(a_i)$ is an infinite indiscernible sequence. Then all elements of $(a_i)$ are $E$-related.
	\end{fct}
	\begin{proof}
		Suppose not. Then for all $i\neq j$ we have $\neg (a_i \Er a_j)$. Since $(a_i)$ is infinite and indiscernible, it can be extended to an arbitrarily long indiscernible sequence (whose elements are pairwise $E$-inequivalent), which contradicts boundedness of $E$.
	\end{proof}
	
	\begin{fct}
		\label{fct:model_coheir}
		Suppose $p\in S(M)$. Then $p$ has a global coheir $p'\in S(\fC)$, i.e.\ an extension of $p$ such that for every $\varphi(x)\in p'$, $\varphi(M)$ is nonempty.
	\end{fct}
	\begin{proof}
		Note that $\{\varphi(M)\mid \varphi\in p \}$ is a centered family of subsets of $M$ (because $M$ is a model), so it can be extended to an ultrafilter $\tilde p$. For every such $\tilde p$, the type $p':=\{\varphi(x,c)\mid c\in \fC\land \varphi(M)\in \tilde p \}$ has the desired property.
	\end{proof}
	
	\begin{fct}
		\label{fct:model_to_indiscernible}
		Suppose $a,b$ are tuples and $M$ is a model such that $a\equiv_M b$. Then there is an infinite sequence $I$ such that $aI$ and $bI$ are indiscernible. In particular, by Fact~\ref{fct:indisc_bdd}, for every bounded ($M$-)invariant equivalence relation $E$ we have $a\Er b$ (and hence $\equiv_M\subseteq E$).
	\end{fct}
	\begin{proof}
		Let $p=\tp(a/M)$, and let $p'\in S(\fC)$ be a global $M$-invariant extension of $p$ (e.g.\ a coheir extension). Then construct recursively sequence $I=(c_n)_{n\in \bN}$ so that $c_n\models p'\restr_{Mabc_{<n}}$.
		
		Let us show that all pairs in $aI$ have the same type over $M$. The case of arbitrarily long tuples follows by straightforward induction. Take any $i_1<i_2$. Then $c_{i_1}\models p'\restr_M=p=\tp(a/M)$, so we can choose some $\sigma\in \Aut(\fC/M)$ such that $\sigma(a)=c_{i_1}$. On the other hand, we had $c_{0}\models p'\restr_{Ma}$. It follows that $\sigma(c_0)\models p'\restr_{Mc_{i_1}}$. Since also $c_{i_2}\models p'\restr_{Mc_{i_1}}$, there is some $\sigma_1\in \Aut(\fC/Mc_{i_1})$ such that $\sigma_1(\sigma(c_0))=c_{i_2}$. but then $\sigma_1(\sigma(ac_0))=c_{i_1}c_{i_2}$ and $\sigma_1\sigma\in\Aut(\fC/M)$, and we are done.
	\end{proof}
	
	\begin{rem}
		\label{rem:bounded_noofclasses}
		Note that Fact~\ref{fct:model_to_indiscernible} immediately implies that every bounded invariant equivalence relation $E$ on a set $X$ of tuples of length $\lambda$ variables has at most $\lvert X_M\rvert\leq 2^{\lvert M\rvert+\lambda+\lvert T\rvert}$, and thus by Löwenheim-Skolem, it has at most $2^{\lambda+\lvert T\rvert}$ classes.\xqed{\lozenge}
	\end{rem}
	
	\begin{fct}
		\label{fct:lascar_finest}
		Let $\Theta$ be the relation on tuples in a given product of sorts of lying in the same infinite indiscernible sequence. Consider the relation $E$, which is the transitive closure of $\Theta$. Then $E$ is the finest bounded invariant equivalence relation. (In particular, such relation exists.)
	\end{fct}
	\begin{proof}
		$\Theta$ is clearly invariant, symmetric and reflexive, so $E$ is an invariant equivalence relation. By Fact~\ref{fct:model_to_indiscernible}, $E$ is refined by $\equiv_M$, so it is bounded. Fact~\ref{fct:indisc_bdd} implies that it refines every bounded invariant equivalence relation, so $E$ has to be the finest such relation.
	\end{proof}
	
	\begin{fct}
		\label{fct:KP_exists}
		On every product of sorts, there is a finest bounded $\emptyset$-type-definable equivalence relation.
	\end{fct}
	\begin{proof}
		Note that the relation $E$ from Fact~\ref{fct:lascar_finest} refines every bounded $\emptyset$-type-definable equivalence relation on a given product of sorts $X$. It follows that there is only a small number (at most $2^{\lvert X/E\rvert}$) of those, which easily implies that the intersection is both type-definable and bounded.
	\end{proof}
	
	\begin{fct}
		\label{fct:indisc_to_model}
		Suppose $I$ is a small indiscernible sequence. Then there is some model $M$ such that $I$ is indiscernible over $M$, and in particular, all of its elements have the same type over $M$.
	\end{fct}
	\begin{proof}
		Fix any small $M'\preceq \fC$. Then by Ramsey's theorem compactness, we can find a sequence $I'$ which is indiscernible over $M'$, and such that $I'\equiv I$. but then there is an automorphism moving $I'$ to $I$, and it moves $M'$ to some model $M$ over which $I$ is indiscernible.
	\end{proof}

	\begin{dfn}
		\label{dfn:class_stp}
		The following are the three classical strong types. (For their existence, see Facts~\ref{fct:lascar_finest} and \ref{fct:KP_exists}.)
		\begin{itemize}
			\index{strong type!Lascar}
			\index{equivL@$\equiv_\Lasc$|see {Lascar strong type}}
			\item the \emph{Lascar strong type} $\equiv_\Lasc$ is the finest bounded, invariant equivalence relation on a given product of sorts,
			\index{strong type!Kim-Pillay}
			\index{equivKP@$\equiv_\KP$|see {Kim-Pillay strong type}}
			\item the \emph{Kim-Pillay strong type} (sometimes called also the \emph{compact strong type}) $\equiv_\KP$ is the finest bounded, $\emptyset$-type-definable equivalence relation on a given product of sorts,
			\index{strong type!Shelah}
			\index{equivSh@$\equiv_\Sh$|see {Shelah strong type}}
			\item the \emph{Shelah strong type} $\equiv_\Sh$ is the intersection of all $\emptyset$-definable equivalence relations with finitely many classes (note that by compactness, a definable equivalence relation is bounded if and only if it has finitely many classes).\xqed{\lozenge}
		\end{itemize}
	\end{dfn}
	
	\begin{rem}
		Strictly speaking, the names given in Definition~\ref{dfn:class_stp} are an abuse of terminology. In more standard terms, a Lascar, Kim-Pillay or Shelah strong type is a single class of $\equiv_\Lasc,\equiv_\KP$ or $\equiv_\Sh$ (respectively). However, we use the names for both the relations and their classes, the same as in the general case, as indicated in Definition~\ref{dfn:stype}. \xqed{\lozenge}
	\end{rem}
	
	\begin{rem}
		It is not hard to see that $\equiv_{\textrm{Sh}}$ is the same as $\equiv_{\acl^{\textrm{eq}}(\emptyset)}$ (i.e.\ the relation of having the same type over the imaginary algebraic closure of the empty set).\xqed{\lozenge}
	\end{rem}
	
	\begin{rem}
		It is easy to see using Fact~\ref{fct:logic_top_cpct_T2} that if $X$ is $\emptyset$-type-definable, then $X/{\equiv_\Lasc}$ is a compact space, $X/{\equiv_\KP}$ is a compact Hausdorff space, and $X/{\equiv_\Sh}$ is a profinite space.\xqed{\lozenge}
	\end{rem}

	\begin{dfn}
		\index{strong automorphism!Lascar}
		\index{Autf(C)@$\Autf(\fC)$|see{Lascar strong automorphism}}
		\label{dfn:autf_L}
		The group $\Autf(\fC)$ \emph{Lascar strong automorphisms} consists of those automorphisms of $\fC$ which preserve all $\equiv_\Lasc$-classes.
		
		\index{strong automorphism!Kim-Pillay}
		\index{strong automorphism!Shelah}
		\index{AutfKP(C)@$\Autf_\KP(\fC)$|see{Kim-Pillay strong automorphism}}
		\index{AutfSh(C)@$\Autf_\Sh(\fC)$|see{Shelah strong automorphism}}
		Analogously, the groups $\Autf_\KP(\fC)$ of \emph{Kim-Pillay strong automorphisms} and $\Autf_\Sh(\fC)$ of \emph{Shelah strong automorphisms} consist of automorphisms fixing all Kim-Pillay and Shelah strong types (respectively).\xqed{\lozenge}
	\end{dfn}

	\begin{rem}
		It is easy to see that all strong automorphism groups are normal subgroups of $\Aut(\fC)$.\xqed{\lozenge}
	\end{rem}
	
	The following fact gives us an explicit description of the Lascar strong type.
	\begin{fct}
		\label{fct:Lascar_equivalent}
		For any tuples $a,b$, the following are equivalent:
		\begin{enumerate}
			\item
			$a\equiv_\Lasc b$,
			\item
			for some $n$, there are small models $M_1,\ldots, M_n$ and tuples $a=a_0,\ldots,a_n=b$ such that for each $i=1,\ldots,n$ we have $a_{i-1}\equiv_{M_{i}} a_{i}$ (so in particular, ${\equiv_M}\subseteq {\equiv_\Lasc}$, and if $p\in S(M)$ and both $a$ and $b$ realise $p$, then $a\equiv_\Lasc b$),
			\item
			for some $n$, there are tuples $a=a_0,\ldots,a_n=b$ such that for each $i=1,\ldots,n$ there is a sequence $(c_n^i)_{n\in \omega}$ such that $a_{i-1}a_{i}\frown (c^i_n)_{n\in \omega}$ is indiscernible (so in particular, if $a$ and $b$ are in an infinite indiscernible sequence, then $a\equiv_\Lasc b$).
		\end{enumerate}
	\end{fct}
	\begin{proof}
		The equivalence of \ref{it:prop:dyn_BFT:untame} and (3) is Fact~\ref{fct:lascar_finest}. (3) implies (2) by Fact~\ref{fct:indisc_to_model}. (2) implies (3) by Fact~\ref{fct:model_to_indiscernible}.
	\end{proof}
	
	The group $\Autf(\fC)$ is especially important. It can be also described in the following way.
	\begin{fct}
		\label{fct:lst_witn_by_aut}
		The group $\Autf(\fC)$ is generated by $\Aut(\fC/M)$, where $M$ runs over all $M\preceq \fC$, and for any tuples $a,b$ we have that $a\equiv_\Lasc b$ if and only if for some $\sigma\in \Autf(\fC)$ we have $\sigma(a)=b$.
	\end{fct}
	\begin{proof}
		Denote by $G$ the group generated by all $\Aut(\fC/M)$. By Fact~\ref{fct:model_to_indiscernible}, each $\Aut(\fC/M)$ is contained in $\Autf(\fC)$, so $G\leq \Autf(\fC)$. On the other hand, by Fact~\ref{fct:Lascar_equivalent}, for all tuples $a\equiv_\Lasc b$ there is some $\sigma\in G$ such that $\sigma(a)=b$. In particular, if $m$ enumerates any model $M$, and $\tau\in \Autf(\fC)$ is arbitrary, then there is some $\sigma\in G$ such that $\sigma(m)=\tau(m)$. But then $\sigma^{-1}\tau\in \Aut(\fC/M)$, so $\tau=\sigma(\sigma^{-1}\tau)\in G$.
	\end{proof}

	\begin{dfn}
		\label{dfn:Lascar distance}
		\index{Lascar!distance}
		\index{dL@$d_\Lasc$}
		The \emph{Lascar distance} $d_\Lasc(a,b)$ is the minimum number $n$ as in Fact~\ref{fct:Lascar_equivalent}(2) (or $\infty$ if it does not exist).
		
		\index{Lascar!diameter}
		The \emph{Lascar diameter} of an automorphism $\sigma\in \Aut(\fC)$ is the largest $n$ such that for some tuple $a$ we have $d_\Lasc(a,\sigma(a))=n$ (or $\infty$ if it does not exist).\xqed{\lozenge}
	\end{dfn}
	
	\begin{fct}
		\label{fct:distance_tdf}
		The Lascar distance is type-definable, i.e.\ for each $n$, $d_\Lasc(a,b)\leq n$ is a type-definable condition about $a$ and $b$ (in a given product of sorts).
	\end{fct}
	\begin{proof}
		First we have the following claim.
		\begin{clm*}
			There is a type $\Pi$ without parameters such that whenever $M\preceq \fC$ has cardinality at most $\lvert T\rvert$, then some enumeration $m$ of $M$ (possibly with repetitions) satisfies $\Pi(m)$, and conversely, whenever a tuple $m$ satisfies $\Pi$, the set it enumerates is an elementary submodel of $\fC$.
		\end{clm*}
		\begin{clmproof}
			We use a variant of the Henkin construction. Suppose for simplicity that the language is countable and one-sorted. Consider all formulas of the form $\varphi(x,y)$ (with no parameters), where $x$ is a single variable and $y$ is a finite tuple. Then we can assign to each pair $(\varphi(x,y),\eta)$, where $\eta$ is a sequence of (not necessarily distinct) natural numbers of length $\lvert y\rvert$ a natural number $n_{\varphi,\eta}$ greater than all elements of $\eta$, and we can do it injectively, so that each pair corresponds to a distinct natural number. Let $z=(z_i)_{i\in \bN}$ be a sequence of distinct variables, and for each $\eta=n_1\ldots n_k$, write $z_\eta$ for the finite tuple $z_{n_1}\ldots z_{n_k}$.
			
			Then let $\Pi(z)=\{\exists x\varphi(x,z_\eta)\rightarrow \varphi(z_{2n_{\varphi,\eta}},z_\eta)\mid \varphi,\eta \}$ shows that the claim is true. Indeed, if $m\models \Pi$, then by the Tarski-Vaught test, it enumerates an elementary submodel of $\fC$. On the other hand, if $M$ enumerates a model, we can enumerate the whole $M$ as $(m_{2n+1})_{n\in \bN}$. Then we can fill the even positions recursively: for each $2n$, if for some $\varphi$ and $\eta$ we have that $n=n_{\varphi,\eta}$ and $\models \exists x\varphi(x,m_{\eta})$, then we take for $m_{2n}$ any witness of that in $M$, and otherwise, we take for it any element of $M$. Then clearly $\models \Pi((m_n)_{n\in \bN})$.
		\end{clmproof}
		Note that by Löwenheim-Skolem, it follows that the models witnessing $d_\Lasc(a,b)\leq n$ can be chosen to be of size at most $\lvert T\rvert$. Thus the fact follows immediately by the Claim.
	\end{proof}
	
	\begin{rem}
		It follows immediately from the definition that the Lascar distance is a metric (which may attain $\infty$), and in particular, it satisfies the triangle inequality.
		Likewise, the Lascar diameter satisfies the inequality $d_\Lasc(\sigma\tau)\leq d_\Lasc(\sigma)+d_\Lasc(\tau)$. \xqed{\lozenge}
	\end{rem}
	
	\begin{fct}
		\label{fct:diameter_witnessed_by_model}
		If $d_\Lasc(a,b)\leq n$, there is some $\sigma\in \Aut(\fC)$ such that $\sigma(a)=b$ and $d_\Lasc(\sigma)=n$.
		
		If $m$ enumerates a model, then whenever $\sigma\in \Aut(\fC)$ satisfies $d_\Lasc(m,\sigma(m))\leq n$, then $d_\Lasc(\sigma)\leq n+1$, i.e.\ for any tuple $a$ we have $d_\Lasc(a,\sigma(a))\leq n+1$.
		
		In particular, for every $\sigma\in\Aut(\fC)$, if $m\equiv_\Lasc\sigma(m)$, then $\sigma\in \Autf(\fC)$ and for every tuple $a$ we have $a\equiv_\Lasc\sigma(a)$.
	\end{fct}
	\begin{proof}
		The first part is immediate by the definitions of $d_\Lasc$.
		
		For the second part, note that since $d_\Lasc(m,\sigma(m))\leq n$, by the first part, there is some $\tau\in \Aut(\fC)$ with $d_\Lasc(\tau)\leq n$ such that $\tau(m)=\sigma(m)$. But then $\tau^{-1}\sigma(m)=m$, so $\tau^{-1}\sigma\in\Aut(\fC/M)$ (where $M$ is the model enumerated by $m$). It follows that $d_\Lasc(\tau^{-1}\sigma)\leq 1$, and hence $d_\Lasc(\sigma)=d_\Lasc(\tau(\tau^{-1}\sigma))\leq d_\Lasc(\tau)+d_\Lasc(\tau^{-1}\sigma)=n+1$.
		
		The third part follows immediately by the definitions and Fact~\ref{fct:Lascar_equivalent}.
	\end{proof}

	\begin{rem}
		It follows immediately from Fact~\ref{fct:Lascar_equivalent} that two tuples are $\equiv_\Lasc$-related precisely when they are in finite Lascar distance from one another. This, along with Fact~\ref{fct:diameter_witnessed_by_model}, implies that a $\sigma\in \Aut(\fC)$ is a Lascar strong automorphism if and only if $d_\Lasc(\sigma)<\infty$.\xqed{\lozenge}
	\end{rem}
	
	\begin{rem}
		The Lascar distance can also be defined in terms of the third bullet in Fact~\ref{fct:Lascar_equivalent}. The resulting metric is bi-Lipschitz equivalent to $d_\Lasc$ defined above (it follows immediately from Facts~\ref{fct:model_to_indiscernible} and \ref{fct:indisc_to_model}), and it is also sometimes called ``the Lascar distance".\xqed{\lozenge}
	\end{rem}
	
	\begin{rem}
		Fact~\ref{fct:Lascar_equivalent} and Fact~\ref{fct:distance_tdf} imply that $\equiv_\Lasc$ is $F_\sigma$ on each product of sorts.\xqed{\lozenge}
	\end{rem}

	\subsection*{Galois groups}
	
	Here, we recall fundamental facts about Galois groups of first order theories.
	
	\begin{dfn}\label{definition: Galois groups}
		\index{Gal(T)@$\Gal(T)$|see{Galois group}}
		\index{Galois group}
		The \emph{(Lascar) Galois group of $T$} is the quotient $\Gal(T)=\Aut(\fC)/\Autf(\fC)$.
		
		\index{GalKP(T)@$\Gal_\KP(T)$|see{Galois group}}
		\index{GalSh(T)@$\Gal_\Sh(T)$|see{Galois group}}
		Similarly, the \emph{Kim-Pillay Galois group} and the \emph{Shelah Galois group} of $T$ are the quotients $\Gal_\KP(T)=\Aut(\fC)/\Autf_\KP(\fC)$ and $\Gal_\Sh(T)=\Aut(\fC)/\Autf_\Sh(\fC)$ (respectively).
		
		\index{Gal0(T)@$\Gal_0(T)$}
		$\Gal_0(T)$ is the preimage in $\Gal(T)$ of the identity in $\Gal_\KP(T)$.\xqed{\lozenge}
	\end{dfn}
	
	\begin{rem}
		Note that the natural maps $\Gal(T)\to \Gal_\KP(T)\to \Gal_\Sh(T)$ are group epimorphisms.\xqed{\lozenge}
	\end{rem}
	
	\begin{rem}
		\label{rem:Lascar_extending_model}
		Note that by Fact~\ref{fct:diameter_witnessed_by_model}, if $a,a'$ are arbitrary small tuples, while $m,m'$ enumerate small models, then $ma\equiv_\Lasc m'a'$ if and only if $m\equiv_\Lasc m'$ and $ma\equiv m'a'$.\xqed{\lozenge}
	\end{rem}

	\begin{fct}
		\label{fct:sm_to_gal}
		If for some tuple $m$ enumerating a model $M$, and for some $\sigma_1,\sigma_2\in \Aut(\fC)$ we have $\tp(\sigma_1(m)/M)\equiv_\Lasc^M\tp(\sigma_2(m)/M)$, then $\sigma_1\Autf(\fC)=\sigma_2\Autf(\fC)$.
		
		Consequently, the map $S_m(M)\to \Gal(T)$ given by $\tp(\sigma(m)/M)\mapsto \sigma\Autf(\fC)$ is a well-defined surjection, and we may identify $\Gal(T)$ and $S_m(M)/{\equiv_\Lasc^M}$
		
		We also have analogous surjections onto $\Gal_\KP(T)$ and $\Gal_\Sh(T)$, which induce bijections with $S_m(M)/{\equiv_\KP}$ and $S_m(M)/{\equiv_\Sh}$ (respectively).
		
		(Throughout the thesis, we freely identify $\Gal(T)$ with $S_m(M)/{\equiv_\Lasc^M}$ for all $M$, and $\Gal_\KP(T)$ with $S_m(M)/{\equiv_\KP}$.)
	\end{fct}
	\begin{proof}
		If $\tp(\sigma_1(m)/M)\equiv_\Lasc^M\tp(\sigma_2(m)/M)$, then by the definition of $\equiv_\Lasc^M$ (see Definition~\ref{dfn:EM}), $\sigma_1(m)\equiv_\Lasc \sigma_2(m)$, so by Fact~\ref{fct:lst_witn_by_aut}, there is some $\tau\in \Autf(\fC)$ such that $\tau\sigma_1(m)=\sigma_2(m)$. Since $\Autf(\fC)$ is normal in $\Aut(\fC)$, there is some $\tau'\in \Autf(\fC)$ such that $\tau\sigma_1=\sigma_1\tau'$, so $\sigma_1\tau'(m)=\sigma_2(m)$, whence $\sigma_1^{-1}\sigma_2(m)\equiv_\Lasc m$. Since $m$ enumerates a model, by Fact~\ref{fct:diameter_witnessed_by_model}, it follows that $\sigma_1^{-1}\sigma_2\in \Autf(\fC)$, i.e.\ $\sigma_1\Autf(\fC)=\sigma_2\Autf(\fC)$.  Conversely, if $\sigma_1\Autf(\fC)=\sigma_2\Autf(\fC)$, then of course $\sigma_1(m)\equiv_\Lasc\sigma_2(m)$, so in particular, $\tp(\sigma_1(m)/M)\equiv_\Lasc^M\tp(\sigma_2(m)/M)$. Thus, we have a bijection between $\Gal(T)=\Aut(\fC)/\Autf(\fC)$ and $S_m(M)/\equiv_\Lasc^M$.
		
		The surjections onto $\Gal_\KP(T)$ and $\Gal_\Sh(T)$ are easily obtained by composing the surjection onto $\Gal(T)$ with the epimorphisms from $\Gal(T)$.
	\end{proof}

	\begin{fct}
		\label{fct:gal_top}
		The quotient topology on $\Gal(T)$ induced by the surjection $S_m(M)\to \Gal(T)$ from Fact~\ref{fct:sm_to_gal} does not depend on $M$ and it makes $\Gal(T)$ a compact (but possibly non-Hausdorff) topological group.
		
		In the same way, we induce a topology on each of $\Gal_\KP(T)$ and $\Gal_\Sh(T)$, with which the first one is a compact Hausdorff group, and the second one is a profinite group.
		
		$\Gal(T)$, $\Gal_\KP(T)$ and $\Gal_\Sh(T)$ do not depend on the choice of $\fC$ (as topological groups).
	\end{fct}
	\begin{proof}
		Choose any models $M$ and $N$, enumerated by $m$ and $n$. We may assume without loss of generality that $m\subseteq n$. Then we have a continuous surjection $S_n(N)\to S_m(M)$ (and hence, by compactness, a topological quotient map), which, by Remark~\ref{rem:Lascar_extending_model} induces a bijection $S_n(N)/{\equiv_\Lasc^N}\to S_m(M)/{\equiv_\Lasc^M}$. It follows that the induced bijection is a quotient map, and hence a homeomorphism.
		
		The fact that $\Gal(T)$ is a topological group is \cite[Lemma 18]{Zie02}.
		
		The fact that $\Gal_\KP(T)$ and $\Gal_\Sh(T)$ are topological groups follows immediately (the quotient of a topological group is a topological group). Since $\equiv_\KP$ is type-definable, it follows by Fact~\ref{fct:logic_top_cpct_T2} that $\Gal_\KP(T)$ is compact Hausdorff, and it is not hard to see that $\Gal_\Sh(T)$ is profinite.
		
		The fact that the Galois groups do not depend on the monster model follows immediately form Fact~\ref{fct:sm_to_gal} --- the space $S_m(M)$ does not depend on $\fC\succeq M$, and neither does $\equiv_\Lasc^M$.
	\end{proof}

	\begin{rem}\label{rem: GalKP is Polish}
		If $T$ and $M$ are countable, then $S_m(M)$ is a compact Polish space, so by Fact~\ref{fct: preservation of metrizability}, if $T$ is countable, $\Gal_\KP(T)$ and $\Gal_\Sh(T)$ are Polish.\xqed{\lozenge}
	\end{rem}
	
	\begin{dfn}
		\index{topology!logic}
		By the \emph{logic topology} on the Galois groups, we mean the topology induced from $S_m(M)$ via Fact~\ref{fct:gal_top}.\xqed{\lozenge}
	\end{dfn}
	
	\begin{rem}
		It follows easily from the definition of the logic topologies on Galois groups that the natural epimorphisms between them are topological group quotient maps.\xqed{\lozenge}
	\end{rem}

	\begin{prop}
		If $T$ is countable, $M$, $N$ are countable models, enumerated by $m$ and $n$, respectively, then the relations $\equiv_\Lasc^M$ and $\equiv_\Lasc^N$ on $S_m(M)$ and $S_n(N)$ (respectively) have the same Borel cardinality.
	\end{prop}
	\begin{proof}
		We may assume without loss of generality that $M\supseteq N$. Then we have a continuous surjection $S_m(M)\to S_n(N)$. From Remark~\ref{rem:Lascar_extending_model}, it follows that it is a reduction of $\equiv_\Lasc^M$ to $\equiv_\Lasc^N$, so the proposition follows by Fact~\ref{fct:borel_section}.
	\end{proof}
	This justifies the following definition.
	\begin{dfn}
		\label{dfn:bcard_galois}
		\index{Borel!cardinality!of the Galois group}
		If the theory is countable, then the \emph{Borel cardinality of the Galois group} $\Gal(T)$ is the Borel cardinality of $\equiv_\Lasc^M$ on $S_m(M)$, for some countable model $M$ enumerated by $m$.\xqed{\lozenge}
	\end{dfn}

	\begin{prop}
		\label{prop:gal_action}
		$\Gal(T)$ acts on $X/E$ for any bounded invariant equivalence relation $E$ defined on an invariant $X$. If $X=p(\fC)$ for some $p=\tp(a/\emptyset) \in S(\emptyset)$, then the orbit map $r_{[a]_E} \colon \Gal(T) \to X/E$ given by $\sigma \Autf(\fC) \mapsto [\sigma(a)]_E$ is a topological quotient map.
		
		If $E$ is type-definable (or, more generally, refined by $\equiv_\KP$), then this action factors through $\Gal_\KP(T)$ (yielding a topological quotient map $\Gal_\KP(T)\to X/E$ if $X=p(\fC)$).
	\end{prop}
	\begin{proof}
		$\Aut(\fC)$ acts on $X$, and by invariance of $E$, it also acts on $X/E$. On the other hand, by definition of $\Autf(\fC)$ (Definition~\ref{dfn:autf_L}), $\Gal(T)=\Aut(\fC)/\Autf(\fC)$ acts on $X/{\equiv_\Lasc}$. Since $E$ is bounded invariant, it is refined by $\equiv_\Lasc$, and we have a natural map $X/{\equiv_\Lasc}\to X/E$, which is trivially $\Aut(\fC)$-equivariant, so it induces an action of $\Gal(T)$ on $X/E$.
		
		If $X=p(\fC)$, this action is clearly transitive, and for any $a\models p$ and $m\supseteq a$ enumerating a model $M$, we have a commutative diagram
		\begin{center}
			\begin{tikzcd}
				S_m(M) \ar[r]\ar[d]& \Gal(T)\ar[d] \\
				S_a(M) \ar[r] & X/E
			\end{tikzcd}
		\end{center}
		where the top map comes from Fact~\ref{fct:sm_to_gal}, the left one is the restriction, the bottom one comes from Fact~\ref{fct:logic_by_type_space}, and the right one is the orbit map. Since the top, left and bottom arrows are topological quotient maps, so is the orbit map (e.g.\ by Remark~\ref{rem:commu_quot} with $A=S_m(M)$, $B=\Gal(T)$ and $C=X/E$).
		
		If $E$ is refined by $\equiv_\KP$, then we can do the same analysis with $\Gal_\KP(T)$ and $\equiv_\KP$ in place of $\Gal(T)$ and $\equiv_\Lasc$.
	\end{proof}

	The logic topology on $\Gal(T)$ can also be characterised in several different ways.
	
	\begin{fct}\label{fct: characterization of topology on Gal_L(T)}
		Denote by $\mu$ the quotient map $\Autf(\fC)\to \Gal(T)$, and take an arbitrary $C\subseteq \Gal(T)$. The following conditions are equivalent:
		\begin{enumerate}
			\item
			$C$ closed.
			\item
			For every (possibly infinite) tuple $\bar a$ of elements of $\fC$, the set $\{\sigma(\bar a)\mid \sigma \in \Aut(\fC)\;\, \mbox{and}\;\, \mu(\sigma) \in C\}$ is type-definable [over some [every] small submodel of $\fC$].
			\item
			There is some  tuple $\bar a$ and a partial type $\pi(\bar x)$ (with parameters) such that $\mu^{-1}[C]=\{ \sigma \in \Aut(\fC)\mid \sigma(\bar a) \models \pi(\bar x)\}$.
			\item
			For some tuple $\bar m$ enumerating a small submodel of $\fC$, the set $\{\sigma(\bar m)\mid \sigma \in \Aut(\fC)\;\, \mbox{and}\;\, \mu(\sigma) \in C\}$ is type-definable [over some [any] small submodel of $\fC$].
		\end{enumerate}
	\end{fct}
	\begin{proof}
		A part of this fact is contained in \cite[Lemma 4.10]{LaPi}. The rest is left as an exercise.
	\end{proof}

	\begin{fct}
		\label{fct:galkp_quot}
		$\Gal_0(T)$ is the closure of the identity in $\Gal(T)$.
	\end{fct}
	\begin{proof}
		This is essentially the first part of \cite[Theorem 21]{Zie02} (note that our $\Gal_0(T)$ is denoted there by $\Gamma_1(T)$).
	\end{proof}
	
	\begin{fct}
		\label{fct:galkpsh_orbital}
		$\equiv_\KP$ and $\equiv_\Sh$ are orbit equivalence relations of $\Autf_\KP(\fC)$ and $\Autf_\Sh(\fC)$ (respectively).
	\end{fct}
	\begin{proof}
		For $\equiv_\KP$, this is \cite[Fact 1.4(ii)]{CLPZ01}. For $\equiv_\Sh$, this is clear, as $\Autf_\Sh(\fC)$ is simply $\Aut(\fC/\acl^{\textrm{eq}}(\emptyset))$.
	\end{proof}

	The following proposition comes from \cite{KPR15} (joint with Krzysztof Krupiński and Anand Pillay).
	\begin{prop}\label{prop:lem_closed}
		Suppose $Y$ is a type-definable set which is $\equiv_\Lasc$-saturated. Then:
		\begin{enumerate}
			\item
			$\Autf(\fC)$ acts naturally on $Y$.
			\item
			The subgroup $S$ of $\Gal(T)$ consisting of all $\sigma\Autf(\fC)$ such that $\sigma[Y]=Y$ (i.e.\ the setwise stabilizer of $Y/{\equiv_\Lasc}$ under the natural action of $\Gal(T)$) is a closed subgroup of $\Gal(T)$. In particular, $\Autf_\KP(\fC)/\Autf(\fC) =\Gal_0(T) \leq S$.
			\item
			$Y$ is a union of $\equiv_\KP$-classes.
		\end{enumerate}
	\end{prop}
	\begin{proof}
		\ref{it:prop:dyn_BFT:untame} follows immediately from the assumption that $Y$ is $\equiv_\Lasc$-saturated.
		
			(2) The fact that $S$ is closed can be deduced from Fact~\ref{fct: characterization of topology on Gal_L(T)} and from the fact that this is a topological group. To see this, note that $S=P \cap P^{-1}$, where $P:= \bigcap_{a \in Y} \{ \sigma/\Autf(\fC)\mid \sigma(a) \in Y\}$ is closed in $\Gal(T)$ by Fact~\ref{fct: characterization of topology on Gal_L(T)}(3). The second part of (2) follows from the first one and the fact that $\Autf_\KP(\fC)/\Autf(\fC) = \Gal_0(T)=\cl (\id/\Autf(\fC))$.
		
		(3) is immediate from (2) and the fact that $\equiv_\KP$ is the orbit equivalence relation of $\Autf_\KP(\fC)$.
	\end{proof}
	
	\subsection*{Model-theoretic group components}
	
	\begin{fct}
		\label{fct:quot_tdgroup_topo}
		If $G$ is a type-definable group and $H\leq G$ has small index and is invariant over a small set, then $G/H$ has a well-defined logic topology (as in Definition~\ref{dfn:logic_topology}), which is compact, and is Hausdorff if and only if $H$ is type-definable.
	\end{fct}
	\begin{proof}
		Note that the coset equivalence relation of $H$ is the preimage of $H$ by the type-definable map $(g_1,g_2)\mapsto g_1^{-1}g_2$, so it is invariant over a small set and type-definable if and only if $H$ is type-definable. Thus, the fact follows from Fact~\ref{fct:logic_top_cpct_T2} (having in mind Remark~\ref{rem:logic_top_larger_sets}, we may add all parameters to the language).
	\end{proof}
	\begin{rem}
		In Fact~~\ref{fct:quot_tdgroup_topo}, in the same way, if $A$ and the theory are countable, while $G$ consists of countable tuples, then the Borel cardinality of $G/H$ is well-defined (per Fact~\ref{fct:cartdf}).
		\xqed{\lozenge}
	\end{rem}
	
	\begin{fct}
		\label{fct:quotient_by_bounded_subgroup}
		Fix a small set $A$ of parameters.
		
		Suppose $G$ is an $A$-type-definable group and $N\unlhd G$ is an $\Aut(\fC/A)$-invariant normal subgroup such that $[G:N]$ is small. Then $G/N$ is a topological group (with the logic topology).
	\end{fct}
	\begin{proof}
		To see that $G/N$ is a topological group, notice that because the map $\mu\colon (g_1,g_2)\mapsto g_1^{-1}g_2$ is type-definable (by type-definability of the group), it follows immediately that the induced map $(G/N)^2\to G/N$ is continuous with respect to the logic topology on $G^2/N^2=(G/N)^2$; we need to show that it is continuous with respect to the product topology (which \emph{a priori} might be coarser). If $N$ is type-definable, then $G/N$ and $G^2/N^2$ are compact Hausdorff, so the logic topology on $G^2/N^2$ and the product topology on $(G/N)^2$ coincide (e.g.\ by Remark~\ref{rem: continuous surjection is closed} applied to the natural bijection $G^2/N^2\to (G/N)^2$).
		
		Otherwise, if $N$ is not type-definable, because it is bounded, there is a minimal type-definable $\overline N\leq G$ containing $N$, and it is easy to see that it is $\emptyset$-type-definable and normal. Now, fix a closed $A\subseteq G/N$. Then there is a, type-definable $A'\subseteq G$ which is a union of cosets of $N$, such that $A=A'/N$. It follows that $A'$ must also be a union of cosets of $\overline N$ (because its setwise stabiliser is an intersection of type-definable unions of cosets of $N$, so it is a type-definable group). Thus, by the preceding paragraph, $\mu^{-1}[A']/{\overline N}^2$ is closed in $(G/{\overline N})^2$. Finally, note that $(G/N)^2\to (G/{\overline N})^2$ is trivially continuous, so the preimage of $\mu^{-1}[A']/{\overline N}^2$ is closed in $(G/N)^2$. But this preimage is just $\mu^{-1}[A']/N^2$, so we  are done.
	\end{proof}

	\begin{dfn}
		\index{model-theoretic connected group components}
		Suppose $A$ is a small set of parameters, and let $G$ be an $A$-type-definable group. Then the following are the classical model-theoretic \emph{connected group components} of $G$:
		\begin{itemize}
			\item
			\index{G000@$G^{000}$|see{model-theoretic connected group components}}
			$G^{000}_A$ (or $G^{\infty}_A$) is the smallest $A$-invariant subgroup of $G$ of bounded index,
			\item
			\index{G00@$G^{00}$|see{model-theoretic connected group components}}
			$G^{00}_A$ is the smallest subgroup of $G$ of bounded index which is type-definable over $A$,
			\item
			\index{G0@$G^{0}$|see{model-theoretic connected group components}}
			$G^{0}_A$ is the intersection of all relatively $A$-definable subgroups of $G$ of bounded index.
		\end{itemize}
	\end{dfn}
	
	\begin{rem}
		It is known that if $T$ has NIP (see Definition~\ref{dfn:NIP_formula_theory}), then for all small $A$ we have $G^{000}_A=G^{000}_\emptyset$, $G^{00}_A=G^{00}_\emptyset$ and $G^{0}_A=G^{0}_\emptyset$, see \cite{Gis11}. When this happens, they are called simply $G^{000}$, $G^{00}$ and $G^{0}$.\xqed{\lozenge}
	\end{rem}

	\begin{fct}
		Suppose $G$ is an $A$-type-definable group. Then $G/G^{00}_A$ is a compact Hausdorff group and $G/G^{0}_A$ is a profinite group.
	\end{fct}
	\begin{proof}
		For $G/G^{00}_A$ follows from the first part and Fact~\ref{fct:quotient_by_bounded_subgroup}. For $G/G^{0}_A$ it is similar, as $G/G^{0}_A$ is the inverse limit of quotients by finite index subgroups.
	\end{proof}
	
	\section{Former state of the art}
	In this section, having recalled the necessary language, we list some former results, along with the result in the thesis which improve them.
	
	I stress that all of the facts listed below only applied to $F_\sigma$ strong types (sometimes with additional constraints), while the main results of this thesis apply to either arbitrary or to analytic strong types (and generally speaking, in those cases where we require analyticity, there are examples when conclusion fails for arbitrary non-analytic strong types).
	
	For more detailed discussion of the historical background, see the introduction.

	\begin{fct}
		\label{fct:newelski}
		Suppose $E$ is an $F_\sigma$ equivalence relation on $X=p(\fC)$ (in particular, if $E$ is the Lascar strong type) (where $p\in S(\emptyset)$) which is not type-definable. Then for every type-definable, $E$-invariant $Y\subseteq X$ we have $\lvert Y/E\rvert\geq 2^{\aleph_0}$.
	\end{fct}
	\begin{proof}
		This is essentially \cite[Corollary 1.12]{Ne03}.
	\end{proof}
	The above fact is entirely superseded by Theorem~\ref{thm:nwg}, where we only require that $E$ is analytic, an not $F_\sigma$. See also Corollary~\ref{cor:nwg2}.
	
	Note that if the language is countable and $E={\equiv_\Lasc}$, the above corollary says that, in particular, either ${\equiv_\Lasc}\restr_{[\bar a]_{\equiv_\KP}}$ has only one class, or $\Delta_{2^\omega}$ Borel reduces to it (by Fact~\ref{fct:silver}).
	
	There is also a corresponding statement for groups
	
	\begin{fct}
		\label{fct:new_group}
		Suppose $G$ is a $\emptyset$-definable group, while $H\leq G$ is $F_\sigma$.
		
		Then either $H$ is type-definable, or $[G:H]\geq 2^{\aleph_0}$.
	\end{fct}
	\begin{proof}
		This is a part of \cite[Theorem 3.1]{Ne03}.
	\end{proof}
	The above Fact is superseded by Corollary~\ref{cor:trich_tdgroups}, as we require only that $H$ is analytic, and in the case of index smaller than $2^{\aleph_0}$, we obtain relative definability of $H$ in $G$.
	
	The following is the main theorem of \cite{KMS14}, proving an earlier conjecture from \cite{KPS13}.
	\begin{fct}
		\label{fct:KMS_theorem}
		Assume that $T$ is a complete theory in a countable language, and consider $\equiv_\Lasc$ on a product of countably many sorts. Suppose $Y$ is an $\equiv_\Lasc$-saturated, $G_\delta$ (i.e.\ the complement of an $F_\sigma$) subset of the domain of $\equiv_\Lasc$. Then either each $\equiv_\Lasc$ class in $Y$ is type-definable, or $E\restr_Y$ is non-smooth.
	\end{fct}
	\begin{proof}
		See \cite[Main Theorem A]{KMS14}.
	\end{proof}
	
	In \cite{KM14} and \cite{KR16}, the last fact was generalized to a certain wider class of bounded $F_\sigma$ relations. In order to formulate this generalization, we need to recall one more definition from \cite{KR16}.
	
	\begin{dfn}
		\label{dfn:orbital_stype}
		\index{equivalence relation!orbital}
		Suppose $E$ is an invariant equivalence relation on a set $X$. We say that $E$ is \emph{orbital} if there is a group $\Gamma\leq \Aut(\fC)$ such that $E$ is the orbit equivalence relation of $\Gamma$ acting on $X$.\xqed{\lozenge}
	\end{dfn}
	
	\begin{rem}
		Note that immediately by the definition, a bounded orbital equivalence relation is a strong type according to Definition~\ref{dfn:stype} (because $\equiv$ is the orbit equivalence relation of the whole $\Aut(\fC)$). Furthermore, each of $\equiv_\Lasc,\equiv_\KP$, and $\equiv_\Sh$ is orbital, by Facts~\ref{fct:lst_witn_by_aut} and \ref{fct:galkpsh_orbital}.\xqed{\lozenge}
	\end{rem}

	\begin{fct}\label{fct:mainA}
		We are working in the monster model $\fC$ of a complete, countable theory. Suppose we have:
		\begin{itemize}
			\item
			a set $X=p(\fC)$ for some $p\in S(\emptyset)$,
			\item
			an $F_\sigma$, bounded invariant equivalence relation $E$ on $X$, which is orbital,
			\item
			a type-definable and $E$-saturated set $Y\subseteq X$ such that $E\restr_Y$ is not type-definable.
		\end{itemize}
		Then $E\restr_Y$ is non-smooth.
	\end{fct}
	\begin{proof}
		This is essentially \cite[Theorem 3.4]{KR16}.
	\end{proof}
	
	Essentially the same was proved independently, using slightly different methods, in \cite[Theorem 3.17]{KM14}. The assumptions in \cite[Theorem 3.17]{KM14} are slightly weaker, namely $X$ is only required to be type-definable with parameters, and $Y$ is only $G_\delta$, and the orbitality assumption is replaced by a slightly weaker one. However, for $X=p(\fC)$ and type-definable $Y$, the two theorems are equivalent.
	
	Both Fact~\ref{fct:KMS_theorem} and Fact~\ref{fct:mainA} are superseded by Corollary~\ref{cor:smt_type} (albeit for type-definable $Y$): we drop orbitality assumption, and $E$ need not be $F_\sigma$.
	
	\begin{fct}
		Suppose $T$ is countable. Let $A\subseteq \fC$ be countable.
		
		Suppose in addition that $G$ is an type-definable group, and $X$ is an type-definable set of countable tuples on which $G$ acts type-definably and transitively (all with parameters in $A$), while $H\leq G$ is $F_\sigma$ over $A$ and has small index in $G$.
		
		Then if for some $a$, the orbit $H\cdot a$ is not type-definable (and $a$ satisfies a mild technical assumption), then the $H$-orbit equivalence relation on $X$ is not smooth.
	\end{fct}
	\begin{proof}
		This is essentially \cite[Theorem 3.33]{KM14}.
	\end{proof}
	This is superseded by Corollary~\ref{cor:trich+_tdf}: using that, we can drop the assumption that $H$ is $F_\sigma$ (as well as the technical assumption), and we obtain a stronger conclusion.
	
	In \cite{KM14}, the following corollary was also obtained.
	\begin{fct}
		\label{fct:KM_about_groups}
		Suppose $G$ is a definable group, while $H\leq G$ is $F_\sigma$ and invariant over a small set, of finite index in $G$. Then $H$ is definable.
	\end{fct}
	\begin{proof}
		This is \cite[Corollary 3.37]{KM14}.
	\end{proof}
	This is superseded by Corollary~\ref{cor:trich_tdgroups}: in our case, we only assume that $G$ is type-definable, and for a definable $G$, to obtain definability of $H$, we only need to assume that $H$ is analytic and of index smaller than $2^{\aleph_0}$.

	\chapter{Basic examples}
	\label{chap:toy}
	In this chapter, we discuss some cases where the analogues of some of the main theorems of the thesis are relatively easy to obtain. Roughly speaking, in proving the general theorems in later parts of this thesis, we will try to imitate the proofs from this chapter. The main problem will be ``getting into the position'' from which we can imitate successfully. The results of the third section in this chapter originate from \cite{KR18} (joint with Krzysztof Krupiński).
	\section{Transitive action of a compact group}
	\begin{prop}
		\label{prop:toy_main}
		Consider a compact Hausdorff group $G$ acting transitively and continuously on a Hausdorff space $X$ (which is also compact, as the continuous image of $G$).
		
		Let $E$ be a $G$-invariant relation on $X$.
		
		Choose any point $x_0\in X$ and let $H$ be the setwise stabiliser of $[x_0]_E$.
		
		Then $G/H$ is homeomorphic to $X/E$ (via the map induced by the orbit map $g\mapsto g\cdot x_0$).
		
		Moreover, whenever one of $H$, $E$ is open, closed, $F_\sigma$, Borel or analytic, so is the other one.
		
		Furthermore, if $G$ and $X$ are Polish, then $G/H\sim_B X/E$ (i.e.\ the Borel cardinality of the relation of lying in the same left coset of $H$ is the same as the Borel cardinality of $E$, cf.\ Definition~\ref{dfn:borel_cardinality}).
	\end{prop}
	\begin{proof}
		Consider the orbit map $R\colon G\to X$, $R(g)=gx_0$. Since $G$ acts transitively, $R$ is onto, and because $G$ is compact and $X$ is Hausdorff, it is a closed map, and as such, a topological quotient map (cf.\ Remark~\ref{rem: continuous surjection is closed}). It follows that the composed map $r\colon G\to X/E$, $r(g)=[gx_0]_E$ is also a quotient map (as the composition of two quotient maps).
		
		Then, since $E$ is $G$-invariant, $G$ acts on $X/E$, so whenever $g_1x_0\Er g_2x_0$, we have $x_0\Er g_1^{-1}g_2 x_0$, or equivalently, $g_1^{-1}g_2\in H$, which means just that $g_1H=g_2H$. It follows that fibres of $r$ are exactly the left cosets of $H$, so by the preceding paragraph, $G/H$ is homeomorphic to $X/E$.
		
		For the ``moreover" part, note that the relation $E_H$ of lying in the same left coset of $H$ is the preimage of $E$ by the continuous surjection $R\times R\colon G\times G\to X\times X$, and it is also the preimage of $H$ by the map $(g_1,g_2)\mapsto g_1^{-1}g_2$, $G^2\to G$, and apply Proposition~\ref{prop:preservation_properties}.
		
		Finally, since the fibres of $r$ are the left cosets of $H$, $R$ is a reduction of the coset equivalence relation $E_H$ to $E$, and since it is also continuous and surjective, using Fact~\ref{fct:borel_section}, we obtain $E\sim_BE_H$.
	\end{proof}
	
	The following lemma is a fairly simple example showing how Proposition~\ref{prop:trichotomy_for_groups} can be used to prove similar results in wider contexts.
	\begin{lem}
		\label{lem:abstract_trich}
		Suppose we have a compact Polish space $X$, an equivalence relation $E$ on $X$, and a compact Polish group $G$ acting transitively on $X/E$, such that for some $x_0\in X$ we have the following:
		\begin{itemize}
			\item
			the orbit map $g\mapsto g[x_0]_E$ is a topological quotient map,
			\item
			the stabiliser $H\leq G$ of $[x_0]_E$ is analytic if $E$ is analytic,
			\item
			if $E$ is smooth (in the sense of Definition~\ref{dfn:smt}), then so is $G/H$ (this is true e.g.\ if $E\geq_B G/H$).
		\end{itemize}
		Then exactly one of the following holds:
		\begin{enumerate}
			\item
			$X/E$ is finite and $E$ is clopen,
			\item
			$|X/E|=2^{\aleph_0}$ and $E$ is closed,
			\item
			$E$ is non-smooth; in this case, if $E$ is analytic, then $|X/E|=2^{\aleph_0}$.
		\end{enumerate}
		In particular, $E$ is smooth if and only if it is closed.
	\end{lem}
	\begin{proof}
		Apply Proposition~\ref{prop:trichotomy_for_groups} to $H\leq G$.
		
		If $H$ is open, then (by Fact~\ref{fct:quotient_by_closed_subgroup}) $G/H$ is discrete, and so is $X/E$. By Fact~\ref{fct:quot_T2_iff_closed}, it follows that $E$ is open (and by compactness of $X$, this implies that $E$ is clopen and $X/E$ is finite).
		
		If $H$ is closed and $[G:H]=2^{\aleph_0}$, then $G/H$ is Hausdorff (by Fact~\ref{fct:quotient_by_closed_subgroup}) and thus so is $X/E$, so $E$ is closed by Fact~\ref{fct:quot_T2_iff_closed}, and of course $\lvert X/E\rvert=[G:H]=2^{\aleph_0}$.
		
		Otherwise, $G/H$ is not smooth, so $E$ is not smooth.
		
		If $G/H$ is not smooth but $E$ is analytic, then $H$ is analytic, so it has the property of Baire, and $\lvert X/E\rvert=\lvert G/H\rvert=2^{\aleph_0}$.
	\end{proof}
	
	\begin{cor}
		\label{cor:toy_trich}
		If $G, X, E$ are as in Proposition~\ref{prop:toy_main}, and $X$ is Polish, then $E$ is smooth (according to Definition~\ref{dfn:smt}) if and only if $E$ is closed (as a subset of $X^2$).
		
		In fact, exactly one of the following holds:
		\begin{enumerate}
			\item
			$E$ is clopen and has finitely many classes,
			\item
			$E$ is closed and has exactly $2^{\aleph_0}$ classes,
			\item
			$E$ is not closed and not smooth. In this case, if $E$ is analytic, then $E$ has exactly $2^{\aleph_0}$ classes.
		\end{enumerate}
	\end{cor}
	\begin{proof}
		Note that by Proposition~\ref{prop:cont_action_factors_through_Polish} we may assume without loss of generality that $G$ is Polish (replacing it by the compact Polish group the action factors through, if necessary).
		
		Then the corollary is immediate by Proposition~\ref{prop:toy_main} and Lemma~\ref{lem:abstract_trich}.
	\end{proof}
	
	\section{Orbital equivalence relations}
	\begin{prop}
		\label{prop:toy_orbital}
		Suppose $G$ is a compact Hausdorff group acting continuously on a Hausdorff space $X$.
		
		Let $E$ be the orbit equivalence relation of some $H\leq G$ (i.e.\ $x_1\Er x_2$ if and only if for some $h\in H$ we have $hx_1=x_2$).
		
		Then $E$ is closed if and only if each class of $E$ is closed.
	\end{prop}
	\begin{proof}
		If $E$ is closed, then trivially each class of $E$ is closed.
		
		In the other direction, suppose all $E$-classes are closed and let $\tilde H$ be the group of all elements $h\in G$ such that for all $x\in X$ we have $hx \Er x$. Then it is easy to see that $\tilde H\supseteq H$, so $E$ is the orbit equivalence relation of $\tilde H$.
		
		Note that $\tilde H=\bigcap_{x\in X} \{h\mid hx\in [x]_E \}$, so it closed in $G$. As such, $\tilde H$ is a compact Hausdorff group, and so, by Fact~\ref{fct:cpct_action}, $X/E=X/\tilde H$ is Hausdorff, which implies that $E$ is closed (as the preimage of the diagonal in $X/E$ via the natural continuous map $X^2\to (X/E)^2$).
	\end{proof}
	The following corollary is a toy version of Corollary~\ref{cor:smt_cpct}.
	\begin{cor}
		\label{cor:toy_orbital}
		Suppose $G,X,E$ are as in Proposition~\ref{prop:toy_orbital}, and in addition $X$ is Polish and $H$ is normal. Then $E$ is closed if and only if it is smooth.
	\end{cor}
	\begin{proof}
		In one direction, if $E$ is closed, then it is smooth by Fact~\ref{fct:clsd_smth}.
		
		Now, suppose $E$ is smooth.
		It is easy to see that $E$ is $G$-invariant, as the orbit equivalence relation of a normal subgroup of $G$. By Corollary~\ref{cor:toy_trich}, for each $X'=G\cdot x$, $E\restr_{X'}$ is smooth if and only if it is closed. But we have trivially $E\restr_{X'}\leq_B E$, so each $E\restr_{X'}$ is smooth, and therefore closed. This implies that each $E$-class is closed, so by Proposition~\ref{prop:toy_orbital}, $E$ itself is closed as well.
	\end{proof}

	\section{Relations coarser than the Kim-Pillay strong type}\label{section: relations coarser than the Kim-Pillay strong type}
	In this section, we will discuss the model-theoretic case of bounded invariant relations coarser than the Kim-Pillay strong type, and prove an analogue of Theorem~\ref{thm:main_galois} (which is also Main~Theorem~\ref{mainthm_group_types}) and of Corollary~\ref{cor:smt_type} (i.e.\ Main~Theorem~\ref{mainthm:smt}). The main point is that --- unlike the general case --- we do not need to construct any group using topological dynamics: we can just use $\Gal_\KP(T)$ instead. This makes the problem much simpler (and quite analogous to Proposition~\ref{prop:toy_main}). Note that this approach applies to all strong types if the underlying theory is G-compact (which includes all stable and, more generally, simple theories).
	
	\begin{lem}
		\label{lem:easy}
		Suppose we have a commutative diagram
		\begin{center}
			\begin{tikzcd}
			A\ar[d, two heads]\ar[r, two heads]&G\ar[d] \\
			C\ar[r]&Q
			\end{tikzcd}
		\end{center}
		
		where:
		\begin{itemize}
			\item
			$A$, $C$ and $G$ are compact Polish spaces,
			\item
			the surjections $A \to C$ and $A \to G$ are continuous.
		\end{itemize}
		Denote by $E|_C$ and $E|_G$ the equivalence relations on $C$ and $G$ (respectively) induced by equality on $Q$. Then:
		\begin{enumerate}
			\item
			\label{it:lem:easy_transfer}
			$E|_G$ is closed [resp. Borel, or analytic, or $F_\sigma$, or clopen (equivalently, with open classes)] if and only if $E|_C$ is such,
			\item
			\label{it:lem:easy_reduc}
			$E|_G\sim _B E|_C$.
		\end{enumerate}
	\end{lem}
	\begin{proof}
		Denote by $E|_A$ the equivalence relation on $A$ induced by equality on $Q$ via the composed map $A\to Q$.
		
		\ref{it:lem:easy_transfer} Since $E|_A$ is the preimage of each of $E|_C$ and $E|_G$ by a continuous surjection between compact Polish spaces, by Proposition~\ref{prop:preservation_properties}, we conclude that closedness [resp. Borelness, or analyticity, or being $F_\sigma$, or being clopen] of $E|_A,E|_C$ and $E|_G$ are all equivalent.
		
		\ref{it:lem:easy_reduc}
		It is clear that the top and the left arrow are continuous, surjective reductions of $E|_A$ to $E|_G$ and $E|_A$ to $E|_C$, respectively. So $E|_G \sim_B E|_A \sim_B E|_C$ by Fact~\ref{fct:borel_section}.
	\end{proof}
	
	The following theorem is a prototype for Theorem~\ref{thm:main_galois}.
	\begin{thm}
		\label{thm:main_over_KP}
		Suppose $E$ is a strong type defined on $p(\fC)$ for some $p\in S(\emptyset)$ (in countably many variables, in an arbitrary countable theory) and $E$ is refined by $\equiv_\KP$. Fix any $a\models p$.
		
		Consider the orbit map $r_{[a]_E}\colon \Gal_\KP(T)\to p(\fC)/E$ given by $\sigma\Autf_\KP(\fC)\mapsto [\sigma(a)]_E$ (the orbit map of the natural action of $\Gal_\KP(T)$ on $p(\fC)/E$ introduced in Proposition~\ref{prop:gal_action}), and put $H=\ker r_{[a]_E}:=r_{[a]_E}^{-1}[[a]_E]$. Then:
		\begin{enumerate}
			\item
			$H\leq \Gal_\KP(T)$ and the fibres of $r_{[a]_E}$ are the left cosets of $H$,
			\item
			$r_{[a]_E}$ is a topological quotient mapping, and so $p(\fC)/E$ is homeomorphic to $\Gal_\KP(T)/H$,
			\item
			$E$ is type-definable [resp. Borel, or analytic, or $F_\sigma$, or relatively definable on $p(\fC) \times p(\fC)$] if and only if $H$ is closed [resp. Borel, or analytic, or $F_\sigma$, or clopen],
			\item
			$E_H\sim_B E$, where $E_H$ is the relation of lying in the same left coset of $H$.
		\end{enumerate}
	\end{thm}
	\begin{proof}
		The first two points follow from Proposition~\ref{prop:gal_action}, as $H$ is just the stabiliser of $[a]_E\in p(\fC)/E$
		
		Let $M$ be a countable model containing $a$, and let $m\supseteq a$ be an enumeration of $M$. Then we have a commutative diagram, as in the proof of Proposition~\ref{prop:gal_action}.
		\begin{center}
			\begin{tikzcd}
			S_m(M) \ar[r,two heads] \ar[d,two heads] & \Gal_\KP(T) \ar[d,"r_{[a]_E}",two heads] \\
			S_a(M) \ar[r,two heads] & {[a]}_\equiv/E
			\end{tikzcd}
		\end{center}
		The top arrow is defined in the same way as the map to $\Gal(T)$ given by Fact~\ref{fct:sm_to_gal}. The left arrow is the restriction map, and the bottom one is the quotient map given by Fact~\ref{fct:logic_by_type_space}.
		
		It is easy to check that this diagram is commutative and consists of continuous maps. Moreover, $S_m(M), S_a(M)$ and $\Gal_\KP(T)$ are all compact Polish (see Remark \ref{rem: GalKP is Polish}).
		
		Since $\Gal_\KP(T)$ is a compact Hausdorff group, we may apply Remark~\ref{rem:group_to_cosets} and Proposition~\ref{prop:preservation_properties} to deduce that $H$ is closed, clopen, Borel, $F_\sigma$, analytic if and only if $E_H$ is such.
		
		By Lemma~\ref{lem:easy}, this is equivalent to $E^M$ having the same property, and $E^M\sim_B E_H$. By Fact~\ref{fct: Borel in various senses}, we obtain (3), and by the definition of Borel cardinality of a bounded invariant equivalence relation (Definition~\ref{dfn:bier_borelcard}), we also have (4).
	\end{proof}

	It is worth noting that with some work, we can actually deduce Theorem~\ref{thm:main_over_KP} from Proposition~\ref{prop:toy_main}.
	
	More precisely, one can show that if $E$ is refined by $\equiv_\KP$, then $\Gal_\KP(T)$ acts continuously on $p(\fC)/{\equiv_\KP}$, and note that $E$ and $E|_{p(\fC)/{\equiv_\KP}}$ (i.e.\ the induced equivalence relation on ${p(\fC)/{\equiv_\KP}}$) are Borel equivalent, that $E$ is type-definable if and only if $E|_{p(\fC)/{\equiv_\KP}}$ is closed, and so on, and then apply Proposition~\ref{prop:toy_main} to $\Gal_\KP(T)$ acting on $p(\fC)/{\equiv_\KP}$.
	
	However, the general case (when $E$ is not refined by $\equiv_\KP$) does not have such a straightforward reduction, as we do not have any obvious choice of a compact Hausdorff group acting on the strong type space. To prove Theorem~\ref{thm:main_abstract} (which will be the main ingredient of the proof of Main Theorem~\ref{mainthm_group_types}), we construct another compact Polish group $\hat G$ acting on a class space instead of $\Gal_\KP(T)$, with properties similar to the action of $\Gal_\KP(T)$ above.

	The following corollary may be considered a toy version of Main~Theorem~\ref{mainthm:abstract_smt}.
	
	\begin{cor}
		\label{cor:main_over_KP}
		Assume $T$ is countable. Let $E$ be a strong type on $p(\fC)$ for some $p\in S(\emptyset)$ (in countably many variables). Assume that $E$ is coarser than $\equiv_\KP$. Then exactly one of the following conditions holds:
		\begin{enumerate}
			\item
			$p(\fC)/E$ is finite and $E$ is relatively definable,
			\item
			$|p(\fC)/E|=2^{\aleph_0}$ and $E$ is type-definable and smooth,
			\item
			$E$ is non-smooth; in this case, if $E$ is analytic, then $|p(\fC)/E|=2^{\aleph_0}$.
		\end{enumerate}
		In particular, $E$ is smooth if and only it it is type-definable.
	\end{cor}
	\begin{proof}
		Fix any $a\models p$, and let $X=p(\fC)$. Note that by Theorem~\ref{thm:main_over_KP}, we have an action of $\Gal_\KP(T)$ on $X/E$ such that that the stabiliser $H$ of $[a]_E$ is analytic if $E$ is, the orbit map of $[a]_E$ is a quotient map, and $E\sim_B G/H$. On the other hand, for any countable model $M$, $X_M$ is a compact Polish space, the quotients $X/E$ and $X_M/E^M$ are homeomorphic, $E^M$ is analytic (by Fact~\ref{fct: Borel in various senses}) if and only if $E$ is and by definition, $E^M\sim_B E$. Thus, the assumptions of Lemma~\ref{lem:abstract_trich} are satisfied for $X=X_M$, $E=E^M$ and $x_0=\tp(a/M)$ (note that since $T$ is countable, $\Gal_\KP(T)$ is a compact Polish group). Since relative definability, type-definability and smoothness of $E$ are equivalent to $E^M$ being clopen, closed or smooth (respectively) and $X/E$ and $X_M/E^M$ have the same cardinality, the conclusion follows.
	\end{proof}

	\chapter{Toolbox}
	\label{chap:toolbox}
	In this chapter, we develop some more advanced tools in topological dynamics, which are more advanced than the ones listed in Chapter~\ref{chap:prelims}. They will be to prove the main theorems in the last two sections of Chapter~\ref{chap:grouplike}. We also explore the connections between the model-theoretic notion of NIP and the dynamical notion of tameness, which will be useful mainly in Chapter~\ref{chap:applications}. Most of the content of this chapter comes from \cite{KR18} (joint with Krzysztof Krupiński).
	\section{From topological dynamics to Polish spaces}
	\label{sec:top_dyn_to_Polish}
	The main outcome of this section is the construction of a Polish compact group associated with a given metrisable dynamical system. We also obtain some more general statements useful in the non-metrisable case.
	
	Throughout this section, $G$ is an abstract group and $(G,X,x_0)$ is a (compact) $G$-ambit, i.e.\ $G$ acts on $X$ by homeomorphisms and $G\cdot x_0$ is dense in $X$.
	We use the notation of Section~\ref{sec:prel_topdyn}. In particular, we use $EL$ for the Ellis semigroup of $G$ acting on $X$, $\cM$ for a fixed minimal left ideal in $EL$, and $u$ for a fixed idempotent in $\cM$.
	
	\subsection*{Good quotients of the Ellis semigroup and the Ellis group}
	In this subsection, we find a rich Polish quotient of the Ellis group of a metric dynamical system (i.e.\ when $X$ is metrisable).
	
	We have a natural map $R\colon EL\to X$ given by $R(f)=f(x_0)$. This gives us an equivalence relation $\equiv$ on $EL$ given by $f_1\equiv f_2$ whenever $R(f_1)=R(f_2)$. Note that $R$ is continuous, so $\equiv$ is closed, and by compactness and the density of $G\cdot x_0$ in $X$, $R$ is surjective, so, abusing notation, we topologically identify $EL/{\equiv}$ with $X$. Similarly, for $A\subseteq EL$, we identify $A/{\equiv}$ with $R[A]\subseteq X$ as sets (the topology need not be the same, but it is if $A$ is closed, by compactness; whenever there is risk of confusion between distinct topologies, it will be clarified). The goal of this subsection is to find a Polish quotient of $u\cM/H(u\cM)$ which will be sufficiently well-behaved with respect to $R$.
	
	\begin{prop}
		\label{prop:commu}
		$R$ commutes with (left) multiplication in $EL$. More precisely, suppose $f_1,f_2\in EL$. Then $R(f_1f_2)=f_1(R(f_2))$. In the same way, $R$ commutes with multiplication by the elements of $G$.
	\end{prop}
	\begin{proof}
		$R(f_1f_2)=(f_1f_2)(x_0)=f_1(f_2(x_0))=f_1(R(f_2))$. From this, the second part follows, since $g\cdot f=\pi_gf$ for $g \in G$.
	\end{proof}
	
	\index{D@$D$}
	Let $D=[u]_{\equiv}\cap u\cM$. More explicitly, $D=\{f\in u\cM \mid f(x_0)=u(x_0) \}$. This $D$ will be important in much of the thesis.
	
	\begin{lem}
		\label{lem:D_closed}
		$D$ is a ($\tau$-)closed subgroup of $u\cM$ (see Fact~\ref{fct:tau_top_pre}(3)).
	\end{lem}
	\begin{proof} Consider any $d \in \cl_\tau(D)$.
		Let $(g_i),(d_i)$ be nets as in the definition of $u\circ D$, i.e.\ such that $g_i\in G$, $g_i\to u$ and $g_id_i\to d$.
		By continuity of $R$, because $R(d_i)=R(u)$ (by the definition of $D$), and by the preceding remark, as well as left continuity of multiplication in $EL$, we have
		\[
		R(d)=\lim R(g_id_i)=\lim g_iR(d_i)=\lim g_iR(u)=R(\lim g_iu)=R(u^2)=R(u).
		\]
		This shows that $D$ is $\tau$-closed.
		
		To see that $D$ is a subgroup of $u\cM$, take any $d,d_1,d_2\in D$. Then:
		\[
		R(d_1d_2)=d_1(R(d_2))=d_1(R(u))=R(d_1u)=R(d_1)=R(u),
		\]
		\[
		R(d^{-1})=R(d^{-1}u)=d^{-1}(R(u))=d^{-1}(R(d))=R(d^{-1}d)=R(u).\qedhere
		\]
	\end{proof}
	
	The following simple example shows that the subgroups $D$ and $H(u\cM)D$ do not have to be normal in $u\cM$.
	\begin{ex}
		\label{ex:D_not_normal}
		Consider $G=S_3$ acting naturally on $X=\{1,2,3\}$ (with the discrete topology), and take $x_0=1$. Then $G=u\cM$ and $D=H(u\cM)D$ is the stabilizer of $1$, which is not normal in $u\cM$. \xqed{\lozenge}
	\end{ex}
	
	\begin{lem}
		\label{lem:D_kernel_equiv}
		Let $f_1,f_2\in u\cM$. Then $f_1\equiv f_2$ (i.e.\ $R(f_1):=f_1(x_0)=f_2(x_0)=:R(f_2)$) if and only if $f_1^{-1}f_2\in D$ (note that here, $f_1^{-1}$ is the inverse of $f_1$ in $u\cM$, not the inverse function), i.e.\ $f_1D=f_2D$. (And thus $u\cM/{\equiv}$ and $u\cM/D$ can and will be identified as sets.)
	\end{lem}
	\begin{proof}
		In one direction, if $f_1\equiv f_2$,
		\[
		R(f_1^{-1}f_2)=f_1^{-1}(R(f_2))=f_1^{-1}(R(f_1))=R(f_1^{-1}f_1)=R(u).
		\]
		In the other direction, if $R(f_1^{-1}f_2)=R(u)$, then
		\[
		R(f_1)=R(f_1u)=f_1(R(u))=f_1(R(f_1^{-1}f_2))=R(f_1f_1^{-1}f_2)=R(f_2)\qedhere
		\]
	\end{proof}
	
	Recall that by Fact~\ref{fct:tau_top_pre}(8), we have the compact Hausdorff topological group $u\cM/H(u\cM)$. Since $D$ is closed in $u\cM$ (and hence compact), it follows that the quotient $H(u\cM)D/H(u\cM)$ is a closed subgroup in $u\cM/H(u\cM)$. Consequently, $u\cM/(H(u\cM)D)$ (which one may also describe as the quotient of $u\cM/H(u\cM)$ by $H(u\cM)D/H(u\cM)$, as in any case the topology is the quotient topology from $u\cM$ induced by the obvious maps) is a compact Hausdorff space (by Fact~\ref{fct:quotient_by_closed_subgroup}). By applying Lemma~\ref{lem:D_kernel_equiv}, we conclude that the quotient map $u\cM\to u\cM/(H(u\cM)D)$ factors through $u\cM/{\equiv}$, which we identify with $R[u\cM]\subseteq X$, giving us the following commutative diagram:
	\begin{center}
		\begin{tikzcd}
		u\cM\arrow[r] \arrow[d,"R"]& u\cM/H(u\cM)\arrow[d]\\
		R[u\cM] \arrow[r,"\widehat j"] & u\cM/(H(u\cM)D).
		\end{tikzcd}
	\end{center}
	
	\begin{prop}
		Suppose $\sim$ is a closed equivalence relation on a compact Hausdorff space $X$, while $F\subseteq X$ is closed. Then the set $[F]_\sim$ of all elements equivalent to some element of $F$ is also closed.\xqed{\lozenge}
	\end{prop}
	\begin{proof}
		$[F]_\sim$ is the projection of $(X\times F)\cap {\sim}$ onto the first axis.
	\end{proof}
	
	\begin{lem}
		\label{lem:jhat_cont}
		On $u\cM/{\equiv}=u\cM/D$, the topology induced from the $\tau$-topology on $u\cM$ is refined by the subspace topology inherited from $EL/{\equiv}=X$.
		
		Consequently, $\widehat j$ in the above diagram is continuous (with respect to the quotient $\tau$ topology on $u\cM/H(u\cM)D$.)
	\end{lem}
	\begin{proof}
		We need to show that if $F\subseteq u\cM$ is $\tau$-closed and right $D$-invariant (i.e.\ $FD=F$), then there is a closed $\equiv$-invariant $\widetilde F\subseteq EL$ such that $\widetilde F\cap u\cM=F$. By the preceding remark, since $\equiv$ is closed, it is enough to check that $[\bar F]_{\equiv}\cap u\cM=F$, where $\bar F$ is the closure of $F$ in $EL$.
		
		Let $f'\in [\bar F]_{\equiv}\cap u\cM$. Then we have a net $(f_i)\subseteq F$ such that $f_i\to f$ and $f\equiv f'$. By Fact~\ref{fct:tau_top_pre}(4), in this case, $f_i$ converges in the $\tau$-topology to $uf$, which is an element of $F$ (because $F$ is $\tau$-closed). Since $F$ is right $D$-invariant (and hence $\equiv$-invariant in $u\cM$), it is enough to show that $f'\equiv uf$. But this is clear since
		\[
		R(uf)=u(R(f))=u(R(f'))=R(uf')=R(f').\qedhere
		\]
	\end{proof}
	
	As indicated before, we want to find diagrams similar to the one used in Lemma~\ref{lem:easy}, which we will later use to prove Theorem~\ref{thm:main_abstract} (and indirectly, Main~Theorems~\ref{mainthm:abstract_smt}, \ref{mainthm_group_types}, \ref{mainthm:smt} and \ref{mainthm:tdgroup}). As an intermediate step, we would like to complete the following diagram (for now, we do not care about continuity).
	
	\begin{center}
		\begin{tikzcd}
		EL\arrow[d]\arrow[r,swap,outer sep=3pt,"f\mapsto fu"] & \cM\arrow[d]\arrow[r,swap,outer sep=3pt,"f\mapsto uf"] & u\cM\arrow[d] \\
		EL/{\equiv} \arrow[r,dashed]\arrow[d,equal] & \cM/{\equiv} \arrow[r,dashed]\arrow[d,equal] & u\cM/D\\[-1em]
		X& R[\cM]& {}
		\end{tikzcd}
	\end{center}
	The dashed arrow on the right exists: if $R(f_1)=R(f_2)$, then $u(R(f_1))=u(R(f_2))$, so, by Proposition~\ref{prop:commu}, also $R(uf_1)=R(uf_2)$, and hence, by Lemma \ref{lem:D_kernel_equiv}, $uf_1D = uf_2D$. Unfortunately, there is no reason for the arrow on the left to exist (i.e.\ $f_1\equiv f_2$ does not necessarily imply $f_1u\equiv f_2u$). However, we can remedy it by replacing $EL/{\equiv}$ with $EL/{\equiv'}$, where $\equiv'$ is given by $f_1\equiv' f_2$ when $R(f_1)=R(f_2)$ and $R(f_1u)=R(f_2u)$. This gives us a commutative diagram, substituting for the above one (again, not all of these arrows have to be continuous):
	\begin{center}
		\begin{tikzcd}
		EL\arrow[d]\arrow[r,swap,outer sep=3pt,"f\mapsto fu"] & \cM\arrow[d]\arrow[r,swap,outer sep=3pt,"f\mapsto uf"] & u\cM\arrow[d] \\
		EL/{\equiv'} \arrow[r]\arrow[d] & \cM/{\equiv} \arrow[r]\arrow[d,equal] & u\cM/D\\[-1em]
		X& R[\cM]& {}
		\end{tikzcd}
	\end{center}
	
	\begin{prop}\label{prop: quotients of EL are Polish}
		$EL/{\equiv}$ and $EL/{\equiv'}$ are both compact Hausdorff spaces.
		
		If $X$ is second-countable (by compactness, equivalently, Polish), so is $EL/{\equiv}$, as well as $EL/{\equiv'}$.
	\end{prop}
	\begin{proof}
		Since $EL/{\equiv}$ is homeomorphic to $X$, the part concerning $EL/{\equiv}$ is clear.
		
		For $EL/{\equiv'}$, note first that $\cM/{\equiv}$ is a closed subspace of $EL/{\equiv}$, and hence it is Polish whenever $X$ is. To complete the proof, use compactness of $EL$, Hausdorffness of $EL/{\equiv}$ and $\cM/{\equiv}$, and continuity of the diagonal map $d \colon EL\to EL/{\equiv}\times \cM/{\equiv}$ given by $f\mapsto ([f]_\equiv,[fu]_\equiv)$ in order to deduce that $EL/{\equiv'}$ is homeomorphic to $d[EL]$ which is closed.
	\end{proof}

	\begin{prop}
		\label{prop:from_cluM}
		The formula $[f]_\equiv=f(x_0)=R(f)\mapsto uf/H(u\cM)D$ describes a well-defined continuous surjection $R[\overline{u\cM}]\to u\cM/H(u\cM)D$, where $\overline{u\cM}$ is the closure in $EL$.
	\end{prop}
	\begin{proof}$\,$
		\begin{center}
			\begin{tikzcd}
			&[-1.5em] \overline{u\cM} \ar{r}\ar{d} & u\cM/H(u\cM)\ar{d} \\
			R[\overline{u\cM}]\arrow[r,equals]&\overline{u\cM}/{\equiv}\ar{r} & u\cM/H(u\cM)D
			\end{tikzcd}
		\end{center}
		
		In the above diagram, the top arrow, given by $f\mapsto uf/H(u\cM)$, is continuous by Proposition~\ref{prop:strange_cont_pre}. The induced map $\overline{u\cM}\to u\cM/H(u\cM)D$ factors through the quotient map $\overline{u\cM} \to \overline{u\cM}/{\equiv}$ yielding a continuous map $\overline{u\cM}/{\equiv}\to u\cM/H(u\cM)D$: to see that, just notice that if $f_1\equiv f_2$, then $uf_1\equiv uf_2$, and then apply Lemma~\ref{lem:D_kernel_equiv}.
	\end{proof}

	\begin{cor}
		\label{cor:uM/HuMD_Polish}
		If $X$ is metrisable, then $u\cM/H(u\cM)D$ is a Polish space.
	\end{cor}
	\begin{proof}
		Note that $\overline{u\cM}$ is a compact space (equipped with the subspace topology from $EL$). Consequently, $R[\overline{u\cM}]=\overline{u\cM}/{\equiv}$ is a compact Polish space.
		
		Hence, by Proposition~\ref{prop:from_cluM}, $u\cM/H(u\cM)D$ is a compact Hausdorff space which is a continuous image of a compact Polish space. As such, it must be Polish by Fact \ref{fct: preservation of metrizability}.
	\end{proof}

	\subsection*{Tameness and Borel ``retractions"}
	
	\begin{prop}
		\label{prop:last_borel}
		Suppose $(G,X)$ is a tame metric dynamical system. Then for any $f_0\in EL$, the map $f\mapsto f_0f$ is $\equiv$-preserving and the induced transformation of $EL/{\equiv}$ is Borel.
		
		In particular, the map $\cM/{\equiv}\to u\cM/{\equiv}$ induced by $p\mapsto up$ is Borel, where both spaces are equipped with the subspace topology from $X=EL/{\equiv}$.
	\end{prop}
	\begin{proof}
		Preserving $\equiv$ follows immediately from Proposition~\ref{prop:commu}.
		The induced transformation of $EL/{\equiv}$ is the same as simply $f_0$ once we identify $X$ with $EL/{\equiv}$, and $f_0$ is Borel by Fact~\ref{fct:tame_borel}.
	\end{proof}
	
	\begin{cor}
		\label{cor:borel_map}
		The map $\cM/{\equiv}\to u\cM/H(u\cM)D$ which takes each $[f]_\equiv$ to $uf/H(u\cM)D$ is Borel, where the former is equipped with subspace topology from $EL$, while the latter has topology induced from the $\tau$ topology.
		
		Similarly, the map $EL/{\equiv'}\to u\cM/H(u\cM)D$, given $[f]_{\equiv'}\mapsto ufu/H(u\cM)D$, is Borel.
	\end{cor}
	\begin{proof}
		The first map is the composition of the continuous map $\hat j\colon u\cM/{\equiv}\to u\cM/H(u\cM)D$ from Lemma~\ref{lem:jhat_cont} and the Borel function from the second part of Proposition~\ref{prop:last_borel}. The second map is the composition of the first one with the continuous map $[f]_{\equiv'}\mapsto [fu]_{\equiv}$.
	\end{proof}

	\subsection*{Polish group quotients of the Ellis group}
	By Corollary~\ref{cor:uM/HuMD_Polish}, we already know that for metric dynamical systems, the quotient $u\cM/H(u\cM)D$ is a Polish space. However, we want to obtain a Polish group, and as we have seen in Example~\ref{ex:D_not_normal}, $H(u\cM)D$ may not be normal, so we need to slightly refine our approach.
	
	\begin{cor}\label{cor:Polish_quotient_Core(D)}
		\index{Core(H(uM)D)@$\Core(H(u\cM)D)$}
		For a metric dynamical system, $u\cM/\Core(H(u\cM)D)$ is a compact Polish group, where $\Core(H(u\cM)D)$ is the normal core of $H(u\cM)D$, i.e.\ the intersection of all conjugates of $H(u\cM)D$ in $u\cM$.
	\end{cor}
	\begin{proof}
		Immediate by the Proposition~\ref{prop:cont_action_factors_through_Polish}, as $u\cM/H(u\cM)$ is a compact Hausdorff group and $u\cM/H(u\cM)D$ is a compact Polish space (by Corollary~\ref{cor:uM/HuMD_Polish}).
	\end{proof}
	
	In the case of \emph{tame} metric dynamical systems, the situation is a little cleaner. Namely, we will show that $u\cM/H(u\cM)$ itself is already Polish.
	
	\begin{dfn}
		\index{space!countably tight}
		A topological space $X$ has {\em countable tightness} (or is \emph{countably tight}) if for every $A\subseteq X$ and every $x\in \overline A$, there is a countable set $B\subseteq A$ such that $x\in \overline B$.
		\xqed{\lozenge}
	\end{dfn}
	
	\begin{fct}[Engelking]\label{fct:tightness}
		A compact Hausdorff topological group of countable tightness is metrisable.
	\end{fct}
	\begin{proof}
		\cite[Corollary 4.2.2]{AT08}.
	\end{proof}
	
	\begin{prop}
		\label{prop:closed_image_is_tight}
		The image of a countably tight space via a closed continuous map is countably tight.
	\end{prop}
	\begin{proof}
		Let $X$ be a countably tight space, and let $f\colon X\to Y$ be a closed and continuous surjection. Choose an arbitrary $A\subseteq Y$ and $y\in \overline A$. Note that since $f$ is closed and onto, we have that $\overline A\subseteq f\left[\overline{f^{-1}[A]}\right]$, so there is some $x\in \overline{f^{-1}[A]}$ such that $f(x)=y$. Choose $B'\subseteq f^{-1}[A]$ countable such that $x\in \overline {B'}$, and let $B=f[B']$. Since $f$ is continuous, $f^{-1}\left[\overline B\right]\supseteq \overline{B'}$, so in particular, $x\in f^{-1}\left[\overline B\right]$, so $y\in \overline B$.
	\end{proof}
	
	\begin{prop}\label{prop:NIP gives metrizability}
		If $(G,X)$ is a tame metric dynamical system, then the group $u\cM/H(u\cM)$ is metrisable (and hence a Polish group).
	\end{prop}
	\begin{proof}
		Note that if $(G,X)$ is tame, then, by Proposition~\ref{prop:dyn_BFT}, $\overline{u\cM}\subseteq EL$ is a Rosenthal compactum, so --- via the Fréchet-Urysohn property we have by Fact~\ref{fct: Rosnthal implies Frechet} --- it is countably tight. Furthermore, by Proposition~\ref{prop:strange_cont}, the function $f\mapsto uf/H(u\cM)$ defines a continuous surjection from $\overline{u\cM}$ to $u\cM/H(u\cM)$, and hence a continuous closed mapping. Hence, the result follows by Proposition~\ref{prop:closed_image_is_tight} and Fact~\ref{fct:tightness}.
	\end{proof}

	\section{Independence, tameness and ambition}\label{section: independence, tameness and ambition}
	In this section, we discuss the relationship between model-theoretic NIP and dynamical tameness. A relationship between the Bourgain-Fremlin-Talagrand dichotomy and NIP seems to have been first noticed independently in \cite{CS18}, \cite{Kha14}, and \cite{Ib16}; see also \cite{Sim15} and \cite{KhP17} for related research. Many parts of this section appear to be folklore, but I have not found them stated and proved in this form. Because of that, and because they are interesting in their own right, we present them along with their proofs. The introduced notions of tame models and ambitious models seem to be new. Ambitious models will be essential later.

	\begin{dfn}
		\index{independence property}
		\index{NIP!formula}
		\index{IP formula|see {NIP formula}}
		\label{dfn:NIP_formula_theory}
		If $A,B\subseteq \fC$, then we say that a formula $\varphi(x,y)$ has the \emph{independence property} (IP) on $A\times B$ if there is an infinite sequence $(b_n)$ of elements of $B$ such that $\varphi(\fC,b_n)\cap A$ are independent subsets of $A$. Otherwise, we say that it \emph{has NIP} on $A\times B$.
		
		We say that $\varphi$ \emph{has IP} if it has IP on the whole $\fC$, otherwise we say that it has NIP.
		
		\index{NIP!theory}
		\index{IP theory|see{NIP theory}}
		We say that \emph{$T$ has NIP} if every formula has NIP. Otherwise, we say that \emph{$T$ has IP}.
		\xqed{\lozenge}
	\end{dfn}
	
	\begin{rem}
		\label{rem:NIP_indiscernible}
		Note that if $A$ and $B$ are type-definable, then in the above definition we can assume without loss of generality that the sequence $(b_n)$ is indiscernible over any given small set of parameters (by Ramsey's theorem and compactness).\xqed{\lozenge}
	\end{rem}
	
	\begin{dfn}
		\label{dfn:fixing_inv_set}
		\index{Aut(M/\{A\})@$\Aut(M/\{A\})$}
		Given a model $M\preceq \fC$ and a set $A$ which is invariant over $M$, by $\Aut(M/\{A\})$ we denote the set of all automorphisms of $M$ which fix the canonical parameter of $A$ (as a hyperimaginary), or equivalently, which fix the set $A_M:=\{\tp(a/M)\mid a\in A \}$ setwise. \xqed{\lozenge}
	\end{dfn}
	
	\begin{dfn}
		\index{tame!formula}
		\label{dfn:tame_formula}
		We say that a formula $\varphi(x,y)$ is \emph{tame} if for every small model $M$ and $b\in M$, the characteristic function of $[\varphi(x,b)]\subseteq S_x(M)$ is tame in $(\Aut(M),S_x(M))$ (in the sense of Definition~\ref{dfn:tame_function_system}).
		
		Similarly, if $A$, $B$ are type-definable sets, we say that $\varphi(x,y)$ is {\em tame} on $A\times B$ if for every small model $M$ over which $A$ and $B$ are type-definable, and every $b\in B(M)$, the characteristic function of $[\varphi(x,b)]\cap A_M\subseteq A_M$ is tame in $(\Aut(M/\{A\}),A_M)$.
		\xqed{\lozenge}
	\end{dfn}
	
	Note that tameness of $\varphi(x,y)$ does not change when we add dummy variables, even allowing infinite sequences of variables.
	
	\begin{lem}\label{lem:NIP_tame}
		[For any type-definable sets $A,B$] $\varphi(x,y)$ is NIP [on $A\times B$] if and only if $\varphi(x,y)$ is tame [on $A\times B$].
	\end{lem}
	\begin{proof}
		For simplicity, we will treat the absolute case here. The relative (i.e.\ $A\times B$) case is proved similarly.
		
		If $\varphi(x,y)$ has IP, there is an indiscernible sequence $(b_n)$ witnessing that, and we can find a small model $M$ which contains $(b_n)$, and such that
		all $b_n$'s lie in a single orbit under $\Aut(M)$.
		It follows from Fact~\ref{fct:ind_untame} that $\varphi$ is untame (which is witnessed in $(\Aut(M),S_x(M))$).
		
		In the other direction, suppose $\varphi(x,y)$ is untame. Fix a small model $M$ and $b\in M$ witnessing that. Then we have a sequence $(\sigma_n)_n$ in $\Aut(M)$ such that $\sigma_n\cdot \chi_{[\varphi(x,b)]}$ is an $\ell^1$ sequence.
		
		Let $\Sigma\leq \Aut(M)$ be the group generated by all $\sigma_n$'s and put $B_0:=\Sigma\cdot b$. Then $B_0$ is countable and $S_{\varphi}(B_0)$ is a totally disconnected, compact metric space. Moreover, the characteristic function of $[\varphi(x,b)]\subseteq S_\varphi(B_0)$ is untame with respect to $(\Sigma,S_{\varphi}(B_0))$. Then, by Prop~\ref{prop:dyn_BFT}, there is a $\varphi$-formula $\psi$ with IP. Since NIP is preserved by Boolean combinations, it follows that $\varphi$ has IP.
	\end{proof}
	
	\begin{rem}
		Lemma~\ref{lem:NIP_tame} is basically equivalent to \cite[Corollary 3.2]{Ib16} (though the latter uses a slightly different language).
		There is also an analogous equivalence between stability and the so-called WAP property of a function in a dynamical system (see e.g.\ \cite{BT16}).\xqed{\lozenge}
	\end{rem}
	
	\begin{lem}
		\label{lem:NIP_local}
		Suppose $\varphi(x,y)$ has IP on $A\times B$, where $A,B$ are type-definable over a small set $C$ of parameters. Then there are $p,q\in S(C)$ such that $p\proves A$, $q\proves B$ and $\varphi(x,y)$ has IP on $p(\fC)\times q(\fC)$.
	\end{lem}
	
	\begin{proof}
		As noticed before, we can choose $(b_n)_{n \in \omega} \subseteq B$ indiscernible over $C$ and such that $\varphi(\fC,b_n) \cap A$ are independent subsets of $A$. So we can choose $a \in A$ such that $\varphi(a,b_n)$ holds if and only if $n$ is even. It is easy to check that $p := \tp(a/C)$ and $q:=\tp(b_0/C)$ satisfy our requirements.
	\end{proof}

	\begin{dfn}
		\label{dfn:tame_model}
		\index{model!tame}
		\index{tame!model}
		We say that $M$ is a \emph{tame model} if for some (equivalently, every) enumeration $m$ of $M$, the system $(\Aut(M),S_m(M))$ is tame.
		\xqed{\lozenge}
	\end{dfn}
	
	\begin{cor}
		\label{cor:NIP_char}
		Let $T$ be any theory. Then the following are equivalent:
		\begin{enumerate}
			\item
			\label{it:cor:NIP_char:NIP}
			$T$ has NIP.
			\item
			\label{it:cor:NIP_char:tame_fla}
			Every formula $\varphi(x,y)$ is tame.
			\item
			\label{it:cor:NIP_char:tame_vars}
			For every small model $M$ and a small tuple $x$ of variables, the dynamical system $(\Aut(M),S_{x}(M))$ is tame.
			\item
			\label{it:cor:NIP_char:tame_elts}
			For every small model $M$ and a small tuple $a$ of elements of $\fC$, the dynamical system $(\Aut(M),S_{a}(M))$ is tame.
			\item
			\label{it:cor:NIP_char:tame_models}
			Every small model of $T$ is tame.
		\end{enumerate}
		Moreover, in \ref{it:cor:NIP_char:tame_vars}--\ref{it:cor:NIP_char:tame_models}, we can replace ``every small model" with ``every model of cardinality $\lvert T\rvert$", and ``small tuple" with ``finite tuple".
	\end{cor}
	\begin{proof}
		The equivalence of \ref{it:cor:NIP_char:NIP} and \ref{it:cor:NIP_char:tame_fla} is immediate by Lemma~\ref{lem:NIP_tame}.
		
		To see that \ref{it:cor:NIP_char:tame_fla} is equivalent to \ref{it:cor:NIP_char:tame_vars}, note that by Corollary~\ref{cor:tame_dense}, tameness can be tested on characteristic functions of clopen sets, so tameness of $(\Aut(M),S_{x}(M))$ follows from tameness of formulas.
		
		Similarly, \ref{it:cor:NIP_char:tame_fla} is equivalent to \ref{it:cor:NIP_char:tame_elts}, because by Lemmas~\ref{lem:NIP_tame} and \ref{lem:NIP_local}, we can test tameness on complete types.
		
		Finally, \ref{it:cor:NIP_char:tame_elts} trivially implies \ref{it:cor:NIP_char:tame_models}.
		In the other direction, if $(\Aut(M),S_a(M))$ is untame and we choose $N\succeq M$ such that $a\in N$ and $N$ is strongly $\lvert M\rvert^+$-homogeneous, then also $(\Aut(N),S_n(N))$ is untame (by Fact \ref{fct:tame_preserved}), where $n$ is an enumeration of $N$.
		
		For the ``moreover" part, for tuples, it is trivial (untameness is witnessed by formulas, and formulas have finitely many variables).
		For models, suppose that $T$ has IP, i.e.\ some formula $\varphi(x,y)$ has IP. By Lemma~\ref{lem:NIP_local}, $\varphi(x,y)$ has IP on $p(\fC) \times \fC$ for some $p \in S(\emptyset)$. Take $a \models p$. The proof of $(\leftarrow)$ in Lemma~\ref{lem:NIP_tame} easily yields a model $M$ of cardinality $|T|$, containing $a$, and such that $(\Aut(M),S_a(M))$ is untame for $a \models p$. Then, by Fact \ref{fct:tame_preserved}, the systems $(\Aut(M),S_x(M))$ and $(\Aut(M),S_m(M))$ are untame as well, where $m$ is an enumeration of $M$.
	\end{proof}

	In the $\omega$-categorical case, we obtain a simpler characterization of NIP.
	\begin{cor}
		Suppose $T$ is a countable $\omega$-categorical theory. The following are equivalent:
		\begin{itemize}
			\item
			$T$ has NIP,
			\item
			the countable model of $T$ is tame.
		\end{itemize}
		More generally, a theory $T$ is NIP if and only if it has a tame, $\aleph_0$-saturated, strongly $\aleph_0$-homogeneous model.
	\end{cor}
	\begin{proof}
		The main part is immediate by Corollary~\ref{cor:NIP_char}.
		Then implication $(\rightarrow)$ in the ``more general" case also follows from Corollary~\ref{cor:NIP_char} (and the existence of $\aleph_0$-saturated and strongly homogeneous models). In the other direction, we argue as in the ``moreover" part of Corollary~\ref{cor:NIP_char}, noticing that $\aleph_0$-saturation and strong $\aleph_0$-homogeneity of $M$ allow us to use $M$ in that argument.
	\end{proof}
	
	\begin{dfn}
		\label{dfn:NIP_set}
		\index{NIP!set}
		If $Y$ is a type-definable set (with parameters), we say that $Y$ is NIP if for every small product of sorts $Z$, every formula $\varphi(y,z)$ is NIP on $Y\times Z$. If $Y$ does not have NIP, we say that it has IP.\xqed{\lozenge}
	\end{dfn}
	
	\begin{cor}
		\label{cor:NIP_implies_tame}
		If $T$ has NIP, then for every small model $M\preceq \fC$ and tuple $a\in \fC$, the dynamical system $(\Aut(M),S_a(M))$ is tame.
		
		More generally, if $T$ is arbitrary, $M$ is a small model and $Y$ is type-definable over $M$ and NIP, then $(\Aut(M/\{Y\}),Y_M)$ is tame (cf.\ Definition~\ref{dfn:fixing_inv_set}).
	\end{cor}
	\begin{proof}
		The first part is contained in Corollary~\ref{cor:NIP_char}. The second part follows similarly from Lemma~\ref{lem:NIP_tame} and Corollary~\ref{cor:tame_dense}.
	\end{proof}
	
	We introduce the following definition.
	\begin{dfn}
		\index{model!ambitious}
		\label{dfn:ambitious_model}
		We say that $M$ is an \emph{ambitious model} if for some (equivalently, for every) enumeration $m$ of $M$, the $\Aut(M)$-orbit of $\tp(m/M)$ is dense in $S_m(M)$ (i.e.\ $(\Aut(M),S_m(M),\tp(m/M))$ is an ambit).
		
		Given a subgroup $G^Y\leq \Gal(T)$, we say that $M$ is \emph{ambitious relative to $G^Y$} if it is ambitious and for $G^Y(M)=\{\sigma\in \Aut(M) \mid [\tp(\sigma(m)/M)]_{\equiv_\Lasc}\in G^Y \}$ (which may also be described as the set of $\sigma\in \Aut(M)$ such that for some (equivalently, every) global extension $\bar\sigma\supseteq \sigma$ we have $\bar\sigma\Autf_\Lasc(\fC)\in G^Y$), the orbit $G^Y(M)\cdot \tp(m/M)$ is dense in $Y'_M$, where $Y':=\{n\in [m]_{\equiv}\mid [n]_{\equiv_\Lasc}\in G^Y \}$ (remember that we identify $[m]_\equiv/{\equiv_\Lasc}$ with $\Gal(T)$).
		\xqed{\lozenge}
	\end{dfn}
	
	\begin{prop}
		\label{prop:amb_exist}
		Any set $A\subseteq \fC$ is contained in an ambitious model $M$ of cardinality $\lvert A\rvert+\lvert T\rvert+\aleph_0$.
		
		More generally, if $G^Y$ is a subgroup of $\Gal(T)$, then we can find such $M$ which is ambitious relative to $G^Y$.
	\end{prop}
	\begin{proof}
		Put $\kappa=\lvert A\rvert+\lvert T\rvert+\aleph_0$. Extend $A$ to some $M_0\preceq \fC$ of cardinality $\kappa$, enumerated by $m_0$. The weight of $S_{m_0}(M_0)$ is at most $\kappa$, so it has a dense subset of size at most $\kappa$, so we can find a group $\Sigma_0\leq \Aut(\fC)$ of size $\kappa$ such that the types over $M_0$ of elements of $\Sigma_0\cdot m_0$ form a dense subset of $S_{m_0}(M_0)$.
		
		Then we extend $M_0$ to a setwise $\Sigma_0$-invariant $M_1\preceq \fC$: namely, we can extend $\Sigma_0\cdot M_0$ to a model $M_0^{1}\preceq \fC$, and then extend $\Sigma_0\cdot M_0^{1}$ to $M_0^2\preceq \fC$ and continue. After countably many steps, we take the union of the elementary chain $M_1=\bigcup_kM_0^k$, and it will be $\Sigma_0$-invariant.
		
		Then we continue, finding an appropriate $\Sigma_1 \supseteq \Sigma_0$ and a $\Sigma_1$-invariant $M_2\preceq \fC$, and so on. Then $M=\bigcup_n M_n$ satisfies the conclusion: if we take $\Sigma=\bigcup_n \Sigma_n$, then $M$ is $\Sigma$-invariant, so $\Sigma\restr_M\leq \Aut(M)$ and $\Sigma\restr_{M}\cdot \tp(m/M)$ is dense in $S_m(M)$.
		
		For the ``more generally'' part, the proof is analogous, only each time we choose $\Sigma_n$, we ensure that it contains enough $\sigma\in G^Y(\fC)$ to witness the appropriate density condition. It works in the end because if $\sigma'\in \Aut(\fC)$ restricts to an automorphism of $M$ (i.e.\ fixes $M$ setwise), then $\sigma'\Autf(\fC)\in G^Y$ if and only if $\sigma'\restr_M\in G^Y(M)$.
	\end{proof}
	
	\begin{rem}
		Alternatively, one can show that if $M$ is a model which together with some group $\Sigma$ acting on it by automorphisms satisfies $(M,\Sigma)\preceq (\fC,\Aut(\fC))$, then $M$ is ambitious, whence the first part of Proposition~\ref{prop:amb_exist} follows from the downward Löwenheim-Skolem theorem.
		\xqed{\lozenge}
	\end{rem}
	
	\begin{rem}
		One can also show that every strongly $\aleph_0$-homogeneous and $\aleph_0$-saturated model is ambitious.\xqed{\lozenge}
	\end{rem}

	One might ask whether we can extend Corollary~\ref{cor:NIP_char} to say that $T$ has NIP if and only if $T$ has a tame ambitious model --- we know that this is the case if $T$ is $\omega$-categorical, but the following example shows that it is not enough in general.
	
	\begin{ex}
		Suppose $M=\dcl(\emptyset)$ is a model (this is possible in an IP theory: for instance if we name all elements of a fixed model of an arbitrary IP theory).
		
		Then $S_{m}(M)$ is a singleton, so $M$ is trivially tame and ambitious.\xqed{\lozenge}
	\end{ex}
	However, any example of this sort will be G-compact, so in this case the the main result (Theorem~\ref{thm:main_galois}) essentially reduces to Theorem~\ref{thm:main_over_KP}, which is simpler by far to prove, and as such, not interesting from the point of view of the following analysis. This leads us to the following question.
	\begin{ques}
		\label{ques:non-g-cpct_tame}
		Is there a countable theory $T$ which is IP but not G-compact, such that some countable $M\models T$ is tame and ambitious?
		\xqed{\lozenge}
	\end{ques}

	\chapter{Group-like equivalence relations}
	\label{chap:grouplike}
	In this chapter, we introduce the notions of group-like and weakly group-like equivalence relations, along with various strengthenings of both. We prove many of their properties, including abstract Theorems~\ref{thm:general_cardinality_intransitive}, \ref{thm:general_cardinality_transitive} and \ref{thm:main_abstract}, which are the formal statements behind Main~Theorems~\ref{mainthm:abstract_card} and \ref{mainthm:abstract_smt}, and which will be later used to prove (the precise versions of) Main~Theorems~\ref{mainthm_group_types}, \ref{mainthm:smt} and \ref{mainthm:nwg}.
	
	Throughout the chapter, unless noted otherwise, we have a fixed $G$-ambit $(X,x_0)$ along with an equivalence relation $E$ on $X$. $EL=E(G,X)$ is the enveloping semigroup of $(G,X)$, $\cM$ is a fixed minimal (left) ideal in $EL$, and $u\in \cM$ is an idempotent.
	
	\section{Closed group-like equivalence relations}
	\begin{dfn}
		\label{dfn:glike}
		\index{equivalence relation!group-like}
		Let $E$ be an equivalence relation on $X$. We say that $E$ is \emph{$G$-group-like} (or just \emph{group-like}) if it is $G$-invariant and the partial operation given by formula
		\[
		[gx_0]_E\cdot [x]_E=g[x]_E,
		\]
		for all $g\in G$ and $x\in X$, extends to a group operation on $X/E$, turning it into (possibly non-Hausdorff) topological group (with the quotient topology).
		
		(It follows that $X/E$ is a topological group (with the quotient topology) and $g\mapsto [gx_0]_E$ is a well-defined group homomorphism.)\xqed{\lozenge}
	\end{dfn}
	Note that when discussing particular group-like equivalence relation, we have in mind a specific group structure on $X/E$. In general, it may be not unique, as we will see in Example~\ref{ex:nonunique_structure}.
	
	\begin{rem}
		Note that the definition of group-likeness implies that $[x_0]_E=[e\cdot x_0]_E$ is the identity in $X/E$.\xqed{\lozenge}
	\end{rem}
	
	\begin{ex}
		\label{ex:group_glike}
		Consider the action of a compact Hausdorff group $G$ on itself by left translations, so that $X=G$ and take $x_0=e$. Then given any normal $N\unlhd G$, the relation $E=E_N$ of lying in the same coset of $N$ is group-like, and it is closed if and only if $N$ is closed. (See Example~\ref{ex:cpct_glike} for more details.)\xqed{\lozenge}
	\end{ex}

	\begin{prop}
		An equivalence relation $E$ is $G$-group-like if and only if $X/E$ has a topological group structure (with the induced topology) and for every $x$ we have $gx\in [gx_0]_E\cdot [x]_E$.
	\end{prop}
	\begin{proof}
		Suppose $E$ is group-like, so we have a group structure on $X/E$ witnessing it. Choose any $g,x$. Then by the assumption $[gx]_E=g[x]_E=[gx_0]_E\cdot [x]_E$, which means that $gx\in [gx_0]_E\cdot [x]_E$.
		
		In the other direction, suppose for all $g,x$ we have that $gx\in [gx_0]_E\cdot [x]_E$. Then, in particular, whenever $[x_1]_E=[x_2]_E$, we have that $gx_1\in [gx_0]_E\cdot [x_1]_E$ and $gx_2\in [gx_0]_E\cdot[x_2]_E$, and since the classes on the right hand side are equal, it follows that $gx_2\Er gx_1$, so $E$ is $G$-invariant. Thus $g[x]_E=[gx]_E$, so $g [x]_E\subseteq [gx_0]_E\cdot [x]_E$, and in fact $g[x]_E=[gx_0]_E\cdot [x]_E$ (because the latter is a single $E$-class).
	\end{proof}
	
	\begin{prop}
		If $E$ is group-like, then for all $g\in G$ we have that whenever $gx_0\Er x_0$, then $gx\Er x$ for all $x\in X$. In particular, the stabilizer of $[x_0]_E$ is normal in $G$.
	\end{prop}
	\begin{proof}
		If $[gx_0]=[x_0]_E=[e_G x_0]_E$, then by group-likeness $g [x]_E=[e_G x_0]_E\cdot [x]_E=e_G[x]_E=[x]_E$.
	\end{proof}
	
	\begin{prop}
		\label{prop:closed_invariant}
		If $E$ is a closed, $G$-invariant equivalence relation on $X$, then it is also $EL$-invariant.
	\end{prop}
	\begin{proof}
		If $f=\lim g_i$, then for any $x_1 \Er x_2$ we have that $g_i(x_1)\Er g_i(x_2)$, so by closedness $f(x_1)\Er f(x_2)$.
	\end{proof}
	
	\begin{dfn}
		\index{r@$r$}
		\index{R@$R$}
		In the remainder of this chapter, we will denote by $R$ the orbit map $EL\to X$, $R(f)=f(x_0)$, an by $r$ the induced map $EL\to X/E$, $r(f)=[f(x_0)]_E$.\xqed{\lozenge}
	\end{dfn}
	
	\begin{lem}
		\label{lem:closed_group_like}
		Suppose $E$ is a closed, group-like equivalence relation on $X$.
		
		Then:
		\begin{enumerate}
			\item
			$E$ is $EL$-invariant,
			\item
			$r\colon EL\to X/E$, $r(f)=[f(x_0)]_E$ is a semigroup homomorphism,
			\item
			$r\restr_{u\cM}$ is onto and a topological quotient mapping (with $u\cM$ equipped with the $\tau$-topology),
			\item
			$H(u\cM)\leq \ker r$ and the induced map $u\cM/H(u\cM)\to X/E$ is a topological group quotient mapping.
		\end{enumerate}
	\end{lem}
	\begin{proof}
		Note that because $E$ is closed, by Fact~\ref{fct:quot_T2_iff_closed}, $X/E$ is Hausdorff, so convergent nets have unique limits.
		
		(1) is immediate by Proposition~\ref{prop:closed_invariant}.
		
		To see that $r$ is a homomorphism, choose any $f_1,f_2\in EL$. We know that $R(f_1f_2)=f_1R(f_2)$. By (1), it follows that $r(f_1f_2)=f_1r(f_2)$. Let $g_i$ converge to $f_1$. Then $f_1r(f_2)=\lim (g_i\cdot r(f_2))$. But by the assumption $g_i\cdot r(f_2)=r(g_i)\cdot r(f_2)$, so by continuity of multiplication and $r$ we have that $\lim(g_ir(f_2))=\lim(r(g_i)r(f_2))=(\lim r(g_i))r(f_2)=r(f_1)r(f_2)$.
		
		Note that $r\restr_{u\cM}$ is surjective, because $u\cM=uELu$, $X/E$ is a group and $r$ is surjective (so $r(u\cM)=r(u)\cdot (X/E)\cdot r(u)=X/E$).
		
		To prove that $r\restr_{u\cM}$ is continuous in the $\tau$ topology, we show that if $F\subseteq X/E$ is closed, then $(u\circ r^{-1}[F])\cap u\cM=r^{-1}[F]\cap u\cM$. Note first that of course $r$ is continuous (as a map $EL\to X/E$). Take any net $(f_i)_i$ in $r^{-1}[F]$ and $g_i\to u$ such that $(g_if_i)_i$ converges to some $h\in u\cM$. We want to show that $r(h)\in F$. Passing to a subnet if necessary, we can assume that $f_i$ converges to some $f\in r^{-1}[F]$. Then we have (by continuity) that $r(h)=\lim r(g_if_i)=\lim (r(g_i)r(f_i))=r(u)r(f)=r(f)$ (because $r(u)$ is the identity, since it is the only idempotent in a group).
		
		In conclusion, $r\restr_{u\cM}\to X/E$ is a continuous surjection from a compact space to a Hausdorff space, and thus it is closed, and in particular a quotient mapping.
		
		The last point follows immediately from Corollary~\ref{cor:H(G)_universal} and the second and third points above.
	\end{proof}
	Note that Lemma~\ref{lem:closed_group_like} implies that if $E$ is closed, then the group structure on $X/E$ witnessing its group-likeness is unique.

	\begin{prop}
		\label{prop:closure_grouplike}
		Suppose $E$ is group-like. Then $\bar E$ defined as $x_1\mathrel{\bar E} x_2$ when $\overline{\{[x_1]_E\}}=\overline{\{[x_2]_E\}}\subseteq X/E$ is a closed group-like equivalence relation Furthermore, $\bar E$ is the finest closed equivalence relation coarser than $E$.
	\end{prop}
	\begin{proof}
		Note that $X/{\bar E}$ is simply the quotient of the topological group $X/E$ by the closure of the identity $[x_0]_E$. As such (by Fact~\ref{fct:quotient_by_closed_subgroup}), it is a Hausdorff group and $X/E\to X/{\bar E}$ is a homomorphism. The conclusion follows.
		
		For the ``furthermore'' part, note that if $F\supseteq E$ is closed, then by Fact~\ref{fct:quot_T2_iff_closed}, $X/F$ is Hausdorff. It follows that the preimage of any point by the map $X/E\to X/F$ is closed in $X/E$, which implies that any class of $F$ contains a class of $\bar E$, which completes the proof.
	\end{proof}
	
	\begin{cor}
		\label{cor:r_uM_cont}
		If $E$ is a group-like equivalence relation (not necessarily closed), then $r\restr_{u\cM}$ is continuous.
	\end{cor}
	\begin{proof}
		Let $\bar E$ be as in Proposition~\ref{prop:closure_grouplike}. Write $r_{\bar E}$ for the induced map $EL\to X/{\bar E}$. Consider the commutative diagrams:
		\begin{center}
			\begin{tabular}{lcr}
				\begin{tikzcd}
				u\cM\arrow[d,"r\restr_{u\cM}"]\arrow[dr,"r_{\bar E}\restr_{u\cM}"] & \\
				X/E\arrow[r] & X/{\bar E}
				\end{tikzcd}
				&&
				\begin{tikzcd}
				EL\arrow[d,"r"]\arrow[dr,"r_{\bar E}"] & \\
				X/E\arrow[r] & X/{\bar E}
				\end{tikzcd}
			\end{tabular}
		\end{center}
		Let $F\subseteq X/E$ be closed. Then $F$ is $\overline{[x_0]_E}$-invariant (because it contains the closure of each of its points), so $F/\overline{[x_0]_E}$ is closed, and $r^{-1}[F]=r_{\bar E}^{-1}[F/\overline{[x_0]_E}]$. But since (by Proposition~\ref{prop:closure_grouplike}) $\bar E$ is closed and group-like, we know by Lemma~\ref{lem:closed_group_like} that $r_{\bar E}\restr_{u\cM}$ is continuous, so $r\restr_{u\cM}^{-1}[F]=r_{\bar E}\restr_{u\cM}^{-1}[F/\overline{[x_0]_E}]$ is closed.
	\end{proof}

	\section{Properly group-like equivalence relations}
	We have seen that if $E$ is closed group-like, then the group structure on $X/E$ is determined uniquely. In general, this need not be true (so in particular, we cannot hope to have a homomorphism as in Lemma~\ref{lem:closed_group_like}), as the following example shows.
	
	\begin{ex}
		\label{ex:nonunique_structure}
		Let $G={\bQ}$ act on the circle $X=\bR/\bZ$ with $x_0=0+\bZ$ by addition; clearly,  $(G,X,x_0)$ is an ambit. Then if we take for $E$ the relation of lying in the same orbit of $G$, then as a topological space, $X/E=X/G=\bR/\bQ$ is a space of cardinality $2^{\aleph_0}$ with trivial (antidiscrete) topology, and $[gx_0]_E=\bQ$ for all $g\in G$. Thus, any group structure on $\bR/\bQ$ such that $\bQ$ is the identity witnesses group-likeness of $E$, and of course there is a large number of such structures.\xqed{\lozenge}
	\end{ex}
	
	Since we do want to treat relations which are not necessarily closed, and still recover a statement in the spirit of Lemma~\ref{lem:closed_group_like}, we impose further restrictions on the group structure of the quotient.
	
	\begin{dfn}
		\label{dfn:prop_glike}
		\index{equivalence relation!group-like!properly}
		We say that an equivalence relation $E$ on $X$ is \emph{properly group-like} if it is group-like and there is a group $\tilde G$, an equivalence relation $\equiv$ on it, an identification of $X$ with $\tilde G/{\equiv}$, such that:
		\index{G@$\tilde G$}
		\begin{figure}[H]
			\begin{tikzcd}
				G \arrow[r]\arrow[dr]& EL=EL(G,X)\arrow[d,"R"]\arrow[dr,"r"]& \\
				\tilde G \arrow[r,"\tilde g\mapsto {[\tilde g]}_\equiv"] & X=\tilde G/{\equiv} \arrow[r] & \tilde G/N=X/E
			\end{tikzcd}
		\end{figure}
		\begin{itemize}
			\item
			$\tilde g\mapsto [[\tilde g]_\equiv]_E$ is a group homomorphism (for brevity, we will denote it by $\tilde r$),
			\item
			(pseudocompleteness) whenever $(g_i)$ and $(p_i)$ are nets in $G$ and $X$ (respectively) such that $g_i\cdot x_0\to x_1$, $p_i\to x_2$ and $g_i\cdot p_i\to x_3$ for some $x_1,x_2,x_3\in X$, there are $\tilde g_1,\tilde g_2\in \tilde G$ such that $[\tilde g_1]_\equiv=x_1$, $[\tilde g_2]_\equiv=x_2$ and $[\tilde g_1\tilde g_2]_\equiv=x_3$,
			\item $F_0=\{[\tilde g_1^{-1}\tilde g_2]_{\equiv}\mid \tilde g_1\equiv \tilde g_2\}$ is closed in $X$.\xqed{\lozenge}
		\end{itemize}
	\end{dfn}
	
	\begin{ex}
		The $E_N$ from Example~\ref{ex:group_glike} is actually properly group-like: indeed, we can just take $\tilde G=G$ with ${\equiv}$ being just equality in $\tilde G=G$. (See Example~\ref{ex:cpct_glike} for more details.).\xqed{\lozenge}
	\end{ex}
	
	\begin{ex}
		\label{ex:closed_glike_is_properly_glike}
		If $E$ is closed group-like, then it is properly group-like. (See Proposition~\ref{prop:closed_glike_is_properly_glike}.)\xqed{\lozenge}
	\end{ex}
	
	In the remainder of this section, unless we specify otherwise, $E$ is a properly group-like equivalence relation on $X$, and we fix $\tilde G$ and $\equiv$ witnessing that.
	
	\begin{prop}
		\label{prop:approx}
		For every $f\in EL$ and $x\in X$ there are $\tilde g_1,\tilde g_2\in \tilde G$ such that we have $[\tilde g_2]_\equiv=x$, $f(x)=[\tilde g_1\tilde g_2]_\equiv$ and $[\tilde g_1]_\equiv=f(x_0)$.
	\end{prop}
	\begin{proof}
		Immediate by pseudocompleteness: take for $(g_i)_i$ a net convergent to $f$ and for $(p_i)_i$ a constant net with all $p_i$ equal to $x$.
	\end{proof}

	\begin{lem}
		\label{lem:r_is_homomorphism}
		$r\colon EL\to X/E$ is a semigroup epimorphism.
	\end{lem}
	\begin{proof}
		The fact that $r$ is onto is trivial, because $R$ is onto.
		
		Take any $f_1,f_2\in EL$. Let $\tilde g_1,\tilde g_2\in \tilde G$ be such that $[\tilde g_1]_\equiv=f_1(x_0)$, $[\tilde g_1\tilde g_2]_{\equiv}=f_1(f_2(x_0))$ and $[\tilde g_2]_\equiv=f_2(x_0)$ (they exist by Proposition~\ref{prop:approx}). Then we have $r(f_i)=\tilde g_iN$ for $i=1,2$. At the same time, $r(f_1f_2)=[R(f_1f_2)]_E=[f_1(f_2(x_0))]_E=[[\tilde g_1\tilde g_2]_\equiv]_E=\tilde g_1\tilde g_2N=\tilde g_1N\tilde g_2N=r(f_1)r(f_2)$.
	\end{proof}
	Note that because $X/E$ is a group, Lemma~\ref{lem:r_is_homomorphism} immediately implies from that for any idempotent $u\in EL$ we have that $u\in \ker r$. Furthermore, Lemma~\ref{lem:r_is_homomorphism} immediately implies that if $E$ is a properly group-like equivalence relation, then the group structure witnessing it is unique, so there is no analogue of Example~\ref{ex:nonunique_structure} for proper group-likeness.

	We have the following proposition, generalising Proposition~\ref{prop:closed_invariant}.
	\begin{cor}
		\label{cor:prop_glike_ellis_invariant}
		If $E$ is closed group-like or properly group-like, then it is $E(G,X)$-invariant. In fact, we have for every $f\in EL$ and $x\in X$ that $f[x]_E=r(f)[x]_E=[f(x)]_E$. Moreover, we have a ``mixed associativity'' law: for every $f\in E(G,X)$ and every $x_1,x_2\in X$, $(f[x_1]_E)[x_2]_E=f([x_1]_E[x_2]_E)$.
	\end{cor}
	\begin{proof}
		Note that in each case, the function $r$ is a semigroup homomorphism (either by Lemma~\ref{lem:closed_group_like} or by Lemma~\ref{lem:r_is_homomorphism}).
		
		Choose any $f\in EL$ and $x\in X$. Then for some $f'\in EL$, $x=R(f)$. Now, note that $[f(x)]_E=[fR(f')]_E=[R(ff')]_E=r(ff')$. Since $r$ is a homomorphism, $[f(x)]_E=r(f)r(f')=r(f)[x]_E$. But the right hand side depends only on $[x]_E$, so $E$ is $EL$-invariant. Since clearly $f(x)\in [f(x)]_E$, it follows that $f[x]_E=[f(x)]_E=r(f)[x]_E$.
		
		For the mixed associativity, just note that by what we have already shown, for every $f\in EL$ and $x_1,x_2\in X$, we have that $(f[x_1]_E)[x_2]_E=(r(f)[x_1]_E)[x_2]_E$ and apply the associativity in $X/E$.
	\end{proof}

	\begin{prop}
		\label{prop:homom}
		$r\restr_{u\cM}\colon u\cM\to X/E$ is a group epimorphism.
	\end{prop}
	\begin{proof}
		Since $u\cM=uELu$ and $X/E$ is a group, it follows that $r(u\cM)=r(u)r(EL)r(u)=r(u)X/E r(u)=X/E$.
	\end{proof}

	\begin{prop}
		\label{prop:restr_quot}
		$r\restr_{\cM}\colon \cM\to \Gal(T)$ is a topological quotient map.
	\end{prop}
	\begin{proof}
		$EL$ is compact and $X$ is Hausdorff, so $R\colon EL\to X$ is a quotient map (because it is closed), and thus so is $r$ (as the composition of $R$ and the quotient $X\to X/E$).
		
		Since the map $f\mapsto fu$ is a quotient map $EL\to \cM$ (by Remark~\ref{rem: continuous surjection is closed}) and $r(f)=r(fu)$ (because $r(u)$ is the identity in $X/E$), $r\restr \cM$ is a factor of $r$ via $f\mapsto fu$, and hence it is also a quotient map, by Remark~\ref{rem:commu_quot} (with $A=EL$, $B=\cM$ and $C=X/E$).
	\end{proof}
	
	\begin{prop}[Corresponding to {\cite[Lemma 4.7]{KP17}}]
		\label{prop:id_clsd}
		Denote by $J$ the set of idempotents in $\cM$. Then $\overline{J}\subseteq \ker r\cap \cM$.
	\end{prop}
	\begin{proof}
		For any given $v\in J$, we have that
		\[
		R(v)\in F_0=\{[\tilde g_1^{-1}\tilde g_2]_{\equiv}\mid \tilde g_1\equiv \tilde g_2\} \}.
		\]
		Indeed, let us fix $v\in J$, and then take $\tilde g_1,\tilde g_2$ according to Proposition~\ref{prop:approx} for $f=v$ and $x=R(v)$. Then
		\[
			[\tilde g_1\tilde g_2]_\equiv=vR(v)=v^2x_0=vx_0=R(v),
		\]
		so $[\tilde g_1\tilde g_2]_\equiv=[\tilde g_1]_\equiv=[\tilde g_2]_\equiv=R(v)$. Since $\tilde g_2=\tilde g_1^{-1}(\tilde g_1\tilde g_2)$, it follows that $R(v)=[\tilde g_2]_\equiv\in F_0$.
		
		On the other hand, if $\tilde g_1\equiv \tilde g_2$, then of course $\tilde r(\tilde g_1)=\tilde r(\tilde g_2)$, so $\tilde g_1^{-1}\tilde g_2\in \ker \tilde r$, and hence $R^{-1}[F_0]\subseteq \ker r$, which (by the assumption in Definition~\ref{dfn:prop_glike}
		that $F_0$ is closed) shows that $\overline{J}\subseteq \ker r$. Since $J\subseteq \cM$ and $\cM$ is closed, we are done.
	\end{proof}

	\begin{lem}
		\label{lem:r_restr_to_top_quot}
		$r\restr_{u\cM}\colon u\cM\to X/E$ is a topological group quotient map (where $u\cM$ is equipped with the $\tau$ topology).
	\end{lem}
	\begin{proof}
		In light of Proposition~\ref{prop:homom}, it is enough to show that $r\restr_{u\cM}$ is a topological quotient map.
		
		We already know that $r\restr_{u\cM}$ is continuous (by Corollary~\ref{cor:r_uM_cont}).
		
		Put $P_v:=\ker r\cap v\cM(=\ker (r\restr_{v\cM}))$ for each idempotent $v\in \cM$, and let $S:=u(u\circ P_u)=\cl_\tau(P_u)$. We will need the following claim.
		
		\begin{clm*}
			$r^{-1}[r[S]]\cap \cM$ is closed.
		\end{clm*}
		By the claim, $r^{-1}[r[S]]\cap \cM$ is closed in $\cM$, so by Proposition~\ref{prop:restr_quot}, $r[S]$ is a closed subset of $X/E$. In particular, it must contain the closure of the identity in $X/E$, i.e.\ $\overline{[x_0]_E}$. On the other hand, by continuity of $r\restr_{u\cM}$, the preimage of $\overline{[x_0]_E}$ by $r\restr_{u\cM}$ is a $\tau$-closed set containing $P_u$, and thus also $S$. It follows that $r[S]=\overline{[x_0]_E}$.
		
		Note that because $X/E$ is a compact topological group and $[x_0]_E$ is the identity, $(X/E)/\overline{[x_0]_E}=X/{\overline E}$ (cf.\ Proposition~\ref{prop:closure_grouplike}) is a compact Hausdorff group.
		
		Now, suppose $F\subseteq X/E$ is such that $r\restr_{u\cM}^{-1}[F]=r^{-1}[F]\cap u\cM$ is $\tau$-closed. Then the preimage is also $S$-invariant (because it is $P_u$-invariant). Since $r[S]=\overline{[x_0]_E}$ and $r$ is a homomorphism, it follows that $F$ is $\overline{[x_0]_E}$-invariant, i.e.\ $F=F\overline{[x_0]_E}$. Thus, $F$ is closed if and only if $F/\overline{[x_0]_E}$ is closed in $X/\overline{E}$. On the other hand, we already know (by Proposition~\ref{prop:closure_grouplike} and Lemma~\ref{lem:closed_group_like}) that the composed map $\bar r\colon u\cM\to X/{\overline E}$ is a quotient map. Since the preimage of $F/\overline{[x_0]_E}$ by $\bar r$ is the same as the preimage of $F$ by $r\restr_{u\cM}$, it follows that $F/\overline{[x_0]_E}$ is closed, and hence so is $F$. Thus, we only need to prove the claim.
		
		\begin{clmproof}[Proof of claim]
			Roughly, we follow the proof of \cite[Lemma 4.8]{KP17}. Denote by $J$ the set of idempotents in $\cM$. First note that (using Fact~\ref{fct:circ_calculations} and Fact~\ref{fct:idempotents_ideals_Ellis}):
			\begin{itemize}
				\item
				For any $v,w\in J$, we have $wP_v=P_w$. Indeed, since $v,w\in \ker r\cap \cM$, we have $vP_w\subseteq P_v$ and $wP_v\subseteq P_w$, and because $wv=w$, we have $P_w=wP_w=wvP_w\subseteq wP_v\subseteq P_w$, and hence $wP_v=P_w$.
				\item
				$S=S\cdot P_u$: for any $f\in P_u$ we have $Sf=u(u\circ P_u) f=u(u\circ (P_uf))$ and clearly $P_uf=P_u$.
				\item
				Since $P_u=\ker (r\restr_{u\cM})$, it follows immediately from the preceding point that $S=r^{-1}[r[S]]\cap u\cM$.
				\item
				$r^{-1}[r[S]]\cap \cM=J\cdot S$: to see $\subseteq$, take any $f\in r^{-1}[r[S]]\cap \cM$. Then $f\in r^{-1}[r[S]]\cap v\cM$ for some $v\in J$, and $r(f)=r(uf)\in r[S]$, so, by the preceding point, $uf\in S$, and thus $f=vf=vuf\in vS$, so $f\in J\cdot S$; the reverse inclusion is clear, as $J\subseteq \ker r$.
				\item
				$r^{-1}[r[S]]\cap \cM=\bigcup_{v\in J} v\circ P_u$. To see $\subseteq$, note that (using Fact~\ref{fct:tau_top_pre}(2)), for every $v\in J$ we have $vS=vu(u\circ P_u)\subseteq (vuu)\circ P_u=v\circ P_u$, so, by the preceding point, $v\circ P_u\supseteq r^{-1}[r[S]]\cap v\cM$, and thus $\bigcup_v v\circ P_u\supseteq r^{-1}[r[S]]\cap (\bigcup_v v\cM)=r^{-1}[r[S]]\cap \cM$. For $\supseteq$, note that because $u\in \ker r$, we have $r[v\circ P_u]=r[u(v\circ P_u)]\subseteq r[(uv)\circ P_u]=r[u\circ P_u]=r[u(u\circ P_u)]=r[S]$.
			\end{itemize}
			
			In summary, to prove the claim, we need only to show that ${\bigcup_v v\circ P_u}$ is closed in $\cM$.
			
			Let $f\in \overline{\bigcup_v v\circ P_u}$. Then we have nets $(v_i)_i$ in $J$ and $(f_i)_i$ in $\cM$ such that $f_i\in v_i\circ P_u$ and $f_i\to f$. By compactness, we can assume without loss of generality that the net $(v_i)$ converges to some $v\in \overline J$. Then, by considering neighbourhoods of $v$ and $f$, we can find nets $(g_j)_j$ in $G$ and $(p_j)_j$ in $P_u$ such that $g_j\to v$ and $g_jp_j\to f$, so $f\in v\circ P_u$. By Proposition~\ref{prop:id_clsd}, $v\in \ker r\cap \cM$, so by the first bullet above, $v\in P_w=wP_u$ for some $w\in J$, so $v=wp$ for some $p\in P_u$. Furthermore, $P_u$ is a group (as the kernel of a group homomorphism), so (using Fact~\ref{fct:tau_top_pre}(2))
			\[
			f\in v\circ P_u=v\circ (p^{-1}P_u)\subseteq v\circ (p^{-1}\circ P_u)\subseteq (vp^{-1})\circ P_u=w\circ P_u
			\]
			(where $p^{-1}$ is the inverse of $p$ in $u\cM$), and we are done.
		\end{clmproof}
	\end{proof}
	
	\begin{dfn}
		\label{dfn:unif_prop_glike}
		\index{equivalence relation!group-like!uniformly properly}
		A properly group-like $E$ is \emph{uniformly properly group-like} if $E=\bigcup \mathcal E$, where $\mathcal E$ is a family of closed, symmetric subsets of $X^2$ containing the diagonal, with the property that (for some $\tilde G$ witnessing proper group-likeness) for any $D\in \mathcal E$ we have some $D'\in \mathcal E$ such that whenever $(x_0,[\tilde g]_\equiv)\in D$, we have that for every $\tilde g'$ also $([\tilde g']_\equiv,[\tilde g\tilde g']_\equiv)\in D'$, and $D\circ D\subseteq D'$ (here, $\circ$ denotes composition of relations; note that since $D$ contains the diagonal, this implies that $D\subseteq D'$).\xqed{\lozenge}
	\end{dfn}

	\begin{prop}
		\label{prop:approx2}
		If $F\subseteq X$ is closed, then for every $f\in EL$ and $f'\in f\circ R^{-1}[F]$, there are $\tilde g_1,\tilde g_2\in \tilde G$ such that $[\tilde g_1]_\equiv=R(f)$, $[\tilde g_2]_\equiv\in F$ and $[\tilde g_1\tilde g_2]_\equiv=R(f')$.\qed
	\end{prop}
	\begin{proof}
		Take $(g_i), (f_i)$ such that $g_i\to f$, $f_i\in R^{-1}[F]$, and $g_if_i\to f'$. Without loss of generality we can assume that $(f_i)$ is convergent to some $f''$. Then $R(f_i)\to R(f'')$ and $g_iR(f_i)=R(g_if_i)\to R(f')$ and by pseudocompleteness (applied to $(g_i)_i$ and $(R(f_i))_i$), we have $\tilde g_1,\tilde g_2$ such that $[\tilde g_1]_\equiv=R(f)$, $[\tilde g_2]_\equiv=R(f'')$ and $[\tilde g_1\tilde g_2]_\equiv=R(f')$. But since $F$ is closed, $R(f'')\in F$ and we are done.
	\end{proof}
	The proof of the following proposition is based on the proofs of \cite[Theorem 0.1]{KP17} and \cite[Theorem 2.7]{KPR15} (the latter paper is joint with Krzysztof Krupiński and Anand Pillay). Recall that $H(u\cM)$ is the intersection of closures of the $\tau$-neighbourhoods of $u\in u\cM$  (see Fact~\ref{fct:tau_top_pre}(8)).
	\begin{prop}
		\label{prop:H(uM)_in_ker}
		Suppose $E$ is uniformly properly group-like.
		
		Then $H(u\cM)\leq \ker(r)$.
	\end{prop}
	\begin{proof}
		Note that $R[\ker r]$ is precisely $[x_0]_E\subseteq X$.
		
		Let $S\in \mathcal E$ be such that $R(u)\in S_{x_0}$ (i.e.\ $(x_0,R(u))\in S$ --- it exists because $u\in \ker r$). Choose $S',S''=(S')'\in \mathcal E$ as in Definition~\ref{dfn:unif_prop_glike}.
		
		Let $U\subseteq X$ be an arbitrary open set such that $K:=\overline U$ is disjoint from $S''_{x_0}$ (i.e.\ for no $k\in K$ we have $(x_0,k)\in S''$). Consider $F_U=R^{-1}[K]\cap u\cM$ and $A_U=R^{-1}[(X\setminus U)\circ S']$ ($\circ$ denotes relation composition, so this makes sense, as $(X\setminus U)\circ S'\subseteq X^1\circ X^2=X^1=X$).
		
		\begin{clm*}\begin{itemize}
				\item $u\notin \cl_\tau(F_U)$,
				\item $\cl_\tau(u\cM\setminus \cl_\tau(F_U))\subseteq A_U$,
				\item $H(u\cM)\subseteq A_U$.
			\end{itemize}
		\end{clm*}
		\begin{clmproof}
			Suppose towards contradiction that $u\in \cl_\tau(F_U)$. Then by Proposition~\ref{prop:approx2} applied to $u\in u\circ R^{-1}[K]$, we have $\tilde g_1,\tilde g_2$ such that $[\tilde g_1]_\equiv=[\tilde g_1\tilde g_2]_\equiv=R(u)$ and $[\tilde g_2]_\equiv\in K$. In particular $[\tilde g_1]_\equiv= R(u)\in S_{x_0}$, so $([\tilde g_2]_\equiv,[\tilde g_1\tilde g_2]_\equiv)\in S'$. Thus, since $[\tilde g_1\tilde g_2]_\equiv=R(u)\in S_{x_0}\subseteq S'_{x_0}$, we have $[\tilde g_2]_\equiv\in S''_{x_0}$ so $[\tilde g_2]_\equiv \in K\cap S''_{x_0}$, a contradiction.
			
			For the second bullet, it is enough to show that $\cl_\tau(u\cM\setminus R^{-1}[U])\subseteq A_U$. Take any $f\in \cl_\tau(u\cM\setminus R^{-1}[U])=\cl_\tau(u\cM\cap R^{-1}[X\setminus U])$. By applying Proposition~\ref{prop:approx2} to $f\in u\circ R^{-1}[X\setminus U]$, we find $\tilde g_1,\tilde g_2$ such that $[\tilde g_1]_\equiv=R(u)$, $[\tilde g_1\tilde g_2]_\equiv=R(f)$ and $[\tilde g_2]_\equiv\in X\setminus U$. Then, as before, $R(f)=[\tilde g_1\tilde g_2]_\equiv\in S'_{[\tilde g_2]_\equiv}\subseteq (X\setminus U)\circ S'$.
			
			For the third bullet, note that $H(u\cM)$ is, by its definition and the first bullet, contained in $\cl_\tau(u\cM\setminus \cl_\tau(F_U))$, and then apply the second bullet.
		\end{clmproof}
		
		By the claim, $\bigcap_U A_U\supseteq H(u\cM)$, where the intersection runs over all $U$ with closures disjoint from $S''_{x_0}$. We will show that the opposite inclusion holds as well, so the two sides are equal.
		
		Notice that $R[A_U]=(X\setminus U)\circ S'$, so for any $f\in A_U$, we have that $S'_{R(f)}\cap (X\setminus U)\neq \emptyset$, and so by compactness, if $f\in \bigcap_U A_U$, then $S'_{R(f)}\cap \bigcap_U (X\setminus U)\neq \emptyset$.
		
		Moreover, because $X$ is normal (as a compact Hausdorff space) and $S''_{x_0}$ is closed, $S''_{x_0}=\bigcap_U X\setminus U$ (where $U$ are as above; this means just that $S''_{x_0}$ is the intersection of all closed sets containing $S''_{x_0}$ in their interior). It follows that for $f\in\bigcap_U A_U$, we have $S'_{R(f)}\cap S''_{x_0}\neq \emptyset$, whence $R(f)\in (S'\circ S'')_{x_0}\subseteq S'''_{x_0}\subseteq [x_0]_E\subseteq R[\ker r]$ (where $S'''=(S'')'\in \mathcal E$ is chosen according to Definition~\ref{dfn:unif_prop_glike}). Thus $f\in \ker r$.
	\end{proof}

	The following lemma summarises the results of this chapter up to this point.
	\begin{lem}
		\label{lem:main_abstract_grouplike}
		Suppose $E$ is a group-like equivalence relation on $X$. Let $\cM$ be a minimal left ideal in $EL=E(G,X)$, and let $u\in \cM$ be an idempotent.
		Consider $r\restr_{u\cM}\colon u\cM\to X/E$ defined by $r(f)=[f(x_0)]_E$.
		Then:
		\begin{enumerate}
			\item
			$r\restr_{u\cM}$ is continuous (where $u\cM$ is equipped with the $\tau$ topology),
			\item
			if $E$ is closed or properly group-like, then it is $E(G,X)$-invariant, $r$ is a homomorphism and $r\restr_{u\cM}$ is a topological quotient map (once more, with $u\cM$ equipped with the $\tau$ topology),
			\item
			if $E$ is closed or uniformly properly group-like, then  $r\restr_{u\cM}$ factors through $u\cM/H(u\cM)$ and induces a topological group quotient mapping from the compact Hausdorff group $u\cM/H(u\cM)$ onto $X/E$.
			
		\end{enumerate}
	\end{lem}
	\begin{proof}
		(1) is Corollary~\ref{cor:r_uM_cont}. (2) follows from Corollary~\ref{cor:prop_glike_ellis_invariant} and Lemmas~\ref{lem:closed_group_like}, \ref{lem:r_is_homomorphism} and \ref{lem:r_restr_to_top_quot}. (3) follows from Lemma~\ref{lem:closed_group_like} and Proposition~\ref{prop:H(uM)_in_ker}.
	\end{proof}
	
	\section{Weakly group-like equivalence relations}
	The following notation will be very convenient throughout this section.
	\begin{dfn}
		\label{dfn:induced_relation}
		\index{EZ@$E"|_Z$}
		Suppose $X$ is a set and $E$ is an equivalence relation on $X$, while $Z$ is another set with some distinguished map $f\colon Z\to X/E$ (which is usually left implicit, but clear from the context, e.g.\ if we have a distinguished map $Z\to X$, the map $Z\to X/E$ would be its composition with the quotient map $X\to X/E$). Then by $E|_Z$ we mean the pullback of $E$ by $f$, i.e.\ $y_1 \mathrel{E|_Z} y_2$ when $f(y_1)\Er f(y_2)$.\xqed{\lozenge}
	\end{dfn}
	
	\begin{rem}
		In the context of the above definition, if $f$ is surjective, we have a canonical bijection between $X/E$ and $Z/E|_Z$.
		
		Furthermore, if $f$ is induced by a continuous surjection $Z\to X$, while $Z$ is compact and $X$ is Hausdorff (which will usually be the case), then $X/E$ and $Z/E|_Z$ are homeomorphic (by Remarks~\ref{rem:commu_quot} and \ref{rem: continuous surjection is closed}).
		
		In both cases, we will freely identify the two quotients.\xqed{\lozenge}
	\end{rem}

	\begin{center}
		\begin{tikzcd}
			Z\ar[d, two heads]\ar[r, two heads] & Z/F\ar[d, two heads] \ar[r, two heads]& (Z/F)/E|_{Z/F}=Z/{E|_Z}\ar[dl, two heads,"1-1"]\\
			X\ar[r, two heads] & X/E
		\end{tikzcd}
	\end{center}
	\begin{dfn}
		\index{domination}
		If $(Z,z_0)$ is a $G$-ambit and $F$ is a group-like equivalence relation on $Z$, while $E$ is an equivalence relation on $X$, we say that $F$ \emph{dominates} $E$ if there is a $G$-ambit morphism $Z\to X$ such that $F$ refines $E|_Z$ and the induced map $Z/F\to X/E$ is $Z/F$-equivariant with respect to some left action of $Z/F$ on $X/E$, i.e.\ $E|_{Z/F}$ is left invariant. (This makes sense because if $F$ refines $E|_Z$, then the morphism $Z\to X$ induces a surjection $Z/F\to X/E$.)\xqed{\lozenge}
	\end{dfn}
	
	\begin{rem}
		\label{rem:domin_orbit}
		Note that if $F$ dominates $E$, as witnessed by $\varphi\colon Z\to X$, then the induced map $Z/F\to X/E$ is not only a topological quotient map (because $\varphi$ is a quotient map, as a continuous surjection between compact spaces), but also an orbit map (at $[x_0]_E$) of the action of $Z/F$ on $X/E$.
		\xqed{\lozenge}
	\end{rem}
	
	\begin{dfn}
		\label{dfn:wkglike}
		\index{equivalence relation!group-like!weakly}
		\index{equivalence relation!group-like!weakly closed}
		\index{equivalence relation!group-like!weakly properly}
		\index{equivalence relation!group-like!weakly uniformly properly}
		We say that an equivalence relation $E$ on the ambit $(X,x_0)$ is \emph{weakly [closed/properly/uniformly properly] ($G$-)group-like} if it is dominated by some [closed/properly/uniformly properly, respectively] ($G$-)group-like equivalence relation.\xqed{\lozenge}
	\end{dfn}

	\begin{ex}
		\label{ex:group_wgl}
		If $G$ is a compact Hausdorff group, acting on $X=G$ by left translations, while $H\leq G$, then the relation $E_H$ of lying in the same left coset of $H$ is weakly uniformly group-like, as it is dominated by equality on $X$, which is closed, and it is not hard to see that it is uniformly group-like in this case. Thus, $E_H$ is weakly uniformly properly group-like and weakly closed group-like.\xqed{\lozenge}
	\end{ex}
	
	\begin{ex}
		\label{ex:noteventhen}
		Recall Example~\ref{ex:nonunique_structure}. There, $X=\bR/\bZ$ is a compact Hausdorff group and $E=E_H$ for $H=\bQ/\bZ$, so this is a special case of Example~\ref{ex:group_wgl}. Thus, $E$ is weakly closed group-like and weakly uniformly properly group-like. This shows that even if $E$ is an equivalence relation which is simultaneously group-like and weakly uniformly properly group-like and weakly closed group-like, there may be multiple group structures on $X/E$ witnessing group-likeness.\xqed{\lozenge}
	\end{ex}
	
	In the definition of a weakly group-like equivalence relation, we did not specify that it needs to be invariant. The following proposition shows that it actually follows from the definition.
	\begin{prop}
		A weakly group-like equivalence relation is
		invariant.
	\end{prop}
	\begin{proof}
		Let $E$ be weakly group-like on $X$, dominated by a group-like $F$ on some $Z$. It is easy to see that $E$ is invariant if and only if $E|_Z$ is, so we can assume without loss of generality that $X=Z$.
		
		Now, for any $g\in G$ and $x_1,x_2\in X$, if $x_1\Er x_2$, then by weak group-likeness, $[x_1]_F \mathrel{E|_{X/F}} [x_2]_F$ and $[gx_1]_F=[gx_0]_F\cdot [x_1]_F\mathrel{E|_{X/F}} [gx_0]_F\cdot [x_2]_F=[gx_2]_F$, so $gx_1\Er gx_2$.
	\end{proof}

	The following proposition describes some basic topological properties of weakly group-like equivalence relations. In particular, it generalises Proposition~\ref{prop:top_gp_R1} for compact Hausdorff groups, and as we will see later (in Remark~\ref{rem:wglike_closedness_to_typdef}), also Proposition~\ref{prop:type-definability_of_relations}.
	\begin{prop}
		\label{prop:top_props_of_wglike}
		Suppose $E$ is a weakly group-like equivalence equivalence relation on $X$. Then:
		\begin{enumerate}
			\item
			$X/E$ is an $R_1$-space (see Definition~\ref{dfn:R0_R_1}),
			\item
			$E$ is closed if and only if it has a closed class, if and only if $X/E$ is Hausdorff.
		\end{enumerate}
	\end{prop}
	\begin{proof}
		Let $F$ be a group-like equivalence relation on $Z$ dominating $E$.
		
		By Remark~\ref{rem:domin_orbit}, if $H\leq Z/F$ is the stabiliser of $[x_0]_E\in X/E$, we have that $X/E$ is homeomorphic to $(Z/F)/H$, which is an $R_1$-space by the Proposition~\ref{prop:top_gp_R1}, which gives us (1).
		
		It follows immediately from (1) that $X/E$ is Hausdorff if and only if all $E$-classes are closed, and by Fact~\ref{fct:quot_T2_iff_closed}, $X/E$ is Hausdorff if and only if $E$ is closed. On the other hand, it is not hard to see that $Z/F$ acts on $X/E$ transitively by homeomorphisms, so if one class of $E$ is closed, then all of them are, which gives us (2).
	\end{proof}

	In general, a product of quotient maps need not be a quotient map. The following proposition establishes a sufficient condition for that to hold, which we will use in a moment.
	\begin{prop}
		\label{prop:product_quotient}
		Suppose $X_1,X_2, Y_1,Y_2$ are compact spaces, and $Y_1,Y_2$ are $R_1$-spaces.
		
		Then if $f_i\colon X_i\to Y_i$ are quotient maps (for $i=1,2$), then the product $f_1\times f_2\colon X_1\times X_2\to Y_1\times Y_2$ is also a quotient map.
	\end{prop}
	\begin{proof}
		Let $\overline{Y_1},\overline{Y_2}$ be the Hausdorff quotients of $Y_1,Y_2$ (respectively) which we have by the $R_1$ condition.
		
		Consider the natural maps $\bar f_i\colon X_i\to \overline{Y_i}$ induced by $f_i$ for $i=1,2$. They are clearly continuous and surjective, and hence so is $\bar f_1\times \bar f_2\colon X_1\times X_2\to \overline{Y_1}\times \overline{Y_2}$. $X_1\times X_2$ is compact and $\overline{Y_1}\times \overline{Y_2}$ is Hausdorff, so by Remark~\ref{rem: continuous surjection is closed}, $\bar f_1\times \bar f_2$ is a quotient map, fitting into the following commutative diagram:
		\begin{center}
			\begin{tikzcd}
			X_1\times X_2\ar[r,two heads,"f_1\times f_2"]\ar[d,two heads,swap,"\bar f_1\times \bar f_2"] & Y_1\times Y_2\ar[dl, two heads]\\
			\overline{Y_1}\times \overline{Y_2}&
			\end{tikzcd}
		\end{center}
		We need to show that if $A\subseteq Y_1\times Y_2$ has closed preimage $A'=(f_1\times f_2)^{-1}[A]\subseteq X_1\times X_2$, then it is closed. Note that since $A'$ is closed, it has closed fibres (both in $X_1$ and in $X_2$). Since $f_1,f_2$ are quotient maps, it follows that fibres of $A$ are closed, so they are preimages of subsets of $\overline{Y_1}$ and $\overline{Y_2}$. Thus $A$ itself is the preimage of some $A''\subseteq \overline{Y_1}\times \overline{Y_2}$. But then it follows that $A'=(\bar f_1\times \bar f_2)^{-1}[A'']$, so, since $\bar f_1\times \bar f_2$ is a quotient map, $A''$ is closed, and thus so is $A$ (as its continuous preimage).
	\end{proof}
	
	Similarly, as mentioned in the introduction, a quotient map need not be open in general. However, suitable algebraic assumptions can force that to be true, as in the following proposition.
	\begin{prop}
		\label{prop:quotient_is_open}
		If $G$ is a left topological group (i.e.\ for any fixed $g_0\in G$, the left multiplication $g\mapsto gg_0$ is continuous), then for any $H\leq G$, the quotient map $\varphi\colon G\to G/H$ is an open mapping.
	\end{prop}
	\begin{proof}
		Note that the assumption implies immediately that for every $g_0$, the left multiplication by $g_0$ is a homeomorphism $G\to G$, and in particular, it is open.
		
		Let $U\subseteq G$ be open. Then $UH=\bigcup_{h\in H} Uh$ is also open (as a union of open sets). Since clearly $\varphi^{-1}[\varphi[U]]=UH$, it follows that $\varphi[U]$ is open.
	\end{proof}
	
	\begin{prop}
		\label{prop:action_factor_is_continuous}
		Suppose we have:
		\begin{itemize}
			\item
			left topological semigroups $S,T$,
			\item
			topological spaces $A,B$,
			\item
			a continuous semigroup action $\mu\colon S\times A\to A$,
			\item
			a topological quotient homomorphism $q_S\colon S\to T$ and a topological quotient map $q_A\colon A\to B$ such that $\mu$ induces an action $\mu_T\colon T\times B\to B$ (satisfying the natural commutativity conditions; see the diagram in the proof).
		\end{itemize}
		Then if either:
		\begin{enumerate}
			\item
			$T$ and $B$ are $R_1$, or
			\item
			$S$ and $T$ are groups and $q_A$ is an open mapping (e.g.\ $B=A$ and $q_A=\id_A$),
		\end{enumerate}
		then $\mu_T$ is continuous.
	\end{prop}
	\begin{proof}
		Note that by the assumption, we have a commutative diagram:
		\begin{center}
			\begin{tikzcd}
				S\times A\ar[d,swap,"q_S\times q_A", two heads]\ar[r,"\mu", two heads]\ar[dr, two heads] & A\ar[d,"q_A", two heads]\\
				T\times B\ar[r,swap,"\mu_T", two heads] & B.
			\end{tikzcd}
		\end{center}
		Since $\mu$ and $q_A$ are continuous, it follows that the diagonal arrow is continuous. Thus, by Remark~\ref{rem:commu_quot} (applied to the lower left triangle), it is enough to show that $q_S\times q_A$ is a topological quotient map.
		
		If we assume (1), this follows from Proposition~\ref{prop:product_quotient}. Under (2), it follows from Proposition~\ref{prop:quotient_is_open} that $q_S$ is open, so $q_S\times q_A$ is open (as a product of open maps). Since it is trivially a continuous surjection, it follows that it is a quotient map.
	\end{proof}

	\begin{prop}
		\label{prop:grouplike_cont_action}
		If $F$ dominates $E$, then the action of $Z/F$ on $X/E$ is (jointly) continuous.
	\end{prop}
	\begin{proof}
		Since $F$ is group-like, the multiplication $Z/F\times Z/F\to Z/F$ is continuous and $Z/F$ is a topological group (so in particular, it is a left topological group).
		
		The conclusion follows immediately by Proposition~\ref{prop:action_factor_is_continuous} to $S=T=Z/F$, $A=Z/F$ and $B=X/E$. In fact, both (1) and (2) apply ((1) by Proposition~\ref{prop:top_props_of_wglike}, and (2) by Proposition~\ref{prop:quotient_is_open}).
	\end{proof}

	\begin{prop}
		\label{prop:epim_ideals_idempotents}
		Suppose $S$ and $T$ are compact Hausdorff left topological semigroups and $\varphi\colon S\to T$ is a continuous epimorphism.
		
		Then for any minimal (left) ideal $\cN\unlhd S$ and idempotent $v\in \cN$, $\cM:=\varphi[\cN]$ is a minimal left ideal in $T$ and $u:=\varphi(v)$ is an idempotent in $\cM$.
		
		Conversely, given a minimal (left) ideal $\cM\unlhd T$ and an idempotent $u\in \cM$, we can find a minimal ideal $\cN\unlhd S$ and an idempotent $v\in \cN$ such that $\varphi(v)=u$ and $\varphi[\cN]=\cM$.
	\end{prop}
	\begin{proof}
		The fact that $\varphi(v)$ is an idempotent is immediate by the fact that $\varphi$ is a homomorphism. By continuity and compactness, $\varphi[\cN]$ is a closed subset of $T$, and because $\varphi$ is an epimorphism, it is a left ideal. In particular, it contains a minimal left ideal $\cM$ of $T$. Similarly, $\varphi^{-1}[\cM]$ is a closed ideal in $S$ containing $\cN$. In particular, $\cM=\varphi[\cN]$.
		
		For the ``conversely" part, notice that $\varphi^{-1}[\cM]$ is an ideal in $S$, so it contains a minimal ideal $\cN$. By the preceding paragraph, $\varphi[\cN]$ is a minimal ideal contained in $\cM$, so it must be equal to $\cM$. It follows that there is some idempotent $v\in \cN$ such that $u\in \varphi[v\cN]=\varphi(v)\cM$. Since $\varphi(v)$ is an idempotent, and since $u$ is the only idempotent in $u\cM$, it follows that $\varphi(v)=u$.
	\end{proof}
	
	\begin{prop}
		\label{prop:preimage_of_derived}
		If $G_1,G_2$ are compact $T_1$ semitopological groups and $\varphi\colon G_1\to G_2$ is both a homomorphism and a topological quotient map, then $\varphi^{-1}[H(G_2)]=H(G_1)\ker \varphi$ (where $H(G_1)$ and $H(G_2)$ are the derived subgroups, see Fact~\ref{fct:semitop_T2_quot}).
	\end{prop}
	\begin{proof}
		Write $K$ for $\ker\varphi$. Consider the composed map $\varphi'\colon G_1\to G_2/H(G_2)$. Clearly, its kernel is $\varphi^{-1}[H(G_2)]$, and it contains $K$. Furthermore, $G_2/H(G_2)$ is Hausdorff (by Fact~\ref{fct:semitop_T2_quot}), so by Corollary~\ref{cor:H(G)_universal}, it also contains $H(G_1)$.
		
		On the other hand, since $G_2$ is $T_1$, $K$ is closed in $G_1$. Since $G_1/H(G_1)$ is Hausdorff and $G_1$ is compact, $KH(G_1)$ is a closed normal subgroup of $G/H(G_1)$, and by Fact~\ref{fct:quotient_by_closed_subgroup} it follows that $G_1/KH(G_1)$ is a Hausdorff topological group. But the map $G_1\to G_1/KH(G_1)$ factors through $G_2=G_1/K$, so by Corollary~\ref{cor:H(G)_universal}, the factor map $G_2\to G_1/KH(G_1)$ factors through $G_2/H(G_2)$, so $\varphi^{-1}[H(G_2)]$ is contained in $KH(G_1)$.
	\end{proof}
	
	The following proposition will allow us to translate the properties of group-like equivalence relations (mainly those from Lemma~\ref{lem:main_abstract_grouplike}) to the weakly group-like quotient.
	\begin{prop}
		\label{prop:induced_epimorphism}
		If $\varphi\colon (Z,z_0)\to (X,x_0)$ is a $G$-ambit morphism, then $\phi$ induces a continuous epimorphism $\varphi_*\colon E(G,Z)\to E(G,X)$ (by the formula $\varphi_*(f)(\varphi(z))=\varphi(f(z))$).
		
		Moreover, for any minimal ideal $\cN$ in $E(G,Z)$, with idempotent $v\in \cN$, the restriction $\varphi_*\restr_{v\cN}$ is a topological quotient map (with respect to $\tau$ topologies), and it induces a topological group quotient map $v\cN/H(v\cN)\to u\cM/H(u\cM)$, where $u=\varphi_*(v)$ and $\cM=\varphi_*[\cN]$ (note that $u\cM$ is an Ellis group in $E(G,X)$ by Proposition~\ref{prop:epim_ideals_idempotents}).
	\end{prop}
	\begin{proof}
		To see that $\varphi_*$ is well-defined, take any $f\in E(G,Z)$ and $z_1,z_2\in Z$ such that $\varphi(z_1)=\varphi(z_2)$. We need to show that $\varphi(f(z_1))=\varphi(f(z_2))$.
		
		Recall that for $g\in G$, by $\pi_{Z,g}$ we denote the function $Z\to Z$ given by $z\mapsto gz$, and by $\pi_{X,g}$ we denote the analogous function $X\to X$.
		
		Take any net $(g_i)_i$ such that $\pi_{Z,g_i}\to f$. Then for all $i$ we have $\varphi(g_i(z_1))=g_i(\varphi(z_1))=g_i(\varphi(z_2))=\varphi(g_i(z_2))$. Since $\varphi$ is continuous, it follows that $\varphi(f(z_1))=\lim_i g_i(\varphi(z_1))=\varphi(f(z_2))$ (which also shows that $\varphi_*(f)=\lim_i \pi_{X,g_i}\in E(G,X)$).
		
		To see that $\varphi_*$ is onto, note that for every $f\in E(G,X)$ we can find some $(g_i)_i$ such that if $\pi_{X,g_i}\to f$. Then by compactness, we can assume without loss of generality that $(\pi_{Z,g_i})_i$ is also convergent to some $f'\in E(G,Z)$. By the parenthetical remark in the last paragraph, $\varphi_*(f')=f$.
		
		To see that $\varphi_*$ is continuous, it is enough to show that the preimage of a subbasic open set of the form $B_{x,U}=\{f\in E(G,X) \mid f(x)\in U \}$ (where $x\in X$ an $U\subseteq X$ is open) is open. But $\varphi_*(f)(x)\in U$ if and only if for some $z$ such that $\varphi(z)=x$ we have that $f(z)\in \varphi^{-1}[U]$, which is an open condition about $f$.
		
		To see that $\varphi_*\restr_{v\cN}$ is continuous, take any $\tau$-closed set $F\subseteq u\cM$, i.e.\ such that $F=u(u\circ F)$, put $F'=\varphi_*^{-1}[F]\cap v\cN$ and take any $f\in (v\cN)\cap (v\circ F')$. We need to show that $f\in F'$. Take any nets $(g_i)_i$ and $(f_i)_i$ such that $g_if_i\to f$, $\pi_{Z,g_i}\to v$ and $f_i\in F'$. Then by continuity of $\varphi_*$, we have $\pi_{X,g_i}\to u$ and $g_i\varphi_*(f_i)=\varphi_*(g_if_i)\to \varphi_*(f)$, and thus $\varphi_*(f)\in u\circ F$. Since $f=vf$ and $\varphi_*(v)=u$, we have $\varphi_*(f)=\varphi_*(vf)=u\varphi_*(f)\in u(u\circ F)$, whence $f\in F'$.
		
		To see that $\varphi_*\restr_{v\cN}$ is a quotient map, take any $F\subseteq u\cM$ such that $F':=\varphi_*^{-1}[F]\cap v\cN$ is $\tau$-closed, and take any $f\in u\circ F$, along with nets $(g_i),(f_i)$ witnessing it, i.e.\ such that $\pi_{X,g_i}\to u$, $f_i\in F$ and $g_if_i\to f$.
		
		We need to show that $uf\in F$. For each $i$, fix some $f_i'\in F'$ such that $\varphi_*(f_i')=f_i$. By compactness, can assume without loss of generality that $\pi_{Z,g_i}\to v'$ and $g_if_i'\to f'$ for some $v',f'\in E(G,Z)$. Note that by continuity of $\varphi_*$, $\varphi_*(v')=u$ and $\varphi_*(vf')=\varphi_*(v)\varphi_*(f')=uf$, so it is enough to show that $\varphi_*(vf')\in F$.
		
		We certainly have $vf'\in v(v'\circ F)$. But using Fact~\ref{fct:tau_top_pre}(2), we have:
		\[
			v(v'\circ F')= vv(v'\circ (vF'))\subseteq v(v\circ (v'\circ(v\circ F')))\subseteq v((vv'v)\circ F').
		\]
		Note that $v''=vv'v\in v\cN$. Thus we have
		\[
			v(v''\circ F')=vv''(v'')^{-1}(v''\circ F')\subseteq vv''(((v'')^{-1}v'')\circ F')=v''v(v\circ F')=v''F'.
		\]
		Since $\varphi_*(v'')=u$, it follows that $\varphi_*[v(v'\circ F')]\subseteq uF=F$, so $\varphi_*(vf')\in F$ and we are done.

		The fact that $\varphi_*\restr_{v\cN}$ is continuous implies (immediately by definition, or by Proposition~\ref{prop:preimage_of_derived}) that $\varphi_*^{-1}[H(u\cM)]\supseteq H(v\cN)$, which gives us an induced mapping $v\cN/H(v\cN)\to u\cM/H(u\cM)$. Since $\varphi_*\restr_{v\cN}$ is a quotient map, so is the induced map.
	\end{proof}
	
	\begin{rem}
		\label{rem:induced_hom_by_D}
		In Proposition~\ref{prop:induced_epimorphism}, if we take $D_Z:=\{f\in v\cN\mid f(z_0)=v(z_0) \}$ and $D=D_X=\{f\in u\cM\mid f(x_0)=u(x_0) \}$, then clearly $\varphi_*^{-1}[D]\supseteq D_Z$, so $\varphi_*$ induces also (among others) a quotient mapping $v\cN/H(v\cN)D_Z\to u\cM/H(u\cM)D$, as well as a topological group quotient mapping from $v\cN/\Core(H(v\cN)D_Z)$ to $u\cM/\Core(H(u\cM)D)$.\xqed{\lozenge}
	\end{rem}
	
	The following Lemma is crucial, and will be one of the most important ingredients in the proofs of all main results in this chapter (and by extension, the main results of the next chapter).
	
	\begin{lem}
		\label{lem:weakly_grouplike}
		Suppose $E$ is weakly group-like. If we let $r\colon EL\to X/E$ be $r(f):=[R(f)]_E(=[f(x_0)]_E)$, then for every minimal ideal $\cM\unlhd E(G,X)$ and idempotent $u\in \cM$:
		\begin{enumerate}
			\item
			$r\restr_{u\cM}$ is continuous (with $u\cM$ equipped with the $\tau$ topology);
			\item
			if $E$ is weakly closed or weakly properly group-like, then:
			\begin{itemize}
				\item
				the action of $G$ on $X/E$ extends to a (jointly) continuous action of $E(G,X)$ by homeomorphisms (given by $f([x]_E)=[f(x)]_E$), and $r$ is its orbit map (at $[x_0]_E$),
				\item
				that action, restricted to $u\cM$, is also jointly continuous (and a group action, i.e.\ with $u$ acting as identity), and $r\restr_{u\cM}$ is its orbit map, and also a topological quotient map;
			\end{itemize}
			\item
			if $E$ is weakly closed or weakly uniformly properly group-like, then the action of ${u\cM}$ on $X/E$ factors through a continuous action of ${u\cM}/H(u\cM)$. Furthermore, $r\restr_{u\cM}$ factors through $u\cM/H(u\cM)$, yielding an orbit map of this action, which is also a topological quotient map.
		\end{enumerate}
	\end{lem}
	\begin{proof}
		The main idea of the proof is to combine Propositions~\ref{prop:epim_ideals_idempotents} and \ref{prop:induced_epimorphism} to translate Lemma~\ref{lem:main_abstract_grouplike} into the weakly group-like context.
		
		Let $\varphi\colon (Z,z_0)\to (X,x_0)$ be the $G$-ambit morphism witnessing that $F$ dominates $E$, where $F$ is group-like (and also closed, properly group-like, or uniformly properly group-like, if possible).
		
		By Propositions~\ref{prop:epim_ideals_idempotents} and \ref{prop:induced_epimorphism}, we have a minimal left ideal $\cN\unlhd E(G,Z)$ and an idempotent $v\in \cN$ such that $\varphi_*(v)=u$ and $\varphi_*[\cN]=\cM$, so that $\varphi_*\restr_{v\cN}$ is an epimorphism and a topological quotient $v\cN\to u\cM$.
		
		Let $r_F\colon E(G,Z)\to Z/F$ be the map $r_F(f)=[f(z_0)]_F$, and let $r_Z:= r\circ \varphi_*$. Then we have a commutative diagram
		\begin{center}
			\begin{tikzcd}
				E(G,Z)\ar[d,"\varphi_*"] \ar[r,"r_F"] \ar[dr,"r_Z"] & Z/F \ar[d,"\varphi_F"] \\
				E(G,X)\ar[r,"r"] & X/E
			\end{tikzcd}
		\end{center}
		(the arrow $\varphi_F$ on the right exists because $\varphi$ witnesses that $F$ dominates $E$, and it is a quotient map, because $\varphi$ is a quotient map).
		
		Now, $r_F\restr_{v\cN}$ is continuous (with respect to the $\tau$ topology on $v\cN$) by Lemma~\ref{lem:main_abstract_grouplike}, and hence so is $r_Z\restr_{v\cN}$. Since $\varphi_*\restr_{v\cN}$ is a quotient map onto $u\cM$, it follows (by Remark~\ref{rem:commu_quot} with $A=v\cN$, $B=u\cM$ and $C=X/E$) that $r\restr_{u\cM}$ is also continuous, which gives us (1).
		
		For (2), note that if $F$ is closed or properly group-like, then by Lemma~\ref{lem:main_abstract_grouplike}, $r_F$ and $r_F\restr_{v\cN}$ are semigroup epimorphisms and topological quotients and $F$ is $E(G,Z)$-invariant. Since $\varphi_F$ is the orbit map of a jointly continuous action (see Proposition~\ref{prop:grouplike_cont_action}), it follows that so are $r_Z$ and $r_Z\restr_{v\cN}$ (where the action factors is the composition of action of $Z/F$ with epimorphisms $r_F$, $r_F\restr_{v\cN}$). Then, notice that not only does $r_Z$ factor through $\varphi_*$, but so do the actions on $X/E$: indeed, we have that for every $f\in E(G,Z)$ and $x\in X$, there is some $z\in Z$ such that $\varphi(z)=x$, and then, by commutativity of the above diagram and Corollary~\ref{cor:prop_glike_ellis_invariant} (applied to $F$), we have, for every $f\in E(G,Z)$ (having in mind the identification of $X/E$, $Z/E|_Z$ and $(Z/F)/(E|_{Z/F})$):
		\begin{multline*}
			f[x]_E=r_F(f)[x]_E=[f(z_0)]_F[x]_E=[f(z_0)]_F[z]_{E|_Z}=[[f(z_0)]_F[z]_F]_{E|_{Z/F}}=\\ =[f[z_0]_F[z]_F]_{E|_{Z/F}}=[[f(z)]_F]_{E|_{Z/F}}=[f(z)]_{E|_Z}=[\varphi_*(f)(\varphi(z))]_E=[\varphi_*(f)(x)]_E.
		\end{multline*}
		Since $\varphi_*$ an epimorphism, this means that $f[x]_E=[f(x)]_E$ describes a well-defined semigroup action of $E(G,X)$ on $X/E$. Furthermore, since $r_F(v)$ is the identity in $Z/F$, $v$ acts as identity, and hence so does $u=\varphi_*(v)$. We need to show that the action is jointly continuous.
		
		First, the action of $Z/F$ on $X/E$ is jointly continuous by Proposition~\ref{prop:grouplike_cont_action}, which immediately implies that the action of $E(G,Z)$ on $X/E$ is also jointly continuous. Then, since $E(G,X)$ is $R_1$ as a Hausdorff space, and $X/E$ is $R_1$ by Proposition~\ref{prop:top_props_of_wglike}, we can apply Proposition~\ref{prop:action_factor_is_continuous}(1) (with $S=E(G,Z)$, $T=E(G,X)$, $A=B=X/E$) to conclude the joint continuity of the action of $E(G,X)$ on $X/E$.

		Since $\varphi_*[v\cN]=u\cM$, it follows that the action of $v\cN$ on $X/E$ (induced by the quotient map $r_F\restr_{v\cN}$) factors through an action of $u\cM$ on $X/E$. As in the preceding paragraph, we easily conclude that $v\cN$ acts continuously on $X/E$. Since $v\cN$ and $u\cM$ are (semitopological) groups, we can then apply Proposition~\ref{prop:action_factor_is_continuous}(2) (with $S=v\cN$, $T=u\cM$ and $A=B=X/E$) to conclude that the factor action of $u\cM$ on $X/E$ is jointly continuous.

		Likewise, since $\varphi_*\restr_{v\cN}$ and $r_Z\restr_{v\cN}=r\restr_{u\cM}\circ\varphi_*\restr_{v\cN}$ are quotient maps, by Remark~\ref{rem:commu_quot} (applied to $A=v\cN$, $B=u\cM$ and $C=X/E$), it follows that $r\restr_{u\cM}$ is also quotient map.
		
		For (3), note that since (by (2)) $u$ acts on $X/E$ as identity, so does $\ker \varphi_*\restr_{v\cN}$. Under the hypotheses of (3), $H(v\cN)$ also acts trivially (by Lemma~\ref{lem:main_abstract_grouplike}(3)) Since --- by Proposition~\ref{prop:preimage_of_derived} --- $\varphi_*\restr_{v\cN}^{-1}[H(u\cM)]=H(v\cN)\ker \varphi_*\restr_{v\cN}$, we conclude that $H(u\cM)$ acts trivially as well. Therefore, both $r\restr_{u\cM}$ and the action of $u\cM$ on $X/E$ factor through $u\cM/H(u\cM)$. The action of $u\cM/H(u\cM)$ is continuous by another application of Proposition~\ref{prop:action_factor_is_continuous} (with $S=u\cM$, $T=u\cM/H(u\cM)$ and $A=B=X/E$). Also as before, the map $u\cM/H(u\cM)\to X/E$ induced by $r\restr_{u\cM}$ is a quotient map by Remark~\ref{rem:commu_quot}. This completes the proof.
	\end{proof}
	If $E$ is group-like, we can ask whether the map $r$ is a homomorphism. The following proposition establishes reasonable sufficient condition for that. Note that in general, this need not be true, not even if $E$ is dominated by an equivalence relation which is both closed and uniformly properly group-like, as Example~\ref{ex:noteventhen} shows.
	\begin{prop}
		\label{prop:wgl_homom}
		In Lemma~\ref{lem:weakly_grouplike}, if $E$ itself is group-like, and $F$ is a properly group-like or closed group-like relation dominating $E$, such that the induced map $Z/F\to X/E$ (denoted by $\varphi_F$ in the proof of Lemma~\ref{lem:weakly_grouplike}) is a homomorphism, then $r$ is also a homomorphism (and thus so is $r\restr_{u\cM}$).
	\end{prop}
	\begin{proof}
		Consider the diagram in the fourth paragraph of the proof of Lemma~\ref{lem:weakly_grouplike}.
		
		Note that $r\circ \varphi_*=\varphi_F\circ r_F$. Since $F$ is properly group-like or closed group-like, $r_F$ is a homomorphism, and by assumption, $\varphi_F$ is also a homomorphism, so $r \circ \varphi_*$ is a homomorphism.
		
		But because $\varphi_*$ is an epimorphism, it follows easily that $r$ must be a homomorphism.
	\end{proof}
	Note that the hypotheses of Proposition~\ref{prop:wgl_homom} are trivially satisfied if $E$ itself is closed group-like or properly group-like, so it extends Lemma~\ref{lem:r_is_homomorphism} and Lemma~\ref{lem:closed_group_like}(2).
	
	We have the following simple property of abstract group actions.
	\begin{prop}
		\label{prop:action_factorization}
		If $G$ is a group acting transitively on a set $X$, and $D\leq G$ stabilises some point $x_0\in X$, then the action factors through $G/\Core(D)$, where $\Core(D)$ is the normal core of $D$, i.e.\ the intersection of all conjugates of $D$.
	\end{prop}
	\begin{proof}
		Since $D$ stabilises $x_0$, for any $g\in G$, it is easy to see that $gDg^{-1}$ stabilises $gx_0$. Since $G$ acts transitively on $X$, it follows that $\Core(D)$ stabilises every point of $X$. The conclusion follows.
	\end{proof}
	
	In the last case described in Lemma~\ref{lem:weakly_grouplike}, we have some additional factorisations. Recall that by $D\leq u\cM$, we denoted the $\tau$-closed group of all $f\in u\cM$ such that $f(x_0)=u(x_0)$ (cf.\ Lemma~\ref{lem:D_closed}).
	\begin{lem}
		\label{lem:factorisation_lemma}
		Suppose $E$ satisfies the conclusion of Lemma~\ref{lem:weakly_grouplike}(3), i.e.\ we have a natural continuous group action of $u\cM/H(u\cM)$ on $X/E$, and its orbit map at $[x_0]_E$ (induced by $r$) is a topological quotient map.
		
		Then $\hat r_1$ factors through a map $\hat r_2\colon u\cM/H(u\cM)D\to X/E$, and also through $\hat r_3\colon u\cM/\Core(H(u\cM)D)$, and both factors are topological quotient maps.
		
		Furthermore, the action of $u\cM/H(u\cM)$ on $X/E$, factors through a continuous action of $u\cM/\Core(H(u\cM)D)$ (consequently, $\hat r_3$ is the orbit map of the factor action, also at $[x_0]_E$).
	\end{lem}
	\begin{proof}
		To obtain the first factorisation, note that if the two cosets $f_1H(u\cM)D$ and $f_2H(u\cM)D$ are equal, then for some $f_1'\in f_1H(u\cM)$ and $f_2'\in f_2H(u\cM)$ we have that the cosets $f_1'D$ and $f_2'D$ are equal as well. Therefore, by Lemma~\ref{lem:D_kernel_equiv}, $f_1'(x_0)=f_2'(x_0)$, so in particular,
		\[
			\hat r_1(f_1)=[f_1(x_0)]_E=[f_1'(x_0)]_E=[f_2'(x_0)]_E=[f_2(x_0)]_E=\hat r_1(f_2).
		\]
		Factoring through $u\cM/\Core(H(u\cM)D)$ follows immediately, as $\Core(H(u\cM)D)\leq H(u\cM)D$. (Note that because $\hat r_1$ is a topological quotient map, so are $\hat r_2$ and $\hat r_3$.)
		
		For the ``furthermore'' part, note that by the first paragraph, $H(u\cM)D$ stabilises $[x_0]_E\in X/E$. On the other hand, since $\hat r_1$ is a surjective orbit map, $u\cM/H(u\cM)$ acts transitively on $X/E$. Thus, by Proposition~\ref{prop:action_factorization}, that action factors through $(u\cM/H(u\cM))/\Core(H(u\cM)D)=u\cM/\Core(H(u\cM)D)$. Continuity of the factor action is an easy consequence of the continuity of action of $u\cM/H(u\cM)$ and Proposition~\ref{prop:action_factor_is_continuous}.
	\end{proof}

	Lemma~\ref{lem:factorisation_lemma} can be further extended with the following ``niceness preservation'' properties, which will be very important for the proofs of the main theorems.
	\begin{lem}
		\label{lem:new_preservation_E_to_H}
		In Lemma~\ref{lem:factorisation_lemma}, let $H_1\subseteq u\cM/H(u\cM)$, $H_2\subseteq u\cM/H(u\cM)D$ and $H_3\subseteq u\cM/\Core(H(u\cM)D)$ be the preimages of $\{[x_0]_E\}$ by the respective $\hat r_1$, $\hat r_2$ or $\hat r_3$. Then:
		\begin{itemize}
			\item
			$E$ is clopen or closed if and only if some (equivalently, all) $H_i$ are such,
			\item
			if $E$ is $F_\sigma$, Borel or analytic, then so is each $H_i$.
		\end{itemize}
	\end{lem}
	\begin{proof}
		For the first bullet, note that since $\hat r_1$ is a quotient map, $X/E$ is homeomorphic to $(u\cM/H(u\cM))/H_1$. Hence, by Fact~\ref{fct:quot_T2_iff_closed} and Fact~\ref{fct:quotient_by_closed_subgroup}, $E$ is closed or clopen if and only if $H_1$ is (respectively). The fact that $\hat r_1$ factors through $\hat r_2,\hat r_3$ easily implies that if one of $H_1,H_2,H_3$ is closed or clopen, so are the other two (e.g.\ by Proposition~\ref{prop:preservation_properties}).
		
		If $E$ is $F_\sigma$, Borel or analytic, then so is $[x_0]_E$, and thus also $[x_0]_E\cap R[\overline{u\cM}]$. But, by the hypotheses, $[x_0]_E\cap R[\overline{u\cM}]$ is the preimage of $H_2$ via the continuous map $R[\overline{u\cM}]\to u\cM/H(u\cM)D$ from Proposition~\ref{prop:from_cluM}. Since $H_1$ and $H_3$ are also continuous preimages of $H_2$, the second bullet follows by several applications of Proposition~\ref{prop:preservation_properties}.
	\end{proof}
	
	\begin{rem}
		\label{rem:homom_factor}
		Note that under the assumptions of of Proposition~\ref{prop:wgl_homom}, if we have $\hat r_1$ as in Lemma~\ref{lem:factorisation_lemma} (e.g.\ $E$ is weakly uniformly properly group like or weakly closed group-like), then $\hat r_1$ is  homomorphism, and thus so is $\hat r_3$. It follows that both are topological group quotient maps, and the subgroups $H_1$ and $H_3$ in Lemma~\ref{lem:new_preservation_E_to_H} are both normal.
		\xqed{\lozenge}
	\end{rem}

	\begin{rem}
		\label{rem:tame_preservation}
		In fact, we can extend Lemma~\ref{lem:new_preservation_E_to_H} to also say that if $(G,X)$ is tame and metrisable, then if one of $E, H_1,H_2,H_3$ is Borel or analytic, then so are the other three. The proof is mostly straightforward, but somewhat technical. (It uses Proposition~\ref{prop:NIP gives metrizability}, Fact~\ref{fct:borel_section}, Corollary~\ref{cor:borel_map}, and the fact that the preimage of an analytic set by a Borel map between Polish spaces is analytic.)
		\xqed{\lozenge}
	\end{rem}
	
	\section{Cardinality dichotomies}
	\label{sec:dichot_cardinality}
	
	In this section, we apply the results of the preceding sections, along with properties of compact Hausdorff groups, to deduce two dichotomies related to weakly group-like equivalence relations. In contrast to the next section, we \emph{do not} assume metrisability. Theorems of this section, along with Lemma~\ref{lem:new_preservation_E_to_H}, and Lemma~\ref{lem:weakly_grouplike}, yield Main~Theorem~\ref{mainthm:abstract_card}.
	
	\begin{thm}
		\label{thm:general_cardinality_intransitive}
		Suppose $E$ is analytic and either weakly uniformly properly group-like or weakly closed group-like. Then either $E$ is closed, or for every $Y\subseteq X$ which is closed and $E$-saturated, we have $\lvert Y/E\rvert\geq 2^{\aleph_0}$.
	\end{thm}
	\begin{proof}
		Suppose $E$ is not closed.
		
		Let $\hat G=u\cM/H(u\cM)$ and $\hat r\colon \hat G\to X/E$ be the induced quotient map we have by Lemma~\ref{lem:weakly_grouplike}(3) (so that in particular, $\hat r$ is an orbit mapping and $\hat r(e_{\hat G})=[x_0]_E$). Because $E$ is not closed, by Lemma~\ref{lem:new_preservation_E_to_H}, $\ker \hat r:=\hat r^{-1}\{[x_0]_E\}$ is not closed, so it is not open (note that it is a subgroup of $\hat G$, as it is just the stabiliser of $[x_0]_E\in X/E$).
		
		Since $Y$ is closed and $E$-saturated, $Y/E\subseteq X/E$ is closed and so is $Y':=\hat r^{-1}[Y/E]$.
		
		We can assume without loss of generality that $e_{\hat G}\in Y'$. Otherwise, for any $g_0\in Y'$, we have $e_{\hat G}\in Y'':=g_0^{-1}Y'$, and then $Y''$ is a closed right $\ker \hat r$-invariant set and we have a bijection between $Y/E=Y'/{\ker \hat r}$ and $Y''/{\ker \hat r}$ (which can be identified with $\hat r[Y'']\subseteq X/E$).
		
		Under this assumption, we also have that $\ker \hat r\subseteq Y'$, and because $Y'$ is closed, $\overline{\ker \hat r}\subseteq Y'$. Since $\ker \hat r$ is not closed, it is not open in $\overline {\ker \hat r}$. On the other hand ${\ker \hat r}$ is analytic --- by Lemma~\ref{lem:new_preservation_E_to_H} --- so it has the Baire property. By Fact~\ref{fct:from_mycielski}, it follows that $\lvert \overline{{\ker \hat r}}/{\ker \hat r}\rvert\geq 2^{\aleph_0}$. It follows that $Y'/{\ker \hat r}=Y/E$ has cardinality at least $2^{\aleph_0}$.
	\end{proof}

	\begin{thm}
		\label{thm:general_cardinality_transitive}
		If $E$ is analytic and weakly closed group-like or weakly uniformly properly group like, then $E$ is clopen or $E$ has at least $2^{\aleph_0}$ many classes.
		
		More generally, suppose $Y\subseteq X$ is closed and $E$-saturated, and suppose that $G^Y$ (the setwise stabiliser of $Y$) has a dense orbit in $Y/E$. Then either $E\restr_{Y}$ is clopen in $Y^2$ (and $Y/E$ is finite) or $\lvert Y/E\rvert\geq 2^{\aleph_0}$.
	\end{thm}
	\begin{proof}
		We will treat the general case. Suppose $\lvert Y/E\rvert <2^{\aleph_0}$. Then by Theorem~\ref{thm:general_cardinality_intransitive}, $E$ is closed (note that this implies that $\ker \hat r$ is closed and $Y/E$ is Hausdorff). If $Y/E$ is finite, it follows easily that $E\restr_Y$ is clopen, so suppose towards contradiction that it is infinite.
		
		Consider $\hat G=u\cM/H(u\cM)$ acting continuously on $X/E$ as in Lemma~\ref{lem:weakly_grouplike}(3), and let $\hat G^Y$ be the setwise stabiliser of $Y/E$. Since $Y$ is closed an $E$-saturated, so is $Y/E$, and thus (by continuity) also $\hat G^Y$, which is therefore a compact group (as a closed subgroup of $\hat G$). Note that for every $g\in G^Y$, we have $uguH(u\cM)\in \hat G^Y$. It follows that $\hat G^Y$ has a dense, and therefore (because $Y/E$ is Hausdorff) infinite orbit in $Y/E$. We may assume without loss of generality that $[x_0]_E\in Y/E$ and $[x_0]_E$ has an infinite $\hat G^Y$-orbit (otherwise, we can replace $Y/E$ by $g^{-1}(Y/E)$ and $\hat G^Y$ by $g^{-1}\hat G^Yg$, for some $g$ such that $g[x_0]_E\in Y$ has infinite $\hat G^Y$-orbit).
		
		Under this assumption, we have a bijection between $\hat G^Y\cdot [x_0]_E\subseteq Y/E$ and $\hat G^Y/(\ker \hat r\cap \hat G^Y)$. Since $E$ is closed, $\ker \hat r$ is a closed subgroup of $\hat G$, so $H^Y:=\ker \hat r\cap \hat G^Y$ is a Baire subgroup of $\hat G^Y$. Using the aforementioned bijection between $\hat G^Y/H^Y$ and the orbit $\hat G^Y\cdot [x_0]_E$, we conclude  that $[\hat G^Y:H^Y]$ is infinite, so by compactness of $\hat G^Y$, $H^Y$ is not open. But then by Fact~\ref{fct:from_mycielski}, it follows that $[\hat G^Y:H^Y]\geq 2^{\aleph_0}$, and thus $\hat G^Y\cdot [x_0]_E\subseteq Y/E$ has cardinality at least $2^{\aleph_0}$, which is a contradiction.
	\end{proof}
	In the case when $Y=X$ and $X$ is metrisable, we can refine the dichotomy from Theorem~\ref{thm:general_cardinality_transitive} by another dividing line, as we will see in Corollary~\ref{cor:metr_smt_cls}. Later, in Chapter~\ref{chap:nonmetrisable_card}, we will discuss a possible variant of this refinement which would apply in the non-metrisable case.

	\section[Group-like quotients and Polish groups and Borel cardinality]{Group-like quotients and Polish groups and Borel cardinality\sectionmark{Polish groups}}
	\sectionmark{Polish groups}
	In this section, we study the consequences of Lemma~\ref{lem:main_abstract_grouplike} for metrisable ambits. In particular, we present the class space of a (weakly uniformly properly or weakly closed) group-like equivalence relation as the quotient of a Polish group by a subgroup, which will later be used to prove Main~Theorem~\ref{mainthm_group_types}.
	
	The following theorem can be considered the principal result of the thesis in the general abstract context. The main results in Chapter~\ref{chap:applications} (in particular, Theorems~\ref{thm:main_aut}, \ref{thm:main_galois} and \ref{thm:main_tdf}) are essentially its specialisations. In Chapter~\ref{chap:nonmetrisable_card}, we discuss possible extensions of this Theorem to the case when $X$ is not metrisable.
	\begin{thm}
		\label{thm:main_abstract}
		Suppose $X$ is metrisable, while $E$ is weakly uniformly properly group-like or weakly closed group-like. Write $D'$ for $\Core(H(u\cM)D)\unlhd u\cM$ and $\hat G$ for the Polish group $u\cM/D'$ (cf.\ Corollary~\ref{cor:Polish_quotient_Core(D)}).
		
		Then $\hat G$ acts continuously and transitively on $X/E$ (as $(fD')\cdot [x]_E=[f(x)]_E$), and the orbit map $\hat r\colon \hat g\mapsto \hat g\cdot [x_0]_E$, the induced equivalence relation $E|_{\hat G}$, and $H\leq \hat G$, defined as the stabiliser of $[x_0]_E$, have the following properties:
		\begin{enumerate}
			\item
			$H\leq \hat G$ and fibres of $\hat r$ are exactly the left cosets of $H$ (so $\hat G/E|_{\hat G}=\hat G/H$),
			\item
			$\hat r$ is a topological quotient map (so it induces a homeomorphism of $\hat G/H$ and $X/E$),
			\item
			$E$ is clopen or closed if and only if $H$ is (respectively)
			\item
			if $E$ is $F_\sigma$, Borel, or analytic (respectively), then so is $H$,
			\item
			$E|_{\hat G}\leq_B E$
		\end{enumerate}
		
		Furthermore:
		\begin{enumerate}
			\setcounter{enumi}{5}
			\item
			if $(G,X)$ is tame, then $E|_{\hat G}\sim_B E$, and
			\item
			if $E$ itself is closed group-like or properly group-like (or, more generally, satisfies the assumptions of Proposition~\ref{prop:wgl_homom}), then $H\unlhd \hat G$ and $\hat r$ is a homomorphism (and hence, by (2), a topological group quotient map).
		\end{enumerate}
	\end{thm}
	\begin{proof}
		Recall that by Corollary~\ref{cor:uM/HuMD_Polish} and Corollary~\ref{cor:Polish_quotient_Core(D)}, we know that the group $\hat G$ and the space $\hat G':=u\cM/H(u\cM)D$ are both compact Polish.
		
		Lemma~\ref{lem:weakly_grouplike} applies, and so do Lemmas~\ref{lem:factorisation_lemma} and \ref{lem:new_preservation_E_to_H}. This gives us the continuous action and (1)-(4) (with $\hat r=\hat r_3$ from Lemma~\ref{lem:factorisation_lemma}).
		
		Note that we have a continuous surjection $\hat G\to \hat G'$, and again by Lemma~\ref{lem:new_preservation_E_to_H}, $\hat r$ factors through it. As both $\hat G$ and $\hat G'$ are compact Polish, Fact~\ref{fct:borel_section} applies, and we have $E|_{\hat G}\sim_B E|_{\hat G'}$.
		
		For (5), note that trivially $E\geq_B E\restr _{\overline{u\cM}/{\equiv}}$ (where we identify $\overline{u\cM}/{\equiv}$ with $R[\overline{u\cM}]\subseteq R[EL/{\equiv}]=X$). On the other hand, we have a commutative diagram:
		\begin{center}
			\begin{tikzcd}
				\overline{u\cM}/{\equiv}\ar[d]\ar[r, two heads] & \hat G'= u\cM/H(u\cM)D \ar[d, two heads] \\
				X\ar[r, two heads] & X/E
			\end{tikzcd}
		\end{center}
		where the map $\overline{u\cM}/{\equiv}\to u\cM/H(u\cM)D$ is the function $[f]_{\equiv}\mapsto ufH(u\cM)D$ given by Proposition~\ref{prop:from_cluM}. Commutativity follows in a straightforward manner from the definitions  (and the fact that $u$ and the whole $H(u\cM)$ act trivially on $X/E$, because the action of $u\cM$ on $X/E$ factors through $u\cM/H(u\cM)$), and it implies that $\overline{u\cM}/{\equiv}\to u\cM/H(u\cM)D$ is a reduction of $E\restr _{\overline{u\cM}/{\equiv}}$ to $E|_{\hat G'}$. By Fact~\ref{fct:borel_section}, we have $E\restr _{\overline{u\cM}/{\equiv}}\sim_B E|_{\hat G'}$, so
		\[
			E|_{\hat G}\sim_BE|_{\hat G'}\sim_B E\restr _{\overline{u\cM}/{\equiv}}\leq_B E.
		\]
		For (6), note that if $(G,X)$ is tame and metrisable, we can apply Corollary~\ref{cor:borel_map} to obtain a commutative diagram:
		\begin{center}
			\begin{tikzcd}
				EL/{\equiv'} \ar[r,"Borel", two heads]\ar[d, two heads] & \hat G'= u\cM/H(u\cM)D\ar[d, two heads]\\
				X \ar[r, two heads] & X/E.
			\end{tikzcd}
		\end{center}
		$EL/{\equiv'}$ is Polish (by Proposition~\ref{prop: quotients of EL are Polish}) and the function $EL/{\equiv'}\to X$ is a continuous surjection, so by Fact~\ref{fct:borel_section}, $E\sim_B E|_{EL/{\equiv'}}$. On the other hand, the Borel function $EL/{\equiv'}\to \hat G'$ is clearly a reduction of $E|_{EL/{\equiv'}}$ to $E|_{\hat G'}$, so
		\[
			E\sim_B E|_{EL/{\equiv'}}\leq_B E|_{\hat G'}\sim_B E|_{\hat G}.
		\]
		
		Finally, for (7), just apply Remark~\ref{rem:homom_factor}.
	\end{proof}
	
	\begin{rem}
		By Proposition~\ref{prop:NIP gives metrizability}, it follows that if $(G,X)$ is tame and metrisable, then $u\cM/H(u\cM)$ is Polish, so in this case, in Theorem~\ref{thm:main_abstract}, we can take $\hat G$ to be $u\cM/H(u\cM)$ instead of $u\cM/D'$ and (with the obvious modifications) the conclusion still holds, with essentially the same proof.
		\xqed{\lozenge}
	\end{rem}

	\begin{cor}
		\label{cor:tame_dominated}
		In Theorem~\ref{thm:main_abstract}, if we have a $G$-ambit morphism $(X,x_0)\to (Z,z_0)$ with $(G,Z)$ tame, and an equivalence relation $F$ on $Z$ such that $E=F|_X$, then $E\sim_B E|_{\hat G}$ (even if $(G,X)$ is untame).
	\end{cor}
	\begin{proof}
		Note that by assumption, $Z$ is metrisable (cf.\ Fact~\ref{fct: preservation of metrizability}).
		
		Choose an Ellis group $u\cM\leq E(G,X)$ and take $D':=\Core(H(u\cM)D)$, so that $\hat G=u\cM/D'$ is as in Theorem~\ref{thm:main_abstract}. Denote by $\varphi$ the morphism $(X,x_0)\to (Z,z_0)$. Then by Proposition~\ref{prop:induced_epimorphism}, $\varphi_*[u\cM]=v\cN$ is an Ellis group in $E(G,Z)$, and if we take $\hat G_Z$ to be $v\cN/D'_Z$ (where $D'_Z=\Core(H(v\cN)D_Z)$ for the naturally defined $D_Z$), then by Remark~\ref{rem:induced_hom_by_D}, $\varphi_*$ induces a topological group quotient mapping $\hat G\to \hat G_Z$.
		
		This quotient map fits into the following commutative diagram.
		\begin{center}
			\begin{tikzcd}
				\hat G\ar[d, two heads]\ar[r, two heads] & X/E\ar[d, two heads] \\
				\hat G_Z\ar[r, two heads] & Z/F.
			\end{tikzcd}
		\end{center}
		The vertical arrows are the quotient map mentioned before and the bijection induced by $\varphi$ (by the assumption that $E=F|_X$). The horizontal arrows are given by $fD'\mapsto [f(x_0)]_E$ and $fD'_Z\mapsto [f(z_0)]_F$ (they are well-defined by Theorem~\ref{thm:main_abstract}). Since for every $f\in E(G,X)$ we have $\varphi_*(f)(z_0)=\varphi(f(x_0))$, the diagram commutes.
		
		It follows immediately that the quotient map $\hat G\to \hat G_Z$ is a continuous reduction of $F|_{\hat G}=E|_{\hat G}$ to $F|_{\hat G_Z}$. Thus, by Fact~\ref{fct:borel_section}, $E|_{\hat G}\sim_B F|_{\hat G_Z}\sim_B F$. But the morphism $(X,x_0)\to (Z,z_0)$ induces a reduction of $E=F|_X$ to $F$, so again by Fact~\ref{fct:borel_section}, $F\sim E$, so $E|_{\hat G}\sim_B E$.
	\end{proof}

	\begin{rem}
		\label{rem:strengthening}
		Similarly, one can show that if either $E$ is weakly closed group-like or $(G,X)$ is tame (or, more generally, it satisfies the assumptions of Corollary~\ref{cor:tame_dominated}), then in Theorem~\ref{thm:main_abstract}, $E|_{\hat G}\sim_B E$ (also in the first case) and $E$ is Borel or analytic if and only if $H$ is. Briefly, in the first case, if we take a closed group-like $F$ dominating $E$, then we can consider $E|_{Z/F}$ as a $Z/F$-invariant equivalence relation on $Z/F$, and then we can apply Proposition~\ref{prop:toy_main} with $G=X=Z/F$ and $E=E|_{Z/F}$, and conclude by successive applications of Fact~\ref{fct:borel_section} and Proposition~\ref{prop:preservation_properties}. In the second case, it can be shown via Remark~\ref{rem:tame_preservation} and Proposition~\ref{prop:preservation_properties}. \xqed{\lozenge}
	\end{rem}
	The following corollary is a part of Main~Theorem~\ref{mainthm:abstract_smt}.
	\begin{cor}
		\label{cor:metr_smt_cls}
		For $E$ as in Theorem~\ref{thm:main_abstract}, $E$ is smooth (according to Definition~\ref{dfn:smt}) if and only if $E$ is closed (as a subset of $X^2$).
		
		Moreover, exactly one of the following holds:
		\begin{enumerate}
			\item
			$E$ is clopen and has finitely many classes,
			\item
			$E$ is closed and has exactly $2^{\aleph_0}$ classes,
			\item
			$E$ is not closed and not smooth. In this case, if $E$ is analytic, then $E$ has exactly $2^{\aleph_0}$ classes.
		\end{enumerate}
	\end{cor}
	\begin{proof}
		Immediate by Theorem~\ref{thm:main_abstract} and Lemma~\ref{lem:abstract_trich}.
	\end{proof}

	\chapter{(Weakly) group-like equivalence relations in model theory and beyond}
	\chaptermark{Group-like equivalence relations in particular contexts}
	\label{chap:applications}
	In this chapter, we apply the main results of Chapter~\ref{chap:grouplike} in specific contexts, mostly model-theoretic. In particular, we prove Main~Theorems~\ref{mainthm_group_types}, \ref{mainthm:smt} and \ref{mainthm:nwg}.
	\section{Compact group actions}
	\begin{ex}
		\label{ex:cpct_glike}
		\begin{figure}[H]
			\begin{tikzcd}
				G \arrow[r]\arrow[dr]& EL=G\arrow[d,"R"]\arrow[dr,"r"]& \\
				\tilde G=G \arrow[r] & X=G \arrow[r]& G/N
			\end{tikzcd}
		\end{figure}
		Suppose $G$ is a compact Hausdorff group and consider $X=G$ with $G$ acting by left translations, with $x_0=e\in G$. Then for any normal $N\unlhd G$ the relation $E=E_N$ of lying in the same $N$-coset is properly group-like on $G$, with ${\equiv}$ on $\tilde G$ being just the equality relation.
		
		Pseudocompleteness means just that whenever $g_i\to x_1$, $h_i\to x_2$ and $g_ih_i\to x_3$ for some nets $(g_i)_i,(h_i)_i$ in $G$ and $x_0,x_1,x_2\in G$, then $x_1x_2=x_3$. But this is clear by joint continuity. Other axioms are easy to verify.
		
		Furthermore, $E_N$ is also uniformly properly group-like: indeed, we can take for $\mathcal E$ to be just the family of sets of the form $D_A:=\{(g,hg)\mid g\in G, h\in A \}$, where $A$ ranges over finite symmetric subsets of $N$ containing the identity. Each $D_A$ is closed (as the continuous image of $G\times A$ by multiplication, which is closed by compactness) and $(D_A)'=D_{A\cdot A}$ has the properties required by Definition~\ref{dfn:unif_prop_glike}.
		
		The conclusion of Lemma~\ref{lem:main_abstract_grouplike} is trivial in this case: the map $G\to G^G$ ($g\mapsto \pi_{G,g}$) is continuous and injective, and therefore an embedding into a closed subset, so $EL=G$, and the only idempotent is the identity in $G$, and thus $u\cM=G$, while $H(u\cM)=\{e\}$.
		
		Likewise, the conclusion of Corollary~\ref{cor:metr_smt_cls} reduces to Proposition~\ref{prop:trichotomy_for_groups}.
		\xqed{\lozenge}
	\end{ex}
	
	\begin{ex}
		Consider a transitive action of a compact Hausdorff group $G$ on a compact Hausdorff space $X$, and let $E$ be any $G$-invariant equivalence relation on $X$. Then $E$ is dominated by equality on $G$, which is clearly closed group-like (and also properly uniformly group-like), so $E$ is weakly closed group-like. By applying Theorem~\ref{thm:main_abstract}, we recover most of Proposition~\ref{prop:toy_main} (the rest can be essentially recovered via Remark~\ref{rem:strengthening}, although this reasoning is circular), and by applying Corollary~\ref{cor:metr_smt_cls}, we recover Corollary~\ref{cor:toy_trich}.
		\xqed{\lozenge}
	\end{ex}
	
	\section{Automorphism group actions}
	In this section, we will be looking at dynamical systems stemming from automorphism group actions, of the form $(\Aut(M),S_m(M),\tp(m/M))$ (and other similar ones). In this context, we will find the naturally occurring (weakly) group-like equivalence relations and then apply to them what we learned from Chapter~\ref{chap:grouplike}. Thus we will recover (with some improvements) the main results of the papers \cite{KPR15} (joint with Krzysztof Krupiński and Anand Pillay) and \cite{KR18} (joint with Krzysztof Krupiński).
	
	\subsection*{Lemmas}
	We intend to apply results of Chapter~\ref{chap:grouplike}. In the following lemmas, we will show that their hypotheses are satisfied in the case of actions of automorphism groups on certain type spaces.
	
	\begin{lem}[pseudocompleteness for automorphism groups]
		\label{lem:pseudocompleteness_for_aut}
		Suppose $M$ is a model and $a$ is an arbitrary tuple in $M$ (e.g.\ an enumeration of $M$), while $N\succeq M$ is an $\lvert M\rvert^+$-saturated and $\lvert M\rvert^+$-strongly homogeneous model (for example, $N=\fC$ and $M$ is small in $\fC$).
		
		Then whenever $(\sigma_i)_i$ and $(p_i)_i$ are nets in $\Aut(M)$ and $S_a(M)$ (respectively) such that $\tp(\sigma_i(a)/M)\to q_1$, $p_i\to q_2$ and $\sigma_i(p_i)\to q_3$ for some $q_1,q_2,q_3\in S_a(M)$, there are $\sigma'_1,\sigma'_2\in \Aut(N)$ such that $\tp(\sigma'_1(a)/M)=q_1$, $\tp(\sigma'_2(a)/M)=q_2$ and $\tp(\sigma'_1\sigma'_2(a)/M)=q_3$. (This is pseudocompleteness for $\tilde G=\Aut(N)$, $X=S_a(M)$ and the map $\tilde G\to S_a(M)$ given by $\sigma\mapsto \tp(\sigma(a)/M)$, see Definition~\ref{dfn:prop_glike}.)
	\end{lem}
	\begin{proof}
		Let for each $i$, choose $b_i\in N$ such that $b_i\models p_i$ (so in particular, $b_i\equiv a$), and extend $\sigma_i$ to $\bar\sigma_i\in \Aut(N)$. Then by the assumptions, for every $\varphi_1(x),\varphi_2(x),\varphi_3(x)$ in $q_1,q_2$ and $q_3$ (resp.) we have, for sufficiently large $i$, $N\models \varphi_1(\bar \sigma_i(a))\land \varphi_2(b_i)\land \varphi_3(\bar\sigma_i(b_i))\land a\equiv b_i\land ab_i\equiv \bar\sigma_i(a)\bar\sigma_i(b_i)$. Now by compactness, we have $a_1,a_2,a_3\in N$ such that $N\models q_1(a_1)\land q_2(a_2)\land q_3(a_3)\land a\equiv a_2\land aa_2\equiv a_1a_3$. Any $\sigma'_1, \sigma'_2\in \Aut(N)$ such that $\sigma'_2(a)=a_2$, $\sigma'_1(aa_2)=a_1a_3$ satisfy the conclusion of the lemma.
	\end{proof}
	
	\begin{lem}
		\label{lem:F_0_for_aut}
		Suppose $A\subseteq \fC$ is small, enumerated by the tuple $a$, while $Y$ is type-definable over $A$.
		
		Then the set $F_0=\{\tp((\sigma'_1)^{-1}\sigma'_2(a)/A) \mid \sigma'_1,\sigma_2'\in \Aut(\fC)\land \sigma'_1(a)\equiv_A \sigma'_2(a)\in Y_A \}$ is closed in $S_a(A)$.
	\end{lem}
	\begin{proof}
		Let $F_1$ be the set of $\tp(b/A)$ such that $\fC\models(\exists a_1,a_2) aa_1\equiv aa_2\land a\equiv a_2\land ab\equiv a_1a_2\land a_1\in Y$. Clearly, $F_1$ is a closed subset of $S_a(A)$. We will show that $F_0=F_1$.
		
		For $\subseteq$, take some $p\in F_0$, as witnessed by $\sigma'_1,\sigma'_2\in \Aut(\fC)$ and take $b=(\sigma'_1)^{-1}\sigma'_2(a)$, $a_1=\sigma'_1(a)$ and $a_2=\sigma'_2(a)$. It is easy to check that they witness that $\tp(b/A)\in F_1$.
		
		For $\supseteq$, take some $p\in F_1$, witnessed by $b,a_1,a_2$. Take $\sigma_1',\sigma'_2\in \Aut(\fC)$, such that $\sigma'_1(ab)=a_1a_2$ and $\sigma'_2(a)=a_2$. It is easy to check that these $\sigma'_1,\sigma'_2$ witness that $p=\tp(b/A)\in F_0$.
	\end{proof}
	
	When reading the proof of Lemma~\ref{lem:lascar_grouplike}, it may be useful to compare the diagram below to the diagram in Definition~\ref{dfn:prop_glike}.
	\begin{center}
		\begin{tikzcd}
			\Aut(M) \ar[r] \ar[dr] & E(\Aut(M), S_m(M)) \ar[d,"R"]\ar[dr,"r"]\\
			\Aut(\fC)\ar[r] & S_m(M)\ar[r] & S_m(M)/{\equiv_\Lasc^M}=\Gal(T)
		\end{tikzcd}
	\end{center}
	\begin{lem}
		\label{lem:lascar_grouplike}
		Let $M$ be a small ambitious model (see Definition~\ref{dfn:ambitious_model}), enumerated by $m$. Consider the $\Aut(M)$-ambit $(\Aut(M),S_m(M),\tp(m/M))$. Then $E={\equiv_\Lasc^M}$ is a uniformly properly group-like equivalence relation on $S_m(M)$.
		
		More generally, if $G^Y\leq \Gal(T)$ is closed, and $M$ is ambitious relative to $G^Y$ (see Definition~\ref{dfn:ambitious_model}), then $(G^Y(M),Y'_M,\tp(m/M))$ (where $Y'$ is as there) is an ambit and ${\equiv_\Lasc^M}\restr_{Y'_M}$ is a uniformly properly group-like equivalence relation on it.
	\end{lem}
	\begin{proof}
		Notice that the ``base" case follows from the ``moreover" case simply by taking $G^Y=\Gal(T)$, so we will treat the second case.
		
		Note that $Y'$ is type-definable over $M$ immediately by Fact~\ref{fct: characterization of topology on Gal_L(T)}. The fact that $(G^Y(M),Y'_M,\tp(m/M))$ is an ambit follows immediately by definition of a relatively ambitious model
		
		Note that almost immediately by our assumptions, $Y'_M$ is the preimage of $G^Y$ by the quotient map $S_m(M)\to \Gal(T)$ from Fact~\ref{fct:sm_to_gal}, so we may identify $G^Y$ and $Y'_M/{\equiv_\Lasc^M}$, and ${\equiv_\Lasc^M}\restr_{Y'_M}$ is evidently group-like.
		
		Let $\tilde G=\{\sigma'\in \Aut(\fC)\mid \sigma'\Autf(\fC)\in G^Y \} $. Immediately by the definitions, $\tilde G$ maps onto $Y'_M$ via $\sigma'\mapsto \tp(\sigma'(m)/M)$, and the induced map $\tilde G\to G^Y$ is just $\sigma'\mapsto \sigma'\Autf(\fC)$, which is of course a homomorphism. By applying Lemma~\ref{lem:pseudocompleteness_for_aut}, we obtain pseudocompleteness (in the sense of Definition~\ref{dfn:prop_glike}) --- to see that, just notice that if $\sigma_i\in G^Y(M)$ and $p_i\in Y'_M$, then (since $Y'_M$ is closed), the $q_1,q_2,q_3$ as in Lemma~\ref{lem:pseudocompleteness_for_aut} are in $Y'_M$, so $\sigma_1',\sigma_2'\in \tilde G$.
		
		By Lemma~\ref{lem:F_0_for_aut} (applied to $Y=Y'$ and $A=M$), we conclude that ${\equiv_\Lasc^M}\restr_{Y'_M}$ is properly group-like.
		
		To see that it is uniformly properly group-like, let $\mathcal E=\{F_n\mid n\in\bN\}$, where $F_n$ is the set of pairs of types $p_1,p_2\in Y'_M$ such that there exist some $m_1,m_2$ satisfying $p_1$ and $p_2$ (respectively) and such that $d_\Lasc(m_1,m_2)\leq n$ (cf.\ Definition~\ref{dfn:Lascar distance}). Clearly, each $F_n$ is symmetric, reflexive and by Fact~\ref{fct:distance_tdf}, they are all closed in $S_m(M)^2$. We will show that $(F_n)':=F_{2n+2}$ has the properties postulated in Definition~\ref{dfn:unif_prop_glike}.
		
		Indeed, if $d_\Lasc(a,b_1)\leq n$ and $d_\Lasc(b_2,c)\leq n$ and $b_1\equiv_M b_2$, then $d_\Lasc(b_1,b_2)\leq 1$, so by triangle inequality $d_\Lasc(a,c)\leq 2n+1$, so $F_n\circ F_n\subseteq F_{2n+1}\subseteq F_{2n+2}$. On the other hand, and if $(\tp(m/M),\tp(\sigma(m)/M))\in F_n$, then there are some $\sigma_1,\sigma_2\in \Aut(\fC/M)$ such that $d_\Lasc(\sigma_1(m),\sigma_2\sigma(m))\leq n$, so $d_\Lasc(m,\sigma_1^{-1}\sigma_2\sigma(m))\leq n$, so by Fact~\ref{fct:diameter_witnessed_by_model}, $d_\Lasc(\sigma_1^{-1}\sigma_2\sigma(m))\leq n+1$, and hence (because $\sigma_1^{-1}\sigma_2$ fixes $M$ pointwise) $d_\Lasc(\sigma)\leq n+2$, i.e.\ for every $a$ we have $d_\Lasc(a,\sigma(a))\leq n+2$, in particular, for every $\sigma'\in \Aut(\fC)$ we have $(\tp(\sigma'(m)/M),\tp(\sigma\sigma'(m)/M))\in F_{n+2}\subseteq F_{2n+2}$. Finally, from Fact~\ref{fct:Lascar_equivalent}, it follows easily that ${\equiv_\Lasc^M}=\bigcup_n F_n$, which completes the proof.
	\end{proof}

	\begin{lem}
		\label{lem:lascar_dominates}
		Given any strong type $E$ on $p(\fC)$, where $p\in S(\emptyset)$, and a small model $M$ enumerated by $m$, such that some $a\subseteq m$ realises $p$, the relation $\equiv_\Lasc^M$ on $S_m(M)$ dominates the relation $E^M$ on $S_a(M)$ (via the restriction $S_m(M)\to S_a(M)$).
		
		More generally, if $Y$ is any $\equiv_\Lasc$-invariant, type-definable set containing $a$, such that $\Aut(\fC/\{Y\})$ acts transitively on $Y$, if $M$ is a small model enumerated by $m\supseteq a$, ambitious relative to $G^Y=\Aut(\fC/\{Y\})/\Autf(\fC)$ (cf.\ Definition~\ref{dfn:ambitious_model}), then $(G^Y(M),Y_M,\tp(a/M))$ is an ambit, and ${\equiv_\Lasc^M}\restr_{Y'_M}$ dominates $E^M\restr_{Y_M}$, where $Y'=\Aut(\fC/\{Y\})\cdot m$.
		
		In particular (by Lemma~\ref{lem:lascar_grouplike}), $E^M$ and $E^M\restr_{Y_M}$ are weakly uniformly properly group-like.
	\end{lem}
	\begin{proof}
		For the first part, just note that the relation $\equiv_\Lasc^M$ on $S_a(M)$ is a refinement of $E^M$, and $\equiv_\Lasc^M$ on $S_m(M)$ dominates $\equiv_\Lasc^M$ on $S_a(M)$ via the restriction mapping $S_m(M)\to S_a(M)$, which follows from the trivial observation that if two tuples are $\equiv_\Lasc$-equivalent, then their subtuples (chosen from corresponding coordinates) are also $\equiv_\Lasc$-equivalent. Since $\Gal(T)=S_m(M)/{\equiv_\Lasc^M}$ acts on $p(\fC)/E$ (and the map $S_m(M)/{\equiv_\Lasc^M}\to S_a(M)/E^M$ is equivariant), this shows that $\equiv_\Lasc^M$ dominates $E^M$.
		
		For the second part, apply Fact~\ref{fct: characterization of topology on Gal_L(T)} and Lemma~\ref{lem:lascar_grouplike} to conclude that $(G^Y(M),Y'_M,\tp(m/M))$ is an ambit and ${\equiv_\Lasc^M}\restr_{Y'_M}$ is uniformly properly group-like. Now, note that the restriction map $S_m(M)\to S_a(M)$, induces an ambit morphism $(\Aut(M),Y'_M,\tp(m/M))\to (\Aut(M),Y_M,\tp(a/M))$. The domination of $E\restr_{Y_M}$ by ${\equiv_\Lasc^M}\restr_{Y'_M}$ follows the same as in the preceding paragraph.
	\end{proof}
	
	\subsection*{Results for automorphism groups}
	\begin{rem}
		\label{rem:wglike_closedness_to_typdef}
		Immediately by Lemma~\ref{lem:lascar_dominates}, we can use Proposition~\ref{prop:top_props_of_wglike} to recover Proposition~\ref{prop:type-definability_of_relations}. We also deduce that if $p\in S(\emptyset)$ and $E$ is an arbitrary strong type on $X=p(\fC)$, then $X/E$ is an $R_1$ space (which was \cite[Proposition 1.12]{KPR15}, but only for $R_0$).\xqed{\lozenge}
	\end{rem}

	The following theorem is Main~Theorem~\ref{mainthm:nwg} and it generalises Fact~\ref{fct:newelski}. It appeared as \cite[Theorem 5.1]{KPR15} (joint with Krzysztof Krupiński and Anand Pillay). Here, we deduce it from the abstract Theorems~\ref{thm:general_cardinality_intransitive} and \ref{thm:general_cardinality_transitive}. Note that --- in contrast to Theorem~\ref{thm:main_aut} below --- we do not require the language to be countable.
	\begin{thm}
		\label{thm:nwg}
		Suppose $E$ is an analytic strong type defined on $X=p(\fC)$ for some $p\in S(\emptyset)$, and $Y\subseteq X$ is type-definable and $E$-saturated. Suppose $\lvert Y/E\rvert<2^{\aleph_0}$.
		
		Then $E$ is type-definable (note that by Remark~\ref{rem:tdf_iff_restr}, this is equivalent to $E\restr_Y$ being type-definable), and if, in addition, $\Aut(\fC/\{Y\})$ acts transitively on $Y/E$, then $E\restr_Y$ is relatively-definable (as a subset of $Y^2$).
	\end{thm}
	\begin{proof}
		Recall from Fact~\ref{fct: Borel in various senses} that $E\restr_Y$ is relatively definable or type-definable if and only if $E^M\restr_{Y_M}$ is clopen or closed (respectively) for some (equivalently, every) small model $M$.
		
		Now, if $M$ is any ambitious model (which exists by Proposition~\ref{prop:amb_exist}), then we can just apply Lemma~\ref{lem:lascar_dominates} and then Theorem~\ref{thm:general_cardinality_intransitive} to conclude that if $Y/E=Y_M/{E^M}$ has cardinality less than $2^{\aleph_0}$, then $E^M\restr_{Y_M}$ is closed, so $E$ is type-definable.
		
		If $\Aut(\fC/\{Y\})$ acts transitively on $Y/E$, then again by Proposition~\ref{prop:amb_exist}, we can choose $M$ to be ambitious relative to $G^Y=\Aut(\fC/\{Y\})/\Autf(\fC)$. Since $Y$ is type-definable, it follows from Fact~\ref{fct: characterization of topology on Gal_L(T)} that $G^Y$ is closed. Thus, we can just apply Lemma~\ref{lem:lascar_dominates} followed by Theorem~\ref{thm:general_cardinality_transitive}.
	\end{proof}

	The following is one of the main results of the thesis, and is a part of Main~Theorem~\ref{mainthm_group_types}. Most of it is \cite[Theorem 8.1]{KR18} (joint with Krzysztof Krupiński). Compared to it, we relax the global NIP assumption for the ``furthermore'' part: we assume only that $Y$ has NIP.
	\begin{thm}
		\label{thm:main_aut}
		Suppose that the theory is countable.
		
		Suppose $E$ is a strong type defined on a type-definable set $X$ (in a countable product of sorts), and $Y\subseteq X$ is type-definable, $E$-saturated and such that $\Aut(\fC/\{Y\})$ acts transitively on $Y$ (e.g.\ $Y=[a]_{\equiv}$ or $Y=[a]_{\equiv_\KP}$ for some tuple $a$). Choose $a\in Y$.
		
		Then there is a compact Polish group $\hat G^Y$ acting continuously on $Y/E$, and such that the stabiliser $H$ of $[a]_E$, and the orbit map $\hat r_Y\colon \hat G^Y\to Y/E$, $\hat g\mapsto \hat g\cdot [a]_E$ have the following properties:
		\begin{enumerate}
			\item
			$H\leq \hat G^Y$ and fibres of $\hat r_Y$ are exactly the left cosets of $H$ (so $\hat G^Y/E|_{\hat G^Y}=\hat G^Y/H$),
			\item
			$\hat r_Y$ is a topological quotient map (so it induces a homeomorphism of $\hat G^Y/H$ and $X/E$),
			\item
			$E\restr_Y$ is relatively definable or type-definable if and only if $H$ is clopen or closed (respectively)
			\item
			if $E$ is $F_\sigma$, Borel, or analytic (respectively), then so is $H$,
			\item
			$\hat G^Y/H\leq_B E$.
		\end{enumerate}
		Furthermore, if $Y$ has NIP (in particular, if $T$ has NIP), then $\hat G^Y/H\sim_B E$.
	\end{thm}
	\begin{proof}
		By Proposition~\ref{prop:amb_exist}, we can find some $M$ which is ambitious relative to $G^Y=\Aut(\fC/\{Y\})/\Autf(\fC)$. Then by Lemma~\ref{lem:lascar_dominates}, $E^M\restr_{Y_M}$ is weakly uniformly properly group-like (in the ambit $(G^Y(M),Y_M,\tp(a/M))$), so Theorem~\ref{thm:main_abstract} applies to $E^M\restr_{Y_M}$.
		
		Now, if $Y$ has NIP, then by Corollary~\ref{cor:NIP_implies_tame}, $(G^Y(M), Y_M)$ is tame, so in particular,  so we also have (6) of Theorem~\ref{thm:main_abstract}.
		
		To complete the proof, just note that:
		\begin{itemize}
			\item
			by definition, the Borel cardinality of $E\restr_Y$ is the Borel cardinality of $E^M\restr_{Y_M}$,
			\item
			we identify $Y/E$ and $Y_M/{E^M}$ (via the natural homeomorphism),
			\item
			by Fact~\ref{fct: Borel in various senses}, $E\restr_Y$ is relatively definable, type-definable, $F_\sigma$, Borel, analytic if and only if $E^M\restr_{Y_M}$ is clopen, closed, $F_\sigma$, Borel or analytic (respectively).
		\end{itemize}
		These observations allow us to translate the conclusion of Theorem~\ref{thm:main_abstract} into the conclusion of the theorem, and thus we are done.
	\end{proof}
	
	\begin{rem}
		By referring back to the statement of Theorem~\ref{thm:main_abstract}, we can see that in Theorem~\ref{thm:main_aut}, the group $\hat G^Y$ is $u\cM/\Core(H(u\cM)D)$ for the ambit $(G(M),Y_M,y_0)$, and the action is induced by the action of $E(G(M),Y_M)$ on $Y_M$.\xqed{\lozenge}
	\end{rem}
	
	\begin{prop}
		\label{prop:orbital_in_main_aut}
		In Theorem~\ref{thm:main_aut}, if the stabiliser of $[a]_E$ is a normal subgroup of $\Aut(\fC/\{Y\})$, then $H$ is a normal subgroup of $\hat G^Y$, and $\hat r_Y$ is a topological group quotient mapping onto then $Y/E$ (equipped with the group structure obtained by identification with $\Aut(\fC/\{Y\})/\Stab_{\Aut(\fC/\{Y\})}\{[a]_E\}$).
	\end{prop}
	\begin{proof}
		It is not hard to see that under the hypotheses, the map $G^Y\to Y/E$ given by $\sigma\Autf(\fC)\mapsto \sigma \Stab_{\Aut(\fC/\{Y\})}\{[a]_E\}$ is a homomorphism. But $G^Y$ is naturally identified with $Y'_M/{\equiv_\Lasc^M}$ (with $Y'$  chosen as in the proof of Lemma~\ref{lem:lascar_dominates}, so that ${\equiv_\Lasc^M}$ is uniformly properly group-like on $Y'_M$ and dominates $E^M$). Hence, the assumptions of Proposition~\ref{prop:wgl_homom} are satisfied. Thus Theorem~\ref{thm:main_abstract}(7) applies, and we are done.
	\end{proof}
	
	\begin{cor}
		\label{cor:galois_quotient}
		The Galois group $\Gal(T)$ of any countable first order theory $T$ is isomorphic as a topological group to the quotient of a compact Polish group by an $F_\sigma$ subgroup.
		
		If the theory is NIP, it also has the same Borel cardinality as that quotient.
				
		The same is true for $\Gal_0(T)$.
	\end{cor}
	\begin{proof}
		Choose any tuple $m$ enumerating a model and apply Theorem~\ref{thm:main_aut} to $Y=[m]_{\equiv}$ (for $\Gal(T)$) and $Y=[m]_{\equiv_\KP}$ (for $\Gal_0(T)$), $a=m$ and $E={\equiv_\Lasc}$, noting that Proposition~\ref{prop:orbital_in_main_aut} applies (e.g.\ because by Fact~\ref{fct:diameter_witnessed_by_model}, the relevant stabiliser is just $\Autf(\fC)$).
	\end{proof}
	
	The following trichotomy appeared (essentially) as \cite[Corollary 6.1]{KPR15} (joint with Krzysztof Krupiński and Anand Pillay). It constitutes most of Main~Theorem~\ref{mainthm:smt} (completed by Corollary~\ref{cor:smt_type}).
	\begin{cor}
		\label{cor:trich_plus}
		Suppose that the theory is countable, while $E$ is a strong type (on a set of countable tuples), and $Y$ is type-definable, $E$-saturated, and such that $\Aut(\fC/\{Y\})$ acts transitively on $Y$ (e.g.\ $Y\in \{[a]_\equiv,[a]_{\equiv_{\textrm{Sh}}},[a]_{\equiv_\KP} \}$ for some countable tuple $a$). Then exactly one of the following is true:
		\begin{enumerate}
			\item
			$E\restr_Y$ is relatively definable (as a subset of $Y^2$) and has finitely many classes,
			\item
			$E\restr_Y$ is type-definable and has exactly $2^{\aleph_0}$ classes,
			\item
			$E\restr_Y$ is not type-definable and not smooth. In this case, if $E\restr_Y$ is analytic, then $E\restr_Y$ has exactly $2^{\aleph_0}$ classes.
		\end{enumerate}
	\end{cor}
	\begin{proof}
		By Theorem~\ref{thm:main_aut}, we can apply Lemma~\ref{lem:abstract_trich}, which (by Fact~\ref{fct: Borel in various senses}) completes the proof.
	\end{proof}
	
	The following corollary of Theorem~\ref{thm:main_aut} gives a partial answer to the question of possible Borel cardinalities of Lascar strong types (and strong types in general), as raised in \cite{KPS13} (for example, it implies that in NIP theories, the Borel cardinality of $[a]_{\equiv_\KP}/{\equiv_\Lasc}$ is the Borel cardinality of the quotient of a compact Polish group by an $F_\sigma$ subgroup).
	\begin{cor}
		Suppose $E$ is a strong type, while $Y$ is a type-definable and $E$-saturated set such that $\Aut(\fC/\{Y\})$ acts transitively on $Y$ (e.g.\ $E$ refines $\equiv_\KP$ and $Y=[a]_{\equiv_\KP}$ for some $a$). Suppose in addition that $Y$ has NIP. Then $E\restr_Y$ is Borel equivalent to the quotient of a compact Polish group by a subgroup (which is $F_\sigma$, Borel or analytic, respectively, whenever $E$ is such).
	\end{cor}
	\begin{proof}
		This is immediate by Theorem~\ref{thm:main_aut}.
	\end{proof}
	
	\begin{prop}
		\label{prop:closure_has_transitive_action}
		If $E$ is a strong type on a ($\emptyset$-)type-definable set $X$, then for any $a_0\in X$, the set $Y_{a_0}$ of $a\in X$ such that $[a]_E\in \overline{\{[a_0]_E\}}\subseteq X/E$ is type-definable and $E$-saturated. Moreover, $\Aut(\fC/\{Y_{a_0}\})$ is equal to $\{\sigma\in \Aut(\fC)\mid \sigma(a_0)\in Y_{a_0} \}$, and it acts transitively on $Y_{a_0}$.
	\end{prop}
	\begin{proof}
		We may assume without loss of generality that $X=p(\fC)=[a_0]_{\equiv}$ (since $E$ is a strong type, we have $[a_0]_E\subseteq [a_0]_{\equiv}$, which implies that $Y_{a_0}\subseteq [a_0]_{\equiv}$).
		
		Type-definability of $Y_{a_0}$ is straightforward by the definition of the logic topology, $E$-invariance is trivial.
		
		Since $Y_{a_0}\subseteq [a_0]_{\equiv}$, we have that for each $a\in Y_{a_0}$, there is some $\sigma\in \Aut(\fC)$ such that $\sigma(a_0)=a$. It is enough to show that $\sigma\in \Aut(\fC/\{Y_{a_0}\})$. But since $\sigma$ is an automorphism, it acts on $X/E$ by homeomorphisms, so $\sigma(\overline{\{[a_0]_E\}})=\overline{\{[a]_E\}}$. Since $[a]_E\in \overline{\{[a_0]_E\}}$, it follows that $\overline{\{[a]_E\}}\subseteq \overline{\{[a_0]_E\}}$.
		It follows that $\sigma[Y_{a_0}/E]\subseteq Y_{a_0}/E$, so $\sigma[Y_{a_0}]\subseteq Y_{a_0}$.

		On the other hand, by Remark~\ref{rem:wglike_closedness_to_typdef}, $X/E$ is an $R_0$ space, whence $[a_0]_E\in \overline{\{[a]_E\}}$, so $a_0\in Y_a$. Arguing as in the preceding paragraph, we conclude that $Y_{a_0}=\sigma^{-1}[Y_a]\subseteq Y_a=\sigma[Y_{a_0}]$, so $\sigma[Y_{a_0}]=Y_{a_0}$.
	\end{proof}
	
	The following corollary is the last part of Main~Theorem~\ref{mainthm:smt}. Essentially, it is a generalisation of the main results of \cite{KMS14} and \cite{KM14}/\cite{KR16} (Facts~\ref{fct:KMS_theorem} and \ref{fct:mainA}). It appeared as \cite[Theorem 4.1]{KPR15} (in the paper joint with Krzysztof Krupiński and Anand Pillay). Loosely speaking, it can be seen as a strengthening of Theorem~\ref{thm:nwg} in the ``countable language case'' (as in that case, if $E$ is Borel, then by Fact~\ref{fct:silver}, non-smoothness implies having $2^{\aleph_0}$ classes).
	\begin{cor}
		\label{cor:smt_type}
		Suppose $T$ is countable. If $E$ is a strong type defined on $X=p(\fC)$ for some $p\in S(\emptyset)$ (in countably many variables) and $Y\subseteq X$ is nonempty, type-definable and $E$-saturated, then either $E$ is type-definable, or $E\restr_Y$ is non-smooth.
	\end{cor}
	\begin{proof}
		Suppose $E\restr_Y$ is smooth. We need to show that $E$ is type-definable, which by Proposition~\ref{prop:type-definability_of_relations} is equivalent to $E$ having a type-definable class. Since $E\restr_{Y}$ is smooth, it follows that for every $a\in Y$, $E\restr_{Y_a}$ is also smooth, where $Y_a$ is as in Proposition~\ref{prop:closure_has_transitive_action}. Clearly, it is enough to show that $E\restr_{Y_a}$ has a type-definable class, and thus we may assume without loss of generality that $Y=Y_a$. But then by Proposition~\ref{prop:closure_has_transitive_action}, $\Aut(\fC/\{Y\})$ acts transitively on $Y$, so by Corollary~\ref{cor:trich_plus}, $E\restr_{Y}$ is type-definable, and by Remark~\ref{rem:tdf_iff_restr}, $E$ is type-definable as well.
	\end{proof}
	(See also Corollary~\ref{cor:smt_aut} for a further generalisation (to sets larger than $p(\fC)$).)
	
	Theorem~\ref{thm:main_aut} shows that every quotient of a sufficiently symmetric type-definable set $Y$ by a strong type $E$ is essentially the quotient of a compact Polish group by a subgroup. It is not hard to see that the compact group does not depend on $E$, only on $Y$.
	
	We can actually show more: essentially, we can find one $\hat G$ witnessing Theorem~\ref{thm:main_aut} for all $Y=p(\fC)$, where $p\in S(\emptyset)$, but first, we need an additional lemma.
	
	\begin{lem}
		\label{lem:every_stype_on_m}
		Suppose $T$ is countable.
		Assume that $E$ is a strong type defined on $p(\fC)$ for $p=\tp(a/\emptyset)$ for some countable tuple $a$, while $M$ is an arbitrary countable model, enumerated by $m$.
		
		Then there is a strong type $E'$ on $[m]_{\equiv}$ such that:
		\begin{itemize}
			\item
			$E$ is type-definable [resp. Borel, or analytic, or $F_\sigma$, or relatively definable] if and only if $E'$ is,
			\item
			there are Borel maps $r_1\colon S_m(M)\to S_a(M)$ and $r_2 \colon S_a(M)\to S_m(M)$ such that $r_1$ and $r_2$ are Borel reductions between $(E')^M$ and $E^M$ (in particular, $E'\sim_B E$), satisfying $r_1(\tp(m/M))=\tp(a/M)$ and $r_2(\tp(a/M))=\tp(m/M)$, and
			\item
			the induced maps $r_1'\colon [m]_{\equiv}/E' \to p(\fC)/E$ and $r_2'\colon p(\fC)/E \to [m]_{\equiv}/E'$ are $\Gal(T)$-equivariant homeomorphisms, and $r_2'$ is the inverse of $r_1'$.
		\end{itemize}
		The maps $r_1'$ and $r_2'$ are uniquely determined by $r_1'([\sigma(m)]_{E'})=[\sigma(a)]_{E}$ and $r_2'([\sigma(a)]_{E})=[\sigma(m)]_{E'}$ for all $\sigma \in \Aut(\fC)$.
	\end{lem}
	\begin{proof}
		Let $N\succeq M$ be a countable model containing $a$, and enumerate it by $n\supseteq am$.
		
		Then we have the restriction maps $S_n(N)\to S_a(M)$, $S_n(N)\to S_m(M)$, which fit in the commutative diagram:
		\begin{center}
			\begin{tikzcd}
			S_m(M) \ar[dr, two heads] & S_n(N) \ar[l, two heads] \ar[r, two heads] \ar[d, two heads] & S_a(M) \ar[d, two heads] \\
			&\Gal(T) \ar[r, two heads] & p(\fC)/E.
			\end{tikzcd}
		\end{center}
		In this diagram, the maps to $\Gal(T)$ are given by Fact~\ref{fct:sm_to_gal}, while the map $\Gal(T)\to p(\fC)/E$ is the orbit map $\sigma\Autf(\fC)\mapsto [\sigma(a)]_E$ (cf. Proposition~\ref{prop:gal_action}).
		
		Recall from Fact~\ref{fct: Borel in various senses} that $E$ is type-definable [resp. Borel, analytic, $F_\sigma$, relatively definable] if and only if the induced relation $E^M$ on $S_a(M)$ is closed [resp. Borel, analytic, $F_\sigma$, clopen]. Note also that $E^M=E|_{S_{a}(M)}$ (using Definition~\ref{dfn:induced_relation}).
		
		Note that $E|_{S_m(M)}$ and $E|_{S_n(N)}$ are both induced by the same left invariant equivalence relation $E|_{\Gal(T)}$ on $\Gal(T)$ (left invariance holds because the map $\Gal(T)\to p(\fC)/E$ is left $\Gal(T)$-equivariant, as the orbit map of a left action).

		Let $E'$ be the $\Aut(\fC/M)$-invariant equivalence relation on $[m]_{\equiv}$ such that $(E')^M$ is $E|_{S_m(M)}$. It is $\Aut(\fC)$-invariant by construction (e.g.\ because $E|_{\Gal(T)}$ is left invariant), and it is clearly bounded by the size of $p(\fC)/E$. We will show that it satisfies the conclusion.
		
		Since $E|_{S_n(N)}$ is the pullback of both $E|_{S_m(M)}$ and $E|_{S_a(M)}$ (and so, a preimage of each by a continuous surjection), the first part follows by Proposition~\ref{prop:preservation_properties}.
		
		Fact~\ref{fct:borel_section} gives us Borel sections $S_m(M)\to S_n(N)$ and $S_a(M)\to S_n(N)$ of the restriction maps, and we can assume without loss of generality that each section maps $\tp(m/M)$ or $\tp(a/M)$ (respectively) to $\tp(n/N)$ (possibly by changing the value of the section at one point). Those sections, composed with the appropriate restrictions from $S_n(N)$, yield Borel maps $r_1\colon S_m(M)\to S_a(M)$ and $r_2\colon S_a(M)\to S_m(M)$, which (by the last sentence) take $\tp(m/M)$ to $\tp(a/M)$ and vice versa. These maps are clearly Borel reductions between $E|_{S_m(M)}$ and $E^M$ (passing via $E|_{S_n(N)}$). Denote by $r_1',r_2'$ the induced maps between the class spaces (as in the statement of the lemma), where we freely identify various homeomorphic quotient spaces (e.g.\ $p(\fC)/E$ and $S_a(M)/E^M$).
		
		Now, note that given any $\sigma\in \Aut(\fC)$, the restriction of $\tp(\sigma(n)/N)\in S_n(N)$ to $S_m(M)$ is $\tp(\sigma(m)/M)$ and likewise, the restriction to $S_a(M)$ is $\tp(\sigma(a)/M)$. It follows easily that for every $\sigma\in \Aut(\fC)$, we have $r_1'([\tp(\sigma(m)/M]_{E|_{S_m(M)}})=[\tp(\sigma(a)/M)]_{E^M}$ and likewise, $r_2'([\tp(\sigma(a)/M)]_{E^M})=[\tp(\sigma(m)/M)]_{E|_{S_m(M)}}$. In particular, $r_1'$ and $r_2'$ are bijections with $r_2'$ being the inverse of $r_1'$, and they are $\Gal(T)$-equivariant.
		
		Note that all the maps in the diagram are quotient maps, so in particular, the composed map $S_m(M)\to p(\fC)/E$ is a quotient map. It is easy to see that this map is the composition of the bijection $r_1'$ and the quotient map $S_m(M)\to [m]_{\equiv}/E'$, which implies that $r_1'$ is a homeomorphism, and hence, so is $r_2'$.
		
		Finally note that the conditions $r_1(\tp(m/M))=\tp(a/M)$ and $r_2(\tp(a/M))=\tp(m/M)$, together with $\Gal(T)$-equivariance of $r_1'$ and $r_2'$, imply that $r_1'$ and $r_2'$ are determined by $r_1'([\sigma(m)]_{E'})=[\sigma(a)]_{E}$ and $r_2'([\sigma(a)]_{E})=[\sigma(m)]_{E'}$, for all $\sigma \in \Aut(\fC)$.
	\end{proof}
	The next theorem (along with Theorem~\ref{thm:main_aut}) completes Main~Theorem~\ref{mainthm_group_types}. It is \cite[Theorem 7.13]{KR18}, and is the main result of that paper (joint with Krzysztof Krupiński).
	\begin{thm}
		\label{thm:main_galois}
		Let $T$ be an arbitrary countable theory, and let $M$ be any countable ambitious model of $T$ (such a model exists by Proposition~\ref{prop:amb_exist}).
		
		Consider the ambit $(\Aut(M),S_m(M),\tp(m/M))$, and let $\hat G$ be the compact Polish group $u\cM/\Core(H(u\cM)D)$ (as in Corollary~\ref{cor:Polish_quotient_Core(D)}). Then the orbit map $E(\Aut(M),S_m(M))\to S_m(M)$, $f\mapsto f(\tp(m/M))$ induces a topological group quotient mapping $\hat r\colon \hat G\to \Gal(T)$ (identified with $S_m(M)/{\equiv_\Lasc^M}$ via Fact~\ref{fct:sm_to_gal}) with the following property.
		
		Suppose $E$ is a strong type defined on $p(\fC)$ for some $p\in S(\emptyset)$ (in countably many variables). Fix any $a\models p$.
		
		Denote by $r_{[a]_E}$ the orbit map $\Gal(T)\to p(\fC)/E$ given by $\sigma \Autf(\fC)\mapsto[\sigma(a)]_E$ (i.e. the orbit map of the natural action of $\Gal(T)$ on $p(\fC)/E$ from Proposition~\ref{prop:gal_action}).
		
		Then for $\hat r_{[a]_E}:=r_{[a]_E}\circ \hat r$ and $H=\ker \hat r_{[a]_E}:=\hat r_{[a]_E}^{-1}[[a]_E]$ we have that:
		\begin{enumerate}
			\item
			$H\leq \hat G$ and the fibres of $\hat r_{[a]_E}$ are the left cosets of $H$,
			\item
			$\hat r_{[a]_E}$ is a topological quotient map, and so $p(\fC)/E$ is homeomorphic to $\hat G/H$,
			\item
			$E$ is type-definable if and only if $H$ is closed,
			\item
			$E$ is relatively definable on $p(\fC) \times p(\fC)$ if and only if $H$ is clopen,
			\item
			if $E$ is Borel [resp. analytic, or $F_\sigma$], then $H$ is Borel [resp. analytic, or $F_\sigma$],
			\item
			$E_H\leq_B E$, where $E_H$ is the relation of lying in the same left coset of $H$.
		\end{enumerate}
		
		Moreover, if $T$ has NIP (or, more generally, if $M$ is a tame ambitious model --- cf.\ Definition~\ref{dfn:tame_model}), then $E_H\sim_B E$.
	\end{thm}
	\begin{proof}
		By Lemma~\ref{lem:lascar_grouplike}, $\equiv_\Lasc^M$ is uniformly properly group-like (in the ambit $(\Aut(M),S_m(M),\tp(m/M))$), so by Theorem~\ref{thm:main_abstract}, we obtain the epimorphism $\hat r \colon \hat G\to [m]_{\equiv}/{\equiv_\Lasc}$ (identified with $\Gal(T)$ and $S_m(M)/{\equiv_\Lasc}$ via Fact~\ref{fct:sm_to_gal} and Fact~\ref{fct:logic_by_type_space}).
		
		Now, fix an $p$, $a\models p$ and $E$. Note that Lemma~\ref{lem:every_stype_on_m} yields a strong type $E'$ on $[m]_\equiv$ and a map $r_1'\colon [m]_\equiv/E' \to p(\fC)/E$ satisfying all the conclusions of that lemma, in particular, $r_1'([\sigma(m)]_{E'})= [\sigma(a)]_E$ for any $\sigma \in \Aut(\fC)$. Therefore, $r_{[a]_E}=r_1'\circ r_{[m]_{E'}}$, and so $\hat r_{[a]_E}=r_1'\circ \hat r_{[m]_{E'}}$. This, together with the conclusions of Lemma~\ref{lem:every_stype_on_m}, shows that we can assume without loss of generality that that $m=a$.
		
		Note that the Borel cardinality of $E$ is by definition the Borel cardinality of $E^M$, $p(\fC)/E$ is homeomorphic to $S_m(M)/E^M$, and by Fact~\ref{fct: Borel in various senses}, we can translate topological and descriptive properties of $E$ to those of $E^M$. Recall also that by definition, if $M$ is a tame ambitious model (which is true for any ambitious model under NIP, see Corollary~\ref{cor:NIP_implies_tame}), then the dynamical system $(\Aut(M),S_m(M))$ is tame.
		
		Now, if $E$ is a strong type defined on $[m]_{\equiv}$, then it is refined by $\equiv_\Lasc$. Therefore, $E^M$ is refined by $\equiv_\Lasc^M$; of course, $\Gal(T)=S_m(M)/{\equiv_\Lasc^M}$ acts on itself preserving $E|_{\Gal(T)}$ so $E^M$ is dominated by $\equiv_\Lasc^M$, and as such it is weakly uniformly properly group-like in the ambit $(\Aut(M),S_m(M),\tp(m/M))$. Thus Theorem~\ref{thm:main_abstract} applies to $E^M$, and the function in its conclusion is just $\hat r_{[m]_{E}}\colon \hat G\to [m]_{\equiv}/E=S_m(M)/E^M$. By the preceding paragraph, the properties of $\hat r_{[m]_{E}}$ given by Theorem~\ref{thm:main_abstract} give us the properties postulated by this theorem, which completes the proof.
	\end{proof}

	\begin{cor}
		In Theorem~\ref{thm:main_galois}, we also have $E_H\sim_B E$ if $E$ is coarser than $\equiv_\KP$.
	\end{cor}
	\begin{proof}
		If $E$ is coarser than $\equiv_\KP$, then by Theorem~\ref{thm:main_over_KP}, the orbit map $r_{[a]_E,KP}$ from $\Gal_\KP(T)$ into $p(\fC)/E$ gives us $E\sim_B E|_{\Gal_\KP(T)}$
		
		On the other hand, if we denote by $\hat r_{KP}$ the function $\hat G\to S_m(M)/{\equiv_\KP^M}=\Gal_\KP(T)$ which we get by applying Theorem~\ref{thm:main_galois} to $\equiv_\KP$ on $[m]_{\equiv}$, then $\hat r_{KP}$ is a continuous epimorphism of compact Polish groups. Because $\hat r_{[a]_E}=r_{[a]_E,KP}\circ \hat r_{KP}$, it follows by Fact~\ref{fct:borel_section} that $E|_{\hat G}\sim_B E|_{\Gal_\KP(T)}\sim_B E$.
		
		(Alternatively, this follows from Remark~\ref{rem:strengthening}.)
	\end{proof}
	
	\begin{rem}
		In Theorem~\ref{thm:main_galois}, it seems plausible that we may also have $E_H\sim_B E$ for all $E$ on a tame $p(\fC)$. If $p$ is realised in the ambitious model $M$ chosen in the proof, then it follows from Corollary~\ref{cor:tame_dominated}, but in general (e.g.\ if $a$ enumerates a larger model), there seems to be no obvious argument. The main obstacle is that there is no clear connection between the Ellis groups of flows $(\Aut(M),S_m(M))$ as we vary the ambitious model $M$. It is possible that one can use methods similar to those introduced in \cite{KNS17} to show this stronger result.\xqed{\lozenge}
	\end{rem}
	
	\begin{rem}
		Theorem~\ref{thm:main_galois} can be extended in the following way: given a subgroup $G_0\leq \Aut(\fC)$ containing $\Autf(\fC)$, and such that $G_0/\Autf(\fC)$ is closed in $\Gal(T)$, we can find a group $\hat G_0$ which witnesses Theorem~\ref{thm:main_aut} for all sets of the form $Y_a:=G_0\cdot a$ (for some countable tuple $a$) in such a way that the action of $\hat G_0$ on $Y/E$ factors through $G_0/\Autf(\fC)$. The proof is analogous, except we have to choose a countable model ambitious relative to $G_0/{\Autf(\fC)}$, and prove a variant of Lemma~\ref{lem:every_stype_on_m} for $Y_a$ in place of $p(\fC)$.
	\end{rem}

	\section{Actions of type-definable groups}
	In this section, we consider a group $G$ acting on a set $X$, such that $G,X$ and the action are type-definable (over $\emptyset$, unless specified otherwise).
	
	We also consider, for a small model $M$, the group $G(M)$ acting on $S_G(M)=G_M$ (the space of types over $M$ concentrated on $G$) and on $X_M$ (the space of $M$-types of elements of $X$). Note that for every $g\in G$, the map $x\mapsto gx$ is type-definable over $g$. It follows immediately that $G(M)$ acts on $X_M$ by homeomorphisms.
	
	\index{EG000@$E_{G^{000}_\emptyset}$}
	Throughout the section, we denote by $E_{G^{000}_\emptyset}$ the coset equivalence relation of $G^{000}_\emptyset$ on $G$.
	
	\begin{rem}
		Note that if $G$ is a group invariant over a small model $M$, then it is easy to see that the set of elements of $G$ invariant over $M$ is a group.
		
		On the other hand, if $g\in G$ is invariant over $M$, then every coordinate of $g$ is definable over $M$ (because it definable and fixed by $\Aut(\fC/M)$), and as such, it is an element of $G$. It follows that $g$ is a tuple of elements of $G$, so $g\in G(M)$, so $G(M)$ is always a subgroup of $G$.\xqed{\lozenge}
	\end{rem}

	\subsection*{Lemmas}
	We intend to apply results of Chapter~\ref{chap:grouplike}. In the following lemmas, we will show that their hypotheses are satisfied in the case of transitive type-definable group actions.
	
	\begin{lem}
		\label{lem:quot_by_normal_grouplike}
		Suppose $G$ is a $\emptyset$-type-definable group and $N\unlhd G$ is an invariant normal subgroup of bounded index. Suppose $M$ is a model such that $G(M)=G(M)\cdot \tp(e/M)$ is dense in $S_G(M)$.
		
		Then for the relation $E_N$ of lying in the same coset of $N$, $E_N^M$ is group-like on $(G(M),S_G(M))$.
	\end{lem}
	\begin{proof}
		Note that $G/E_N=G/N$ and so, by Fact~\ref{fct:logic_by_type_space}, topologically $S_G(M)/E_N^M=G/N$. On the other hand, $G/N$ is a topological group by Fact~\ref{fct:quotient_by_bounded_subgroup}. The fact that $G(M)\to G/N$ is a group homomorphism is trivial.
	\end{proof}

	\begin{lem}[pseudocompleteness for type-definable groups]
		\label{lem:pseudocompleteness_for_groups}
		Let $M$ be a model, and suppose $G$ is a group type-definable over $M$. Let $N\succeq M$ be $\kappa^+$-saturated, where $\kappa$ is the maximum of $\lvert M\rvert$ and the length of an element of $G$, considered as a tuple (e.g.\ $N=\fC$ and $M$ is small in $\fC$).
		
		Then whenever $(g_i)_i$ and $(p_i)_i$ are nets in $G(M)$ and $S_G(M)$ (respectively) such that $\tp(g_i/M)\to q_1$, $p_i\to q_2$ and $g_i(p_i)\to q_3$ for some $q_1,q_2,q_3\in S_G(M)$, there are $g'_1,g'_2\in G(N)$ such that $\tp(g'_1/M)=q_1$, $\tp(g'_2/M)=q_2$ and $\tp(g'_1g'_2/M)=q_3$. (This is pseudocompleteness for $\tilde G=G(N)$, $X=S_G(M)$ and the map $\tilde G\to X$ given by $g\mapsto \tp(g/M)$, see Definition~\ref{dfn:prop_glike}.)
	\end{lem}
	\begin{proof}
		Take any net $(a_i)_i$ in $N$ such that for all $i$, $a_i\models p_i$. Then for each $\varphi_1,\varphi_2,\varphi_3$ in $q_1,q_2,q_3$ (respectively), we have for sufficiently large $i$ that $N\models \varphi_1(g_i)\land \varphi_2(a_i)\land \varphi_3(g_ia_i)$. The conclusion follows easily by compactness.
	\end{proof}
	
	\begin{lem}
		If $M$ is any model, and If $G$ is an $M$-type-definable group, then the set $F_0=\{\tp(g_1^{-1}g_2/M)\mid g_1\equiv_M g_2\in G \}$ is closed in $S_G(M)$.
	\end{lem}
	\begin{proof}
		Straightforward.
	\end{proof}
	
	When reading the proof of Lemma~\ref{lem:coset_rel_is_glike}, it may be helpful to compare the diagram below to the one in Definition~\ref{dfn:prop_glike}.
	
	\begin{center}
		\begin{tikzcd}
			G(M)\ar[r]\ar[dr] & E(G(M),S_G(M))\ar[d,"R"]\ar[dr,"r"] \\
			G(\fC)\ar[r] & S_G(M)\ar[r] & G/G^{000}_{\emptyset}
		\end{tikzcd}
	\end{center}
	\begin{lem}
		\label{lem:coset_rel_is_glike}
		Given a type-definable group $G$ and a small model $M$ such that $G(M)\cdot \tp(m/M)$ is dense in $S_G(M)$, we have that the coset relation $E_{G^{000}_{\emptyset}}^M$ on $S_G(M)$ is uniformly properly group-like (according to Definition~\ref{dfn:unif_prop_glike}).
	\end{lem}
	\begin{proof}
		From the preceding lemmas it follows easily that $\tilde G=G=G(\fC)$ witnesses that $E_{G^{000}_{\emptyset}}^M$ is properly group-like, with $[\tilde g]_{\equiv}=\tp(\tilde g/M)$ (in the sense of Definition~\ref{dfn:prop_glike}).
		
		To see that it is uniformly properly group-like, denote by $A$ the symmetric subset of $G$ consisting of products $g^{-1}g'$, where $g\equiv_M g'$, and let $\mathcal E$ be the family of sets of the form $F_n=\{
		(\tp(g/M),\tp(g'g/M))\mid g \in G\land g'\in A^n \}$ (where $A^n$ is the set of all products of $n$ elements of $A$). They are clearly closed, symmetric and contain the diagonal in $S_G(M)^2$.
		
		We have that $F_{2n+1}\supseteq F_{n}\circ F_n$. Indeed, suppose we have two pairs in $F_n$: $(\tp(g/M),\tp(g'g/M))$ and $(\tp(h/M), \tp(h'h/M))$ (i.e.\ $g',h'\in A^n$) and $\tp(h/M)=\tp(g'g/M)$, then for $h''=g'g$ we have $h\equiv_M h''$, so $h(h'')^{-1}\in A$, and
		\[
			h'h=h'(h(h'')^{-1})h''=h'(h(h'')^{-1})g'g,
		\]
		so we have $(\tp(g/M),\tp(h'h/M))\in F_{2n+1}$.
		
		On the other hand, suppose $(\tp(e/M),\tp(g/M))\in F_n$. Then for some $g'\in A^n$ we have $g\equiv_M g'$. Since $A^n$ is clearly invariant over $M$, it follows that $g\in A^n$, so for any $g''\in G$ we have $(\tp(g''/M),\tp(g''g/M))\in F_n\subseteq F_{2n+1}$, which completes the proof.
	\end{proof}
	
	\begin{prop}
		\label{prop:bdd_iff_invariant}
		Suppose $A\subseteq \fC$ is a small set, $G$ is a group acting transitively on a set $X$, and $E$ is a $G$-invariant equivalence relation on $X$. Suppose in addition that $G$, $X$, $E$ and the action are all $\Aut(\fC/A)$-invariant.
		
		Then $E$ is bounded if and only if its classes are setwise $G^{000}_A$-invariant.
	\end{prop}
	\begin{proof}
		Note that since $E$ is $G$-invariant, $G$ acts on $X/E$.
		
		If $E$ is bounded, then $X/E$ is small, so the kernel of this action has small index. By the assumptions, the kernel is also invariant over $A$, so it contains $G^{000}_A$, which implies that the classes are $G^{000}_A$-invariant.
		
		In the other direction, if all classes of $E$ are setwise $G^{000}_A$-invariant, then for any $x_0\in X$, the assignment $gG^{000}_A\mapsto [g\cdot x_0]_E$ yields a well-defined function $G/G^{000}_A\to X/E$. Because $G$ acts transitively on $X$, this function is surjective. In particular, $\lvert X/E\rvert\leq [G:G^{000}_A]$, so $E$ is bounded.
	\end{proof}

	\begin{lem}
		\label{lem:weakly_grouplike_tdf}
		Suppose $G$ is a type-definable group acting type-definably and transitively on a type-definable set $X$ (all without parameters).
		
		Suppose in addition that $E$ is a $G$-invariant bounded invariant equivalence relation on $X$.
		
		Let $M$ be a small model such that $X(M)$ is nonempty, while $G(M)$ is dense in $S_G(M)$, and choose some $x_0\in X(M)$.
		
		Then the relation $E_{G^{000}_\emptyset}^M$ on the ambit $(S_G(M),\tp(e/M))$ dominates $E^M$ on the ambit $(X_M,\tp(x_0/M))$, and in particular (by Lemma~\ref{lem:coset_rel_is_glike}), $E^M$ is weakly uniformly properly group-like on the ambit $(G(M),X_M,\tp(x_0/M))$.
	\end{lem}
	\begin{proof}
		Since $G$ acts type-definably on $X$, it follows that the orbit map $g\mapsto g\cdot x_0$ is type-definable over $M$, so it induces a continuous map $S_G(M)\to X_M$. Because the action of $G$ on $X$ is transitive, this map is onto. By Proposition~\ref{prop:bdd_iff_invariant}, $E$-classes are setwise $G^{000}_\emptyset$-invariant, which implies that $E^M|_{S_G(M)}$ (see Definition~\ref{dfn:induced_relation}) is refined by $E_{G^{000}_\emptyset}^M$. Since $E$ is $G$-invariant, it follows that $S_G(M)/E_{G^{000}_\emptyset}^M=G/G^{000}_\emptyset$ acts on $X/E=X_M/E^M$, witnessing the domination.
	\end{proof}
	
	\begin{prop}
		\label{prop:amb_for_groups}
		Let $x_0\in \fC$ be an arbitrary tuple, and let $G$ be a $\emptyset$-type-definable group, consisting of tuples of length at most $\lambda$.
		
		Then there is a model $M$ of cardinality at most $\lvert T\rvert+\lvert x_0\rvert+\lambda$ containing $x_0$, such that $G(M)$ is dense in $S_G(M)$.
	\end{prop}
	\begin{proof}
		The proof is analogous to that of Proposition~\ref{prop:amb_exist}
		
		Roughly, start with any model $M_0$ containing $x_0$, and find a small group $G_0\leq G$ such that $\{\tp(g_0/M_0)\mid g_0\in G_0 \}$ is dense in $S_G(M_0)$, and expand $M_0$ to a small model $M_1$ containing $G_0$, and continue. After $\omega$ steps, take the union of the resulting elementary chain.
		
		(Note that if $G$ is a \emph{definable} group, then $G(M)$ is always dense in $S_G(M)$.)
	\end{proof}
	
	We have the following proposition, analogous to Corollary~\ref{cor:NIP_implies_tame}.
	\begin{prop}
		\label{prop:nip_tame_group}
		Suppose $G$ is a type-definable group acting type-definably on a type-definable set $X$. Let $M$ be a model over which $G,X$ and the action are type-definable. Then if $X$ has NIP, then the dynamical system $(G(M),X_M)$ is tame (cf.\ Definition~\ref{dfn:NIP_set} and Definition~\ref{dfn:tame_function_system}).
	\end{prop}
	\begin{proof}
		The proof is by contraposition. Suppose $(G(M),X_M)$ is untame. We will show that $X$ has IP (i.e.\ does not have NIP).
		
		Since $X_M$ is totally disconnected, by Proposition~\ref{prop:dyn_BFT}, there is a clopen subset $U\subseteq X_M$ and a sequence $(g_n)_{n\in \bN}$ in $G(M)$ such that the sets $g_nU$ are independent. Fix a formula $\varphi(x)$ with parameters in $M$ giving $U$ (i.e.\ such that $[\varphi(x)]\cap X_M=U$).
		
		Write $\mu$ for the multiplication $G\times X\to X$. Note that since $\mu$ is type-definable, the preimage $\mu^{-1}[\varphi(\fC)]$ is relatively definable over $M$ in $G\times X$ (because it is type-definable and so is its preimage, $\mu^{-1}[\neg \varphi(\fC)]$). Let $\psi(y,x)$ be a formula (with parameters from $M$) such that $\psi(G,X)=\mu^{-1}[\varphi(\fC)\cap X]$. Then by the assumption, $\psi(g_n,\fC)\cap X=g_n(\varphi(\fC)\cap X)$ are independent subsets of $X$, so $\psi$ does not have NIP on $G\times X$, so $X$ does not have NIP.
	\end{proof}
	
	The following proposition shows that NIP assumption on $G$ implies NIP for all transitive (type-definable) $G$-spaces.
	\begin{prop}
		The image of an NIP set by a type-definable surjection is NIP. In particular, if $G$ is a type-definable group with NIP, acting transitively and type-definably on $X$, then $X$ also has NIP.
	\end{prop}
	\begin{proof}
		The proof is by contraposition. We will show that if the range of a type-definable function does not have NIP, then neither does the domain.
		
		Fix a small set of parameters $A$, a set $X$, and a surjection $f\colon Z\to X$, all type-definable over $A$.
		
		Suppose $X$ has IP. Then there is a formula $\varphi(x,y)$ witnessing it; we may assume without loss of generality that all parameters in $\varphi$ are from $A$ (making it larger if necessary). In particular, by Remark~\ref{rem:NIP_indiscernible}, we can find a sequence $(b_n)_{n\in\bN}$, indiscernible over $A$, such that the sets $\varphi(\fC,b_n)\cap X$ are independent in $X$. Then clearly the sets $f^{-1}[\varphi(\fC,b_n)\cap X]$ are independent in $Z$, and we only need to show that they are uniformly definable.
		
		First, note that the set $f^{-1}[\varphi(\fC,b_0)]$ is relatively definable in $Z$ (it is obviously type-definable, and the same is true about its complement in $Z$), so there is some definable set $W$ such that $W\cap Z=f^{-1}[\varphi(\fC,b_0)]$. Now, since the sequence $(b_n)_{n\in \bN}$ is indiscernible over $A$, we can find, for each $n$, some automorphism $\sigma_n\in \Aut(\fC/A)$ such that $\sigma_n(b_0)=b_n$. But since $f$ and $X$ are invariant over $A$ and $\varphi(x,y)$ has parameters only from $A$, it follows that $\sigma_n(W)\cap Z=f^{-1}[\varphi(\fC,\sigma_n(b_0))\cap X]=f^{-1}[\varphi(\fC,b_n)\cap X]$. Since $\sigma_n(W)$ are clearly uniformly definable, we are done.
	\end{proof}
	\begin{rem}
		Note also that it is not hard to see that if $(G(M),G_M)$ is a tame dynamical system, then for every $M$-type-definable transitive $G$-space $X$ with nonempty $X(M)$, the system $(G(M),X_M)$ is also tame: under those hypotheses, we can have a $G(M)$-ambit morphism $(G_M,\tp(e/M))\to (X_M,\tp(x_0/M))$ (where $x_0\in X(M)$ is arbitrary), and apply Fact~\ref{fct:tame_preserved}.\xqed{\lozenge}
	\end{rem}
	
	\subsection*{Results for type-definable group actions}
	
	Now, Lemma~\ref{lem:weakly_grouplike_tdf} allows us to apply preceding results, including Theorem~\ref{thm:general_cardinality_intransitive}, Theorem~\ref{thm:general_cardinality_transitive} and Theorem~\ref{thm:main_abstract}. In particular, we have the following theorem (originally, \cite[Theorem 8.4]{KR18}, joint with Krzysztof Krupiński).

	\begin{thm}
		\label{thm:main_tdf}
		Suppose that the theory is countable, and $A\subseteq \fC$ is a countable set of parameters.
		
		Let $G$ be an type-definable group (of countable tuples), acting type-definably and transitively on a type-definable set $X$ (of countable tuples), all with parameters in $A$. Let $E$ be a bounded, $G$-invariant and $\Aut(\fC/A)$-invariant equivalence relation on $X$.
		
		Then there is a compact Polish group $\hat G$ acting continuously on $X/E$, and such that for any $x_0\in X$, the stabiliser $H$ of $[x_0]_E$, and the orbit map $\hat r\colon \hat G\to X/E$, $\hat g\mapsto \hat g\cdot [x_0]_E$, have the following properties:
		\begin{enumerate}
			\item
			$H\leq \hat G$ and fibres of $\hat r$ are exactly the left cosets of $H$ (so $\hat G/E|_{\hat G}=\hat G/H$),
			\item
			$\hat r$ is a topological quotient map (so it induces a homeomorphism of $\hat G/H$ and $X/E$),
			\item
			$E$ is relatively definable (as a subset of $X^2$) or type-definable if and only if $H$ is clopen or closed (respectively)
			\item
			if $E$ is $F_\sigma$, Borel, or analytic (respectively), then so is $H$,
			\item
			$\hat G/H\leq_B E$.
		\end{enumerate}
		Furthermore, if $X$ has NIP (in particular, if $G$ has NIP or, yet more generally, if $T$ has NIP), then $\hat G/H\sim_B E$.
	\end{thm}
	\begin{proof}
		Note first that we may assume without loss of generality that $A=\emptyset$ (if necessary, we may add some countably many parameters to the language).
		
		Fix $x_0$ and find a countable model $M$ as in Proposition~\ref{prop:amb_for_groups}. Then Lemma~\ref{lem:weakly_grouplike_tdf} applies, and we can apply Theorem~\ref{thm:main_abstract}, arguing as in the proof of Theorem~\ref{thm:main_aut}.
		
		More precisely, by Proposition~\ref{prop:amb_for_groups}, we can fix a model $M$ satisfying the hypotheses of Lemma~\ref{lem:weakly_grouplike_tdf}, and then for any $x_0\in X(M)$, $(G(M),X_M,\tp(x_0/M))$ is an ambit and $E^M$ is weakly uniformly properly group-like. Furthermore, by Proposition~\ref{prop:nip_tame_group}, if $X$ has NIP, then $(G(M),X_M)$ is tame.
		
		Recall that we identify $X/E$ and $X_M/E^M$ (and the identification is homeomorphic), the Borel cardinality of $E$ is by definition the Borel cardinality of $E^M$, and by Fact~\ref{fct: Borel in various senses}, we have that $E$ is relatively definable in $X^2$, type-definable, $F_\sigma$, Borel, or analytic if and only if $E^M$ is clopen, closed, $F_\sigma$, Borel or analytic (respectively).
		
		Thus, the by the third paragraph, the assumptions of Theorem~\ref{thm:main_abstract} are satisfied, and by the fourth paragraph, its conclusion gives us the desired $\hat G$, action and $\hat r$.
	\end{proof}
	
	\begin{rem}
		By going back to Theorem~\ref{thm:main_abstract}, we see that the group $\hat G$ in Theorem~\ref{thm:main_tdf} is actually the quotient $u\cM/\Core(H(u\cM)D)$ calculated for the ambit $(G(M),S_X(M),\tp(x_0/M))$.
		\xqed{\lozenge}
	\end{rem}
	
	\begin{rem}
		In Theorem~\ref{thm:main_tdf}, if the stabiliser of $[x_0]_E$ is normal in $G$, then so is the stabiliser in $G/G^{000}_A$, which gives $X/E$ a topological group structure such that the orbit map $G/G^{000}_A\to X/E$ (at $[x_0]_E$) is a homomorphism.
		
		It is not hard to see that this implies that that $E^M$ satisfies the assumptions of Proposition~\ref{prop:wgl_homom}, and so by Theorem~\ref{thm:main_abstract}(7), $H$ is normal.\xqed{\lozenge}
	\end{rem}

	We can also apply Corollary~\ref{cor:metr_smt_cls}, yielding the following. See also Corollary~\ref{cor:smt_def} for related statement which applies to intransitive actions.
	\begin{cor}
		\label{cor:trich+_tdf}
		Suppose $T$ is countable. Let $A\subseteq \fC$ be countable.
		
		Suppose in addition that $G$ is a type-definable group, and $X$ is an type-definable set of countable tuples on which $G$ acts transitively and type-definably (all with parameters in $A$), while $E$ is a bounded $G$-invariant and $\Aut(\fC/A)$-invariant equivalence relation on $X$. Then exactly one of the following holds:
		\begin{enumerate}
			\item
			$E$ is relatively definable (as a subset of $X^2$) and has finitely many classes,
			\item
			$E$ is type-definable and has exactly $2^{\aleph_0}$ classes,
			\item
			$E$ is not type-definable and not smooth. In this case, if $E$ is analytic, then it has exactly $2^{\aleph_0}$ classes.
		\end{enumerate}
		In particular, $E$ is smooth if and only if it is type-definable.
		
		Furthermore, if $X$ has NIP, then the Borel cardinality of $E$ is the Borel cardinality of the coset equivalence relation of a subgroup of a compact Polish group (which is $F_\sigma$, Borel or analytic, respectively, whenever $E$ is such).
	\end{cor}
	\begin{proof}
		By Theorem~\ref{thm:main_tdf}, we can apply Lemma~\ref{lem:abstract_trich}, which (by Fact~\ref{fct: Borel in various senses}) completes the proof, apart from the ``furthermore'' part, which follows directly from Theorem~\ref{thm:main_tdf}.
	\end{proof}
	
	The following corollary (Main~Theorem~\ref{mainthm:tdgroup}) is a strengthening of Fact~\ref{fct:KM_about_groups} from \cite{KM14}; it partially appeared in \cite{KPR15} and in \cite{KR18} (cf.\ the comments preceding Main~Theorem~\ref{mainthm:tdgroup}).
	\begin{cor}
		\label{cor:trich_tdgroups}
		Suppose $G$ is a type-definable group, while $H\leq G$ is an analytic subgroup, invariant over a small set. Then exactly one of the following holds:
		\begin{itemize}
			\item
			$[G:H]$ is finite and $H$ is relatively definable,
			\item
			$[G:H]\geq 2^{\aleph_0}$, but is bounded, and $H$ is not relatively definable.
			\item
			$[G:H]$ is unbounded (i.e.\ not small).
		\end{itemize}
		In particular, $[G:H]$ cannot be infinite and smaller than $2^{\aleph_0}$.
		
		Moreover, if the language is countable, $G$ consists of countable tuples, and $G$ and $H$ are invariant over a countable set, then we can divide the second case further: either $H$ is type-definable, or $G/H$ is not smooth.
	\end{cor}
	\begin{proof}
		Choose $M$ as in Proposition~\ref{prop:amb_for_groups} for $x_0=e_G$.
		
		If $[G:H]$ is unbounded, there is nothing to prove. Otherwise, the left coset equivalence relation $E_H$ on $G$ is bounded, $G$-invariant, and invariant over the set over which $G$ and $H$ are invariant.
		Furthermore, $G(M)$ is dense in $S_G(M)$, so it has a dense orbit in $G/H$, i.e.\ $G(M)\cdot e_{G/H}=G(M)/H$ is dense in $G/H=S_G(M)/E_H^M$.
		
		Therefore, we can apply Lemma~\ref{lem:weakly_grouplike_tdf} to $E_H$, and it follows that Theorem~\ref{thm:general_cardinality_transitive} applies with $X=Y=S_G(M)$ and $E_H^M$. It is easy to see that $E_H^M$ is clopen if and only if both $E_H$ and $H$ are relatively definable, which completes the proof of the trichotomy.
		
		Under the countability assumptions, we can apply Corollary~\ref{cor:trich+_tdf} with $X=G$ and $E=E_H$, noting that $E_H$ is type-definable if and only if $H$ is type-definable. This gives us the ``moreover'' part.
	\end{proof}
	
	\begin{rem}
		It is possible to have $[G:H]=2$ for an invariant and not relatively definable $H$ (see \cite[Example 3.39]{KM14}), but then $H$ is necessarily non-analytic. (Indeed, in the cited example, the $H$ is only obtained existentially, and is a kind of ``Vitali set".)\xqed{\lozenge}
	\end{rem}

	\begin{rem}
		In Theorem~\ref{thm:main_tdf}, in the ``Furthermore" part, one can weaken the assumption that $X$ has NIP to say only that there is no $\varphi(x)$ with parameters in $M$ such that $\{g\cdot [\varphi(x)]\mid g\in G(M) \}$ contains an independent family, as that is enough to guarantee that $(G(M),S_G(M))$ is tame.
		\xqed{\lozenge}
	\end{rem}
	
	\begin{rem}
		One may also show that we have the analogue of Lemma~\ref{lem:every_stype_on_m}, and using that, obtain an analogue of Theorem~\ref{thm:main_galois}. Roughly speaking, given a fixed $G$ and $A$, there is a single $\hat G$ witnessing Theorem~\ref{thm:main_tdf} for all $X$ and $E$.\xqed{\lozenge}
	\end{rem}
	
	\section{Other applications in model theory}
	\label{sec:other_apps}
	As mentioned in the introduction, Theorem~\ref{thm:main_abstract} (as well as Theorems~\ref{thm:general_cardinality_intransitive} and \ref{thm:general_cardinality_transitive}) may be used to deduce virtually all the similar results in the model-theoretic contexts, either by directly showing that some equivalence relation is (weakly) uniformly properly group-like, or by some reduction to Theorem~\ref{thm:main_aut} or Theorem~\ref{thm:main_tdf}.
	
	Before, we have seen how we can recover and even improve the results from the papers \cite{Ne03}, \cite{KMS14}, \cite{KM14}, \cite{KP17}, \cite{KR16}, \cite{KPR15} and \cite{KR18}. Below, we briefly describe a couple of other examples.
	
	\subsection*{Definable components in classical topological dynamics}
	
	In \cite{KP16}, the authors consider a topological group $G=G(M)$, definable in a structure $M$ with predicates for all open subsets of $G$, denoting $G(\fC)$ by $G^*$. They denote by $\mu$ the subgroup of $G^*$ of infinitesimal elements, that is, $\bigcap_U U(\fC)$, where $U$ ranges over all neighbourhoods of the identity in $G$. Using $\mu$, they define the group $G^{*000}_{\topo}$ as the smallest $M$-invariant normal subgroup of $G^*$ which contains $\mu$ and has bounded index. Then $G^*/G^{*000}_\topo$ is a new invariant of the topological group $G$ (as one can show that it does not depend on the choice of the model $M$, as long as it defines $G$ and has predicates for all its open subsets).
	
	They also define the space $S_{G^*}^\mu(M)$ as the quotient of $S_{G^*}(M)$ by $\mu$ (i.e.\ two types $p,q\in S_{G^*}(M)$ are identified if $\mu p(\fC)=\mu q(\fC)$).
	
	Then, since $G^*$ is definable, $G(M)$ is dense in $S_{G^*}(M)$, and so it is also dense in $S_{G^*}^\mu(M)$, so $(G(M),S_{G^*}^\mu(M),\tp(e/M))$ is an ambit. In fact, it is exactly the classical universal (topological) $G$-ambit.  It turns out that $S_{G^*}^\mu(M)$ has a natural semigroup structure which makes it isomorphic to $E(G(M),S_{G^*}^\mu(M))$, so in particular, we can find inside the Ellis group $u\cM$ and the quotient $u\cM/H(u\cM)$, which is exactly the generalized Bohr compactification of $G$, as introduced by Glasner in \cite[Chapter VIII]{Gl76}. They turn to state, in \cite[Theorem 2.24, Theorem 2.25]{KP16} (without proof, beyond very broad description how one can adapt \cite{KP17}) that we have a well-defined topological quotient map $u\cM/H(u\cM)\to G^*/G^{*000}_\topo$, and that $G^{*00}_\topo/G^{*000}_\topo$ is the quotient of a compact Hausdorff group by a dense subgroup, for $G^{*00}_\topo$ defined analogously.
	
	It is not hard to show that, in their context, the coset equivalence relation $E_{G^{*000}_\topo}$ of $G^{*000}_\topo$ induces a uniformly properly group-like $F_\sigma$ equivalence relation on the ambit $(G(M),S_{G^*}^\mu(M),\tp(e/M))$, and the quotient of $S_{G^*}^\mu(M)$ by this relation can be naturally identified with $G^*/G^{*000}_\topo$. Thus, by Lemma~\ref{lem:main_abstract_grouplike}, we recover the quotient map $u\cM/H(u\cM)\to G^*/G^{*000}_\topo$, concluding (using Lemma~\ref{lem:new_preservation_E_to_H}) that $G/G^{*000}_\topo$ is the quotient of the compact group $u\cM/H(u\cM)$ by an $F_\sigma$ normal subgroup. Since $G^{*00}_\topo/G^{*000}_\topo$ is the closure of the identity in $G^*/G^{*000}_\topo$, it follows that it is the quotient of a compact Hausdorff group by a dense subgroup.
	\subsection*{Relative Galois groups}
	\newcommand{\res}{{\mathrm{res}}}
	\newcommand{\fix}{{\mathrm{fix}}}
	In \cite{DKL17}, the authors study several variants of the Galois group. For each partial type $\Sigma$ over $\emptyset$, they put:
	\begin{itemize}
		\item
		$\Aut(\Sigma(\fC))=\{\sigma\restr_{\Sigma(\fC)}\mid \sigma\in \Aut(\fC) \}$,
		\item
		$\Autf_\res(\Sigma(\fC))=\{\sigma\restr_{\Sigma(\fC)}\mid \sigma\in \Autf(\fC) \}$,
		\item
		$\Autf_\fix(\Sigma(\fC))=\{\sigma\in \Aut(\Sigma(\fC))\mid \sigma(a)\equiv_\Lasc a\}$, where $a$ is a tuple enumerating $\Sigma(\fC)$.
	\end{itemize}
	Using these, they define the relative Galois groups in the following way.
	\begin{itemize}
		\item
		$\Gal^\res(\Sigma(\fC))=\Aut(\Sigma(\fC))/\Autf_\res(\Sigma(\fC))$
		\item
		$\Gal^\fix(\Sigma(\fC))=\Aut(\Sigma(\fC))/\Autf_\fix(\Sigma(\fC))$
	\end{itemize}
	It is easy to see that $\Autf_\res(\Sigma(\fC))$ and $\Autf_\fix(\Sigma(\fC))$ are normal subgroups of $\Aut(\Sigma(\fC))$ and $\Autf_\res(\Sigma(\fC))\leq \Autf_\fix(\Sigma(\fC))$, so $\Gal^\res(\Sigma(\fC))$ and $\Gal^\fix(\Sigma(\fC))$ are groups and we have a natural epimorphism $\Gal^\res(\Sigma(\fC))\to \Gal^\fix(\Sigma(\fC))$. Furthermore, the restriction epimorphism $\Aut(\fC)\to \Aut(\Sigma(\fC))$ induces an epimorphism $\Gal(T)\to \Gal^\res(\Sigma(\fC))$, which turns $\Gal^\res(\Sigma(\fC))$ and $\Gal^\fix(\Sigma(\fC))$ into topological groups. Furthermore, by considering the compositions of the epimorphisms with the function $S_m(M)\to \Gal(T)$ from Fact~\ref{fct:sm_to_gal}, we can also conclude that each relative Galois group also has a well-defined Borel cardinality (provided the theory is countable and $\Sigma$ has countably many free variables).
	
	In both cases, we can show that the Galois groups are actually quotients of the space $S_m(M)$ by a uniformly properly group-like equivalence relation. Thus, we can apply Lemma~\ref{lem:weakly_grouplike} to present them as quotients of compact Hausdorff groups, and if the language is countable, we may also apply Theorem~\ref{thm:main_abstract} to present them as quotients of compact Polish groups, and Corollary~\ref{cor:metr_smt_cls} to see that they are smooth if and only if they are Hausdorff  (i.e.\ they coincide with the appropriately defined relative Kim-Pillay Galois groups). Similarly to Theorem~\ref{thm:main_galois}, we also recover the full Borel cardinality under NIP, although in this case, it is enough to assume NIP on $[a]_{\equiv}$ for a suitable tuple of realisations of $\Sigma$.

	\section{Examples}
	\label{sec:examples}
	
	In this section, we analyse examples of non-G-compact theories $T$ from \cite{CLPZ01} and \cite{KPS13} and see how Theorem~\ref{thm:main_galois} can be applied to them. Namely, we describe the compact group $\hat G$ (which turns out to be the Ellis group) and the kernel of $\hat r\colon \hat G\to \Gal(T)$ in those cases. In order to do that, we compute the Ellis groups of the appropriate dynamical systems. This allows us to describe the group $\Gal(T)$ in each of these examples. Further, because the examples have NIP, this description also yields the Borel cardinality of the Galois group.
	
	The contents of this section are based on the appendix of \cite{KR18} (joint with Krzysztof Krupiński), expanded with more details of the proofs.
	
	(The topological group structure in the first example (Example~\ref{ex:CLPZ}) was described in \cite{Zie02}, by a more direct method. In \cite{KPS13}, the authors describe the topological group structure the second example (Example~\ref{ex:KPS}) and the Borel cardinality in both cases, but use completely different methods and give very few details.)
	
	\subsection*{Lemmas}
	First, we prove some auxiliary lemmas.
	
	\begin{rem}
		\label{rem:action_factors}
		If $(G,X)$ is a dynamical system and the action of $G$ on $X$ factors through another group $G'$, then it is easy to see that $E(G,X)=E(G',X)$, and the $\tau$ topologies on the ideal groups coincide.\xqed{\lozenge}
	\end{rem}

	\begin{lem}
		\label{lem:projlim_ellis}
		Consider a projective system of dynamical systems $(G_i,X_i)$ for $i\in I$ (where $i$ is some downwards directed set), i.e.\ for each pair $i<j$ we have an epimorphism $\pi_{i,j}\colon G_i\to G_j$, and a $G_i$-equivariant continuous surjection $\pi_{i,j}\colon X_i\to X_j$. Let $G:=\varprojlim_i G_i$ act naturally on $X:=\varprojlim_i X_i$, and for each $i$, let $\pi_i$ denote the projection $\pi_i\colon X\to X_i$ and abusing the notation, also the projection $\pi_i\colon G\to G_i$.
		
		Then we have a natural isomorphism $E(G,X)\cong \varprojlim_i E(G_i,X_i)$ (as semitopological semigroups and as a $G$-flows), consistent with the maps $\pi_i$ given in the preceding paragraph. Let us abuse the notation further, and write $\pi_i$ for the epimorphism $E(G,X)\to E(G_i,X_i)$.
		
		Then, for every minimal left ideal $\cM \unlhd E(G,X)$, each $\cM_i=\pi_i[\cM]$ is a minimal left ideal in $E(G_i,X_i)$ and $\cM=\varprojlim_i \pi_i[\cM]$. If $u\in \cM$ is an idempotent, then each $u_i=\pi_i(u)$ is an idempotent in $\cM_i=\pi_i[\cM]$.
		
		In particular, $u\cM=\varprojlim_i u_i\cM_i$. Furthermore, the $\tau$ topology on $u\cM$ is the projective limit topology, with each $u_i\cM_i$ equipped with its $\tau$ topology.
		
		Conversely, if $(\cM_i)_i$ is a consistent system of minimal left ideals in $E(G_i,X_i)$ and for each $i$, $u_i$ is an idempotent in $\cM_i$, then $\varprojlim_i \cM_i$ is a minimal left ideal in $E(G,X)$ and $u=(u_i)_i$ is an idempotent in $\cM$.
	\end{lem}
	\begin{proof}
		Note that immediately by the assumptions, $G$ acts on each $X_i$ via $G_i$, and $\pi_i\colon X\to X_i$ is $G$-equivariant. By Remark~\ref{rem:action_factors}, we may assume without loss of generality that $G=G_i$ for all $i$.
		
		Now, using Proposition~\ref{prop:induced_epimorphism}, we obtain the epimorphisms $\pi_i\colon E(G,X)\to E(G,X_i)$, and they obviously commute with the covering maps in the projective system, whence $E(G,X)=\varprojlim_i E(G,X_i)$.
		
		Since each $\pi_i$ is an epimorphism, preimages and images of ideals by $\pi_i$ are ideals. This easily implies that if $\cM$ is a minimal ideal, then each $\pi_i[\cM]$ is also minimal. Since $\cM$ is closed (as a minimal ideal), it is the inverse limit of $\pi_i[\cM]$. Conversely, if $(\cM_i)_i$ is a consistent system of minimal ideals, then $\varprojlim_i \cM_i=\bigcap_i \pi_i^{-1}[\cM_i]$, so it is an ideal (as an intersection of ideals, which is nonempty by compactness, because the system is consistent). Thus, it contains a minimal ideal $\cM$. If $\cM\subsetneq \varprojlim_i \cM_i$, then for some $i$ we have $\pi_i[\cM]\subsetneq \cM_i$, which contradicts minimality of $\cM_i$, so $\cM= \varprojlim_i \cM_i$, and the latter is a minimal ideal.
		
		It is clear that $u\in E(G,X)$ is an idempotent if and only if each $\pi_i(u)$ is an idempotent (because multiplication in the inverse limit is coordinatewise). Therefore, if $u=(u_i)_i\in \cM=\varprojlim_i \cM_i$ is an idempotent, then for each $f\in u\cM$ and each $i$ we have $\pi_i(f)=\pi_i(uf)=\pi_i(u)\pi_i(f)\in u_i \cM_i$, and conversely, if for each $i$ we have $\pi_i(f)\in u_i\cM_i$, then $uf=(u_i\pi_i(f))_i=(\pi_i(f))_i=f$, so as sets, $u\cM=\varprojlim_i u_i\cM_i$.
		
		What is left is to show that the $\tau$-topology on an Ellis group $u\cM$ is the limit of the $\tau$ topologies on projections. Let us denote the limit topology by $\pi$.
		
		In one direction, this is trivial: a subbasic $\pi$-closed set is clearly $\tau$-closed, so $\tau$ refines $\pi$.
		
		In the other direction, let $A$ be a $\tau$-closed set in $u\cM$. Take any $f$ which is in $\pi$-closure of $A$, any open $U\ni u$ and $V\ni f$, with the aim to apply Proposition~\ref{prop:circ_description} to show that $f\in u\circ A$. Then for some $i\in I$ and open $U',V'\subseteq EL_i=E(G,X_i)$, we have $U=\pi_i^{-1}[U'],V=\pi_i^{-1}[V']$, so $u_i\in U'$ and $f_i:=\pi_i(f)\in V'$. But then by the assumption and Proposition~\ref{prop:circ_description}, there is some $g_i\in G$ and $a_i\in \pi[A]$ such that $\pi_{X_i,g_i}\in U'$ and $g_ia_i\in V'$. But then for any $a\in A$ such that $\pi_i(a)=a_i$ we have $\pi_{X,g_i}\in U$ and $g_ia\in V$. Since $U,V$ were arbitrary, by Proposition~\ref{prop:circ_description}, $f\in u\circ A$, and since $f\in u\cM$ and $A$ is $\tau$-closed, we have $f\in A$.
	\end{proof}
	
	\begin{lem}
		\label{lem:tau_accumulation}
		Fix arbitrary dynamical system $(G,X)$, and consider its Ellis group $u\cM$.
		
		Given a net $(f_i)_i$ in $u\cM$ and $f\in u\cM$, the following are equivalent:
		\begin{itemize}
			\item
			$f$ is a $\tau$-accumulation point of $(f_i)_i$ (i.e.\ for every $i_0$, $f$ is in the $\tau$-closure of $(f_i)_{i>i_0}$),
			\item
			there is a subnet $(f'_{j'})_{j'}$ of $(f_j)_j$ such that for some net $(g_{j'})_{j'}$ in $G$ such that $g_{j'}\to u$ we have $g_{j'}f'_{j'}\to f$.
		\end{itemize}
		In particular, $(f_i)_i$ $\tau$-converges to $f$ if and only if every subnet of $(f_i)_i$ has a further subnet with the second property.
	\end{lem}
	\begin{proof}
		It is clear that the second condition implies the first. For the converse, by Proposition~\ref{prop:circ_description}, it is enough to show that for every $i_0$, for every open $U\ni u$ and $V\ni f$, there is some $i>i_0$ and $g_i\in G$ such that $g_i\in U$ and $g_if_i\in V$. But by the assumption, we can find some net $(g'_j)_j$ and a net $(f'_j)_j$, where each $f'_j\in f_{>{i_0}}$, such that $g'_j\to u$ and $g'_jf'_j\to f$. But then for sufficiently large $j$ we have $g'_j\in U$ and $g'_jf'_j\in V$, so we can just take any $i>i_0$ such that $f_i=f'_j$ and $g_i=g'_j$.
		
		The ``in particular" follows easily, as $(f_i)_i$ converges to $f$ exactly when $f$ is the accumulation point of every subnet of $(f_i)_i$.
	\end{proof}

	\begin{lem}
		\label{lem:product_ellis}
		Consider dynamical systems $(G_i,X_i)$ for $i\in I$ (where $I$ is some index set). Put $G=\prod_i G_i$ acting naturally on $X=\prod_i X_i$, and for each $i$, let $\pi_i$ be the projection $X\to X_i$ and, abusing the notation, $G\to G_i$.
		
		Then $E(G,X)\cong \prod_i E(G_i,X_i)$ (as a semitopological semigroup and as a $G$-flow).
		
		Furthermore, if $\cM$ is a minimal left ideal in $E(G,X)$, then each $\cM_i=\pi_i[\cM]$ is a minimal left ideal in $E(G_i,X_i)$ and $\cM=\prod_i \pi_i[\cM]$. If $u\in \cM$ is an idempotent, then each $u_i=\pi_i(u)$ is an idempotent in $\cM_i=\pi_i[\cM]$.
		
		In particular, $u\cM=\prod_i u_i\cM_i$. Furthermore, the $\tau$ topology on $u_i\cM_i$ is the product topology.
		
		Conversely, if $\cM_i$ is a minimal left ideal in $E(G_i,X_i)$ and $u_i$ is an idempotent in $\cM_i$, then $\prod_i \cM_i$ is a minimal left ideal in $E(G,X)$ and $u=(u_i)_i$ is an idempotent in $\cM$.
	\end{lem}
	\begin{proof}
		Since every product is the inverse limit of its finite subproducts, by Lemma~\ref{lem:projlim_ellis}, it is enough to consider the case when $I$ is finite. Moreover, a straightforward inductive argument shows that the case of finite products follows from the case of products of two elements.
		
		Thus, we may assume that $G=G_1\times G_2$ and $X=X_1\times X_2$. The fact that $E(G,X)$ is the product $E(G_1,X_1)\times E(G_2,X_2)$ is straightforward, as is the fact that minimal ideals in $E(G,X)$ are exactly the products of minimal ideals in $E(G_i,X_i)$, and that idempotents are those elements which have idempotents on both coordinates.
		
		The only nontrivial statement is about the $\tau$ topology being equal to the product topology. As in the case of inverse limit, let us call the latter topology $\pi$. Also as there, we see immediately that subbasic $\pi$-closed sets are $\tau$-closed, so $\tau$ refines $\pi$.
		
		In the other direction, consider any $A\subseteq u\cM=u_1\cM_1\times u_2\cM_2$ and let $f$ be a point in the $\pi$-closure of $A$. We will show that $f$ is also in the $\tau$-closure of $A$. We have a net $(a_i)_i$ in $A$ which is $\pi$-convergent to $f$, i.e.\ for $j=1,2$ we have $(a_{j,i})\xrightarrow{\tau} f_j$, where $f=(f_1,f_2)$ and each $a_i=(a_{1,i},a_{2,i})$.
		
		By applying Lemma~\ref{lem:tau_accumulation} to $(a_{1,i})_i$, we may assume without loss of generality that there is a net $(g_{1,i})_i$ in $G_1$ such that $g_{1,i}\to u_1$ and $g_{1,i}a_{1,i}\to f_1$. By applying it again, we may assume without loss of generality that there is also a net $(g_{2,i})_i$ in $G_2$ such that $g_{2,i}\to u_2$ and $g_{2,i}a_{2,i}\to f_2$. But then $(g_{1,i},g_{2,i})\to (u_1,u_2)=u$ and $(g_{1,i},g_{2,i})(a_{1,i},a_{2,i})\to (f_1,f_2)=f$, so $f$ is in $\tau$-closure of $A$, and we are done.
	\end{proof}

	\begin{prop}\label{prop:product of Ellis groups}
		Suppose we have a multi-sorted structure $M=(M_n)_n$, where the sorts $M_n$ are arbitrary, without any functions or relations between them. Enumerate each $M_n$ by $m_n$ and put $m=(m_n)_n$. Then $E(\Aut(M),S_m(M))\cong \prod_n E(\Aut(M_n),S_{m_n}(M))$, and similarly, the minimal left ideals and the Ellis groups (equipped with the $\tau$-topology) are the products of minimal left ideals and Ellis groups, respectively.
	\end{prop}
	\begin{proof}
		Under the given assumptions, it is easy to see that $\Aut(M)=\prod_n \Aut(M_n)$ and $S_m(M)=\prod_n S_{m_n}(M_n)$. The proposition follows from Lemma~\ref{lem:product_ellis}
	\end{proof}

	\subsection*{Examples}
	
	In this section, unless otherwise stated, $M_n$ denotes the countable structure $(M_n,R_n,C_n)$, where $n>1$ is a fixed natural number, the underlying set is ${\bQ}/{\bZ}$, $R_n$ is the unary function $x\mapsto x+1/n$, and $C_n$ is the ternary predicate for the natural (dense, strict) circular order. Let a tuple $m_n$ enumerate $M_n$. It is easy to show (see \cite[Proposition 4.2]{CLPZ01}) that $\Th(M_n)$ has quantifier elimination and the real circle $S^1_n=\bR/\bZ$ equipped with the rotation by the angle $2\pi/n$ and the circular order is an elementary extension of $M_n$. As usual, $\fC \succ S^1_n$ is a monster model.
	
	Given any $c' \in \fC$, by $\st(c')$ we denote the standard part of $c'$ computed in the circle $S^1=\bR/\bZ$.
	As $\st(c')$ depends only on $\tp(c'/M_n)$, this extends to a standard part mapping on the space of 1-types $S_1(M_n)$.
	
	\begin{prop}
		\label{prop:group_onetypes}
		If $u$ is an idempotent in a minimal left ideal $\cM$ of  the Ellis semigroup $E(\Aut(M_n),S_1(M_n))$, then $u\cM$ is generated by $R_nu$ and cyclic of order $n$. In particular, it is isomorphic to ${\bZ}/n{\bZ}$.
	\end{prop}
	
	\begin{proof}
		Note that $R_n$ is a $\emptyset$-definable automorphism of $M_n$, and as such, it is in the centre of $\Aut(M_n)$, and so it is also central in the Ellis semigroup.
		
		Now, since for any two Ellis groups $u\cM,v\cN$, the map $f\mapsto vfv$ is an isomorphism $u\cM\to v\cN$ (cf.\ Remark~\ref{rem:explicit_ellisgroup_isomorphism}), and since $R_n$ is central, we have $vR_n^juv=R_n^jvuv=R_n^jv$. Thus if the conclusion holds for $u\in \cM$, then it also holds for $v
		\in \cN$. Hence, it is enough to show that it holds for \emph{some} idempotent in \emph{some} minimal ideal.
		
		In the rest of this proof, by \emph{short} interval we mean an interval of length less than $1/n$. We also identify $\Aut(M_n)$ with its image in the Ellis semigroup.
		
		From quantifier elimination, it follows easily that $M_n$ is $\omega$-categorical, and $\Aut(M_n)$ acts transitively on the set of short open intervals in $M_n$.
		
		Denote by $J$ the set of $p\in S_1(M_n)$ with $\st(p)\in [0,1/n) + \bZ \subseteq \bR/\bZ$.
		
		\begin{clm*}
			For any non-isolated type $p\in S_1(M_n)$, there is a unique $f_{p}\in EL:=E(\Aut(M_n),S_1(M_n))$ such that for all $q\in J$ we have $f_p(q)=p$.
		\end{clm*}
		\begin{clmproof}
			Enumerate $M_n$ as $(a_k)_{k\in \bN}$.
			
			Since $p$ is non-isolated, for each $k\in \bN$ there is a short open interval $I_k$ such that $p$ is concentrated on $I_k$ and $a_0,\ldots,a_k\notin I_k$. By quantifier elimination, it is easy to see that $p$ is the only type in $S_1(M_n)$ concentrated on all $I_k$'s.
			
			Now, let $J_k:=(\frac{-1}{2kn},\frac{1}{n}-\frac{1}{kn})$. Notice that if $q\in J$, then $q$ is concentrated on all but finitely many $J_k$'s.
			
			Since each $I_k$ and $J_k$ is a short open interval, we can find for each $k$ some $\sigma_k\in \Aut(M_n)$ such that $\sigma_k[J_k]=I_k$. It follows that for any $q\in J$ we have $\lim_k\sigma_k(q)=p$. Thus, if we take any $f_p\in EL$ which is an accumulation point of $(\sigma_k)_k$, we will have $f_p(q)=p$ for all $q \in J$.
			
			To see that $f_p$ is unique, note that for each integer $j$ and $q\in R_n^j[J]$, $f_p(q)\in f_p[R_n^j[J]]=f_pR_n^j[J]=R_n^j f_p[J]=\{R_n^j(p)\}$. Since $J\cup R_n[J]\cup\ldots\cup R_n^{n-1}[J]=S_1(M_n)$, uniqueness follows.
		\end{clmproof}
		
		Take any non-isolated $p_0\in J$, and let $u=f_{p_0}$ (as in the claim). By uniqueness in the claim, $u$ is an idempotent. Denote by $\mathcal O$ the $R_n$-orbit of $p_0$.
		
		Note that every $f \in ELu$ is constant on $J$. As in the above proof of uniqueness, since $u$ and $uf$ commute with $R_n$, we easily see that the image of $uf$ equals $\mathcal O$.
		
		Now, we show that $\cM:=ELu$ is a minimal left ideal. Consider any $f \in \cM$.
		By the last paragraph, $uf(p_0)=R_n^j(p_0)$ for some $j$. Then $R_n^{-j}uf(p_0)=p_0$ and $R_n^{-j}uf$ is constant on $J$, so by uniqueness in the claim, $R_n^{-j}uf=u$. It follows that $ELf=ELu=\cM$, so $\cM$ is a minimal left ideal.

		By the preceding paragraph, we see also that for any $uf\in u\cM$, there is some $j$ such that $uf=R_n^ju$. Conversely, since $R_n$ is central, $R_nu=uR_nu\in u\cM$, so $u\cM$ is cyclic, generated by $R_nu$. As $R_n^ju(p_0)=R_n^j(p_0)$, $R_nu$ has order $n$ in $u\cM$, so $u\cM\cong \bZ/n\bZ$.
	\end{proof}
	
	\begin{lem}
		\label{lem:ellis_group_onesort}
		Suppose $n>1$.
		
		The restriction $S_{m_n}(M_n)\to S_1(M_n)$ to the first variable induces an isomorphism of Ellis semigroups $E(\Aut(M_n),S_{m_n}(M_n))\cong E(\Aut(M_n),S_1(M_n))$
		
		In particular, every Ellis group $u\cM$ of $(\Aut(M_n),S_{m_n}(M_n))$ is generated by $R_nu$ and isomorphic to $\bZ/n\bZ$.
	\end{lem}
	\begin{proof}
		We have the following ``orthogonality" property.
		
		\begin{clm*}
			Let $p,q\in S_{m_n}(M_n)$ satisfy the condition that for each single variable $x$, $p\restr_x=q\restr_x$. Then $p=q$.
		\end{clm*}
		\begin{clmproof}
			For $c_1',c_2'\in \fC$, write $c_1'<c_2'$ for $C_n(c_1',c_2',R_n(c_1'))$. Note that for each $r\in S^1$, this is a linear ordering on the set of all $c'$ with $\st(c')=r$. Furthermore, for any $c_1',c_2',c_3'$ we have that $C_n(c_1',c_2',c_3')$ holds if and only if one of the following holds:
			\begin{itemize}
				\item
				$\st(c_1'),\st(c_2'),\st(c_3')$ are all distinct and they are in the standard circular order on $S^1$,
				\item
				$\st(c_1')=\st(c_2')\neq \st(c_3')$ and $c_1'<c_2'$,
				\item
				$\st(c_1')\neq\st(c_2')=\st(c_3')$ and $c_2'<c_3'$,
				\item
				$\st(c_2')\neq\st(c_1')=\st(c_3')$ and $c_1'>c_3'$,
				\item
				$\st(c_1')=\st(c_2')=\st(c_3')$ and ($c_1'<c_2'<c_3'$ or $c_3'<c_1'<c_2'$ or $c_2'<c_3'<c_1'$).
			\end{itemize}
			
			We need to show that for each tuple $m'=(m'_k)_{k\in \bN}$ satisfying $\tp(m_n/\emptyset)$, we have the implication $\tp(m_n/\emptyset)\cup\bigcup_k \tp(m'_k/M_n)\vdash \tp(m'/M_n)$. By quantifier elimination, it is enough to show that the type on the left implies each atomic formula (or negation) in $\tp(m'/M_n)$. The only nontrivial cases are of the form $C_n(R_n^i(x_1),R_n^j(x_2),c)$, $C_n(R_n^i(x_1),c,R_n^j(x_2))$, $C_n(c,R_n^i(x_1),R_n^j(x_2))$ (or negations), where $i,j \in \{0,\dots,n-1\}$ and $c \in M_n$. But that follows immediately from the preceding paragraph (and the fact that the standard part is determined by the type over $M_n$).
		\end{clmproof}
		
		It follows from quantifier elimination that there is a unique 1-type over $\emptyset$, so the restriction to the first variable $S_{m_n}(M_n)\to S_1(M_n)$ is surjective, and (since it is obviously equivariant) it gives us a surjective homomorphism $E(\Aut(M_n),S_{m_n}(M_n))\to E(\Aut(M_n),S_1(M_n))$. We need to show that it is injective.
		
		Suppose $f_1,f_2\in E(\Aut(M_n),S_{m_n}(M_n))$ are distinct, so there is some $p\in S_{m_n}(M_n)$ such that $f_1(p)\neq f_2(p)$. But then, by the claim, there is a variable $x_k$ such that $f_1(p)\restr x_k\neq f_2(p)\restr x_k$. Choose $m'=(m'_k)_{k\in \bN}\models p$; then $m'$ enumerates a countable $M'\preceq \fC$. By $\omega$-categoricity and the fact that there is a unique 1-type over $\emptyset$, there is $\sigma\in \Aut(M')$ such that $\sigma(m'_1)=m'_k$. Now, if we put $p':=\tp(\sigma(m')/M_n)$, we have that $p'\restr_{x_1}=p\restr_{x_k}$. From that, we obtain $f_1(p')\restr_{x_1}=f_1(p)\restr_{x_k}\neq f_2(p)\restr_{x_k}=f_2(p')\restr_{x_1}$. It follows that the epimorphism $E(\Aut(M_n),S_{m_n}(M_n))\to E(\Aut(M_n),S_1(M_n))$ induced by the restriction to the first variable is injective.
		
		To complete the proof, notice that because restriction to $S_1(M_n)$ induces an isomorphism of $E(\Aut(M_n),S_{m_n}(M_n))$ and $E(\Aut(M_n),S_1(M_n))$, for any Ellis group $u\cM$ in the former, $u\cM\restr_{S_1(M_n)}$ is an Ellis group in the latter. Furthermore, it is easy to see that $\pi_{R_n,S_{m_n}(M_n)}\restr_{S_1(M_n)}=\pi_{R_n,S_{1}(M_n)}$, so since --- by Proposition~\ref{prop:group_onetypes} --- $u\cM\restr_{S_1(M_n)}$ is cyclic of order $n$, generated by $R_nu\restr_{S_1(M_n)}=\pi_{R_n,S_{1}(M_n)}u\restr_{S_1(M_n)}$ , it follows that $u\cM$ is generated by $\pi_{R_n,S_{m_n}(M_n)}u=R_nu$.
	\end{proof}

	\begin{prop}
		\label{prop:connecting_maps_nn'}
		If $n,n'$ are positive integers and $n'$ divides $n$, then the map $M_n\to M_{n'}$ given by multiplication by $k=n/n'$ induces an epimorphism $\varphi_1\colon \Aut(M_n)\to\Aut(M_{n'})$ and a continuous surjection $\varphi_2\colon S_{m_n}(M_n)\to S_{m_{n'}}(M_{n'})$ which is equivariant with respect to the induced action of $\Aut(M_n)$ on $S_{m_{n'}}(M_{n'})$.
	\end{prop}
	\begin{proof}
		
		Denote by $H$ the group generated by $R_n^{n'}$ in $\Aut(M_n)$. Then the $H$-orbit equivalence relation on $M_n$ is definable in $M_n$, and thus $M_n/H$ is an imaginary sort in $M_n$. $R_n$ and $C_n$ induce an unary function $R_n'$ and a ternary relation $C_n'$ (both $M_n$-definable) on $M_n/H$ by putting $R_n'(Hx):=HR_n(x)$ and declaring that $C_n'(Hx_1,Hx_2,Hx_3)$ if we have $C_n(x_1',x_2',x_3')$ for the representatives $x_1',x_2',x_3'$ (of the respective orbits) in $[0,1/k)+\bZ$. Then it is not hard to see that $(M_n/H,C_n')$ is a dense circular order and $R_n'$ is its automorphism of order $n/k=n'$. Furthermore, it is easy to see that the map $M_n\to M_{n'}$ given by $x\mapsto kx$ factors through $Hx\mapsto kx$, which defines an isomorphism $\varphi_0\colon (M_n/H,R_n',C_n')\to(M_{n'},R_{n'},C_{n'})$.
		
		Then the isomorphism $\varphi_0$ induces an action of $\Aut(M_n)$ on $M_{n'}$ by automorphisms (given by $\sigma(x')=\varphi_0(\sigma(\varphi_0^{-1}(x')))$), and thus gives us a homomorphism $\varphi_1\colon \Aut(M_n)\to \Aut(M_{n'})$. Note that in particular, if $x'=kx$ for some $x\in M_n$, then $\varphi_1(\sigma)(kx)=\varphi_0(\sigma(Hx))=\varphi_0(H\sigma(x))=k\sigma(x)$. Since $x\mapsto kx$ is onto $M_{n'}$, this determines $\varphi_1(\sigma)$ uniquely, and in this sense $\varphi_1$ is induced by $x\mapsto kx$.
		
		Let $m_n$ be enumerated as $(m_n^i)_{i\in \bN}$. Then the natural map $[m_n]_{\equiv}\to [(Hm_n^i)_i]_{\equiv}$ (given by taking each coordinate to its $H$-orbit) is type-definable and induces a natural continuous surjection $S_{m_n}(M_n)\to S_{(Hm_n^i)_i}(M_n)$. Via the isomorphism between $(M_n/H,C_n',R_n')$ and $(M_{n'},C_{n'},R_{n'})$, we obtain a continuous surjection $S_{(Hm_n^i)_i}(M_n)\to S_{m_{n'}}(M_{n'})$. By composing the two, we obtain a surjection $\varphi_2\colon S_{m_n}(M_n)\to S_{m_{n'}}(M_{n'})$, which is easily seen to be $\Aut(M_n)$-equivariant, furthermore, it is not hard to see that it is induced by $x\mapsto kx$ in the sense that if $m_n'\equiv m_n$ is a tuple in $M_n$, then, the type $\tp(m_n'/M_n)$ is mapped to $\tp(km_n'/M_{n'})$ (note that since $M_n$ is $\omega$-categorical, realised types are dense in $S_{m_n}(M_n)$, so by continuity, this uniquely determines $\varphi_2$).
		
		What is left is to show that $\varphi_1$ is surjective. By the description of $\varphi_1$ given before, it is enough to show that for every $\sigma'\in \Aut(M_{n'})$ there is some $\sigma\in\Aut(M_n)$ such that for all $x\in M_{n}$, $k\sigma(x)=\sigma'(kx)$.
		
		We may assume without loss of generality that $\sigma'(0)=0$: otherwise, $\sigma'-\sigma'(0)$ is also an automorphism of $\Aut(M_{n'})$, and if we find $\sigma$ such that $k\sigma(x)=\sigma'(kx)-\sigma'(0)$, then for any $\alpha$ such that $k\alpha=\sigma'(0)$, $\sigma(x)+\alpha$ is an automorphism of $M_n$ and we have $k(\sigma(x)+\alpha)=\sigma'(kx)$.
		
		Put $I_j:=([j/k,(j+1)/k)\cap \bQ)+\bZ\subseteq \bQ/\bZ$ (where $j\in \bZ$). Then we can define $\sigma(x)$ for $x\in I_j$ by letting $\sigma(x)$ be the unique element of $I_j$ such that $k\sigma(x)=\sigma'(kx)$. We need to show that $\sigma$ is an automorphism, and then we will obviously have $\varphi_1(\sigma)=\sigma'$.
		
		Note that we can also describe $\sigma(x)$ in the following way: if $x\in I_j$ and $\beta\in [0,1)$ is a representative of $\sigma'(kx)$, then $\sigma(x)=(\beta+j)/k+\bZ$.
		
		It is not hard to see that for each $j$, $\sigma$ restricts to a bijection from $I_j$ to itself. In particular, $\sigma$ is a bijection. It is also not hard to see that it preserves the circular ordering (roughly, because it preserves the circular order between the chunks $I_j$, it preserves the circular order for triples where not all elements are in a single $I_j$, while for triples lying in a single $I_j$, it follows from the description of $C_n'$ and the fact that $\varphi_0$ is an isomorphism).
		
		Note that since $\sigma'(0)=0$, it follows that for any $x'\in M_{n'}$, $x'\in [1-1/n',1)+\bZ$ if and only if $\sigma'(x')\in [1-1/n',1)+\bZ$. Next, notice that for $x\in I_j$, we have that $kx\in [1-1/n',1)+\bZ$ if and only if $x\in [(j+1)/k-1/n,(j+1)/k)$, which (since $x\in I_j$) is equivalent to $x+1/n\in I_{j+1}$. Thus, for $x\in I_j$, we have $\sigma'(kx)\in [1-1/n',1)+\bZ$ if and only if $R_n(x)\in I_{j+1}$.
		
		It follows that $\sigma$ preserves $R_n$: fix any $x\in I_j$ and let $\beta\in [0,1)$ be a representative of $\sigma'(kx)$. Note that in this case $R_n(x)\in I_j$ or $R_n(x)\in I_{j+1}$. By the preceding paragraph it easily follows that either
		\begin{itemize}
			\item
			 $R_n(x)\in I_j$ and $\beta+1/n'\in [0,1)$, or
			\item
			$R_n(x)\in I_{j+1}$ and $\beta+1/n'-1\in [0,1)$.
		\end{itemize}
		Because $\sigma'$ is an automorphism of $M_{n'}$, it commutes with addition of $1/n'$, so
		\[
			\sigma'(k(x+1/n))=\sigma'(kx+1/n')=\sigma'(kx)+1/n'=(\beta+\bZ)+1/n',
		\]
		which shows that both $\beta+1/n'$ and $\beta+1/n'-1$ are representatives of $\sigma'(kR_n(x))$. Thus if $R_n(x)\in I_j$, then
		\[
			\sigma(R_n(x))=(\beta+1/n'+j)/k+\bZ=(\beta+j)/k+1/n+\bZ=\sigma(x)+1/n=R_n\sigma(x).
		\]
		Likewise, if $R_n(x)\in I_{j+1}$, then
		\[
			\sigma(R_n(x))=((\beta+1/n'-1)+(j+1))/k+\bZ=R_n\sigma(x).
		\]
		Thus, $\sigma$ is an automorphism such that $\varphi_1(\sigma)=\sigma'$, so $\varphi_1$ is onto, and we are done.
	\end{proof}

	\begin{cor}
		\label{cor:ellis_groups_allsorts}
		For all positive integers $n$, the Ellis group of $S_{m_n}(M_n)$ is isomorphic to $\bZ/n\bZ$, generated by $R_nu_n$, where $u_n$ is the identity in the Ellis group.
	\end{cor}
	\begin{proof}
		For $n\neq 1$, this is Lemma~\ref{lem:ellis_group_onesort}, so we only need to consider $n=1$.
		
		By Proposition~\ref{prop:connecting_maps_nn'}, Remark~\ref{rem:action_factors} and Proposition~\ref{prop:induced_epimorphism}, for each $n$, we have an epimorphism from the Ellis group of $S_{m_n}(M_n)$ onto the Ellis group of $S_{m_1}(M_1)$. In particular, since the Ellis groups of $S_{m_2}(M_2)$ and $S_{m_3}(M_3)$ are isomorphic to $\bZ/2\bZ$ and $\bZ/3\bZ$ respectively, the Ellis group of $S_{m_1}(M_1)$ is cyclic of order dividing $2$ and $3$. As such, it must be trivial.
		
		Note that $R_1$ is simply the identity, so $R_1u_1=u_1$ indeed generates the trivial group $\{u_1\}$.
	\end{proof}

	\begin{ex}
		\label{ex:CLPZ}
		Consider the theory $T$ of the multi-sorted structure $M=(M_n)_{n\in \bN^+}$, where each $M_n=(M_n,R_n,C_n)$ is the countable model as described at the beginning of this section. Then, if we enumerate $M$ as $m$, then $M$ is ambitious (because it is $\omega$-categorical). For each $n$, choose an idempotent $u_n'$ as i
		
		By Corollary~\ref{cor:ellis_groups_allsorts} and Proposition~\ref{prop:product of Ellis groups}, the Ellis group $u\cM$ of  the dynamical system $(\Aut(M),S_m(M))$ is isomorphic to $\prod_n \bZ/n\bZ$ with the product topology. Moreover, each element $f\in u\cM$ can be uniquely represented as $(R_n^{b_n}u_n)_n$, where $u_n$ is the restriction $u\restr_{S_{m_n}(M_n)}$, while $b_n$ is an integer in the interval $(-n/2,n/2]$. In particular, $u\cM$ is a Hausdorff (compact and Polish) group, so $H(u\cM)$ is trivial.
		
		Moreover, the group $D$ (i.e.\ $[u]_\equiv\cap u\cM$) is trivial. Indeed, if $f\in u\cM$ is nontrivial, then for some $n$, $f\restr_{S_{m_n}(M_n)}$ and $u\restr_{S_{m_n}(M_n)}$ are distinct. Therefore, $f\restr_{S_{m_n}(M_n)}=R_n^{b_n}(u\restr_{S_{m_n}(M_n)})$ for some $b_n$ not divisible by $n$, so in particular, $f(\tp(m_n/M_n))=R_n^{b_n}u(\tp(m_n/M_n))$, which is clearly distinct from $u(\tp(m_n/M_n))$. Hence, also $f(\tp(m/M))\neq u(\tp(m/M))$, i.e.\ $f\notin D$.
		
		We have proved that $u\cM/H(u\cM)D=u\cM/H(u\cM)=u\cM \cong \prod_n \bZ/n\bZ$, so the group $\hat G$ from Theorem~\ref{thm:main_galois} is $u\cM$, which we identify with $\prod_n \bZ/n\bZ$.
		
		We claim that $g\in \ker \hat r$ if and only if the $g_n$'s are absolutely bounded.
		
		By \cite[Corollary~4.3]{CLPZ01}, for any $a\in M_n(\fC)$ and integer $k \in (-n/2,n/2]$ we have $d_\Lasc(a,R_n^k(a)) \geq k$, which easily implies (having in mind the precise identification of $u\cM$ with $\prod_n \bZ/n\bZ$) that unbounded sequences are not in the kernel.
		
		On the other hand, to show that absolutely bounded sequences are in $\ker \hat r$, it is enough to show this for sequences bounded by $1$. But then (again, having in mind the identification of $u\cM$ with $\prod_n \bZ/n\bZ$ from the second paragraph) for an element $f \in u\cM$ corresponding to such a sequence, $f\restr_{S_{m_n}(M_n)} =R_n^{\epsilon_n}u \restr_{S_{m_n}(M_n)}=u \restr_{S_{m_n}(M_n)}R_n^{\epsilon_n}$ for some $\epsilon_n \in \{-1,0,1\}$. By \cite[Lemma 3.7]{CLPZ01}, it is enough to show that $d_\Lasc(m_n,R_n(m_n))$ is bounded (when $n$ varies).
		By $\omega$-categoricity, we can replace $m_n$ by an enumeration $m_n'$ of any other countable model $M_n'$. So let $m'_n$ be an enumeration of $(\bQ\cap ([0,1/3n)+\bZ/n ))/\bZ\subseteq \bQ/\bZ$. Furthermore, put $m''_n:=m'_n+1/3n$ and $m'''_n:=m'_n+2/3n$, and write $M'_n,M''_n,M'''_n$ for the respective models they enumerate. Then $\tp(m'_n/M'''_n)=\tp(m''_n/M'''_n)$, $\tp(m''_n/M'_n)=\tp(m'''_n/M'_n)$, $\tp(m'''_n/M''_n)=\tp(R_n(m'_n)/M''_n)$, so $d_\Lasc(m_n',R_n(m_n')) \leq 3$.
		
		Note that $T$ has NIP (e.g.\ because it is interpretable in an o-minimal theory), so the full Theorem~\ref{thm:main_galois} applies, and the Galois group $\Gal(T)$ is the quotient of $\prod_n \bZ/n\bZ$ by the subgroup of bounded sequences. As a topological group, this is exactly the description given by \cite[Theorem~28]{Zie02}; note that the topology is trivial. In terms of Borel cardinality, we obtain $\ell^\infty$
		(see the paragraph following the proof of Lemma 3.10 in \cite{KPS13}).
		\xqed{\lozenge}
	\end{ex}
	
	\begin{ex}
		\label{ex:KPS}
		Consider the theory $T$ of the multi-sorted structure $M=(M_n,h_{nn'})_{n,n'}$, where $M_n$ are as before, $n$ runs over positive integers, while $n'$ rangers over divisors of $n$; for each pair $n' \mid n$, $h_{nn'}\colon M_n\to M_{n'}$ is the multiplication by $n/n'$. Enumerate each $M_n$ by $m_n$ in such a way that $h_{nn'}(m_{n})=m_{n'}$, and then enumerate $M$ by $m=(m_n)_n$.
		
		By Proposition~\ref{prop:connecting_maps_nn'}, we see that each map $h_{nn'}$ induces a natural epimorphism $\Aut(M_n)\to \Aut(M_{n'})$, and a continuous, $\Aut(M_n)$-equivariant surjection $S_{m_n}(M_n)\to S_{m_{n'}}(M_{n'})$, so we have epimorphisms of dynamical systems $(\Aut(M_n),S_{m_n}(M_n))\to (\Aut(M_{n'}),S_{m_{n'}}(M_{n'}))$. Furthermore, if $n''\mid n' \mid n$, then it is easy to see that $h_{n'n''}\circ h_{nn'}=h_{nn''}$, so these epimorphisms are compatible. Using that, it is not hard to see that $\Aut(M)=\varprojlim_n \Aut(M_n)$ (because $\Aut(M)$ acts on each $M_n$ by automorphisms and it has to be compatible with $h_{nn'}$) and $S_m(M)=\varprojlim_n S_{m_n}(M_n)$ (because the type on the $n$-th coordinate determines the type on $n'$-th coordinate, when $n'$ divides $n$). By Lemma~\ref{lem:projlim_ellis}, it follows that $E(\Aut(M),S_{m}(M))\cong \varprojlim_n E(\Aut(M_n),S_{m_n}(M_n))$.
		
		In particular, by Corollary~\ref{cor:ellis_groups_allsorts}, the Ellis group $u\cM$ of $E(\Aut(M),S_m(M))$ is isomorphic to the profinite completion of integers $\widehat \bZ=\varprojlim_n\bZ/n\bZ$. By analysis analogous to the preceding example, we see that $H(u\cM)$ and $D$ are trivial, and $\ker \hat r$ corresponds to the elements of $\widehat \bZ$ represented by bounded sequences. Those sequences are exactly the elements of $\bZ\subseteq \widehat \bZ$ (this follows from the observation that a bounded sequence representing an element of $\widehat \bZ$ has to eventually stabilize).
		
		Thus, by Theorem~\ref{thm:main_galois}, $\Gal(T)$ is the quotient $\widehat \bZ/\bZ$ (which, again, has trivial topology), and, since the theory is NIP (e.g.\ because, as it is easy to see, it is interpretable in $(\bR,+,\cdot,\leq)$), $\Gal(T)$ also has the Borel cardinality of $\widehat\bZ/\bZ$ which is $E_0$ (which can be seen as a consequence of the fact that it is hyperfinite (as an orbit equivalence relation of a $\bZ$-action) and non-smooth (as the quotient of a compact Polish group by a non-closed subgroup), see \cite[Theorem 8.1.1]{kanovei}.
		\xqed{\lozenge}
	\end{ex}

	\chapter{Group actions which are not (point-)transitive}
	\chaptermark{Intransitive actions}
	\label{chap:intransitive}
	The results obtained in previous chapters applied only in context of group actions which were transitive, or at least had dense orbits. In this chapter, we find classes of equivalence relations which, while not group-like (because they are not defined on an ambit), share some good properties that we have shown before.
	
	More precisely, the goal is to find a general context where the analogue of Proposition~\ref{prop:top_props_of_wglike}(2) holds, and to use that to extend the first part of Corollary~\ref{cor:metr_smt_cls} (equivalence of closedness and smoothness) and its analogues to wider classes of relations. To that end, we introduce the notion of an orbital and weakly orbital equivalence relation.
	
	The main results of this chapter can be found in \cite{Rz16} (my own paper).
	
	\section[Abstract orbital and weakly orbital equivalence relations]{Abstract orbital and weakly orbital equivalence relations\sectionmark{Abstract (weakly) orbital equivalence relations}}
	\sectionmark{Abstract (weakly) orbital equivalence relations}
	\label{sec:general}
	In this section, $G$ is an arbitrary group, while $X$ is a $G$-space, and neither has any additional structure. The goal of this section is to define and understand orbital and weakly orbital equivalence relations in this abstract context.
	
	\subsection*{Orbital equivalence relations}
	\label{ssec:orbital}
	To every ($G$-)invariant equivalence relation on $X$, we can attach a canonical subgroup of $G$.
	\begin{dfn}
		\label{dfn:HE}
		\index{HE@$H_E$}
		If $E$ is an invariant equivalence relation, then we define $H_E$ as the subgroup of all elements of $G$ which preserve every $E$-class setwise.\xqed{\lozenge}
	\end{dfn}
	
	A dual concept to that of $H_E$ is $E_H$.
	\begin{dfn}
		\index{EH@$E_H$}
		For any $H\leq G$, we denote by $E_H$ the equivalence relation on $X$ of lying in the same $H$-orbit.\xqed{\lozenge}
	\end{dfn}
	
	\begin{ex}
		$E_H$ need not be invariant: for example, if $G=X$ is acting on itself by left translations, then $E_H$ (whose classes are just the right cosets of $H$) is invariant if and only if $H$ is a normal subgroup of $G$.\xqed{\lozenge}
	\end{ex}
	
	\begin{prop}
		\label{prop:orb_norm}
		If $E$ is an invariant equivalence relation on $X$, then $H_E$ is normal in $G$.
		
		In the other direction, if $H\unlhd G$, then $E_H$ is an invariant equivalence relation on $X$.
	\end{prop}
	\begin{proof}
		Let $h\in H_E$ and $g\in G$. We need to show that for any $x$ we have $x \Er ghg^{-1}x$.
		
		Put $y=g^{-1}x$. Then $y\Er hy$. By invariance, $gy\Er ghy$. But $gy=x$ and $ghy=ghg^{-1}x$. This completes the proof of the first part.
		
		For the second part, just note that if $H$ is normal, then $g[x]_{E_H}=gHx=Hgx=[gx]_{E_H}$.
	\end{proof}
	
	The orbital equivalence relations -- defined below -- are extremely well-behaved among the invariant equivalence relations. (Note that a particular case are the relations orbital with respect to the action of $\Aut(\fC)$ in model theory, as per Definition~\ref{dfn:orbital_stype}.)
	
	\begin{dfn}
		\index{equivalence relation!orbital}
		An invariant equivalence relation $E$ is said to be \emph{orbital} if there is a subgroup $H$ of $G$ such that $E=E_H$ (i.e.\ $E$ is the relation of lying in the same orbit of $H$).\xqed{\lozenge}
	\end{dfn}
	
	(Note that if $E$ is orbital, then $E\subseteq E_G$, so $E$-classes are subsets of $G$-orbits.)
	
	\begin{ex}
		\label{ex:single_rotation}
		Let $G=\SO(2)$ act naturally on $X=S^1$. Then for each angle $\theta$, we have the invariant equivalence relation $E_\theta$, such that $z_1\Er_{\theta} z_2$ exactly when $z_1$ and $z_2$ differ by an integer multiple of $\theta$.
		
		This equivalence relation is orbital: if $H\leq SO(2)$ is the group generated by the rotation by $\theta$, then $E_\theta=E_H$.\xqed{\lozenge}
	\end{ex}
	\begin{prop}
		\label{prop:orb_from_group}\leavevmode
		\begin{itemize}
			\item
			If $E$ is an invariant equivalence relation on $X$, then $E_{H_E}\subseteq E$.
			\item
			If, in addition, $E$ is orbital, then $E=E_{H_E}$.
		\end{itemize}
	\end{prop}
	\begin{proof}
		The first part is obvious.
		
		The second is immediate from the assumption that $E$ is orbital: if $E=E_H$, then clearly $H\leq H_E$, so $E_H\subseteq E_{H_E}$ and the proposition follows.
	\end{proof}

	\begin{prop}
		\label{prop:orb_to_group}
		Let $H\leq G$.
		\begin{itemize}
			\item
			If $E_H$ is an invariant equivalence relation, we have $H\leq H_{E_H}$.
			\item
			If, in addition, the action of $G$ on $X$ is free, then $H=H_{E_H}$ and $H\unlhd G$.
		\end{itemize}
	\end{prop}
	\begin{proof}
		The first part is obvious.
		
		For the second part, let us fix some $g\in H_{E_H}$, i.e.\ $g\in G$ which preserves all $E_H$-classes. Then for any $x\in X$ we have $x \Er_H g\cdot x$, so $g\cdot x= h\cdot x$ for some $h\in H$. But then by freeness $g=h$, so we have $g=h\in H$, and hence $H_E\leq H$. Finally, $H=H_{E_H}$ is normal by Proposition~\ref{prop:orb_norm}.
	\end{proof}

	\begin{cor}
		\label{cor:orb_bijection}
		Every orbital equivalence relation is of the form $E_H$ for some $H\unlhd G$. If the action is free, then the correspondence is bijective: every $N\unlhd G$ is of the form $H_E$ for some orbital $E$.
		
		In particular, if $G$ is simple, then the only orbital equivalence relations on $X$ are the equality and $E_G$.
	\end{cor}
	\begin{proof}
		Immediate by Propositions~\ref{prop:orb_norm}, \ref{prop:orb_from_group}, and \ref{prop:orb_to_group}.
	\end{proof}
	
	\begin{prop}
		\label{prop:comm+trans}
		If $G$ is commutative and the action of $G$ on $X$ is transitive, then all invariant equivalence relations on $X$ are orbital. In particular, they all correspond to subgroups of $G$.
	\end{prop}
	\begin{proof}
		Let $E$ be an invariant equivalence relation on $X$. Fix an $x\in X$ and let $H$ be the stabiliser of $[x]_E$. Then $Hx=[x]_E$ (because the action is transitive) and for any $g\in G$, the stabiliser of $[gx]_E$ is $g^{-1}Hg$. But since $G$ is commutative, $gHg^{-1}=H$, so $[gx]_E=Hgx$. Since the action is transitive, it follows that $E=E_H$.
	\end{proof}
	
	\begin{ex}
		The action of $\SO(2)$ on $S^1$ is free, and the group is commutative. This implies that the orbital equivalence relations correspond exactly to subgroups of $\SO(2)$ (in fact, because the action is transitive, those are all the invariant equivalence relations).\xqed{\lozenge}
	\end{ex}
	
	\begin{ex}
		\label{ex:3drotations}
		Consider the natural action of $\SO(3)$ on $S^2$. Certainly, the trivial and total relations are both invariant equivalence relations. Moreover, it is not hard to see that the equivalence relation identifying antipodal points is also invariant.
		
		In fact, those three are the only invariant equivalence relations: if a point $x\in S^2$ is $E$-equivalent to some $y\in S^2\setminus \{x,-x\}$ for some invariant equivalence relation $E$, then by applying rotations around axis containing $x$, we deduce that $x$ is equivalent to every point in a whole circle containing $y$, and in particular, all those points are in $[x]_E$.
		
		But then any rotation close to $\id\in \SO(3)$ takes the circle to another circle which intersects it -- and thus, by transitivity, the ``rotated" circle will still be a subset of $[x]_E$. It is easy to see that we can then ``polish" the whole sphere with (compositions of) small rotations applied to the initial circle, so in fact $[x]_E$ is the whole $\SO(3)$.
		
		The total and trivial equivalence relations are orbital, but the antipodism is not -- it is not hard to verify directly, but it also follows from Corollary~\ref{cor:orb_bijection}, as $\SO(3)$ is a simple group.\xqed{\lozenge}
	\end{ex}
	
	\subsection*{Weakly orbital equivalence relations}
	We want to find a generalisation of orbitality which includes equivalence relations invariant under transitive group actions. Here we define such a notion, and in later on, we will see that it does indeed include both cases.
	
	\begin{dfn}
		\index{equivalence relation!orbital!weakly}
		\index{X@$\tilde X$}
		\label{dfn:worb}
		We say that $E$ is a \emph{weakly orbital} equivalence relation if there is some $\tilde X\subseteq X$ and a subgroup $H\leq G$ such that
		\[
		x_1\Er x_2 \iff \exists g\in G \; \exists h\in H \ \ gx_1=hgx_2\in \tilde X.
		\]
		or equivalently,
		\[
		x_1\Er x_2\iff \exists g_1,g_2\in G\ \ g_1x_1=g_2x_2\in\tilde X\land g_2g_1^{-1}\in H,
		\]
		(In other words, two points are equivalent if we can find some $g\in G$ which takes the first point to some $\tilde x\in \tilde X$, and the second point to something $E_H$-related to $\tilde x$.)\xqed{\lozenge}
	\end{dfn}
	
	Note that, as in the orbital case, a weakly orbital equivalence relation is always a refinement of $E_G$, and it is always $G$-invariant. We will soon see (in Proposition~\ref{prop:orb_is_worb}) that orbital equivalence relations are, as expected, weakly orbital.
	
	\begin{ex}
		The antipodism equivalence relation from Example~\ref{ex:3drotations} is weakly orbital: choose any point $\tilde x\in S^2$, and then choose a single rotation $\theta\in \SO(3)$ which takes $\tilde x$ to $-\tilde x$. Put $\tilde X:=\{\tilde x\}$ and $H:=\langle \theta\rangle$. Then $H$ and $\tilde X$ witness that the antipodism is weakly orbital. (We will see in Proposition~\ref{prop:single_orbit} that this is no coincidence: transitivity of the group action implies that every invariant equivalence relation is weakly orbital.)\xqed{\lozenge}
	\end{ex}
	
	Further examples of weakly orbital equivalence relations will be examined at the end of this section.

	We can justify the name ``weakly orbital": suppose $E$ is weakly orbital on $X$, as witnessed by $\tilde X$ and $H$.
	
	We can attempt to define a right action of $G$ on $X$ thus: for any $x\in X$, pick some $\tilde x\in \tilde X$ such that $x\in G\cdot\tilde x$, and let $g_0\in G$ be such that $g_0\tilde x=x$. Then it seems natural to define $x\cdot g:=g_0\cdot g\cdot \tilde x$.
	
	The problem with this is that it is not, in general, well-defined: the value on the right hand side may depend on the choice of $\tilde x$, and even if we do fix $\tilde x$, it may nonetheless depend on the choice of $g_0$.
	
	If, however, we are somehow in a situation where this is a well-defined right group action, $E$ would be the orbit equivalence relation of $H$ acting on the right.
	
	To sum up, one can say that a weakly orbital equivalence relation is the relation of lying in the same orbit of an ``imaginary group action".
	
	We will see later, in Remark~\ref{rem:worb_interpret}, another interpretation of weak orbitality, which is more closely tied to the applications in later sections.

	The following notation (and its alternative definition in Remark~\ref{rem:alt_rsub}) is very convenient, and we will use it frequently in the rest of this chapter.
	
	\begin{dfn}
		\index{RHX@$R_{H,\tilde X}$}
		\label{dfn:rsub}
		For arbitrary $H\leq G$ and $\tilde X\subseteq X$, let us denote by $R_{H,\tilde X}$ the relation on $X$ (which may not be an equivalence relation) defined by
		\[
		x_1\Rr_{H,\tilde X} x_2\iff \exists g\in G \; \exists h\in H \ \ gx_1=hgx_2\in \tilde X.\xqed{\lozenge}
		\]
	\end{dfn}
	
	\begin{rem}
		\label{rem:alt_rsub}
		Note that $R_{H,\tilde X}$ may also be defined as the smallest relation $R$ such that:
		\begin{itemize}
			\item
			$R$ is invariant,
			\item
			for each $\tilde x\in \tilde X$ and $h\in H$ we have $\tilde x\Rr h\tilde x$. \xqed{\lozenge}
		\end{itemize}
	\end{rem}
	
	\begin{rem}\leavevmode
		\begin{enumerate}
			\item
			$R_{?,?}$ is monotone, i.e.\ if $H_1\subseteq H_2$ and $\tilde X_1\subseteq \tilde X_2$, then $R_{H_1,\tilde X_1}\subseteq R_{H_2,\tilde X_2}$.
			\item
			If $R_{H,\tilde X}$ is an equivalence relation (or even merely a reflexive one), then for any $x\in X$ we have $(G\cdot x)\cap \tilde X\neq \emptyset$ (or, equivalently, $G\cdot \tilde X=X$).
			\item
			an equivalence relation $E$ on $X$ is weakly orbital if and only if $E=R_{H,\tilde X}$ for some $H$, $\tilde X$.
			\item
			If for every $\tilde x\in \tilde X$ we have $H_1\tilde x=H_2\tilde x$, then $R_{H_1,\tilde X}=R_{H_2,\tilde X}$.\xqed{\lozenge}
		\end{enumerate}
	\end{rem}
	
	We can give an explicit description of ``classes" of $R_{H,\tilde X}$.
	
	\begin{lem}
		\label{lem:worb_class_description}
		Let $R=R_{H,\tilde X}$. For every $x_0\in X$, we have
		\[
		\{x\mid x_0\mathrel R x \}=\bigcup_{g} g^{-1} H g\cdot x_0,
		\]
		where the union runs over $g\in G$ such that $g\cdot x_0\in \tilde X$.
	\end{lem}
	\begin{proof}
		If $x_0\mathrel R x $, then we have $gx_0=hgx\in \tilde X$ for some $g\in G$ and $h\in H$. But then $x= g^{-1}h^{-1}gx_0$. On the other hand, if $x=g^{-1}hgx_0$ for some $h\in H$ and $g\in G$ such that $g x_0\in \tilde X$, then $h^{-1}g x=g x_0\in \tilde X$.
	\end{proof}

	\begin{dfn}
		\index{maximal witness (of orbitality)}
		\label{dfn:max_witness}
		We say that $H$ is a \emph{maximal witness} for weak orbitality of $E$ if there is some $\tilde X\subseteq X$ such that $E=R_{H,\tilde X}$ and for any $H'\gneq H$ we have $E\neq R_{H',\tilde X}$.
		
		Similarly, we say that $\tilde X$ is a \emph{maximal witness} for weak orbitality of $E$ if there is some $H\leq G$ such that $E=R_{H,\tilde X}$ and for any $\tilde X'\supsetneq \tilde X$ we have $E\neq R_{H,\tilde X'}$.
		
		We say that a pair $(H,\tilde X)$ is a \emph{maximal pair of witnesses} for weak orbitality of $E$ if $E=R_{H,\tilde X}$ and for any $H'\geq H$, $\tilde X\supseteq X$ we have that $E=R_{H',\tilde X'}$ if and only if $H=H'$ and $\tilde X=\tilde X'$\xqed{\lozenge}
	\end{dfn}
	
	The following proposition is, in part, an analogue of Proposition~\ref{prop:orb_from_group} (only for weakly orbital equivalence relations, instead of orbital).
	
	\begin{lem}
		\label{lem:worb_maximal}
		Consider $R=R_{H,\tilde X}$. Then:
		\begin{itemize}
			\item
			$R=R_{H,\tilde X'}$, where $\tilde{X'}:=\{x\in X \mid \forall h\in H\ \ x\Rr hx\}$, and
			\item
			$R=R_{H',\tilde X}$, where $H':=\{g\in G \mid \forall \tilde x\in \tilde X \ \ \tilde x \Rr g\tilde x \}$.
		\end{itemize}
		
		Moreover, if $R=E$ is an equivalence relation, then:
		\begin{itemize}
			\item
			each of $\tilde X'$ and $H'$ is a maximal witness in the sense of Definition~\ref{dfn:max_witness},
			\item
			applying the two operations, in either order, yields a maximal pair of witnesses,
			\item
			every maximal witness $\tilde X$ is a union of $E$-classes.
		\end{itemize}
	\end{lem}
	\begin{proof}
		For the first bullet, $R$ is an invariant relation such that for all $\tilde x\in \tilde X'$ and $h\in H$ we have $\tilde x \Rr h\tilde x$. Since $R_{H,\tilde X'}$ is, by Remark~\ref{rem:alt_rsub}, the finest such relation, it follows that $R_{H,\tilde X'}\subseteq R$. On the other hand, $\tilde X\subseteq \tilde X'$, so $R=R_{H,\tilde X}\subseteq R_{H,\tilde X'}$. The second bullet is analogous.
		
		The first two bullets of the ``moreover" part are clear. For the third, we only need to see that $\tilde X'$ is a union of $E$-classes. For that, just notice that if $x\Er hx$ and $y\Er x$, then $y\Er hx$ and (by invariance of $E$) $hx\Er hy$, so in fact $y\Er hy$.
	\end{proof}
	
	Note that, in contrast to the orbital case, where we have a canonical maximal witness (namely, $H_E$), in the weakly orbital case, the maximal pairs of witnesses are, in general, far from canonical, for instance because for any $g\in G$ we have $R_{H,\tilde X}=R_{gHg^{-1},g\cdot \tilde X}$. But even up to this kind of conjugation, the choice may not be canonical, as we see in the following example.
	
	\begin{ex}
		Let $F$ be any field, and consider the action of the affine group $F^3\rtimes \GL_3(F)$ on itself by left translations. Put:
		\begin{itemize}
			\item
			$H_1=(F^2\times \{0\})\times \{I\}$,
			\item
			$H_2=(F\times \{0\}^2)\times\{I\}$,
			\item
			$\tilde X_1=F^3\times \{g\in \GL_3(F)\mid g^{-1}H_1g=H_1 \}$, and
			\item
			$\tilde X_2=F^3\times \{g\in \GL_3(F)\mid g^{-1}H_2g\subseteq H_1\}$.
		\end{itemize}
		Then, using Lemma~\ref{lem:worb_class_description}, we deduce that $E=R_{H_1,\tilde X_1}=R_{H_2,\tilde X_2}$ is an orbital equivalence relation (which on $F^3\times \{I\}$ is just lying in the same plane parallel to $F^2\times \{0\}$). On the other hand, both pairs $H_1, \tilde X_1$ and $H_2,\tilde X_2$ are maximal, and simultaneously, $\tilde X_1\subsetneq \tilde X_2$ and $H_2\subsetneq H_1$. If $F$ is a finite field, the sets and groups don't even have the same cardinality, and they are certainly not conjugate even if $F$ is infinite.\xqed{\lozenge}
	\end{ex}
	
	\subsection*{Orbitality and weak orbitality; transitive actions}
	\label{ssec:orb+trans_as_worb}
	Now, we proceed to show that weakly orbital equivalence relations do indeed include all orbital equivalence relations (as well as \emph{all} invariant equivalence relations when the action is transitive), and to investigate what makes a weakly orbital equivalence relation actually orbital.
	\begin{prop}
		\label{prop:orb_is_worb}
		Every orbital equivalence relation is weakly orbital. In fact, if $E=E_H$ is an invariant equivalence relation, then $E=R_{H,X}(=R_{H_E,X})$ (i.e.\ we have $\tilde X=X$).
		
		Conversely, if $E=R_{H,\tilde X}$ is an invariant equivalence relation and $H\unlhd G$, then $E=E_H$, so a weakly orbital equivalence relation is orbital precisely when there is a normal group witnessing the weak orbitality.
	\end{prop}
	\begin{proof}
		For the first half, $E_H$ is by definition the finest relation such that for each $x\in X$ and $h\in H$ we have $x\Er_H hx$. If it is also invariant, we have by Remark~\ref{rem:alt_rsub} that $R_{H,X}=E_H$.
		
		The second half is an immediate consequence of Lemma~\ref{lem:worb_class_description}.
	\end{proof}
	
	\begin{cor}
		\label{cor:comm_worb}
		If $G$ is a commutative group, then every weakly orbital equivalence relation is orbital.
	\end{cor}
	\begin{proof}
		If $G$ is commutative and $E=R_{H,X}$, then $H\unlhd G$, and thus $E=E_H$ by Proposition~\ref{prop:orb_is_worb}.
	\end{proof}
	
	The next corollary shows that a proper inclusion $\tilde X\subsetneq X$ is another obstruction of orbitality (other than non-normality of $H$). This shows that $\tilde X$ is necessary in the definition of weak orbitality.
	
	\begin{cor}
		\label{cor:orb_gap2}
		A weakly orbital equivalence $E$ relation is orbital precisely when we can choose as the witness $\tilde X$ the whole domain $X$, i.e.\ $E=R_{H,X}$ for some $H\leq G$.
	\end{cor}
	\begin{proof}
		That weak orbitality of an orbital equivalence relation is witnessed by $X=\tilde X$ is a part of Proposition~\ref{prop:orb_is_worb}.
		
		In the other direction, if $E=R_{H,X}$, then we can define the maximal witnessing group $H'$ as in Lemma~\ref{lem:worb_maximal}. Since $X=\tilde X$, this $H'$ coincides with $H_E$, so it is normal by Proposition~\ref{prop:orb_norm}. But then by Lemma~\ref{lem:worb_maximal} and Proposition~\ref{prop:orb_is_worb} we have $R_{H,X}=R_{H',X}=E_{H'}$, so $E$ is orbital.
	\end{proof}
	
	Orbital equivalence relations are weakly orbital as witnessed by $\tilde X=X$. Equivalence relations invariant under transitive actions are, in a sense, an orthogonal class: in their case, we can choose a singleton $\tilde X$.
	
	\begin{prop}
		\label{prop:single_orbit}
		Suppose the action of $G$ is transitive, and $\tilde x \in X$ is arbitrary. Then for any invariant equivalence relation $E$ we have $E=R_{\Stab_G\{[\tilde x ]_E\},\{\tilde x \}}$ (where $\Stab_G\{[\tilde x ]_E\}$ is the setwise stabiliser of $[\tilde x ]_E$, i.e.\ $\{g\in G\mid \tilde x\Er g\tilde x \}$).
	\end{prop}
	\begin{proof}
		Choose any $x_1,x_2$ and let $g_1,g_2$ be such that $g_1x_1=g_2x_2=\tilde x $ (those exist by the transitivity). To complete the proof, it is enough to show that $h=g_2g_1^{-1}\in \Stab_G\{[\tilde x ]_E\}$ if and only if $x_1\Er x_2$ (because $hg_1x_2=g_2x_2$).
		
		Note that $g_2g_1^{-1}\in \Stab_G\{[\tilde x ]_E\}$ if and only if $\tilde x \Er g_2g_1^{-1}(\tilde x )$. But, since $E$ is invariant, this is equivalent to $g_2^{-1}\tilde x \Er g_1^{-1}\tilde x $. But $g_1^{-1}\tilde x =x_1$ and $g_2^{-1}\tilde x =x_2$, so we are done.
	\end{proof}
	
	Note that Corollary~\ref{cor:comm_worb} and Proposition~\ref{prop:single_orbit} provide an alternative proof of Proposition~\ref{prop:comm+trans}.
	
	Note also that even for transitive actions, $R_{H,\{\tilde x\}}$ is not in general an equivalence relation, as it may fail to be transitive.
	
	\begin{ex}
		Let $G$ be the free group of rank $2$, and consider its free generators $a$ and $b$. Let $H=\langle b\rangle$ and $K=\langle a\rangle$. Put $X=G/K$ with $G$ acting on $X$ by left translations, and finally, put $\tilde x=eK\in X$ and $R=R_{H,\{\tilde x \}}$. Then $eK\Rr bK$, and also $eK=a\cdot (eK)\Rr a\cdot (bK)=abK$. Therefore, $b\cdot (eK)=bK\Rr b\cdot(abK)=babK$. But it is not true that $eK\Rr babK$ (indeed, by Lemma~\ref{lem:worb_class_description}, $eK\Rr gK$ if and only if $gK=a^nb^mK$ for some integers $n,m$), so $R$ is not transitive.\xqed{\lozenge}
	\end{ex}

	\begin{rem}
		\label{rem:worb_interpret}
		Given an arbitrary invariant equivalence relation $E$ on $X$, we can attach to each $G$-orbit $G\cdot \tilde x$ with a fixed ``base point" $\tilde x$ a subgroup $H_{\tilde x}:=\Stab_G\{[\tilde x ]_E\}$, such that $G\cdot {\tilde x}/E$ is isomorphic (as a $G$-space) with $G/H_{\tilde x}$.
		
		Along with Lemma~\ref{lem:worb_class_description}, this gives an intuitive description of weakly orbital equivalence relations (among the invariant equivalence relations) as those for which we have a set $\tilde X$ which restricts choice of ``base points", and at the same time ``uniformly" limits the manner in which the group $H_{\tilde x}$ changes between various orbits: we can only take a union of conjugates of a fixed subgroup.
		
		In the later sections, we will consider the behaviour of $E$ when $\tilde X$ is somehow well-behaved, and we can think of that as somehow ``smoothing" the manner in which the group changes between the $G$-orbits.\xqed{\lozenge}
	\end{rem}

	\subsection*{Further examples of weakly orbital equivalence relations}
	\label{ssec:worb_ex}
	In the first example, we define a class of examples of weakly orbital equivalence relations which are not orbital, on spaces $X$ such that $\lvert X/G\rvert$ is large (so the action is far from transitive).
	\begin{ex}
		Consider the action of $G$ on $X=G^2$ by left translation in the first coordinate only. Let $\tilde X\subseteq G^2$ be the diagonal, and $H$ be any subgroup of $G$. Then $R_{H,\tilde X}$ is a weakly orbital equivalence relation on $X$ whose classes are sets of the form $(g_1g_2^{-1}Hg_2)\times \{g_2\}$. The relation is orbital if and only if $H$ is normal (because the action is free), while $\lvert X/G\rvert=\lvert G\rvert$.\xqed{\lozenge}
	\end{ex}
	
	The second example is a vast generalisation of its predecessor.
	
	\begin{ex}
		\label{ex:separate_orbits}
		Let $G$ be any group, while $H\leq G$ is a subgroup. Suppose $(X_i)_{i \in I}$ are disjoint $G$-spaces with subsets $\tilde X_i$, such that $R_{H,\tilde X_i}$ is a weakly orbital equivalence relation on $X_i$. Let $X$ be the disjoint union $\coprod_{i\in I} X_i$. Put $\tilde X=\bigcup_{i\in I}\tilde X_i$. Then $E=R_{H,\tilde X}$ is a weakly orbital equivalence relation (which is just the union $\coprod_{i\in I} R_{H,\tilde X_i}$).\xqed{\lozenge}
	\end{ex}
	
	The third example shows how we can, in a way, join weakly orbital equivalence relations on different $G$-spaces, for varying $G$.
	
	\begin{ex}
		Suppose we have a family of groups $(G_i)_{i\in I}$ acting on spaces $(X_i)_{i\in I}$ (respectively), and that on each $X_i$ we have a [weakly] orbital equivalence relation $E_i$. Then the product $G=\prod_{i\in I}G_i$ acts on the disjoint union $X=\coprod_{i\in I}X_i$ naturally (i.e\ $(g_i)_i\cdot x=g_j\cdot x$ when $x\in X_j$) and $E=\coprod_{i\in I} E_i$ is [weakly] orbital.\xqed{\lozenge}
	\end{ex}
	
	The final example shows us that we cannot, in general, choose $\tilde X$ as a transversal of $X/G$ (i.e.\ a set intersecting each orbit at precisely one point).
	
	\begin{ex}
		\label{ex:worb_nontrivial}
		Let $F$ be an arbitrary field. Consider the affine group $G=F^3\rtimes \GL_3(F)$, and let $G'$ be a copy of $G$, disjoint from it.
		
		Let $\ell\subseteq F^3$ be a line containing the origin. Choose a plane $\pi\subseteq F^3$ containing $\ell$. Let $E$ be the invariant equivalence relation on $X=G\sqcup G'$ (on which $G$ acts by left translations) which is:
		\begin{itemize}
			\item
			on $G$: $(x_1,g_1)\Er (x_2,g_2)$ whenever $g_1=g_2$ and $x_1-x_2\in g_2\cdot \pi$,
			\item
			on $G'$: $(x_1',g_1')\Er (x_2',g_2')$, whenever $g_1'=g_2'$ and $x_1'-x_2'\in g_2'\cdot \ell$ (slightly abusing the notation).
		\end{itemize}
		
		Put $H=\ell\times \{I\}\leq G$, let $A\subseteq \GL_3(F)$ be such that $A^{-1}\cdot l=\pi$ and $I\in A$, and finally let $\tilde X=(\{0'\}\times \{I'\})\cup (\{0\}\times A)$ (where $0'$ is the neutral element in the vector space component of $G'$).
		
		Then $E$ is weakly orbital, as witnessed by $\tilde X$ and $H$ (to see this, recall Lemma~\ref{lem:worb_class_description} and notice that the $E$-class and $R_{\tilde X,H}$-``class" of $I$ both are the union of $I+a^{-1}\cdot \ell$ over $a\in A$, while the class of $I'$ is just $I'+\ell$, and then use the fact that both $E$ and $R_{\tilde X,H}$ are invariant).
		
		We will show that $E$ does not have any $\tilde X_1$ witnessing weak orbitality which intersects each of $G$ and $G'$ at exactly one point. Suppose that $E=R_{H_1,\tilde X_1}$ and $\tilde X_1\cap G=\{g\}$, while $\tilde X_1\cap G'=\{g'\}$. Since the action of $G$ on $X$ is free, it follows from Lemma~\ref{lem:worb_class_description} that $[g]_E=H_1g$ and $[g']_E=H_1g'$. This implies that in fact $H_1\leq F^3$, and the first equality implies that $H_1$ is a plane, while the second one implies that it is a line, which is a contradiction.\xqed{\lozenge}
	\end{ex}

	\section{Abstract structured equivalence relations}
	\label{sec:structured}
	In this section, we consider an action of a group $G$ on a set $X$, but we also put on them some additional structure: namely, we have on each finite product of $G$ and $X$ (as sets) a lattice of sets (i.e.\ a family closed under finite unions and intersections) which we call pseudo-closed, such that the empty set and the whole space is always pseudo-closed. In the remainder of this section, the lattices are implicitly present and fixed.
	
	Note that the lattices may not be closed under arbitrary intersection (so they need not be the lattices of closed sets in the topological sense) and we do not necessarily assume that the lattice on a product is the product lattice (so even if, say, the pseudo-closed sets on $X$ are actually closed sets in a topology, there might be a pseudo-closed set in $X^2$ which is not closed in the product topology).
	
	Naturally, we need to impose some compatibility conditions on the group action and the lattices of pseudo-closed sets.
	
	\begin{dfn}
		\index{agreeable group action}
		\index{pseudo-closed set}
		\label{dfn:agree}
		We say that the lattices of pseudo-closed sets \emph{agree} with the group action of $G$ on $X$ (or, leaving the lattice implicit, $G$ acts \emph{agreeably} on $X$) if we have the following:
		\begin{enumerate}
			\item
			sections of pseudo-closed sets are pseudo-closed,
			\item
			products of pseudo-closed sets are pseudo-closed,
			\item
			the map $G\times X\to X$ defined by the formula $(g,x)\mapsto g\cdot x$ is pseudo-continuous (i.e.\, the preimages of pseudo-closed sets are pseudo-closed),
			\item
			for each $g\in G$, the map $X\to X\times X$ defined by the formula $x\mapsto (x,g\cdot x)$ is pseudo-continuous,
			\item
			\label{it:dfn:agree:proj}
			$\pi\restr_{E_G}$, the restriction to $E_G\subseteq (X^2)^2$ of the projection onto the first two coordinates $\pi\colon (X^2)^2\to X^2$, is a pseudo-closed mapping (i.e.\ images of relatively pseudo-closed sets are pseudo-closed), where $(x_1,x_2)\Er_G (x_1',x_2')$ when there is some $g\in G$ such that $x_1'=gx_1$ and $x_2'=gx_2$,
			\item
			the map $X\times G\to X\times X$ defined by the formula $(x,g)\mapsto(x,g\cdot x)$ is pseudo-closed.\xqed{\lozenge}
		\end{enumerate}
	\end{dfn}
	
	\begin{ex}
		\label{ex:agree_cpct}
		A prototypical example of an agreeable group action is any continuous action of a compact Hausdorff group on a compact Hausdorff space, where pseudo-closed means simply closed. Definition~\ref{dfn:agree} is easy to verify there, as all the functions under consideration are continuous, and hence closed (as continuous functions between compact spaces and Hausdorff spaces). Section~\ref{sec:cpct} is dedicated to the generalisation of this example where we only assume that the group is compact, not the space it acts on.\xqed{\lozenge}
	\end{ex}
	
	\begin{lem}
		\label{lem:psclsd}
		If $G$ acts agreeably on $X$ and $E$ is an invariant equivalence relation, then for each $h\in G$ and $\tilde x\in X$:
		\begin{itemize}
			\item
			$\Stab_G\{[\tilde x]_E\}=\{g\in G\mid \tilde x\Er g\tilde x \}$ is pseudo-closed whenever $[x]_E$ is pseudo-closed,
			\item
			$\{x\in X\mid x\Er hx \}$ is pseudo-closed whenever $E$ is pseudo-closed.
		\end{itemize}
	\end{lem}
	\begin{proof}
		For the first bullet, the set in question is a section at $\tilde x$ of the preimage of $[\tilde x]_E$ via the map $(g,x)\mapsto g\cdot x$, so it is pseudo-closed by agreeability.
		
		For the second bullet, the set is the preimage of $E$ via $x\mapsto (x,hx)$, so it is pseudo-closed by agreeability.
	\end{proof}

	\begin{thm}
		\label{thm:orb}
		If $G$ acts agreeably on $X$, while $E$ is an orbital equivalence relation, and the lattice of pseudo-closed sets in $G$ is downwards $[G:H_E]$-complete (i.e.\ closed under intersections of at most $[G:H_E]$ sets), then the following are equivalent:
		\begin{enumerate}
			\item
			\label{it:thm:orb:clsd}
			$E$ is pseudo-closed,
			\item
			\label{it:thm:orb:clclsd}
			each $E$-class is pseudo-closed,
			\item
			\label{it:thm:orb:HEclsd}
			$H_E$ is pseudo-closed,
			\item
			\label{it:thm:orb:Hclsd}
			$E=E_H$ for some pseudo-closed $H\leq G$.
		\end{enumerate}
	\end{thm}
	\begin{proof}
		If $E$ is pseudo-closed, then each $E$-class is pseudo-closed (as a section of $E$), so we have \ref{it:thm:orb:clsd}$\Rightarrow$\ref{it:thm:orb:clclsd}.
		
		To see that \ref{it:thm:orb:clclsd}$\Rightarrow$\ref{it:thm:orb:HEclsd}, note that
		\[
		H_E=\bigcap_{\tilde x\in X}\Stab_G\{[\tilde x]_E\}.
		\]
		By Lemma~\ref{lem:psclsd}, the stabilisers are pseudo-closed. Ostensibly, there are $\lvert X\rvert$ factors in the intersection, so completeness does not apply directly. However, each of the stabilisers is a group containing $H_E$, and as such, it is some union of cosets of $H_E$ in $G$. It follows that to calculate the intersection, we only need to see which cosets are excluded from it, and since there are only $[G:H_E]$ many cosets, the intersection can be realised as the intersection of at most $[G:H_E]$ factors (one for each coset excluded from the intersection). Therefore, by completeness, $H_E$ is pseudo-closed.
		
		\ref{it:thm:orb:HEclsd}$\Rightarrow$\ref{it:thm:orb:Hclsd} is a weakening (by Proposition~\ref{prop:orb_from_group}).
		
		For \ref{it:thm:orb:Hclsd}$\Rightarrow$\ref{it:thm:orb:clsd}, notice that $E_H$ is the image of $X\times H$ by the map $(x,g)\mapsto (x,g\cdot x)$, which is pseudo-closed by agreeability.
	\end{proof}
	
	It makes sense to consider the question about when the closedness of classes implies the closedness of the whole equivalence relation. The following example shows that if simply we drop the orbitality assumption in Theorem~\ref{thm:orb}, the implication no longer holds.
	\begin{ex}
		\label{ex:not_orbital}
		Consider the action of $G={\bf Z}/2{\bf Z}$ on $X=\{0,1\}\times \{0,\frac{1}{n}\mid n\in {\bf N}^+ \}$ by changing the first coordinate. This is an agreeable action, as a special case of Example~\ref{ex:agree_cpct} (so pseudo-closed = closed).
		
		Consider the equivalence relation $E$ on $X$ such that its classes are $\{(0,0)\}$, $\{(1,0)\}$, and $\{(0,\frac{1}{n}),(1,\frac{1}{n})\}$, where $n\in {\bf N}^+$. Clearly, $E$ is invariant and all its classes are closed, but it is not itself closed.\xqed{\lozenge}
	\end{ex}
	
	The following definition makes for more elegant statements of the remaining results.
	
	\begin{dfn}
		\index{weakly orbital by pseudo-closed}
		We say that an invariant equivalence relation $E$ is \emph{weakly orbital by pseudo-closed} if there is a pseudo-closed set $\tilde X\subseteq X$ and a (not necessarily pseudo-closed) group $H\leq G$ such that $E=R_{H,\tilde X}$.\xqed{\lozenge}
	\end{dfn}
	
	(We will also replace the epithet ``pseudo-closed" in the above definition by others in the more concrete applications, so e.g.\ in the context of Example~\ref{ex:agree_cpct}, we would talk about ``weakly orbital by closed" equivalence relations.)

	\begin{thm}
		\label{thm:worb}
		If in Theorem~\ref{thm:orb} we assume that $E$ is only weakly orbital (instead of orbital), and add the assumption that the lattice of pseudo-closed sets in $X$ is also downwards $[G:H_E]$-complete, then the following are equivalent:
		\begin{enumerate}
			\item
			\label{it:thm:worb_1}
			$E$ is pseudo-closed,
			\item
			\label{it:thm:worb_2}
			each $E$-class is pseudo-closed and $E$ is weakly orbital by pseudo-closed,
			\item
			\label{it:thm:worb_3}
			$E=R_{H,\tilde X}$ for some pseudo-closed $H$ and $\tilde X$,
			\item
			\label{it:thm:worb_4}
			for every $H\leq G$ and $\tilde X\subseteq X$, if either of $H$ or $\tilde X$ is a \emph{maximal} witness to weak orbitality of $E$, then it is also pseudo-closed.
		\end{enumerate}
	\end{thm}
	\begin{proof}
		We will show the implications \ref{it:thm:worb_1}$\Rightarrow$\ref{it:thm:worb_2}$\Rightarrow$ \ref{it:thm:worb_3} $\Rightarrow$ \ref{it:thm:worb_1}, and on the way, that the three conditions imply \ref{it:thm:worb_4} (which implies \ref{it:thm:worb_3} by Lemma~\ref{lem:worb_maximal}).
		
		If we assume \ref{it:thm:worb_1}, then clearly all the classes are pseudo-closed (as sections of $E$), and we have $H_1$, $\tilde X_1$ such that $E=R_{H_1,\tilde X_1}$ (because we have assumed that $E$ is weakly orbital). Then we can put
		\[
		\tilde X:=\bigcap_{h\in H_1} \{x\in X\mid x\Er hx \}.
		\]
		The sets we intersect are then pseudo-closed by Lemma~\ref{lem:psclsd}. As before, the intersection may have more than $[G:H_E]$ factors, but there are at most $[G:H_E]$-many distinct sets of the form $\{x\mid x\Er hx \}$, because every such set depends only on the left $H_E$-coset of $h$. To see this, note that for every $x\in X$ and $h_0\in H_E$, we have $x\Er h_0 x$, and therefore --- by invariance --- also $hx\Er hh_0x$. Thus, $x \Er hx$ implies that $x\Er hh_0x$.
		
		It follows that we have the following equality:
		\[
		\tilde X=\bigcap_{hH_E\in H_1/H_E} \{x\in X\mid x\Er hx \},
		\]
		and therefore, by completeness, $\tilde X$ is pseudo-closed, and by Lemma~\ref{lem:worb_maximal}, we have $E=R_{H_1,\tilde X}$, and hence \ref{it:thm:worb_2}. This also gives us the part of \ref{it:thm:worb_4} pertaining to $\tilde X$: if $\tilde X_1$ was already maximal, as witnessed by $H_1$, then we would have $\tilde X=\tilde X_1$.
		
		The implication \ref{it:thm:worb_2}$\Rightarrow$\ref{it:thm:worb_3} is showed the same way as the one in Theorem~\ref{thm:orb} (using Lemma~\ref{lem:worb_maximal}), only we take the intersection over the pseudo-closed set $\tilde X$ (which we have by definition of weakly orbital by pseudo-closed) instead of the whole $X$. The same reasoning shows the remaining part of \ref{it:thm:worb_4}.
		
		For \ref{it:thm:worb_3}$\Rightarrow$\ref{it:thm:worb_1}, notice that $x_1 \Rr_{H,\tilde X} x_2$ if and only if there are $x_1'$ and $x_2'$ such that $(x_1,x_2)\Er_G (x_1',x_2')$ (where $E_G$ is defined as in Definition~\ref{dfn:agree}\ref{it:dfn:agree:proj}), $x_1'\in \tilde X$ and $x_1' \Er_H x_2'$. Since $H$ is pseudo-closed, we also have that $E_H$ is pseudo-closed (just as in the final paragraph of the proof of Theorem~\ref{thm:orb}), so overall, this is a condition about $(x_1,x_2,x_1',x_2')$ which is relatively pseudo-closed in $E_G$, and the projection onto the first two coordinates (which is just $E=R_{H,\tilde X}$) is also pseudo-closed (by Definition~\ref{dfn:agree}\ref{it:dfn:agree:proj}).
	\end{proof}
	
	\begin{rem}
		Since $X$ itself is its own pseudo-closed subset, one can use Proposition~\ref{prop:orb_is_worb} and Corollary~\ref{cor:orb_gap2} to show that for orbital $E$, the conclusion of Theorem~\ref{thm:worb} implies the conclusion of Theorem~\ref{thm:orb}, so, if we ignore the completeness assumptions, Theorem~\ref{thm:worb} implies Theorem~\ref{thm:orb}.\xqed{\lozenge}
	\end{rem}
	
	One might ask whether in Theorem~\ref{thm:worb}, we could have weakened the condition \ref{it:thm:worb_2} to say only that each class is pseudo-closed (or, equivalently, the condition \ref{it:thm:worb_3} to say only that $H$ is pseudo-closed). But this is not the case -- as explained in Remark~\ref{rem:worb_interpret}, we need the $\tilde X$ to control the way $E$ changes between $G$-orbits. This is shown in the following example.
	
	\begin{ex}
		Let $G=S_3$ act naturally on $S_3\times \{0,\frac{1}{n}\mid n\in {\bf N}^+ \}$. This is another special case of Example~\ref{ex:agree_cpct}, so this action is agreeable. Let $H=\{{\id},(1,2) \}$, and let $\tilde X=\{((1,2,3),0),({\id},\frac{1}{n})\mid n\in {\bf N}^+\}$. Then $E=R_{H,\tilde X}$ is weakly orbital, $H$ is pseudo-closed, as are all the $E$-classes, but $E$ is not (because $({\id},0)$ and $((1,2),0)$ are not related whereas each $({\id},\frac{1}{n})$ is related to $((1,2),\frac{1}{n})$).\xqed{\lozenge}
	\end{ex}

	\section[Compact group actions]{(Weakly) orbital equivalence relations for compact group actions\sectionmark{Compact group actions}}
	\sectionmark{Compact group actions}
	\label{sec:cpct}
	
	In this section, $X$ is a (Hausdorff) $G$-space for a compact Hausdorff group $G$ (and the action is continuous). The pseudo-closed sets are just the closed sets in respective spaces. Then pseudo-continuity and pseudo-closedness of functions are just the usual topological continuity and closedness.
	
	\subsection*{Preparatory lemmas in the case of compact group actions}
	
	We have seen in Example~\ref{ex:agree_cpct} that a continuous action of a compact Hausdorff group $G$ on a compact Hausdorff space $X$ is agreeable. It turns out that compactness of $X$ is not necessary.
	\begin{lem}
		\label{lem:agree_cpct}
		Actions of compact groups are agreeable (with respect to the standard closed sets, according to Definition~\ref{dfn:agree}).
	\end{lem}
	\begin{proof}
		Recall from Fact~\ref{fct:cpct_action} that if $G$ is a compact Hausdorff group acting continuously on a Hausdorff space $X$, then the multiplication $X\times G\to X$ ($(x,g)\mapsto(g\cdot x)$) and the quotient $X\to X/G$ are both closed.
		
		It is enough to demonstrate the last two points of Definition~\ref{dfn:agree}: that the projection mapping from $E_G$ onto $X^2$ and the mapping $(x,g)\mapsto (x,g\cdot x)$ are both closed. The rest is straightforward (and does not rely on compactness of $G$).
		\begin{figure}[H]
			\begin{tikzcd}
				\Delta(X^2)\arrow[r,hook,"\subseteq"] & E_G\arrow[r,hook,"\subseteq"] \arrow[dl,two heads,"\pi\restr _{E_G}"] \arrow[d, two heads, "q"] &(X^2)^2\\
				X^2\arrow{u}{\approx}[swap]{\Delta} \arrow[r, "\approx"]& E_G/G \arrow[l] &
			\end{tikzcd}
			\caption{The commutative diagram of the functions discussed in the proof. Each of them is continuous and closed.}
		\end{figure}
		For the first one, consider the diagonal embedding $\Delta\colon X^2\to \Delta(X^2)$ (as a subset of $X^4$) composed with the quotient map $q\colon E_G\to E_G/G\subseteq X^2\times (X^2/G)$, where the quotient is with respect to the action defined by the formula $g\cdot (x_1,x_2,x_1',x_2')=(x_1,x_2,g\cdot x_1',g\cdot x_2')$. The first function is a homeomorphic embedding with closed image, and the second is a closed map (by Fact~\ref{fct:cpct_action}, as a quotient map with respect to a compact group action).
		
		Note that both $q$ and $\pi\restr_{E_G}$ are onto, and they glue together exactly those points which share the first two coordinates. It follows that we have an induced bijection between $X^2$ and $E_G/G$. But this bijection is just $q\circ \Delta$, which is continuous and closed, and therefore a homeomorphism.
		
		This implies that $\pi\restr_{E_G}$ must be closed (as the composition of $q$ -- which is closed -- and $(q\circ \Delta)^{-1}$ -- which is a homeomorphism).
		
		The second one is immediate by Fact~\ref{fct:cpct_proper}.
	\end{proof}

	\subsection*{Results in the case of compact group actions}

	\begin{thm}
		\label{thm:orb_cpct}
		Suppose $G$ is a compact Hausdorff group acting continuously on a Hausdorff space $X$.
		
		The following are equivalent for an orbital invariant equivalence relation $E$ on $X$:
		\begin{enumerate}
			\item
			\label{it:thm:orb_cpct_1}
			$E$ is closed
			\item
			each $E$-class is closed,
			\item
			\label{it:thm:orb_cpct_3}
			$H_E$ is closed,
			\item
			\label{it:thm:orb_cpct_4}
			$E=E_H$ for a closed subgroup $H\leq G$,
			\item
			\label{it:thm:orb_cpct_5}
			$X/E$ is Hausdorff.
		\end{enumerate}
	\end{thm}
	\begin{proof}
		\ref{it:thm:orb_cpct_5} clearly implies \ref{it:thm:orb_cpct_1} (because $E$ is the preimage of the diagonal by the quotient map $X^2\to (X/E)^2$), while the implication from \ref{it:thm:orb_cpct_4} to \ref{it:thm:orb_cpct_5} is a consequence of Fact~\ref{fct:cpct_action}.
		
		Notice that the lattice of closed sets is simply downwards complete, so the rest follows immediately from Theorem~\ref{thm:orb} and Lemma~\ref{lem:agree_cpct}.
	\end{proof}
	
	The following examples show that we cannot drop the assumption that $G$ is compact in Theorem~\ref{thm:orb_cpct}, even if $G$ is otherwise very tame.
	\begin{ex}
		Consider the action of $G={\bf R}$ on a two-dimensional torus $X={\bf R}^2/{\bf Z}^2$ by translations along a line with an irrational slope (e.g.\ $t\cdot [x_1,x_2]=[x_1+t,x_2+ t\sqrt 2]$). Then for $H=G$ the relation $E_G$ has dense (and not closed) orbits, so in particular, $X/E_G$ has trivial topology.\xqed{\lozenge}
	\end{ex}
	
	\begin{ex}
		Consider the action of $G={\bf R}$ on $X={\bf R}^2$ defined by the formula $t\cdot (x,y)= (x+ty,y)$. Then for $H=G$, the classes of $E_G$ are the singletons along the line $y=0$ and horizontal lines at $y\neq 0$, so they are closed, but $E_G$ is not closed.\xqed{\lozenge}
	\end{ex}

	\begin{cor}
		\label{cor:orb_closed}
		Suppose $G$ is a compact Hausdorff group acting continuously on a Hausdorff space $X$.
		
		Then every closed orbital equivalence relation on $X$ is of the form $E_H$ for some closed $H\unlhd G$. If the action is free, the correspondence is bijective: every closed $N\unlhd G$ is of the form $H_E$ for some closed orbital $E$.
		
		In particular, if $G$ is topologically simple, then the only closed orbital equivalence relations on $X$ are the equality and $E_G$.
		
		On the other hand, if $G$ is commutative and the action is transitive, then all the closed invariant equivalence relations on $X$ are of the form $E_H$ for closed $H\leq G$.
	\end{cor}
	\begin{proof}
		Immediate from Corollary~\ref{cor:orb_bijection}, Proposition~\ref{prop:comm+trans} and Theorem~\ref{thm:orb_cpct}.
	\end{proof}

	Recall that an invariant equivalence relation $E$ is ``weakly orbital by closed" if there is a closed $\tilde X\subseteq X$ and any $H\leq G$ such that $E=R_{H,\tilde X}$.
	
	\begin{thm}
		\label{thm:worb_cpct}
		Suppose $G$ is a compact Hausdorff group acting continuously on a Hausdorff space $X$.
		Then the following are equivalent for a weakly orbital equivalence relation $E$:
		\begin{enumerate}
			\item
			$E$ is closed,
			\item
			each $E$-class is closed and $E$ is weakly orbital by closed,
			\item
			$E=R_{H,\tilde X}$ for some closed $H$ and $\tilde X$,
			\item
			for every $H\leq G$ and $\tilde X\subseteq X$, if either of $H$ or $\tilde X$ is a \emph{maximal} witness to weak orbitality of $E$, then it is also closed.
		\end{enumerate}
	\end{thm}
	\begin{proof}
		Immediate from Theorem~\ref{thm:worb} and Lemma~\ref{lem:agree_cpct}.
	\end{proof}
	
	Notice that if $X$ is compact, then by Fact~\ref{fct:quot_T2_iff_closed}, the conditions in Theorem~\ref{thm:worb_cpct} imply that $X/E$ is Hausdorff, but for arbitrary $X$ (in contrast to Theorem~\ref{thm:orb_cpct}), we do not know whether this is true.

	\begin{cor}
		\label{cor:smt_cpct}
		Suppose that $G$ is a compact Hausdorff group acting on a Polish space $X$. Suppose that $E$ is an invariant equivalence relation on $X$ which is orbital or, more generally, weakly orbital by closed.
		Then the following are equivalent:
		\begin{enumerate}
			\item
			\label{it:cor:smt_cpct:clsd}
			$E$ is closed,
			\item
			\label{it:cor:smt_cpct:clses}
			each $E$-class is closed,
			\item
			\label{it:cor:smt_cpct:smt}
			$E$ is smooth,
			\item
			\label{it:cor:smt_cpct:rest}
			for each $x\in X$, the restriction $E\restr_{G\cdot x}$ is closed.
		\end{enumerate}
	\end{cor}
	\begin{proof}
		Clearly, \ref{it:cor:smt_cpct:clsd} implies \ref{it:cor:smt_cpct:rest}, which implies \ref{it:cor:smt_cpct:clses}.
		
		By Theorem~\ref{thm:orb_cpct} or \ref{thm:worb_cpct}, \ref{it:cor:smt_cpct:clses} implies \ref{it:cor:smt_cpct:clsd}.
		
		By Fact~\ref{fct:clsd_smth}, \ref{it:cor:smt_cpct:clsd} implies \ref{it:cor:smt_cpct:smt}.
		
		Finally, \ref{it:cor:smt_cpct:smt} implies that each restriction $E\restr_{G\cdot x}$ is smooth, and as such --- by Corollary~\ref{cor:toy_trich} --- it is closed. Therefore, its classes are closed in $G\cdot x$, which --- by compactness of $G$ --- is closed in $X$, so we have \ref{it:cor:smt_cpct:clses}.
	\end{proof}
	
	(Note that, since every orbital equivalence relation is weakly orbital by closed, the ``orbital or" part of Corollary~\ref{cor:smt_cpct} is redundant.)

	\section[Type-definable group actions]{(Weakly) orbital equivalence relations for type-definable group actions
		\sectionmark{Type-definable group actions}}
	\sectionmark{Type-definable group actions}
	\label{sec:def}
	
	\subsection*{Preparatory lemmas in the case of type-definable group actions}
	In this section, $G$ is a type-definable group, $X$ is a type-definable set, while the action of $G$ on $X$ is also type-definable (in the sense that it has a type-definable graph), all in the monster model $\fC$.
	
	\begin{lem}
		\label{lem:agree_def}
		If $X$ is a type-definable set, $G$ is a type-definable group acting in a type-definable way on $X$, then the action is agreeable (with respect to type-definable sets as pseudo-closed sets), according to Definition~\ref{dfn:agree}.
	\end{lem}
	\begin{proof}
		Sections of (relatively) type-definable are clearly type-definable, as are products. Since the projection of a type-definable set is type-definable, the remaining points are straightforward as well (analogously to Example~\ref{ex:agree_cpct}).
	\end{proof}

	\subsection*{Results in the case of type-definable group actions}

	\begin{thm}
		\label{thm:orb_def}
		Let $G$ be a type-definable group acting type-definably on a type-definable set $X$.
		
		Suppose $E$ is an orbital, $G$-invariant equivalence relation on $X$ with $G^{000}_A$-invariant classes (for some small set $A$). Then the following are equivalent:
		\begin{enumerate}
			\item
			$E$ is type-definable,
			\item
			each $E$-class is type-definable,
			\item
			$H_E$ is type-definable,
			\item
			there is a type-definable subgroup $H\leq G$ such that $E=E_H$.
		\end{enumerate}
		In addition, if $E$ is bounded (equivalently, if $X/G$ is small), then the conditions are equivalent to the statement that $X/E$ is Hausdorff with the logic topology.
	\end{thm}
	\begin{proof}
		Immediate from Theorem~\ref{thm:orb}, Lemma~\ref{lem:agree_def} and Fact~\ref{fct:logic_top_cpct_T2}. Note that the completeness needed for Theorem~\ref{thm:orb} follows from the fact that $G^{000}_A\leq H_E$, so by definition $[G:H_E]\leq [G:G^{000}_A]$ is small.
	\end{proof}
	
	Recall that an invariant equivalence relation $E$ is ``weakly orbital by type-definable" if there is a type-definable $\tilde X\subseteq X$ and any $H\leq G$ such that $E=R_{H,\tilde X}$
	
	\begin{thm}
		\label{thm:worb_def}
		In context of Theorem~\ref{thm:orb_def}, if we assume instead that $E$ is only weakly orbital, then the following are equivalent:
		\begin{enumerate}
			\item
			$E$ is type-definable,
			\item
			each $E$-class is type-definable and $E$ is weakly orbital by type-definable,
			\item
			$E=R_{H,\tilde X}$ for some type-definable $H$ and $\tilde X$,
			\item
			for every $H\leq G$ and $\tilde X\subseteq X$, if either of $H$ or $\tilde X$ is a \emph{maximal} witness for weak orbitality of $E$, then it is also type-definable.
		\end{enumerate}
		In addition, if $E$ is bounded (equivalently, $X/G$ is bounded), then the conditions are equivalent to statement that $X/E$ is Hausdorff with the logic topology.
	\end{thm}
	\begin{proof}
		Immediate from Theorem~\ref{thm:worb}, Lemma~\ref{lem:agree_def} and Fact~\ref{fct:logic_top_cpct_T2}.
	\end{proof}
	
	\begin{cor}
		\label{cor:smt_def}
		Assume that the theory is countable, and fix a countable set $A$ of parameters. Suppose $G$ is a type-definable group acting type-definably on $X$ (with both $G$ and $X$ consisting of countable tuples, all over $A$), while $E$ is a bounded, $G$-invariant and $\Aut(\fC/A)$-invariant equivalence relation on $X$. Assume in addition that $E$ is orbital or, more generally, weakly orbital by type-definable. Then the following are equivalent:
		\begin{enumerate}
			\item
			\label{it:cor:smt_def:clsd}
			$E$ is type-definable,
			\item
			\label{it:cor:smt_def:clses}
			each $E$-class is type-definable,
			\item
			\label{it:cor:smt_def:smt}
			$E$ is smooth,
			\item
			\label{it:cor:smt_def:T2}
			$X/E$ is Hausdorff,
			\item
			\label{it:cor:smt_def:rest}
			for each $x\in X$, the restriction $E\restr_{G\cdot x}$ is type-definable.
		\end{enumerate}
	\end{cor}
	\begin{proof}
		Clearly, \ref{it:cor:smt_def:clsd} implies \ref{it:cor:smt_def:rest}, which implies \ref{it:cor:smt_def:clses}.
		
		By Theorem~\ref{thm:orb_def} or \ref{thm:worb_def}, \ref{it:cor:smt_def:clses} implies \ref{it:cor:smt_def:clsd} (note that the assumptions that $E$ is bounded and $A$-invariant imply together that all $E$-classes are $G^{000}_A$-invariant, cf.\ Proposition~\ref{prop:bdd_iff_invariant} --- but note that here, the action need not be transitive, so we only have one implication).
		
		\ref{it:cor:smt_def:clsd} implies \ref{it:cor:smt_def:smt} by Remark~\ref{rem:tdf_implies_smooth}, and it is equivalent to \ref{it:cor:smt_def:T2} by Fact~\ref{fct:logic_top_cpct_T2}.
		
		Finally, \ref{it:cor:smt_def:smt} implies that each restriction $E\restr_{G\cdot x}$ is smooth, which -- by Corollary~\ref{cor:trich+_tdf} -- implies \ref{it:cor:smt_def:rest}.
	\end{proof}
	(Note that any orbital equivalence relation is also weakly orbital by type-definable, so the ``orbital or" part is redundant.)

	\section[Automorphism group actions]{(Weakly) orbital equivalence relations for automorphism groups\sectionmark{Automorphism group actions}}
	\sectionmark{Automorphism group actions}
	\label{sec:aut}
	In this section, $X$ is a $\emptyset$-type-definable subset of a small product of sorts in $\fC$, while $G$ is just $\Aut(\fC)$. (In particular, in this case, orbitality coincides with Definition~\ref{dfn:orbital_stype}, at least for bounded invariant equivalence relations.) We use the letters $\Gamma$ and $\gamma$ where we would use $H$ and $h$ in the rest of this chapter, following the notation of \cite{KR16} in that respect; for example, we write $\Gamma_E$ instead of $H_E$ (cf.\ Definition~\ref{dfn:HE}) and we typically denote the group witnessing weak orbitality by $\Gamma$ instead of $H$ as before.
	
	It is worth noting that in contrast to Sections~\ref{sec:cpct} and \ref{sec:def}, we will not apply Theorems~\ref{thm:orb} and \ref{thm:worb} directly. Instead, we will apply them to the action of $\Gal(T)$ on $X/{\equiv_\Lasc}$, and the preparatory lemmas will provide us with tools to translate the result back to $\Aut(\fC)$ and $X$.
	
	More precisely, we identify $\Gal(T)$ with $[m]_\equiv/{\equiv_\Lasc}$ for a tuple $m$ enumerating a small model (cf.\ Fact~\ref{fct:sm_to_gal}), and the pseudo-closed sets are the sets closed in the logic topology: for example, a pseudo-closed set in $\Gal(T)\times X/{\equiv_\Lasc}=([m]_\equiv\times X)/({\equiv_\Lasc}\times {\equiv_\Lasc})$ is a set whose preimage in $[m]_\equiv\times X$ is type-definable. Note that it is \emph{not} a priori the same as being closed in the product of logic topologies on $\Gal(T)$ and $X/{\equiv_\Lasc}$! (More precisely, the product topology might be coarser.) Similarly, the product relation ${\equiv_\Lasc}\times{\equiv_\Lasc}$ on a product of two invariant sets is usually not the finest bounded invariant equivalence relation on it (so it coarser than $\equiv_\Lasc$ on the product).
	
	As a side result, we will show that orbitality and weak orbitality are well-defined for bounded invariant equivalence relations, see Corollary~\ref{cor:mtprop}.

	\subsection*{Preparatory lemmas in the case of automorphism group action}
	
	\begin{lem}
		\label{lem:agree_aut}
		If $X$ is a type-definable set, then the action of $\Gal(T)$ on $X/{\equiv_\Lasc}$ is agreeable (with respect to sets closed in logic topology, according to Definition~\ref{dfn:agree}).
	\end{lem}
	\begin{proof}
		For brevity, let us write $\bar x$, for any $[x]_{\equiv_\Lasc} \in X/{\equiv_\Lasc}$, as well as $\bar \sigma$ for $\sigma\Autf(\fC)\in \Gal(T)$, and $\bar n$ for $[n]_{\equiv_\Lasc}\in [m]_\equiv/{\equiv_\Lasc}$ (which, by Fact~\ref{fct:sm_to_gal}, we identify with the sole $\bar{\sigma}\in \Gal(T)$ such that $\bar{\sigma}(\bar m)=\bar n$, where $\bar m=[m]_{\equiv_\Lasc}$).
		
		Consider the partial type $\Phi(n,x,y)= (mx\equiv ny\land x\in X)$.
		\begin{clm*}
			For any $n\in [m]_\equiv$ and $x,y\in X$, the following are equivalent:
			\begin{itemize}
				\item
				$\bar n\cdot \bar x=\bar y$,
				\item
				$\models (\exists y')\Phi(n,x,y')\land y\equiv_\Lasc y'$, and
				\item
				$\models (\exists n')\Phi(n',x,y)\land n\equiv_\Lasc n'$.
			\end{itemize}
		\end{clm*}
		\begin{clmproof}
			Suppose $\bar n\cdot \bar x=\bar y$. Then we have some $\sigma\in \Aut(\fC)$ such that $\sigma(\bar m)=\bar n$ and $\sigma(\bar x)=\bar y$. This means that we have some $\tau\in \Autf(\fC)$ such that $\tau(\sigma(m))=n$. But $\overline{\tau\circ \sigma}=\bar \sigma$, and taking $y'=\tau\circ\sigma(x)$ gives us the second bullet. For the reverse implication, if $\sigma$ witnesses that $mx\equiv ny'$, then in particular $\bar{\sigma}(\bar m)=\bar{n}$, so by definition $\bar n\cdot \bar x=\bar{\sigma}(\bar x)=\overline{\sigma(x)}=\bar {y'}=\bar y$. The proof that the third bullet is equivalent to the first is analogous.
		\end{clmproof}
		
		It follows that for any $A\subseteq X/{\equiv_\Lasc}$ we have $\bar n\cdot \bar x\in A$ if and only if $\models (\exists y)\,\, \bar y\in A\land \Phi(n,x,y)$ (because ``$\bar y\in A$" is a $\equiv_\Lasc$-invariant condition), which is a type-definable condition about $n$ and $x$, if $A$ is closed. This gives us the third point from Definition~\ref{dfn:agree} (continuity of $(\bar n,\bar x)\mapsto \bar n\cdot \bar x$).
		
		To obtain the fourth point (continuity of $\bar x\mapsto (\bar x,\bar n\cdot \bar x)$ for all $\bar n$), note that if we fix any $\bar n\in \Gal(T)$ and some $A\subseteq (X\times X)/({\equiv_\Lasc}\times {\equiv_\Lasc})$, then likewise $(\bar x,\bar n\cdot \bar x)\in A$ exactly when $\models(\exists y)\,\, (\bar x,\bar y)\in A\land \Phi(n,x,y)$, which is again a type-definable condition about $x$ whenever $A$ is closed.
		
		For the fifth point, note that the $E_G$ from Definition~\ref{dfn:agree} is just the relation $\equiv$ on $X^2$, which is of course type-definable as a subset of $(X^2)^2$. Thus, any relatively type-definable subset of it is actually type-definable, and thus so is its projection onto $X^2$.
		
		Similarly, for the sixth point (closedness of $(\bar n,\bar x)\mapsto (\bar x,\bar n\cdot \bar x)$), note that $(\bar x,\bar y)$ is in the image of $A\subseteq ([m]_{\equiv}\times X)/({\equiv_\Lasc}\times {\equiv_\Lasc})$ exactly when $\models (\exists n)\,\, (\bar n,\bar x)\in A\land \Phi(n,x,y)$, which is a type-definable condition about $x$ and $y$, as long as $A$ is closed.
		
		The remaining parts of Definition~\ref{dfn:agree} are easy to verify.
	\end{proof}
	
	\index{E@${\bar E}$}
	From now on, given a bounded invariant equivalence relation on an invariant set $X$, denote by $\bar E$ the induced equivalence relation on $X/{\equiv_\Lasc}$.
	\begin{prop}
		\label{prop:orb_to_gal}
		Suppose $E$ is a bounded invariant equivalence relation on an invariant set $X$. Suppose also that $\Gamma\leq \Aut(\fC)$ contains $\Autf(\fC)$ [and suppose $\tilde X\subseteq X$]. Write $\bar{\Gamma}:=\Gamma/\Autf(\fC)$ [and $\bar{\tilde X}:=\tilde X/{\equiv_\Lasc}$]. Then the following are equivalent.
		\begin{enumerate}
			\item
			\label{it:prop:orb_to_gal:1}
			$E=E_\Gamma$ [$E=R_{\Gamma,\tilde X}$]
			\item
			\label{it:prop:orb_to_gal:2}
			$\bar E=E_{\bar \Gamma}$ [$\bar E=R_{\bar \Gamma,\bar {\tilde X}}$].
		\end{enumerate}
		In particular, $E$ is [weakly] orbital if and only if $\bar E$ is (because we can always assume that $\Autf(\fC)$ is contained in the group witnessing [weak] orbitality, as $\Autf(\fC)$ fixes each class setwise).
	\end{prop}
	\begin{proof}
		In this proof, for brevity, whenever $x\in X$ and $\sigma \in \Aut(\fC)$, we will write $\bar x$ and $\bar \sigma$ as shorthands for $[x]_{\equiv_\Lasc}\in X/{\equiv_\Lasc}$ and $\sigma\cdot \Autf(\fC)\in \Gal(T)$, respectively.
		
		Consider the quotient map $q\colon X\to X/{\equiv_\Lasc}$. Then we have:
		\begin{equation}
		\label{eq:prop:orb_to_gal:1}
		\tag{$\dagger$}
		q[\Gamma\cdot x]=\bar{\Gamma}\cdot \bar x,
		\end{equation}
		and, since $\Gamma\cdot x$ is an $\equiv_\Lasc$-saturated set, conversely:
		\begin{equation}
		\label{eq:prop:orb_to_gal:2}
		\tag{$\dagger\dagger$}
		\Gamma\cdot x=q^{-1}[\bar{\Gamma}\cdot \bar x].
		\end{equation}

		For the orbital case, the implication \ref{it:prop:orb_to_gal:1}$\Rightarrow$\ref{it:prop:orb_to_gal:2} is an immediate consequence of \eqref{eq:prop:orb_to_gal:1} -- $\bar E$-classes are the $q$-images of $E$-classes, while $\bar{\Gamma}$-orbits are the $q$-images of $\Gamma$-orbits. The converse is analogous, as by \eqref{eq:prop:orb_to_gal:2}, $\Gamma$-orbits are the $q$-preimages of $\bar \Gamma$-orbits and of course $E$-classes are $q$-preimages of $\bar E$-classes.
		
		The weakly orbital case can be proved similarly: by Lemma~\ref{lem:worb_maximal}, we can assume that
		\begin{equation}
		\label{eq:prop:orb_to_gal_3}
		\tag{$*$}
		\tilde X=\Autf(\fC)\cdot\tilde X,
		\end{equation}
		Then we can just apply Lemma~\ref{lem:worb_class_description}: if
		\[
		[\bar x]_{\bar E}=\bigcup_{\bar \sigma}\bar \sigma^{-1}[\bar \Gamma\cdot \bar\sigma(\bar x)]
		\]
		(where the union runs over $\bar \sigma$ such that $\bar \sigma(\bar x)\in \bar{\tilde X}$), then also
		\[
		[x]_E=q^{-1}[[\bar x]_{\bar E}]=\bigcup_{\bar \sigma}q^{-1}[\bar \sigma^{-1}[\bar \Gamma\cdot \bar\sigma(\bar x)]]=\bigcup_{\sigma}\sigma^{-1}[\Gamma\cdot \sigma(x)],
		\]
		where the last union runs over $\sigma$ such that $\sigma(x)\in \tilde X$. To see the last equality, just note that (by \eqref{eq:prop:orb_to_gal_3}) $\sigma(x)\in \tilde X$ if and only if $\bar\sigma(\bar x)\in \bar{\tilde X}$. This yields \ref{it:prop:orb_to_gal:2}$\Rightarrow$\ref{it:prop:orb_to_gal:1}, and the opposite implication is analogous.
	\end{proof}

	\begin{cor}
		\label{cor:mtprop}
		{}[Weak] orbitality of a bounded invariant equivalence relation is a model-theoretic property, i.e.\ it does not depend on the choice of the monster model.
	\end{cor}
	\begin{proof}
		$X/{\equiv_\Lasc}$, $\Gal(T)$, $\bar E$ and the action of $\Gal(T)$ on $X/{\equiv_\Lasc}$ do not depend on the monster model, so the result is immediate from Proposition~\ref{prop:orb_to_gal}.
	\end{proof}
	(It is not clear whether orbitality or weak orbitality is a model-theoretic property for an unbounded invariant equivalence relation.)

	\begin{lem}
		\label{lem:closed_cl_to_closed_witn}
		If $E=R_{\Gamma,\tilde X}$ is bounded invariant and either:
		\begin{itemize}
			\item
			for each $\tilde x\in \tilde X$, $[\tilde x]_E$ is type-definable, or
			\item
			$\Autf_\KP(\fC)\leq \Gamma$
		\end{itemize}
		then $\tilde X':=\{x\in X\mid \exists \tilde x\in \tilde X \ \ x \equiv_\KP \tilde x \}$ satisfies $E=R_{\Gamma,\tilde X'}$. (Note that if $\tilde X$ is type-definable, so is $\tilde X'$, and $\Autf(\fC)\cdot \tilde X'=\tilde X'$.)
	\end{lem}
	\begin{proof}
		By Lemma~\ref{lem:worb_maximal}, the first bullet implies that we can assume the second one: each $[\tilde x]_E$ is $\equiv_\Lasc$-saturated, so if it is type-definable, it is also $\equiv_\KP$-saturated (see Proposition~\ref{prop:lem_closed}), i.e.\ $\Autf_\KP(\fC)$-invariant.
		
		Now, assuming the second bullet: the $\tilde X'$ considered here contains $\tilde X$ and it is contained in the maximal one defined as in Lemma~\ref{lem:worb_maximal} (because the maximal one is $\Gamma$-invariant, and hence $\Autf_\KP(\fC)$-invariant), so $E=R_{\Gamma,\tilde X'}$.
	\end{proof}

	\subsection*{Results in the case of automorphism group action}

	\begin{thm}
		\label{thm:orb_aut}
		Suppose $E$ is a bounded invariant, orbital equivalence relation on $X$. Then the following are equivalent:
		\begin{enumerate}
			\item
			$E$ is type-definable,
			\item
			each $E$-class is type-definable,
			\item
			$\Gamma_E$ is the preimage of a closed subgroup of $\Gal(T)$,
			\item
			$E=E_\Gamma$ for some $\Gamma$ which is the preimage of a closed subgroup of $\Gal(T)$,
			\item
			$X/E$ is Hausdorff.
		\end{enumerate}
	\end{thm}
	\begin{proof}
		$(1)$ and $(5)$ are equivalent by Fact~\ref{fct:logic_top_cpct_T2}.
		
		The rest follows readily from Theorem~\ref{thm:orb}, Lemma~\ref{lem:agree_aut} and Proposition~\ref{prop:orb_to_gal}. For example, if $E=E_{\Gamma}$ for some $\Gamma$ which is the preimage of a closed subgroup of $\Gal(T)$, then $\bar \Gamma$ is closed and by Proposition~\ref{prop:orb_to_gal}, $\bar E=E_{\bar \Gamma}$, so (by Theorem~\ref{thm:orb} and Lemma~\ref{lem:agree_aut}) $\bar E$ is closed, and hence $E$ is type-definable.
	\end{proof}
	
	Recall that an invariant equivalence relation $E$ is ``weakly orbital by type-definable" if there is a type-definable $\tilde X\subseteq X$ and any $\Gamma\leq \Aut(\fC)$ such that $E=R_{\Gamma,\tilde X}$.
	
	\begin{thm}
		\label{thm:worb_aut}
		Suppose $E$ is bounded invariant, weakly orbital equivalence relation on $X$. Then the following are equivalent:
		\begin{enumerate}
			\item
			\label{it:thm:worb_aut:closed}
			$E$ is type-definable,
			\item
			\label{it:thm:worb_aut:closedcl}
			each $E$-class is type-definable and $E$ is weakly orbital by type-definable,
			\item
			\label{it:thm:worb_aut:closedgp}
			$E=R_{\Gamma,\tilde X}$ for some type-definable $\tilde X$ and a group $\Gamma\leq \Aut(\fC)$ which is the preimage of a closed subgroup of $\Gal(T)$,
			\item
			for every $\Gamma\leq \Aut(\fC)$ and $\tilde X\subseteq X$, if $\Gamma$ or $\tilde X$ is a \emph{maximal} witness for weak orbitality of $E$, then it is also the preimage of a closed subgroup of $\Gal(T)$ (in the case of $\Gamma$) or type-definable (in the case of $\tilde X$),
			\item
			\label{it:thm:worb_aut:T2}
			$X/E$ is Hausdorff.
		\end{enumerate}
	\end{thm}
	\begin{proof}
		Points \ref{it:thm:worb_aut:closed} and \ref{it:thm:worb_aut:T2} are equivalent by Fact~\ref{fct:logic_top_cpct_T2}.
		
		As for the rest, the only added difficulty compared to Theorem~\ref{thm:orb_aut} comes from the fact that the type-definable $\tilde X$ we have by the assumptions of \ref{it:thm:worb_aut:closedcl} and \ref{it:thm:worb_aut:closedgp} may not be $\equiv_\Lasc$-saturated, but thanks to Lemma~\ref{lem:closed_cl_to_closed_witn}, we can assume that without loss of generality. Once we have that, we finish as before, using Theorem~\ref{thm:worb}, Lemma~\ref{lem:agree_aut} and Proposition~\ref{prop:orb_to_gal}.
		
		For example, if we have \ref{it:thm:worb_aut:closedgp}, then -- by Lemma~\ref{lem:closed_cl_to_closed_witn} -- we can assume without loss of generality that $\tilde X=\Autf(\fC)\cdot \tilde X$. This implies that $\bar{\tilde X}=\tilde X/{\equiv_\Lasc}$ is closed. Moreover, $E=R_{\Gamma,\tilde X}$, so by Proposition~\ref{prop:orb_to_gal}, we also have $\bar E=R_{\bar \Gamma,\bar{\tilde X}}$, so by Lemma~\ref{lem:agree_aut} and Theorem~\ref{thm:worb}, $\bar E$ is closed, which immediately gives us \ref{it:thm:worb_aut:closed}.
	\end{proof}
	
	The following corollary is Main~Theorem~\ref{mainthm_worb}.
	
	\begin{cor}
		\label{cor:smt_aut}
		Assume that the theory is countable. Suppose that $E$ is a bounded, invariant, countably supported equivalence relation on $X$. Assume in addition that $E$ is orbital or, more generally, weakly orbital by type-definable. Then the following are equivalent:
		\begin{enumerate}
			\item
			\label{it:cor:smt_aut:clsd}
			$E$ is type-definable,
			\item
			\label{it:cor:smt_aut:clses}
			each $E$-class is type-definable,
			\item
			\label{it:cor:smt_aut:smt}
			$E$ is smooth,
			\item
			\label{it:cor:smt_aut:T2}
			$X/E$ is Hausdorff,
			\item
			\label{it:cor:smt_aut:rest}
			for each complete $\emptyset$-type $p\vdash X$, the restriction $E\restr_{p(\fC)}$ is type-definable.
		\end{enumerate}
	\end{cor}
	\begin{proof}
		Clearly, \ref{it:cor:smt_aut:clsd} implies \ref{it:cor:smt_aut:rest}, which implies \ref{it:cor:smt_aut:clses}.
		
		By Theorem~\ref{thm:orb_aut} or \ref{thm:worb_aut}, \ref{it:cor:smt_aut:clses} implies \ref{it:cor:smt_aut:clsd}.
		
		\ref{it:cor:smt_aut:clsd} implies \ref{it:cor:smt_aut:smt} by Remark~\ref{rem:tdf_implies_smooth}, and it is equivalent to \ref{it:cor:smt_aut:T2} by Fact~\ref{fct:logic_top_cpct_T2}.
		
		Finally, \ref{it:cor:smt_aut:smt} implies that each restriction $E\restr_{p(\fC)}$ is smooth, which -- by Corollary~\ref{cor:smt_type} -- implies \ref{it:cor:smt_aut:rest}.
	\end{proof}
	(Note that every orbital equivalence relation is also weakly orbital by type-definable, so the ``orbital or" part is redundant.)
	
	\begin{rem}
		Note also that, in the context of Corollary~\ref{cor:smt_aut}, Corollary~\ref{cor:smt_type} implies that if $Y\subseteq X$ is a type-definable and $E$-invariant subset of $X$ such that $\Aut(\fC)\cdot Y=X$ and $E\restr_Y$ is smooth, then the condition \ref{it:cor:smt_aut:rest} from Corollary~\ref{cor:smt_aut} is satisfied (and hence also all the others).
		
		Moreover, if $X=p(\fC)$, then it is a single $\Aut(\fC)$ orbit, so by Proposition~\ref{prop:single_orbit}, every invariant equivalence relation is weakly orbital by type-definable (as singletons are certainly type-definable). Thus, Corollary~\ref{cor:smt_aut} extends Corollary~\ref{cor:smt_type}.\xqed{\lozenge}
	\end{rem}

	\chapter{``Borel cardinality" in the non-metrisable case}
	\chaptermark{Non-metrisable case}
	\label{chap:nonmetrisable_card}
	In this chapter, we discuss a possible variant of Corollary~\ref{cor:metr_smt_cls} applicable when $X$ is not metrisable, but more subtle than simple cardinality estimates given by Theorem~\ref{thm:general_cardinality_intransitive} and Theorem~\ref{thm:general_cardinality_transitive}.
	The content of this chapter is heavily based on the Section~6.2 of \cite{KPR15} (joint with Krzysztof Krupiński and Anand Pillay).
	
	Recall that Corollary~\ref{cor:metr_smt_cls} tells us that for a wide class of weakly group-like equivalence relations, smoothness and closedness are equivalent conditions. Note tht for a non-metrisable compact Hausdorff space, a closed equivalence relation need not be smooth in the naïve sense that we have a reduction to equality on a Polish space, because that would imply having no more than $2^{\aleph_0}$ classes, which rule out, for example, equality on any non-metrisable compact Hausdorff group (which is trivially closed group-like). Instead, we could try to conceive generalisations of smoothness given by studying reductions to relations on ``higher reals'', such as $2^\kappa$ for uncountable cardinals $\kappa$. Unfortunately, if we go in that direction, many of the properties used in study of Borel cardinalities are no longer true (e.g.\ the analogues of the Silver dichotomy and Harrington-Kechris-Louveau may not hold, see \cite{THK14} for examples).
	
	For those reasons, we use a weak form of non-smoothness of an equivalence relation, which is essentially due to \cite{KMS14}.

	Recall that if $E$ is a non-smooth Borel equivalence relation on a Polish space $X$, then $E$ is non-smooth if and only if $\EZ\leq_B E$, and in fact, on this case, the reduction can be chosen as a homeomorphic embedding of $2^{\bN}$ into $X$ (see Fact~\ref{fct:Harrington-Kechris-Louveau dichotomy}).
	
	The idea that non-smoothness corresponds to some embedding of $\EZ$ can be used to generalise the notion of non-smoothness. Recall the notion of the sub-Vietoris topology, a coarsening of the Vietoris topology, introduced in \cite{KR16}.
	
	\begin{dfn}
		\index{topology!sub-Vietoris}
		Suppose $X$ is a topological space. Then by the {\em sub-Vietoris topology} we mean the topology on $\powerset(X)$ (i.e.\ on the family of all subsets of $X$), or on any subfamily of $\powerset(X)$, generated by subbasis of open sets of the form $\{A\subseteq X\mid A\cap F=\emptyset\}$ for $F\subseteq X$ closed.
		\xqed{\lozenge}
	\end{dfn}
	
	\begin{dfn}
		\index{weakly nonsmooth}
		\label{dfn:weak_nonsmt}
		Suppose $X$ is a topological space, while $E$ is an equivalence relation on $X$. We say that $E$ is \emph{weakly non-smooth} if there is a homeomorphic embedding $\psi\colon 2^{\omega}\to \powerset(X)$ (where $\powerset(X)$ is the power set of $X$, equipped with the sub-Vietoris topology) such that for any $\eta,\eta'\in 2^{\omega}$:
		\begin{enumerate}
			\item
			$\psi(\eta)$ is a nonempty closed set,
			\item
			if $\eta,\eta'$ are $\EZ$-related, then $[\psi(\eta)]_{E}=[\psi(\eta')]_{E}$,
			\item
			if $\eta,\eta'$ are distinct, then $\psi(\eta)\cap\psi(\eta')=\emptyset$,
			\item
			if $\eta,\eta'$ are not $\EZ$-related, then $(\psi(\eta)\times \psi(\eta'))\cap E=\emptyset$.\xqed{\lozenge}
		\end{enumerate}
	\end{dfn}
	
	\begin{rem}
		Note that if $E$ is a non-smooth Borel equivalence relation on a Polish space, then it is also weakly non-smooth: if $\psi'$ is a homeomorphic reduction of $\EZ$ to $E$ (given by Fact~\ref{fct:Harrington-Kechris-Louveau dichotomy}), then the formula $\psi(\eta)=\{\psi'(\eta)\}$ clearly satisfies the properties listed in Definition~\ref{dfn:weak_nonsmt}.\xqed{\lozenge}
	\end{rem}

	\begin{qu}
		\label{qu:broad_nonmetrisable}
		Suppose $(G, X, x_0)$ is an ambit, and $E$ is an equivalence relation on $X$ which is analytic and either weakly uniformly properly group-like or weakly closed group-like.
		
		Is $E$ closed if and only if $E$ is not weakly non-smooth?
	\end{qu}
	Note that weak non-smoothness immediately implies having at least $2^{\aleph_0}$ classes, so a positive answer to the question would imply Theorem~\ref{thm:general_cardinality_intransitive} for $Y=X$. It would also allow us to obtain a trichotomy similar to Corollary~\ref{cor:metr_smt_cls}, for non-metrisable $X$.
	
	Furthermore, applied in model-theoretic context, it would be (essentially) a generalization of \cite[Theorem 3.18]{KR16} -- which, in turn, is a generalization of \cite[Theorem 5.1]{KMS14} (see also \cite[Theorems 2.19, 3.19]{KM14}), a variant of Fact~\ref{fct:KMS_theorem} for uncountable languages. See Corollary~\ref{cor:nwg2} for an example of such application.
	
	As we will see in Proposition~\ref{prop:conj_converse}, a weakly non-smooth equivalence relation $E$ is \emph{not} closed, which gives us one direction.
	
	To show this, we first prove the following topological lemma.
	
	\begin{lem}
		\label{lem:subv_closed_relation}
		Let $X$ be a compact, Hausdorff space. Suppose $E$ is a binary relation on $X$. Write $\overline E$ for the relation on $2^X$ (the hyperspace of closed subsets of $X$) defined by
		\[
		K_1 \mathrel{\overline E} K_2 \iff \exists k_1\in K_1\exists k_2\in K_2\quad k_1\Er k_2
		\]
		Then, if $E$ is a closed relation, so is $\mathrel{\overline E}$ (on $2^X$ with the sub-Vietoris topology).
	\end{lem}
	\begin{proof}
		Choose an arbitrary net $(K_i,K'_i)_{i\in I}$ in $\mathrel{\overline E}$ converging to some $(K,K')$ in $2^X$. We need to show that $(K,K')\in \overline E$.
		
		Let $k_i\in K_i, k_i'\in K_i'$ be such that $k_i \Er k_i'$. By compactness, we can assume without loss of generality that $(k_i,k_i')$ converges to some $(k,k') \in E$ (as $E$ is closed). If $k\in K$ and $k'\in K'$, we are done.
		
		Let us assume towards contradiction that $k\notin K$. Then, since $K$ is closed, and $X$ is compact, Hausdorff (and thus regular), we can find disjoint open sets $U,V$ such that $K\subseteq U$ and $k\in V$. Then we can assume without loss of generality that all $k_i$ are in $V$ (passing to a subnet if necessary). We see that $F:=X\setminus U$ is a closed set such that $F\cap K=\emptyset$. But for all $i$ we have $k_i\in F\cap K_i$, which gives us a (sub-Vietoris) basic open set separating $K$ from all $K_i$, a contradiction; therefore, we must have $k\in K$.
		
		Similarly, it cannot be that $k'\notin K'$, which completes the proof.
	\end{proof}
	(In fact, the converse is also true, because the map $x\mapsto \{x\}$ is a homeomorphic embedding of $X$ into $2^X$ with the sub-Vietoris topology.)
	
	Without further ado, we can prove the aforementioned proposition.
	\begin{prop}
		\label{prop:conj_converse}
		If $E$ is a closed equivalence relation, then it is not weakly nonsmooth.
	\end{prop}
	\begin{proof}
		Suppose towards contradiction that $E$ is closed and weakly non-smooth, which is witnessed by some $\psi\colon 2^\omega\to \mathcal P(X)$. Denote by $\mathcal F$ the range of $\psi$.
		
		Since $\mathcal F$ consists of closed sets, by Lemma~\ref{lem:subv_closed_relation}, the restriction $\overline{E}\restr_{\mathcal F}$ is a closed relation. On the other hand, by the properties of $\psi$, for any $\eta_1,\eta_2 \in2^{\omega}$, $\eta_1 \EZ \eta_2 \iff \psi(\eta_1) \mathrel{\overline{E}} \psi(\eta_2)$. Since $\psi$ is a homeomorphism from $2^{\omega}$ to ${\mathcal F}$, we conclude that $\EZ$ is a closed relation which is a contradiction.
	\end{proof}
	
	Proposition~\ref{prop:conj_weak} below (along with Proposition~\ref{prop:conj_converse}) gives a positive answer to Question~\ref{qu:broad_nonmetrisable} in a weakened form, namely, we assume that $E$ is $F_\sigma$, we only require that $\psi$ is continuous (and not a homeomorphism), and we drop the property that $\psi$ takes distinct points to disjoint sets (which would imply that it is a homeomorphism, by Fact~\ref{fct:subVt} below).
	
	\begin{prop}
		\label{prop:conj_weak}
		Suppose $(G,X,x_0)$ is an ambit, while $E$ is an equivalence relation on $X$ which $F_\sigma$ and either weakly uniformly properly group-like or weakly closed group-like. Assume that $E$ is not closed.
		
		Then there is a continuous function $\phi\colon 2^\omega\to \powerset(X)$ (equipped with the sub-Vietoris topology), such that:
		\begin{itemize}
			\item
			$\phi(\eta)$ is a nonempty closed set,
			\item
			if $\eta,\eta'$ are $\EZ$-related, then $[\phi(\eta)]_{E}=[\phi(\eta')]_{E}$,
			\item
			if $\eta,\eta'$ are not $\EZ$-related, then $(\phi(\eta)\times \phi(\eta'))\cap E=\emptyset$.
		\end{itemize}
	\end{prop}
	
	Before the proof we need to recall a few facts and make some observations. The descriptive set theoretic tools which we use to prove the proposition are similar to those from \cite{KMS14} and \cite{KR16}.
	
	\begin{dfn}
		The {\em strong Choquet game} on a topological space $X$ is the following two-player game in $\omega$-rounds. In round $n$, player A chooses an open set
		$U_n \subseteq V_{n-1}$ and $x_n \in U_n$, and player B responds by choosing an open set $V_n \subseteq U_n$
		containing $x_n$. Player B wins when the intersection $\bigcap \{V_n \mid n < \omega\}$ is nonempty.
		
		A topological space X is a {\em strong Choquet space} if player B has a winning strategy in the strong Choquet game on $X$. For more details see Sections 8.C and 8.D of \cite[Chapter I]{Kec95}.
		
		Given a subset $C$ of $X$, we say that $X$ is {\em strong Choquet over $C$} to mean that the points that player A chooses are taken from $C$ (and player $B$ has a winning strategy in the modified game).
	\end{dfn}
	
	\begin{fct}
		A compact Hausdorff space is strong Choquet.
	\end{fct}
	\begin{proof}
		By compactness, if player $B$ chooses at step $n$ a $V_n$ such that $\overline{V_n}\subseteq U_n$, then he wins. A compact Hausdorff space is normal, so he can always do that.
	\end{proof}
	
	\begin{rem}
		Note that a strong Choquet space is trivially strong Choquet over each of its subsets.\xqed{\lozenge}
	\end{rem}

	As usual, given a set $X$ and a relation $R \subseteq X \times X$, and $x \in X$, by $R_x$ we denote the section of $R$ at $x$, i.e.\ $\{y \in X \mid x \mathrel{R} y \}$.

	\begin{fct}
		\label{fct:dtmtoolu}
		Suppose that $X$ is a regular topological space, $\langle R_n\mid n\in \omega\rangle$ is a sequence of $F_\sigma$ subsets of $X^2$, $\Sigma$ is a group of homeomorphisms of $X$, and $\mathcal O\subseteq X$ is an orbit of $\Sigma$ with the property that for all $n\in \omega$ and open sets $U\subseteq X$ intersecting $\mathcal O$, there are distinct $x,y\in\mathcal O\cap U$ with $\mathcal O\cap (R_n)_x\cap (R_n)_y=\emptyset$. If $X$ is strong Choquet over $\mathcal O$, then there is a function $\tilde\phi\colon 2^{<\omega}\to \powerset(X)$ such that for any $\eta\in 2^\omega$ and any $n\in \omega$:
		\begin{itemize}
			\item
			$\tilde\phi(\eta\restr n)$ is a nonempty open set,
			\item
			$\overline{\tilde\phi(\eta\restr{(n+1)})}\subseteq \tilde\phi(\eta\restr{n})$
		\end{itemize}
		Moreover, $\phi(\eta)=\bigcap_n \tilde\phi(\eta\restr n)=\bigcap_n \overline{\tilde\phi(\eta\restr n)}$ is a nonempty closed $G_\delta$ set such that for any $\eta,\eta'\in 2^\omega$ and $n\in\omega$:
		\begin{itemize}
			\item
			if $\eta \EZ \eta'$, then there is some $\sigma\in\Sigma$ such that $\sigma\cdot \phi(\eta)=\phi(\eta')$,
			\item
			if $\eta(n)\neq \eta'(n)$, then $(\phi(\eta)\times \phi(\eta'))\cap R_n=\emptyset$, and if $\eta,\eta'$ are not $\EZ$-related, then $(\phi(\eta)\times \phi(\eta'))\cap \bigcup R_n=\emptyset$.
		\end{itemize}
	\end{fct}
	\begin{proof}
		This is \cite[Theorem 3.14]{KR16}.
	\end{proof}

	\begin{fct}
		\label{fct:subVt}
		Suppose $X$ is a normal topological space (e.g.\ a compact, Hausdorff space) and $\mathcal A$ is any family of pairwise disjoint, nonempty closed subsets of $X$. Then $\mathcal A$ is Hausdorff with the sub-Vietoris topology.$\qed$
	\end{fct}
	\begin{proof}
		This is \cite[Proposition 3.16]{KR16}.
	\end{proof}
	
	Using the last two facts, we obtain a corollary reminiscent of \cite[Theorem 3.18]{KR16} (albeit topological group theoretic, and not model theoretic in nature), which will be used in the proof of Proposition~\ref{prop:conj_weak}.
	
	\begin{cor}
		\label{cor:Vietoris_embed}
		Suppose $G$ is a compact, Hausdorff group, while $H\leq G$ is $F_\sigma$ and not closed. Then there is a homeomorphic embedding $\phi\colon 2^\omega\to \powerset(G)$ (with the sub-Vietoris topology) such that for any $\eta,\eta'\in 2^\omega$:
		\begin{itemize}
			\item
			$\phi(\eta)$ is a nonempty closed set,
			\item
			if $\eta \EZ \eta'$, then there is some $h\in H$ such that $\phi(\eta)h=\phi(\eta')$,
			\item
			if $\eta\neq \eta'$, then $\phi(\eta)\cap \phi(\eta')=\emptyset$,
			\item
			if $\eta,\eta'$ are not $\EZ$-related, then $\phi(\eta)H\cap \phi(\eta')H=\emptyset$.
		\end{itemize}
		In particular, $[G:H]\geq 2^{\aleph_0}$.
	\end{cor}
	\begin{proof}
		We can assume without loss of generality that $H$ is dense in $G$ (by replacing $G$ with $\overline H$). Since $H$ has the Baire property (as an $F_\sigma$ subset of a compact space), by the Pettis theorem (i.e.\ Fact~\ref{fct:pettis}) it follows that $H$ is meagre in $G$ (because $H$ is not closed, and so not open). Therefore, since $H$ is $F_\sigma$ and closed meagre sets are nowhere dense, there are nonempty closed, nowhere dense sets $F_n\subseteq G$, $n \in \omega$, such that $H = \bigcup_n F_n$. We can assume without loss of generality that the $F_n$'s are symmetric (i.e.\ $F_n=F_n^{-1}$ and $e \in F_n$), increasing, and satisfy $F_nF_m\subseteq F_{n+m}$.
		
		$H$ acts by homeomorphisms on $G$ (by right translations by inverses). Let us denote by $R_n$ the preimage of $F_n$ by $(g_1,g_2)\mapsto g_1^{-1}g_2$. We intend to show that the assumptions of Fact~\ref{fct:dtmtoolu} are satisfied, with $X:=G$, $\mathcal O=\Sigma:=H$ and $R_n$ just defined.
		
		Since $G$ is compact Hausdorff, it is strong Choquet over $\mathcal O$ (even over itself) and regular. Fix any open set $U$ and any $n\in \omega$. Then pick any $h\in H\cap U$ (which exists by density). Then $h\in F_N$ for some $N\in \omega$.
		
		From the fact that $H$ is dense and the $F_m$'s are closed nowhere dense, it follows that for each $m$, $H\setminus F_m$ is dense, so we can find some $h'\in U\cap (H\setminus F_{2n+N})$. Since the $F_n$'s are increasing, we see that $h \ne h'$.
		Moreover, we have
		\[
		H\cap (R_n)_h\cap (R_n)_{h'}=H\cap hF_n \cap h'F_n \subseteq F_N F_n \cap h' F_n .
		\]
		But if this last set was nonempty, we would have $h'\in F_NF_nF_n^{-1}\subseteq F_{2n+N}$ -- which would contradict the choice of $h'$ -- so $H\cap (R_n)_h\cap(R_n)_{h'}=\emptyset$, and the assumptions of Fact~\ref{fct:dtmtoolu} are satisfied. This gives us the map $\phi$, which satisfies all the bullets, as well as the auxiliary map $\tilde{\phi}$. What is left is to show that $\phi$ is a homeomorphic embedding.
		
		$\phi$ is clearly injective by the third bullet, and by the preceding fact, the range of $\phi$ is a Hausdorff space, so we only need to show that it is continuous. To do that, consider a subbasic open set $U=\{F\mid F\cap K= \emptyset\}$, and notice that by compactness, $\phi(\eta)\in U$ if and only if $\overline{\tilde{\phi}(\eta\restr n)}\cap K= \emptyset$ for some $n$, which is an open condition about $\eta$.
	\end{proof}
	
	\begin{prop}
		\label{prop:subVcont}
		Consider a map $f\colon X\to Y$ between topological spaces and the induced image and preimage maps $\mathcal F\colon \powerset(X)\to \powerset(Y)$ and $\mathcal G\colon \powerset(Y)\to \powerset(X)$ (where $\powerset(X)$ and $\powerset(Y)$ are equipped with the sub-Vietoris topology). Then:
		\begin{itemize}
			\item
			If $f$ is continuous, so is $\mathcal F$.
			\item
			If $f$ is closed, $\mathcal G$ is continuous.
		\end{itemize}
		In particular, if $f$ is continuous, $Y$ is Hausdorff and $X$ is compact, then both $\mathcal F$ and $\mathcal G$ are continuous.\xqed{\lozenge}
	\end{prop}
	\begin{proof}
		For the first point, consider a subbasic open set $B=\{A \mid A\cap F=\emptyset \}\subseteq \powerset(Y)$. Then $\mathcal F^{-1}[B]=\{A \mid f[A]\cap F=\emptyset\}=\{A \mid A\cap f^{-1}[F]=\emptyset\}$ (this is because any $a\in A$ witnessing that $A$ is not in one of the sets will witness the same for the other). The third set is clearly open in $\powerset(X)$. The second point is analogous.
	\end{proof}

	\begin{proof}[Proof of Proposition~\ref{prop:conj_weak}]
		By Proposition~\ref{prop:strange_cont_pre}, we have a continuous function $\zeta_1\colon \overline{u\cM}\to u\cM/H(u\cM)$, given by $f\mapsto ufH(u\cM)$. Furthermore, since $E$ is weakly closed group-like or weakly uniformly properly group-like, by Lemma~\ref{lem:weakly_grouplike}(3), we have an action of $u\cM/H(u\cM)$ on $X/E$ and an orbit map $u\cM/H(u\cM)\to X/E$ which completes the following commutative diagram (similar to the diagram from the proof of Proposition~\ref{prop:from_cluM}):
		\begin{center}
			\begin{tikzcd}
			\overline{u\cM}\ar[r,"\zeta_1",two heads]\ar[d,"R"] & u\cM/H(u\cM)\ar[d,two heads] \\
			X\ar[r,two heads] & X/E.
			\end{tikzcd}
		\end{center}
		Commutativity is clear by the definition of all the maps involved (and the fact that they are well-defined): the action of $u\cM/H(u\cM)$ on $X/E$ is induced by the action of $E(G,X)$ (given by $f([x]_E)=[f(x)]_E$, cf.\ Lemma~\ref{lem:weakly_grouplike}(2)), so it is given by $fH(u\cM)([x]_E)=[f(x)]_E$. In particular, for every $f\in \overline{u\cM}$, we have $\zeta_1(f)[x_0]_E=[f(x_0)]_E=[R(f)]_E$, which means exactly that the diagram commutes.

		Let $H$ be the preimage of $[x_0]_E$ in $u\cM/H(u\cM)$. By Lemma~\ref{lem:new_preservation_E_to_H}, since $E$ is $F_\sigma$ and not closed, the same is true about $H$. Hence, Corollary~\ref{cor:Vietoris_embed} applies, giving us a function $\varphi'\colon 2^\omega\to \mathcal P(u\cM/H(u\cM))$ as there. Note that it witnesses weak non-smoothness of $E|_{u\cM/H(u\cM)}=E_H$ (the orbit equivalence relation of $H$). Now, using Proposition~\ref{prop:subVcont}, it is straightforward to check that $\varphi\colon 2^\omega\to \mathcal P(X)$ defined as $\varphi(\eta)=R[\zeta_1^{-1}[\varphi'(\eta)]]$ is continuous, has only closed sets in its image. Furthermore, since $E_H=E|_{u\cM/H(u\cM)}$, it is not hard to see that it satisfies the two ``reduction'' properties postulated in Proposition~\ref{prop:conj_weak} as well, which completes the proof.
	\end{proof}
	
	\begin{rem}
		It is not hard to see that if we were able, in Corollary~\ref{cor:Vietoris_embed}, to weaken the assumption that $H$ is $F_\sigma$ to say only that $H$ is analytic, then the same thing could be done in Proposition~\ref{prop:conj_weak}. However, it would still fall short of a positive answer to Question~\ref{qu:broad_nonmetrisable}, as we would not have the property of mapping distinct points onto disjoint sets.\xqed{\lozenge}
	\end{rem}
	
	The following corollary of Proposition~\ref{prop:conj_weak} is \cite[Proposition 5.12]{KPR15} (joint with Krzysztof Krupiński and Anand Pillay).
	
	\begin{cor}
		\label{cor:nwg2}
		Suppose we have $E$ is an $F_\sigma$ strong type on $X=p(\fC)$ for some $p\in S(\emptyset)$, while $Y\subseteq X$ is type-definable and $E$-saturated (i.e.\ it is a union of classes of $E$). Suppose moreover that $E$ is not type-definable.
		
		Then for every model $M$, there is a continuous function $\varphi\colon 2^{\omega}\to \powerset(Y_M)$ (where $\powerset(Y_M)$ is equipped with the sub-Vietoris topology) such that for any $\eta,\eta'\in 2^{\omega}$:
		\begin{itemize}
			\item
			$\varphi(\eta)$ is a nonempty closed set,
			\item
			if $\eta,\eta'$ are $\EZ$-related, then $[\varphi(\eta)]_{E^M}=[\varphi(\eta')]_{E^M}$,
			\item
			if $\eta,\eta'$ are not $\EZ$-related, then $(\varphi(\eta)\times \varphi(\eta'))\cap E^M=\emptyset$.
		\end{itemize}
	\end{cor}
	\begin{proof}
		Note that if $\varphi_M$ is as in the conclusion for a given model $M$, while $N$ is another model, then $\varphi_N$, defined as $\varphi_N(\eta)=\{\tp(a/N)\mid \tp(a/M)\in \varphi_M(\eta) \}$, witnesses the conclusion for $N$. Briefly, we have $\tp(a/M)\Er^M\tp(b/M)$ if and only if $\tp(a/N)\Er^N \tp(b/N)$; using that and Proposition~\ref{prop:subVcont}, it follows immediately that when $N\subseteq M$ or $N\subseteq M$, then $\varphi_N$ is as prescribed, and otherwise, we can argue the same in two steps, using a model $N'\supseteq M\cup N$. Thus, it is enough to find one model $M$ for which $\varphi$ exists.
		
		By Proposition~\ref{prop:closure_has_transitive_action}, may assume without loss of generality that $\Aut(\fC/\{Y\})$ acts transitively on $Y$ (if necessary, making $Y$ smaller).
		
		Fix any $a\in Y$. Then by Proposition~\ref{prop:amb_exist}, we can find a model $M$ containing $a$, ambitious relative to $G^Y=\Aut(\fC/\{Y\})/\Autf(\fC)$. Then by Lemma~\ref{lem:lascar_dominates}, $E^M\restr_{Y_M}$ is weakly uniformly properly group-like with respect to $G^Y(M)$. The conclusion follows immediately from Proposition~\ref{prop:conj_weak} applied to the ambit $(G^Y(M),Y_M,\tp(a/M))$ and the relation $E^M\restr_{Y_M}$.
	\end{proof}
	
	Since the conclusion of Corollary~\ref{cor:nwg2} easily implies that $E\restr_Y$ has at least as many classes as $\EZ$ (that is, at least $2^{\aleph_0}$ classes), it is a strengthening of Fact~\ref{fct:newelski} (in a different direction from Theorem~\ref{thm:nwg}: instead of weakening the assumptions about $E$, we have a stronger conclusion).
	
	Note that \cite[Theorem 3.18]{KR16} is a very similar to Corollary~\ref{cor:nwg2}. The difference is that the assumption strengthened that $E$ is an \emph{orbital} (in the sense of Definition~\ref{dfn:orbital_stype}) $F_\sigma$ equivalence relation, whereas the conclusion is strengthened to say also that $\varphi$ is a homeomorphic embedding, and maps distinct points to disjoint sets, but weakened  Thus, the main advantage of Corollary~\ref{cor:nwg2} lies in dropping the ``orbital" part of the assumption. (And, more vaguely, in maybe giving some hint how to proceed in the general case.) See also \cite[Theorem 5.1]{KMS14} for a related fact for $\equiv_\Lasc$.

	\appendix
	\chapter{Basic facts in topological dynamics}
	\chaptermark{Topological dynamics}
	\label{app:topdyn}
	In this section, we build the framework for topological dynamics in the generality needed in the thesis. The majority of the facts proven here are folklore, and their proofs are, for the most part, straightforward adaptations of the proofs from \cite{Gl76}, where the main focus is on flows of the form $(G,\beta G)$. Since I have not found them in the literature in the required generality, they are included in the thesis, with complete proofs, for the convenience of the reader (and possibly future reference).
	
	The definitions also essentially the same as in \cite{Gl76}; the only nontrivial change is the definition of the operation $\circ$ --- see Definition~\ref{dfn:circ} --- as the one introduced in \cite[Section IX.1]{Gl76} does not seem generalise to arbitrary Ellis groups (but the two definitions coincide in the case of $(G,\beta G)$, see Remark~\ref{rem:circ_equivalence}).
	\subsection*{Preparatory facts}

	\begin{dfn}
		\index{ideal}
		\index{left ideal}
		A (left) ideal $I\unlhd S$ in a semigroup $S$ is a subset such that $IS\subseteq I$.\xqed{\lozenge}
	\end{dfn}

	\begin{dfn}
		\index{semitopological group}
		A group $G$ endowed with a topology is a \emph{semitopological group} if multiplication is separately continuous in $G$.
		\xqed{\lozenge}
	\end{dfn}
	
	\begin{rem}
		\label{rem:semitop_acts_on_itself}
		Note that, as an immediate consequence of the definition, semitopological group acts on itself by homeomorphisms by left and right multiplication, as well as by conjugation. In particular, inner automorphisms are homeomorphisms.\xqed{\lozenge}
	\end{rem}
	
	\begin{dfn}
		\index{left topological semigroup}
		A semigroup $S$ equipped with a topology is called a \emph{left topological semigroup} if the multiplication is continuous on the left, i.e.\ for every $s_0\in S$, the map $s\mapsto ss_0$ is continuous.
		\xqed{\lozenge}
	\end{dfn}
	
	\begin{ex}
		For every topological space $X$, the semigroup of functions $X^X$ with pointwise convergence (i.e.\ Tychonoff product) topology is a left topological semigroup.
		\xqed{\lozenge}
	\end{ex}
	
	\begin{fct}[Ellis joint continuity theorem]
		\label{fct:ellis_joint_cont}
		Suppose $G$ is a locally compact Hausdorff semitopological group. Then $G$ is a topological group (i.e.\ multiplication is jointly continuous and inversion is continuous).
	\end{fct}
	\begin{proof}
		This is \cite[Theorem 2]{ellis57}.
	\end{proof}
	
	\begin{fct}
		\label{fct:elementary_algebra}
		If $S$ is a semigroup and $S$ has a left identity element and left inverses, then $S$ is a group.
	\end{fct}
	\begin{proof}
		Let $e$ be a left identity in $S$, and let $a\in S$ be arbitrary. If $b$ is the left inverse of $a$, it is enough to show that $ae=a$ and $ab=e$.
		
		Let $c$ be the left inverse of $b$, so that $ba=cb=e$.
		
		Then
		\[
			e=cb=c(eb)=c(ba)b=(cb)(ab)=e(ab)=ab,
		\]
		and thus
		\[
			ae=a(ba)=(ab)a=ea=a.\qedhere
		\]
	\end{proof}
	
	\begin{fct}
		\label{fct:minimal_ideals_idempotents}
		Suppose $S$ is a semigroup with a compact $T_1$ topology such that for any $s_0\in S$, the map $s\mapsto ss_0$ is continuous and a closed mapping (the latter follows immediately from continuity and compactness if $S$ is Hausdorff).
		
		Then there is a minimal (left) ideal $\cM$ in $S$ (i.e.\ a minimal set such that $S\cM=\cM$), and every such $\cM$ satisfies the following:
		\begin{enumerate}
			\item
			every $s\in \cM$ generates it as a left ideal: $\cM=Ss$ (thus, by the assumptions, $\cM$ is closed),
			\item
			\index{idempotent}
			if $u\in \cM$ is an idempotent (i.e.\ it satisfies $uu=u$), then $u\cM$ is a group (with composition as group operation and $u$ as identity),
			\item
			each $\cM$ is the disjoint union $\bigsqcup_{u} u\cM$, where $u$ ranges over idempotents $u\in \cM$; in particular, idempotents exist,
			\item
			for every idempotent $u\in \cM$ and every $s\in \cM$, we have that $su=s$,
			\item
			given two idempotents $u,v\in \cM$, the map $s\mapsto vs$ defines an isomorphism $u\cM\to v\cM$
			\item
			every two groups of the form $v\cN$ (where $\cN$ is a minimal left ideal in $S$ and $v$ is an idempotent in $\cN$) in $S$ are isomorphic as groups (even if they are contained in distinct ideals).
		\end{enumerate}
	\end{fct}
	\begin{proof}
		(The proof below is the same as in the case of compact Hausdorff left topological semigroups, which is classical.)
		
		Note that if $S'\subseteq S$ is a closed subsemigroup, then $S'$ also satisfies the hypotheses of the fact we are proving.
		
		For existence of minimal left ideals, notice that by the assumption, principal left ideals in $S$ (i.e.\ ideals of the form $Ss_0$ for some $s_0\in S$) are compact, so it is not hard to see that the family of all principal left ideals satisfies the assumptions of the Kuratowski-Zorn Lemma. This implies the existence of minimal principal left ideals. But every ideal contains a principal ideal, so the minimal principal left ideals are also minimal left ideals.
		
		Now, fix a minimal (left) ideal $\cM\unlhd S$.
		
		(1) follows easily from the minimality assumption: if $a\in \cM$, then $Sa\subseteq \cM$ is an ideal.
		
		(2): Clearly, $u\cM$ is a semigroup, and $u$ is a left identity. By Fact~\ref{fct:elementary_algebra}, it is enough to show that $u\cM$ has left inverses. But if $f\in u\cM$ is arbitrary, then by minimality of $\cM$, we have $\cM f=\cM$, so $u\cM f=u\cM$, so $u\in u\cM f$, i.e.\ for some $g\in \cM$ we have $(ug)f=u$.
		
		(3): Applying the Kuratowski's Lemma again, there is a minimal closed subsemigrup $K\leq \cM$. Pick any $u\in K$. Then $Ku\subseteq K$ is closed, and trivially, $Ku$ is a semigroup, so by minimality, $Ku=K$, so there is some $k\in K$ such that $ku=u$. But the set of all such $k$ is closed (by left continuity of multiplication and the $T_1$ assumption) and a subsemigroup of $K$, so it is just $K$ itself. In particular, $uu=u$, so $u$ is an idempotent.
		
		It follows that if $u\cM\cap v\cM\neq \emptyset$, then $u=v$ (indeed, if $f\in u\cM\cap v\cM$, then by (2) we have $f'\in \cM$ such that $ff'=u$, and it follows that $u\in v\cM f'=v\cM$; since $v\cM$ is a group, the only idempotent in it is $v$, so $u=v$).
		
		Finally, take any $m\in \cM$. We need to show that $m\in u\cM$ for some idempotent $u$. But the set of $f\in \cM$ such that $fm=m$ is a nonempty (because $\cM m=\cM$), closed subsemigroup of $\cM$, so it satisfies the assumptions of the fact we are proving, and so, by (3), it must contain an idempotent $u$. But then $m=um\in u\cM$.
		
		(4): Fix $u,s$. Then by minimality, $\cM=\cM u$, so there is some $s'\in \cM$ such that $s'u=s$. But then $su=(s'u)u=s'(uu)=s'u=s$.
		
		(5): Fix any idempotents $u,v\in \cM$. Then $uv=u$ and $vu=v$, so $s\mapsto us$, $v\cM\to u\cM$ and $s\mapsto vs$, $u\cM\to v\cM$ are inverse to one another, so they are bijections. Furthermore, it is easy to see that they are homomorphisms, as by the preceding point, for any $f,g\in \cM$ we have $ufug=ufg$ and $vfvg=vfg$.
		
		(6): By (5), it is enough to show that for every $u\in \cM$ and every minimal left ideal $\mathcal N$, there is an idempotent $v\in \mathcal N$ such that $v\mathcal N\cong u\cM$. First, note that $\mathcal N u$ is a left ideal contained in $\cM=Su$, so $\mathcal N u=\cM$, and there is some $f\in \mathcal N$ such that $fu=u$. The set of all such $f$ is a closed subsemigroup of $\mathcal N$, so it contains an idempotent $v$ by (3). Thus, we have $vu=u$. By analogous consideration, there is an idempotent $u'\in \cM$ such that $u'v=v$. But then $u=vu=(u'v)u=u'(vu)=u'u=u'$ (by (4)), so in fact $uv=v$. Similarly to (5), we conclude that $s\mapsto su$ yields a homomorphism $v\mathcal N\to u\cM$ with inverse $s\mapsto sv$.
	\end{proof}
	\index{u@$u$}
	\index{M@$\cM$}
	Throughout, we denote minimal ideals by $\cM$ or $\cN$ and idempotents in minimal ideals by $u$ or $v$.
	
	\begin{rem}
		\label{rem:explicit_ellisgroup_isomorphism}
		It is easy to see from the proof of Fact~\ref{fct:minimal_ideals_idempotents}(6), for arbitrary ideal groups $u\cM,v\cN$, we have an isomorphism $u\cM\to v\mathcal N$ given by $s\mapsto vsv$.\xqed{\lozenge}
	\end{rem}
	
	\begin{dfn}
		\index{ideal group}
		If $S$ is a semigroup as above, $\cM$ is a minimal left ideal, while $u\in \cM$ is an idempotent, we call $u\cM$ an \emph{ideal group} of $S$. By ``the" ideal group of $S$ we mean the unique isomorphism type of an ideal group of $S$.
		\xqed{\lozenge}
	\end{dfn}

	\begin{fct}
		Suppose $G$ is a semitopological group. If $A\subseteq G$, then $\overline{A}=\bigcap_V V^{-1}A$, where $V$ ranges over the neighbourhoods of the identity in $G$.
	\end{fct}
	\begin{proof}
		Fix any $V\ni e$. Let $x\in \overline{A}$. Then $Vx$ is a neighbourhood of $x$, so $Vx\cap A\neq \emptyset$, and thus $x\in V^{-1}A$.
		
		On the other hand, if $x\in V^{-1}A$ for every neighbourhood $V$ of $e$, then also $Vx\cap A\neq\emptyset$, so $V\cap Ax^{-1}\neq\emptyset$. Since $V$ was an arbitrary neighbourhood of $e$, by Remark~\ref{rem:semitop_acts_on_itself}, it follows that $e\in \overline{Ax^{-1}}=\overline{A} x^{-1}$, so $ex=x\in \overline{A}$.
	\end{proof}
	
	\begin{fct}
		\label{fct:semitop_T2_quot}
		\index{H(G)@$H(G)$}
		Suppose $G$ is a compact $T_1$ semitopological group. Then the derived subgroup $H(G):=\bigcap_V \overline V$ (where the intersection runs over all neighbourhoods of the identity in $G$) is a closed normal subgroup of $G$
		
		Furthermore, $G/H(G)$ is a compact Hausdorff topological group.
	\end{fct}
	\begin{proof}
		The proof is essentially given in \cite{Gl76} for the special case of the Ellis groups of a certain class of dynamical systems. We present the full proof for completeness.
		
		Recall that a semitopological group acts on itself by homeomorphisms on the left and on the right (Remark~\ref{rem:semitop_acts_on_itself}). For brevity, write $H$ for $H(G)$. It is clear by definition that $H$ is closed and contains the identity $e\in G$.
		
		\begin{clm*}
			If $V$ is an open neighbourhood of the identity $e\in G$, then $\overline{V}H\subseteq \overline{V}$.
		\end{clm*}
		\begin{clmproof}
			Choose any $v\in \overline{V}$. Then for any open $W\ni e$, we have $Wv\cap V\neq \emptyset$ (because $v\in Wv\cap \overline{V}$ and $Wv$ is open), so $wv\in V$ for some $w\in W$. But then by continuity, there is some open $U\ni e$ such that $wvU\subseteq V$, so $wv\overline{U}=\overline{wvU}\subseteq\overline V$ But by definition of $H$, we have $H\subseteq \overline{U}$, so $wvH\subseteq \overline{V}$, so $vH\subseteq w^{-1}\overline{V}\subseteq W^{-1}\overline{V}$. But since $v$ and $W$ were arbitrary, by the preceding fact, $\overline{V}H\subseteq \overline{\overline{V}}=\overline{V}$.
		\end{clmproof}
		By taking the intersection over all $V$ in the claim, we have that $H^2=HH\subseteq H$, so $H$ is a subsemigroup of $G$. It follows that for every $h\in H$, $hH$ is also a subsemigroup. Because multiplication by $h$ is a homeomorphism, $hH$ is also closed (and thus compact). Furthermore, since multiplication by any element of $G$ is a homeomorphism, $hH$ satisfies the assumptions of Fact~\ref{fct:minimal_ideals_idempotents}, so it contains an idempotent. But the only idempotent in $G$ is the identity, so $hH$ contains the identity, and hence $H$ must contain the inverse of $h$. Since $h$ was arbitrary, $H$ is a group.
		
		The fact that $H$ is normal follows immediately from the fact that inner automorphisms are homeomorphisms (cf.\ Remark~\ref{rem:semitop_acts_on_itself}) and they fix the identity.
		
		Compactness of $G/H$ is immediate. Separate continuity of multiplication in $G/H$ follows immediately from the separate continuity of multiplication in $G$. By the Ellis joint continuity theorem (Fact~\ref{fct:ellis_joint_cont}), it is enough to show that $G/H$ is Hausdorff.
		
		Let $fH,gH$ be distinct elements of $G/H$. Then $fg^{-1}\notin H$, so there is a neighbourhood $V$ of $e\in G$ such that $fg^{-1}\notin \overline V$, so in particular, there is another neighbourhood $W$ of $e\in G$ such that $Wfg^{-1}\cap \overline{V}=\emptyset$. By the claim, $\overline{V} H\subseteq \overline V$,
		so it follows that $Wfg^{-1}\cap \overline{V}H=\emptyset$. But --- since $H$ is normal --- this is equivalent to $Wf\cap \overline{V}gH=\emptyset$, and because $H$ is a subgroup, this is equivalent to $WfH\cap \overline{V}gH=\emptyset$, so in particular, we have $WfH\cap VgH=\emptyset$, which completes the proof.
	\end{proof}

	\begin{fct}
		A topological space $X$ is Hausdorff if and only if for all $x\in X$ we have that
		\[
		\bigcap_{V}\overline V=\{x\},
		\]
		where the intersection runs over the neighbourhoods of $x$ in $X$.
	\end{fct}
	\begin{proof}
		Straightforward.
	\end{proof}
	
	\begin{cor}
		\label{cor:T2_quotient}
		Suppose $f\colon X\to Y$ is a continuous map between topological spaces, where $Y$ is Hausdorff. Then for any $x\in X$ we have that $f$ is constant on $\bigcap_V{\overline V}$, where the intersection runs over all neighbourhoods of $x\in X$.
	\end{cor}
	\begin{proof}
		Since $Y$ is Hausdorff, we have that $\bigcap_W \overline W=\{f(x)\}$, where the intersection runs over all neighbourhoods of $f(x)$ in $Y$. But then
		\[
		f^{-1}[f(x)]=f^{-1}\bigg[\bigcap_W \overline W\bigg]=\bigcap_W f^{-1}[\overline W]\supseteq\bigcap_W \overline {f^{-1}[W]}\supseteq \bigcap_V{\overline V}\qedhere
		\]
	\end{proof}
	
	\begin{cor}
		\label{cor:H(G)_universal}
		If $G$ is a compact $T_1$ semitopological group, while $G'$ is a Hausdorff topological group, and $\varphi\colon G\to G'$ is a continuous homomorphism, then $\varphi$ factors through $G/H(G)$ (where $H(G)$ is the derived subgroup of $G$).
	\end{cor}
	\begin{proof}
		This follows immediately from Corollary~\ref{cor:T2_quotient}.
	\end{proof}
	
	\subsection*{Dynamical systems}
	\begin{dfn}
		\label{dfn:dyn_syst}
		\index{G-flow@$G$-flow}
		\index{dynamical system}
		A {\em $G$-flow} or a \emph{dynamical system} is a pair $(G,X)$, where $G$ is a topological group acting continuously on a compact, Hausdorff space $X$.
		
		\index{G-ambit@$G$-ambit}
		A \emph{$G$-ambit} is a triple $(G,X,x_0)$ such that $(G,X)$ is a $G$-flow and $x_0\in X$ is a point with dense $G$-orbit.
		\xqed{\lozenge}
	\end{dfn}
	\begin{rem}
		In this thesis, most of the time, we do not care about the topology on $G$, so we can think of it as a discrete group acting on $X$ by homeomorphisms.
		\xqed{\lozenge}
	\end{rem}

	\begin{dfn}
		\index{Ellis!semigroup}
		\index{pig@$\pi_g$}
		\label{dfn:ellis_semigroup}
		The \emph{Ellis} or \emph{enveloping} semigroup of the flow $(G,X)$, denoted by $E(G,X)$, is the closure of the collection of functions $\{\pi_g \mid g \in G\}$ (where $\pi_g: X \to X$ is given by $\pi_g(x)=gx$) in the space $X^X$ equipped with the product topology, with composition as the semigroup operation.
		\xqed{\lozenge}
	\end{dfn}
	\begin{rem}
		When there is little risk of confusion, we sometimes abuse the notation and write $g$ instead of $\pi_g$ (and similarly, for $A\subseteq E(G,X)$, we write $A\cap G$ for the set of $g\in G$ such that $\pi_g\in A$).
		
		\index{piXg@$\pi_{X,g}$}
		If, on the contrary, we have more than one $G$-flow around, we add suitable indices, e.g.\ we have $(G,X)$ and $(G,Y)$, we write $\pi_{X,g}$ and $\pi_{Y,g}$ for the appropriate multiplication functions.
		
		When $(G,X)$ is fixed, we frequently write $EL$ instead of $E(G,X)$.
		\xqed{\lozenge}
	\end{rem}
	
	\begin{fct}
		The Ellis semigroup $E(G,X)$ is a compact Hausdorff left topological semigroup, i.e.\ given any $f_0\in E(G,X)$, the function $f\mapsto ff_0$ (the composition) is continuous.
	\end{fct}
	\begin{proof}
		Since $X$ is compact Hausdorff, so is $X^X$. Function composition is trivially left continuous, so $X^X$ is a left topological semigroup. Since $G$ acts on $X$ (and thus also, coordinatewise, on $X^X$) by homeomorphisms, we see that $GE(G,X)=E(G,X)$, and hence (by left continuity of composition and the fact that $E(G,X)$ is the closure of $G\subseteq X^X$) also $E(G,X)E(G,X)=E(G,X)$, so $E(G,X)$ is a closed subsemigroup of $X^X$, which completes the proof.
	\end{proof}
	As an immediate consequence, we obtain the following:
	\begin{fct}
		\label{fct:idempotents_ideals_Ellis}
		There is a minimal (left) ideal $\cM$ in $E(G,X)$ (i.e.\ a minimal set such that $E(G,X)\cM=\cM$). Every such $\cM$ satisfies the following:
		\begin{enumerate}
			\item
			every $s\in \cM$ generates it as a left ideal: $\cM=E(G,X)s$,
			\item
			if $u\in \cM$ is an idempotent (i.e.\ $uu=u$), then $u\cM$ is a group (with composition as group operation and $u$ as identity),
			\item
			each $\cM$ is the disjoint union $\bigsqcup_{u} u\cM$, where $u$ ranges over idempotents $u\in \cM$; in particular, idempotents exist,
			\item
			for every idempotent $u\in \cM$ and every $s\in \cM$, we have that $su=s$,
			\item
			given two idempotents $u,v\in \cM$, the map $s\mapsto vs$ defines an isomorphism $u\cM\to v\cM$,
			\item
			every two ideal groups in $E(G,X)$ are isomorphic as groups.
		\end{enumerate}
	\end{fct}
	\begin{proof}
		Immediate by the Fact~\ref{fct:minimal_ideals_idempotents}.
	\end{proof}

	\begin{dfn}
		\label{dfn:circ}
		\index{o@$\circ$}
		For each $a\in EL$, $B\subseteq EL$, we write $a\circ B$ for the set of all limits of nets $(g_ib_i)_i$, where $g_i\to a$ (by which we mean, abusing the notation, that $\pi_{g_i}\to a$, where $\pi_g\colon X\to X$ is defined by $x\mapsto g\cdot x$), $g_i\in G$ and $b_i\in B$.\xqed{\lozenge}
	\end{dfn}
	\begin{rem}
		\label{rem:circ_equivalence}
		In \cite{Gl76}, the author uses a different definition of $\circ$, which, however, does not appear to work in sufficient generality. However, in the case of $(G,\beta G)$ considered there, the two definitions are equivalent, by \cite[Lemma 1.1(1), \S IX.1.]{Gl76}.
		\xqed{\lozenge}
	\end{rem}
	The following proposition gives a useful description of the $\circ$ operation, which frequently allows us to avoid cumbersome calculations with nets.
	\begin{prop}
		\label{prop:circ_description}
		For any $a,b\in E(G,X)$ and $C\subseteq E(G,X)$, we have that $b\in a\circ C$ if and only if for every open $U\ni a$ and $V\ni b$, we have some $g\in G$ and $c\in C$ such that $g\in U$ an $gc\in V$ (equivalently, $c\in g^{-1}V$).
	\end{prop}
	\begin{proof}
		It is clear that if $b\in a\circ C$, then we can find the appropriate $g$ and $c$.
		
		In the other direction: take a directed set consisting of pairs $(U,V)$ of neighbourhoods of $a$ and $b$, respectively, ordered by reverse inclusion (separately on each coordinate), and for each $(U,V)$ take $g_{U,V}$ and $c_{(U,V)}$ as in the assumption. Then $g_{(U,V)}\to a$ and $g_{(U,V)}c_{(U,V)}\to b$.
	\end{proof}
	
	\begin{fct}
		\label{fct:circ_calculations}
		For any $B\subseteq EL$ and $a,b\in EL$, we have:
		\begin{enumerate}
			\item
			$(a\circ B)c=a\circ (Bc)$,
			\item
			$a\circ(b\circ B)\subseteq (ab)\circ B$,
			\item
			$aB\subseteq a\circ B$,
			\item
			$a\circ(B\cup C)=(a\circ B)\cup (a\circ C)$,
			\item
			$a\circ(bC)\subseteq (ab)\circ C$ and $a(b\circ C)\subseteq (ab)\circ C$.
		\end{enumerate}
	\end{fct}
	\begin{proof}
		(1) is easy by left continuity: if $g_i\to a$ and $b_i\in B$, then $b_ic\in Bc$ and $(\lim g_ib_i)c=\lim(g_ib_ic)$, provided the limit on the left exists, which can be forced by compactness (and passing to a subnet).
		
		For (2), take any $f\in (a\circ (b\circ B))$ and take any neighbourhoods $U_f$ of $f$ and $U_{ab}$ of $ab$. Since $ab\in U_{ab}$, by left continuity, there is a neighbourhood $U_a$ of $a$ such that for all $a'\in U_a$ we have $a'b\in U_{ab}$.
		
		Now, since $f\in a\circ (b\circ B)$, by applying Proposition~\ref{prop:circ_description} to it, we can find some $g_a\in U_a\cap G$ and $b'\in b\circ B$ such that $g_a\in U_{a}$ and $b'\in g_{a}^{-1}U_f$. Then $g_ab\in U_{ab}$, so also $b\in U_b:=g_a^{-1}U_{ab}$.
		
		Applying Proposition~\ref{prop:circ_description} for $b'\in b\circ B$, we get $g_b\in U_b\cap G$ and $b''\in B$ such that $g_bb''\in g_a^{-1}U_f$. Thus $g_ag_bb''\in U_f$ and $g_ag_b\in g_aU_b\cap G=U_{ab}\cap G$.
		
		Since $U_{ab}$ and $U_f$ were arbitrary, by Proposition~\ref{prop:circ_description} for $(ab)\circ B$ (in the opposite direction), we are done.
		
		For (3), just take constant net in $B$ and any net $(g_i)_i$ in $G$ converging to $a$.
		
		For (4), just note that of $(a_i)_i$ is a net in $B\cup C$, then we have a cofinal subnet contained in one of $B$ and $C$.
		
		For (5), just notice that by (3), $bC\subseteq b\circ C$, and by (2), $a\circ(b\circ C))\subseteq (ab)\circ C$, and likewise $a(b\circ C)\subseteq a\circ (b\circ C)\subseteq (ab)\circ C$.
	\end{proof}
	
	\begin{fct}
		\label{fct:circ_with_closure}
		$a\circ B=a\circ \overline B$.
	\end{fct}
	\begin{proof}
		$\subseteq$ is clear. For the opposite inclusion, take any $b\in a\circ \overline{B}$ and any neighbourhoods $U_a\ni a$ and $U_b\ni b$. By the assumption and Proposition~\ref{prop:circ_description}, we have $g_a\in U_a\cap G$ and $b'\in \overline{B}$ such that $g_ab'\in U_b$, i.e.\ $b'\in g_a^{-1}U_b$. But since $b'\in \overline{B}$, we can find some $b''\in B\cap g_a^{-1}U_b$. Since $U_a,U_b$ were arbitrary, we are done.
	\end{proof}
	
	\begin{fct}
		\label{fct:circ_closed}
		$a\circ B$ is closed.
	\end{fct}
	\begin{proof}
		Since every neighbourhood of an $f\in \overline{a\circ B}$ is also a neighbourhood of some $f'\in{a\circ B}$, the fact follows immediately by Proposition~\ref{prop:circ_description}.
	\end{proof}
	
	\begin{fct}
		\label{fct:circ_stays_in_ideal}
		For any $a\in EL$, any closed left ideal $I\unlhd EL$ (e.g.\ any minimal left ideal) and $B\subseteq I$ we have $a\circ B\subseteq I$.
	\end{fct}
	\begin{proof}
		Since $B\subseteq I$, for any $g_i\in G$, $b_i\in B$ we have $g_ib_i=\pi_{g_i}b_i\in ELb_i\subseteq ELI\subseteq I$. Because $I$ is closed, the result follows.
	\end{proof}
	
	\begin{fct}
		\label{fct:circ_with_idemp}
		If $v\in EL$ is any idempotent, then for any $a\in EL$ and $B\subseteq EL$ we have $v(a\circ B)\subseteq v((va)\circ B)$.
	\end{fct}
	\begin{proof}
		First, since $v$ is an idempotent, $v(a\circ B)=v(v(a\circ B))$. By Fact~\ref{fct:circ_calculations}(5), $v(a\circ B)\subseteq (va)\circ B$, so $v(a\circ B)\subseteq v((va)\circ B)$.
	\end{proof}
	
	\begin{fct}
		\label{fct:tau_closure}
		\index{clt@$\cl_{\tau}$}
		Given a minimal left ideal $\cM\unlhd E(G,X)$ and an idempotent $u\in \cM$, $\cl_\tau(A):=(u\cM)\cap (u\circ A)$ is a closure operator on $u\cM$.
	\end{fct}
	\begin{proof}
		The identity $\cl_{\tau}(\emptyset)=\emptyset$ is trivial by the definition of $\circ$.
		
		Idempotence follows from Fact~\ref{fct:circ_calculations}(2), as $u\circ(u\circ A)\subseteq (u^2)\circ A=u\circ A$.
		
		Likewise, extensivity follows easily by Fact~\ref{fct:circ_calculations}(3): if $A\subseteq u\cM$, then $A=uA\subseteq u\circ A$.
		
		Finally, additivity is immediate by Fact~\ref{fct:circ_calculations}(4).
	\end{proof}
	
	\begin{dfn}
		\label{dfn:tau_topology}
		\index{topology!t@$\tau$}
		By the \emph{$\tau$ topology} we mean the topology on $u\cM$ given by the closure operator $\cl_\tau$ from Fact~\ref{fct:tau_closure}.
		\xqed{\lozenge}
	\end{dfn}
	The following fact gives us another description of the operator $\cl_\tau$.
	\begin{fct}
		\label{fct:taucl_alt}
		$\cl_\tau(A)=u(u\circ A)$
	\end{fct}
	\begin{proof}
		By Fact~\ref{fct:circ_calculations}(5), $u(u\circ A)\subseteq (uu)\circ A=u\circ A$, and by Fact~\ref{fct:circ_stays_in_ideal}, $u\circ A\subseteq \cM$, so $u(u\circ A)\subseteq u\cM$. Hence $u(u\circ A)\subseteq (u\cM)\cap u\circ A$. The other inclusion is trivial.
	\end{proof}
	
	When $A\subseteq u\cM$, when we write $\overline{A}$, we always mean the closure of $A$ in $EL$. The closure in the $\tau$ topology we write as $\cl_\tau(A)$, or explicitly as $u\cM\cap(u\circ A)$ or ---  via Fact~\ref{fct:taucl_alt} --- as $u(u\circ A)$.
	
	\begin{fct}
		\label{fct:tau_coarser}
		For all $A\subseteq u\cM$ we have $\overline A\subseteq u\circ A$. In particular, the $\tau$ topology on $u\cM$ is coarser than the subspace topology inherited from $EL$.
	\end{fct}
	\begin{proof}
		We have $A=uA\subseteq u\circ A$ and $u\circ A$ is closed by Fact~\ref{fct:circ_closed}.
	\end{proof}
	
	\begin{rem}
		Note that the $\tau$ topology is not necessarily Hausdorff, so limits of nets may not be unique. In particular, if $(a_i)_i$ is a net in $u\cM$ converging to some $a$ in $EL$, and $a'\in u\cM$ is distinct from $ua$, then even though $a_i\xrightarrow{\tau} ua\neq a'$, we may have $a_i\xrightarrow{\tau} a'$. This makes some facts more difficult to prove than it may appear at first.
		\xqed{\lozenge}
	\end{rem}
	
	\begin{fct}
		\label{fct:ulimit}
		If $(a_i)_i$ is a net in $u\cM$ converging to $a\in \overline{u\cM}$, then $(a_i)_i$ converges to $ua$ in the $\tau$-topology.
	\end{fct}
	\begin{proof}
		Suppose that $(a_i)$ is a net in $u\cM$. Note that for any $i_0$ we have $ua_{\geq i_0}=a_{\geq i_0}$. Since by definition $a\in \overline {a_{\geq i_0}}$ and by Fact~\ref{fct:tau_coarser} $\overline {a_{\geq i_0}}\subseteq u\circ a_{\geq i_0}$, it follows that $ua\in u(u\circ a_{\geq i_0})=\cl_{\tau}(a_{\geq i_0})$, so (because $i_0$ is arbitrary) $ua$ is a $\tau$-limit of a subnet of $(a_i)_i$. Since the same is true about any subnet of $(a_i)_i$, $ua$ is a $\tau$-limit of $(a_i)_i$.
	\end{proof}

	\begin{fct}
		\label{fct:tau_T1}
		$u\cM$ with the $\tau$ topology is a compact $T_1$ semitopological group (i.e.\ multiplication is separately continuous).
	\end{fct}
	\begin{proof}
		$T_1$ is immediate by the left continuity of multiplication in $EL$. Compactness follows from Fact~\ref{fct:ulimit} and compactness of $\overline{u\cM}$: given any net in $u\cM$, we can find a subnet convergent in $\overline{u\cM}$, and this subnet will be $\tau$-convergent. What is left is to show separate continuity of multiplication.
		
		Choose any $a\in u\cM$ and a $\tau$-closed $B\subseteq u\cM$. Then Facts~\ref{fct:circ_calculations} and \ref{fct:taucl_alt} yield
		\[
		\cl_\tau (Ba)= u(u\circ Ba)=u((u\circ B)a)=(u(u\circ B))a=Ba,
		\]
		which gives us continuity on the left (note that $Ba$ is the preimage of $B$ by the multiplication on the right by $a^{-1}$).
		
		For the right continuity, note that $b\in (a\circ B)\cap u\cM$ and Fact~\ref{fct:circ_calculations}(5) imply $a^{-1}b\in a^{-1}(a\circ B)\subseteq (a^{-1}a)\circ B=u\circ B$, so $a^{-1}b\in B$ (because $a^{-1}b\in u\cM$ and $B$ is $\tau$-closed), whence $b\in aB$, so $(a\circ B)\cap u\cM\subseteq aB$. In particular, by Fact~\ref{fct:circ_calculations}(5) and the assumption that $a\in u\cM$,
		\[
		\cl_\tau (aB)=(u\circ(aB))\cap u\cM\subseteq ((ua)\circ B)\cap u\cM=(a\circ B)\cap u\cM\subseteq aB.\qedhere
		\]
	\end{proof}
	
	\begin{fct}
		\label{fct:ellis_isomorphic}
		All ideal groups of $E(G,X)$ are isomorphic as semitopological groups.
	\end{fct}
	\begin{proof}
		First, consider the case where the groups are contained in a single left ideal $\cM$. As in the proof of Fact~\ref{fct:minimal_ideals_idempotents}(6), it is enough to show that for any idempotents $u,v\in \cM$, the map $s\mapsto vs$, $u\cM\to v\cM$ is closed in $\tau$ topology (then by the same token, the inverse map will also be closed, so it will be a topological isomorphism). In other words, we need to show that if $A\subseteq u\cM$ is $\tau$-closed, then so is $vA$. But since $vu=v$ and $uv=u$, we have (using Fact~\ref{fct:circ_calculations}(5)):
		\begin{multline*}
			\cl_{\tau}(vA)=v(v\circ(vA)))=vu(v\circ (vuA))=v(u(v\circ (v(uA))))\subseteq\\\subseteq v((uvv)\circ (uA))=v(u(u\circ(uA)))=v\cl_{\tau}(A)=vA
		\end{multline*}
		
		Likewise (having in mind the proof of Fact~\ref{fct:minimal_ideals_idempotents}), for varying ideals, it is enough to show that if $u\in \cM$ and $v\in \mathcal N$ are idempotents in minimal left ideals $\cM$ and $\mathcal N$ such that $uv=v$ and $vu=u$, then $s\mapsto sv$ is closed. We have a similar calculation, for $\tau$-closed $A\subseteq u\cM$ (again, using Fact~\ref{fct:circ_calculations}):
		\begin{equation*}
			\begin{multlined}
				\cl_{\tau}(Av)=v(v\circ Av)=uv(v\circ A)v\subseteq  v((vv)\circ A)v = u(v\circ A)v=\\=u(v\circ (uA))v\subseteq u((vu)\circ A)v=u(u\circ A)v=\cl_{\tau}(A)v=Av
			\end{multlined}\qedhere
		\end{equation*}
	\end{proof}
	
	\begin{rem}
		\label{rem:explicit_ellisgroup_isomorphism_2}
		It is not hard to see from the proof of Fact~\ref{fct:ellis_isomorphic} that the isomorphism described by Remark~\ref{rem:explicit_ellisgroup_isomorphism} is actually a homeomorphism.\xqed{\lozenge}
	\end{rem}
	
	\begin{dfn}
		\index{Ellis!group}
		\label{dfn:ellis_group}
		The \emph{Ellis group of $(G,X)$} is the (unique isomorphism type of an) ideal group $u\cM$ of $E(G,X)$.
		\xqed{\lozenge}
	\end{dfn}
	
	\begin{fct}
		\index{uM/H(uM)@$u\cM/H(u\cM)$}
		\index{H(uM)@$H(u\cM)$}
		\label{fct:H(uM)}
		$H(u\cM)=\bigcap_V \cl_\tau(V)$, where $V$ runs over the $\tau$-open neighbourhoods of $u$ in $u\cM$ is a ($\tau$-)closed normal subgroup of $u\cM$, and $u\cM/H(u\cM)$ is a compact Hausdorff group.
	\end{fct}
	\begin{proof}
		By Fact~\ref{fct:tau_T1}, we can apply Fact~\ref{fct:semitop_T2_quot}, which finishes the proof.
	\end{proof}
	
	The next proposition is Lemma 4.1 from \cite{KPR15} (joint with Krzysztof Krupiński and Anand Pillay).
	
	\begin{figure}[H]
		\centering
		\includegraphics[width={\textwidth}]{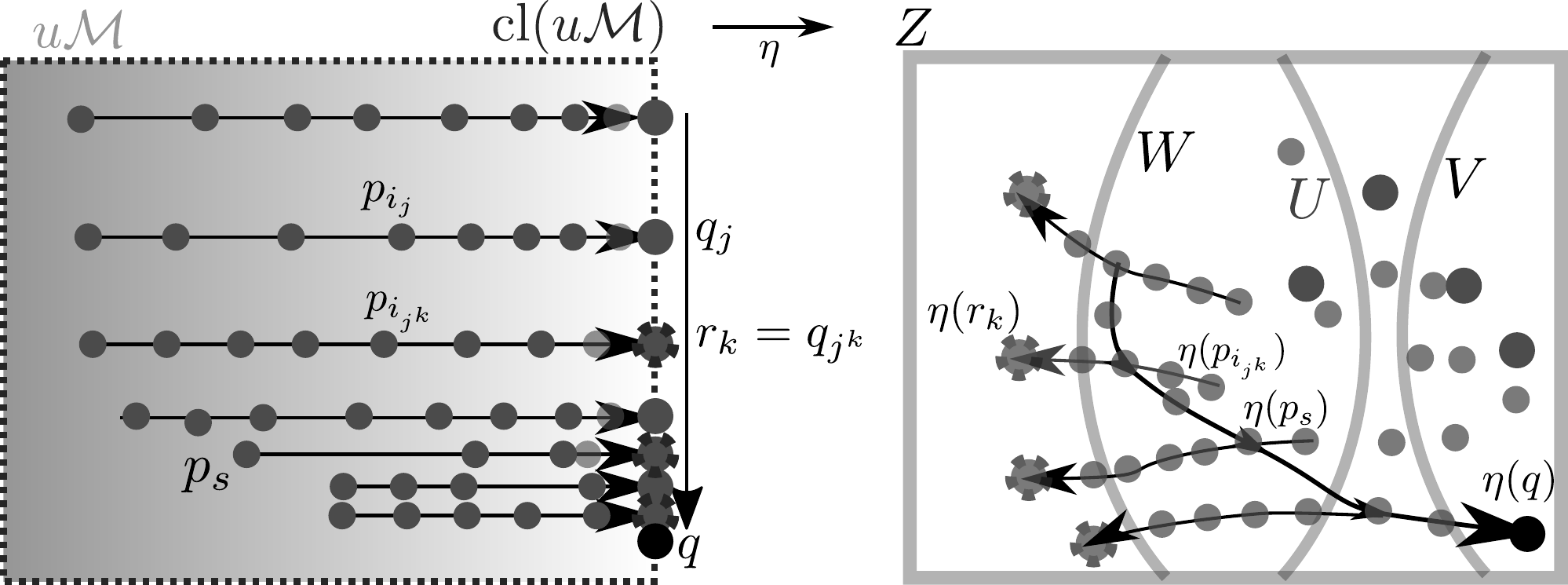}
		{Illustration of the nets described in the following proof.}
	\end{figure}
	
	\begin{prop}
		\label{prop:strange_cont}
		Let $\zeta\colon \overline{u\cM} \to u\cM $ be the function defined by $\zeta(x)=ux$ and let $\xi\colon u\cM \to Z$ be a continuous function, where $Z$ is a regular (e.g.\ compact, Hausdorff) space and $u\cM $ is equipped with the $\tau$-topology. Then $\xi \circ \zeta:\overline{u\cM} \to Z$ is continuous, where $\overline{u\cM}$ is equipped with the topology induced from the Ellis semigroup $EL$.
		
		In particular, the function $\overline{u\cM}\to u\cM/H(u\cM)$, $f\mapsto ufH(u\cM)$ is continuous.
	\end{prop}
	\begin{proof}
		Denote $\xi \circ \zeta$ by $\eta$.
		By Fact~\ref{fct:ulimit}, we know that for any net $(p_i)_i$ in $u\cM $ and $ p \in \overline{u\cM}$ such that $\lim p_i = p$ one has $\tau\mbox{-}\!\lim p_i=up$. So, in such a situation, $\eta(p)=\xi(up)=\lim_i \xi(p_i)=\lim_i \xi(up_i)=\lim_i \eta(p_i)$.
		
		Consider any net $(q_j)_{j\in J}$ in $\overline{u\cM}$ converging to $q$ in $\overline{u\cM}$. The goal is to show that $\lim_j \eta(q_j)=\eta(q)$. Suppose for a contradiction that there is an open neighbourhood $W$ of $\eta(q)$ and a subnet $(r_k)$ of $(q_j)$ such that all points $\eta(r_k)$ belong to $W^c$. Since $Z$ is regular, we can find open subsets $U$ and $V$ such that $W^c \subseteq U$, $\eta(q) \in V$ and $U \cap V=\emptyset$.
		
		For each $j$ we can choose a net $(p_{i^j})_{i^j \in I^j}$ in $u\cM $ such that $\lim_{i^j} p_{i^j}=q_j$.
		
		For each $k$, $r_k = q_{j_{k}}$ for some $j_k \in J$, and $\eta(r_k) \in U$. Hence, since by the first paragraph of the proof $\eta(r_k)=\eta(q_{j_k})= \lim_{i^{j_k}} \eta(p_{i^{j_k}})$, we see that for big enough $i^{j_k} \in I^{j_k}$ one has $\eta(p_{i^{j_k}}) \in U$.
		
		On the other hand, let $S:=J \times \prod_{j \in J} I^j$ be equipped with the product order. For $s \in S$, put $p_s:= p_{i^{j_s}}$, where $j_s$ is the first coordinate of $s$ and $i^{j_s}$ is the $j_s$-coordinate of $s$. Since $\lim_{j \in J} q_j=q$ and $\lim_{i^j \in I^j} p_{i^j}=q_j$, we get $\lim_s p_s=q$. So, by the first paragraph of the proof, $\lim_s \eta(p_s)=\eta(q) $, and hence, for $s\in S$ big enough, $\eta(p_s) \in V$.
		
		By the last two paragraphs, we can find $j \in J$ and $i^j \in I^j$ (big enough) so that $\eta(p_{i^j}) \in U \cap V$, a contradiction, as the last set is empty.
	\end{proof}

	\chapter{Side results}
	\label{app:side}
	This appendix contains various results which have turned up in the context of the thesis, but remain tangential to the main results.
	
	\section[On the existence of a semigroup structure on the type space \texorpdfstring{${S_{\bar c}(\mathfrak{C})}$}{Sc(C)}]{On the existence of a semigroup structure on the type space \texorpdfstring{${S_{\bar c}(\mathfrak{C})}$}{Sc(C)}}
	\sectionmark{Semigroup structure on ${S_{{\bar c}}(\mathfrak{C})}$}
	\label{section: semigroup operation}
	This section was originally the appendix in \cite{KPR15} (joint with Krzysztof Krupiński and Anand Pillay). The original proof of Corollary~\ref{cor:smt_type} in that paper used what we may consider an application of Lemma~\ref{lem:weakly_grouplike} to the case of the $\Aut(\fC)$-ambit $(\Aut(\fC),S_{\bar c}(\fC),\tp(\bar c/\fC))$, where $\fC$ is the monster model and $\bar c$ is its enumeration. In particular, it used the enveloping semigroup $E(\Aut(\fC),S_{\bar c}(\fC))$. In many natural cases, the enveloping semigroup of a dynamical system is naturally isomorphic to that system itself. For example, if we consider $(G,\beta G)$, then $E(G,\beta G)\cong \beta G$.
	
	The question that arises is whether or not $S_{\bar c}(\fC)$ is its own Ellis semigroup, or in other words, whether it admits a left topological semigroup structure (compatible with the action of $\Aut(\fC)$). In this section, we show that the answer is no, unless the underlying theory is stable (in which case, the answer is yes).
	
	The general idea is as follows. We establish the inclusion of $\Aut(\fC)$ in the type space $S_{\bar c}(\fC)$ as a universal object in a certain category. This allows us to describe the existence of a semigroup operation on $S_{\bar c}(\fC)$ in terms of a ``definability of types'' kind of statement, which in turn can be related to stability using a type counting argument.
	
	\begin{prop}\label{proposition: category C}
		Consider $\Aut(\fC)\subseteq S_{\bar c}(\fC)$ given by $\sigma \mapsto \tp(\sigma(\bar c)/\fC)$.
		Consider the category $\catg$ whose objects are maps $\Aut(\fC)\to K$ such that:
		\begin{itemize}
			\item
			$K$ is a compact, zero-dimensional, Hausdorff space,
			\item
			preimages of clopen sets in $K$ are relatively $\fC$-definable in $\Aut(\fC)$, i.e.\ for each clopen $C$ there is a formula $\varphi(x,a)$ with $a$ from $\fC$ such that $\sigma$ is in the preimage of $C$ if and only if $\models \varphi(\sigma(\bar c),a)$,
		\end{itemize}
		where morphisms are continuous maps between target spaces with the obvious commutativity property.
		Then the inclusion of $\Aut(\fC)$ into the space $S_{\bar c}(\fC)$ is the initial object of $\catg$.
	\end{prop}
	\begin{proof}
		Firstly, $\Aut(\fC)$ is dense in $S_{\bar c}(\fC)$, so the uniqueness part of the universal property is immediate. What is left to show is that for every $h\colon \Aut(\fC)\to K$, $h\in \catg$, we can find a continuous map $\bar h\colon S_{\bar c}(\fC)\to K$ extending $h$.
		
		Choose any $p\in S_{\bar c}(\fC)$ and consider it as an ultrafilter
		on relatively $\fC$-definable subsets of $\Aut(\fC)$, and then consider $K_p:=\bigcap \left\{\overline{h[D]}\mid D\in p\right\}\subseteq K$.
		It is the intersection of a centered (i.e.\ with the finite intersection property) family of nonempty, closed subsets of $K$, so it is nonempty. In fact, it is a singleton.
		If not, there are two distinct elements $k_1,k_2 \in K_p$. Take a clopen neighbourhood $U$ of $k_1$ such that $k_2 \notin U$. Since $h \in \catg$, $h^{-1}[U]= \{\sigma \in \Aut(\fC) \mid \;\models \varphi(\sigma(\bar c),a)\}$ for some formula $\varphi(\bar x, a)$. If $\varphi(\bar x, a) \in p$, then $K_p \subseteq \overline{h[h^{-1}[U]]} \subseteq \overline{U}=U$, a contradiction as $k_2 \notin U$. If $\neg \varphi(\bar x, a) \in p$, then $K_p \subseteq \overline{h[\Aut(\fC) \setminus h^{-1}[U]]} \subseteq \overline{K \setminus U}=K \setminus U$, a contradiction as $k_1 \notin K \setminus U$.
		In conclusion, we can define $\bar h(p)$ to be the unique point in $K_p$.
		
		We see that $\bar h$ extends $h$. Moreover, $\bar h$ is continuous, because the preimage of a clopen set $C\subseteq K$ is the basic open set in $S_{\bar c}(\fC)$ corresponding to the relatively definable set $h^{-1}[C]$.
	\end{proof}
	
	In Corollary~\ref{cor: semigroup iff definability of types}, we will establish the aforementioned ``definability of types''-like condition from the existence of a semigroup operation.
	For this we will need the following definition.

	\begin{dfn}
		\index{piecewise definable type}
		Let $M$ be a model (e.g.\ $M=\fC$). A type $q(x) \in S(M)$ is {\em piecewise definable} if for every type $p(y)\in S(\emptyset)$ and formula $\varphi(x,y)$ the set of $a \in p(M)$ for which $q\vdash \varphi(x,a)$ is relatively $M$-definable in $p(M)$ (that is, there is a formula $\delta(y,c)$ (with $c$ from $M$) such that for any $a\in p(M)$ we have $q\vdash \varphi(x,a)$ if and only if $\delta(a,c)$).\xqed{\lozenge}
	\end{dfn}

	\begin{rem}\label{remark: piecewise weak definability equals piecewise definability}
		Let $q(x) \in S(\fC)$. The following conditions are equivalent.
		\begin{enumerate}
			\item $q(x)$ is piecewise definable.
			\item For every type $p(y)\in S(\emptyset)$ and formula $\varphi(x,y)$ there is a type $\bar p(yz) \in S(\emptyset)$ extending $p(y)$ such that the set of $ab \models \bar p$ for which $q\vdash \varphi(x,a)$ is relatively $\fC$-definable in $\bar p(\fC)$ (that is, there is a formula $\delta(yz,c)$ (with $c$ from $\fC$) such that for any $ab \models \bar p$ we have $q\vdash \varphi(x,a)$ if and only if $\delta(ab,c)$).
			\item For every $\varphi(x,y)$ and every $a$ from $\fC$ there is some $b$ from $\fC$ such that the set of all $a'b'$ from $\fC$ with $a'b'\equiv ab$ and $q\vdash \varphi(x,a')$ is relatively definable over $\fC$ (among all $a'b'$ from $\fC$ equivalent to $ab$).
		\end{enumerate}\xqed{\lozenge}
	\end{rem}
	
	\begin{proof}
		The equivalence $(2) \Leftrightarrow (3)$ is obvious. It is also clear that $(1) \Rightarrow (2)$, by taking $z$ and $b$ to be empty in (2). It remains to prove $(2) \Rightarrow (1)$.
		
		Take any type $p(y) \in S(\emptyset)$ and formula $\varphi(x,y)$. By (2), there is a type $\bar p(yz) \in S(\emptyset)$ extending $p(y)$ and a formula $\delta(yz,c)$ (with $c$ from $\fC$) such that for any $ab \models \bar p$ we have $q\vdash \varphi(x,a)$ if and only if $\delta(ab,c)$. This implies that for any $a,b,b'$ such that $ab \models \bar p$ and $ab' \models \bar p$ we have $\delta(ab,c) \leftrightarrow \delta(ab',c)$. By compactness, there is a formula $\psi(y,z) \in \bar p(yz)$ such that for any $a,b,b'$ with $ab \models \psi(y,z)$ and $ab' \models \psi(y,z)$ we have $\delta(ab,c) \leftrightarrow \delta(ab',c)$. Put
		$$\delta'(y,c) : = (\exists z) (\psi(y,z) \wedge \delta(yz,c)).$$
		It remains to check that for any $a \models p$, $q\vdash \varphi(x,a)$ if and only if $\delta'(a,c)$.
		
		First, assume $q\vdash \varphi(x,a)$ and $a \models p$. Take $b$ such that $ab \models \bar p$. Then $\psi(a,b) \wedge \delta(ab,c)$, and so $\delta'(a,c)$.
		
		Now, assume that $\delta'(a,c)$ and $a \models p$. Then there is $b$ such that $\psi(a,b)$ and $\delta(ab,c)$. There is also $b'$ with $ab' \models \bar p$, and then $\psi(a,b')$. By the choice of $\psi$, we conclude that $\delta(ab',c)$. Hence, $q\vdash \varphi(x,a)$.
	\end{proof}

	\begin{cor}\label{cor: semigroup iff definability of types}
		The natural action $\Aut(\fC)\times S_{\bar c}(\fC)\to S_{\bar c}(\fC)$ extends to a left-continuous semigroup operation on $S_{\bar c}(\fC)$ if and only if each complete type over $\fC$ is piecewise definable.
	\end{cor}
	
	\begin{proof}
		First, using Proposition~\ref{proposition: category C}, we will easily deduce:
		\begin{clm*}
			The action $\Aut(\fC)\times S_{\bar c}(\fC)\to S_{\bar c}(\fC)$ extends to a left-continuous semigroup operation on $S_{\bar c}(\fC)$ if and only if for each $q\in S_{\bar c}(\fC)$ the mapping $h_q \colon \Aut(\fC)\to S_{\bar c}(\fC)$ given by $\sigma \mapsto \sigma(q)$ is in the category $\catg$ (i.e.\ the preimages of clopen sets are relatively $\fC$-definable).
		\end{clm*}
		
		\begin{clmproof}[Proof of claim]
			$(\Rightarrow)$ Let $*$ be a left-continuous semigroup operation on $S_{\bar c}(\fC)$ extending the action of $\Aut(\fC)$. Consider any $q \in S_{\bar c}(\fC)$. Define $\bar h_q \colon S_{\bar c}(\fC) \to S_{\bar c}(\fC)$ by $\bar h_q(p) := p*q$. Then $\bar h_q$ is a continuous extension of $h_q$. By continuity, the preimages of clopen sets by $\bar h_q$ are clopen, and therefore their intersections with $\Aut(\fC)$ (which are exactly the preimages of clopen sets by the original map $h_q$) are relatively $\fC$-definable.
			
			$(\Leftarrow)$ By Proposition~\ref{proposition: category C}, for any $q \in S_{\bar c}(\fC)$ there exists a continuous function $\bar h_q \colon S_{\bar c}(\fC) \to S_{\bar c}(\fC)$ which extends $h_q$. For $p,q \in S_{\bar c}(\fC)$ define $p*q := \bar h_q(p)$. It is clear that $*$ (treated as a two-variable function) is left continuous and extends the action of $\Aut(\fC)$ on $S_{\bar c}(\fC)$. We leave as a standard exercise on limits of nets to check that $*$ is also associative.
		\end{clmproof}
		
		By the claim and Remark~\ref{remark: piecewise weak definability equals piecewise definability}, the whole proof boils down to showing that for any type $q(x) \in S_{\bar c}(\fC)$ we have the following equivalence: the preimage by $h_q$ of any clopen subset of $S_{\bar c}(\fC)$ is relatively definable in $\Aut(\fC)$ if and only if $q(x)$ satisfies item (3) of Remark~\ref{remark: piecewise weak definability equals piecewise definability}.

		Let us fix an arbitrary $q(x) \in S_{\bar c}(\fC)$, and any formula $\varphi(x,a)$ for some
		$a$ from $\fC$. The preimage by $h_q$ of the clopen set $[\varphi(x,a)]$ equals
		\[
		\{\sigma \in \Aut(\fC)\mid \sigma(q)\vdash \varphi(x,a)\}=\{\sigma \in \Aut(\fC)\mid q\vdash \varphi(x,\sigma^{-1}(a))\}.
		\]
		
		Now, for $(\Leftarrow)$, suppose there is some $b$ from $\fC$ and a formula $\delta(yz,c)$ (with $c$ from $\fC$) such that for any $a'b'\equiv ab$ we have $q\vdash \varphi(x,a')$ if and only if $\models \delta(a'b',c)$. Then (taking $a'b'=\sigma^{-1}(ab)$) we have that
		\[
		q\vdash \varphi(x,\sigma^{-1}(a))\iff \models \delta(\sigma^{-1}(ab),c)\iff \models \delta(ab,\sigma(c)),
		\]
		and the last statement is clearly relatively $\fC$-definable about $\sigma$.
		
		For $(\Rightarrow)$, suppose $\{\sigma\in \Aut(\fC)\mid q\vdash \varphi(x,\sigma^{-1}(a))\}$ is defined by some formula $\delta$, i.e.\ for some $d,c$ from $\fC$, for any $\sigma \in \Aut(\fC)$ we have
		\[
		q\vdash \varphi(x,\sigma^{-1}(a))\iff \models\delta(d,\sigma(c)) \iff \models \delta(\sigma^{-1}(d),c).
		\]
		We can assume without loss of generality that $d=ab$ for some $b$ from $\fC$ (adding dummy variables to $\delta$ if necessary). But then, for $a'b'\equiv ab$ there is some automorphism $\sigma$ such that $\sigma(a'b')=ab$, so we have
		\[
		q\vdash \varphi(x,a')\iff \models \delta(a'b',c).\qedhere
		\]
	\end{proof}
	
	This easily implies that stability is sufficient for the existence of a semigroup structure.
	
	\begin{cor}\label{corollary: stability implies semigroup}
		If $T$ is stable, then $S_{\bar c}(\fC)$ has a left-continuous semigroup operation extending the action of $\Aut(\fC)$ on $S_{\bar c}(\fC)$
	\end{cor}
	\begin{proof}
		If $T$ is stable, then every type over $\fC$ is definable, so in particular it is piecewise definable, which by Corollary~\ref{cor: semigroup iff definability of types} implies that the semigroup structure exists.
	\end{proof}
	
	For the other direction, we will use Corollary~\ref{cor: semigroup iff definability of types} and an easy counting argument. But before that we need to establish a transfer property for piecewise definability.
	
	\begin{prop}\label{proposition: transfer of piecewise definability}
		Suppose that each complete type over $\fC$ is piecewise definable. Then each complete type over any model $M$ of cardinality less then $\kappa$ (where $\kappa$ is the degree of saturation of $\fC$) is piecewise definable.
	\end{prop}
	
	\begin{proof}
		Take any $q(x) \in S(M)$.
		Consider any type $p(y) \in S(\emptyset)$ and formula $\varphi(x,y)$.
		Take a coheir extension $\bar q \in S(\fC)$ of $q$. Then $q$ is invariant over $M$. By assumption, $\bar q$ is piecewise definable. So there is a formula $\delta(y,c)$ (with $c$ from $\fC$) such that for any $a \in p(\fC)$, $\bar q\vdash \varphi(x,a)$ if and only if $\delta(a,c)$. Denote by $A$ the set of all $a \in p(\fC)$ satisfying these equivalent conditions; so $A$ is a relatively definable subset of $p(\fC)$. By the invariance of $\bar q$ over $M$, we see that $A$ is invariant over $M$, and so, by $\kappa$-saturation and strong $\kappa$-homogeneity of $\fC$, the subset $A$ of $p(\fC)$ is relatively definable over $M$. In other words, there is a formula $\delta'(y,m)$ (with $m$ from $M$) such that for any $a \in p(\fC)$, $\bar q\vdash \varphi(x,a)$ if and only if $\delta'(a,m)$. Hence, for any $a \in p(M)$, $q\vdash \varphi(x,a)$ if and only if $\delta'(a,m)$.
	\end{proof}
	
	\begin{cor}\label{cor: semigroup = stability}
		$S_{\bar c}(\fC)$ has a left-continuous semigroup operation extending the action of $\Aut(\fC)$ on $S_{\bar c}(\fC)$ if and only if $T$ is stable.
	\end{cor}
	\begin{proof}
		The ``if'' part is the content of Corollary~\ref{corollary: stability implies semigroup}.
		
		$(\Rightarrow)$ Assume $S_{\bar c}(\fC)$ has a left-continuous semigroup operation extending the action of $\Aut(\fC)$ on $S_{\bar c}(\fC)$. Then, by Corollary~\ref{cor: semigroup iff definability of types}, all complete types over $\fC$ are piecewise definable. We will show that this implies that $T$ is $\beth_2(|T|)$-stable (where $\beth_2(|T|): = 2^{2^{|T|}}$). Consider any $M \models T$ of cardinality at most $\beth_2(|T|)$. We need to show that $|S_1(M)| \leq \beth_2(|T|)$. For this it is enough to prove that for any $\varphi(x,y)$ (where $x$ is a single variable) $|S_\varphi(M)| \leq \beth_2(|T|)$. Without loss of generality $M \prec \fC$.
		By Proposition~\ref{proposition: transfer of piecewise definability}, each complete type over $M$ is piecewise definable. This implies that each type $q \in S_\varphi(M)$ is determined by a function $S_y(\emptyset) \to \lang(M)$ which takes $p(y)$ to $\delta(y,c)$ witnessing piecewise definability of $q$ (or, more precisely, of an arbitrarily chosen extension of $q$ to a type in $S_1(M)$) for the formula $\varphi(x,y)$. So $|S_\varphi(M)| \leq |\lang(M)|^{|S_y(\emptyset)|} \leq (\beth_2(|T|))^{2^{|T|}}=\beth_2(|T|)$.
	\end{proof}
	
	It is well known that if $T$ is stable, then it is $2^{|T|}$-stable. The reason why we worked with $\beth_2(|T|)$ in the above proof is that this is the ``degree'' of stability which we can deduce directly from piecewise definability. Then, knowing that $T$ is stable, we have the usual definability of types which implies $2^{|T|}$-stability.

	\section{Closed group-like implies properly group-like}
	Here, we show that closed group-like equivalence relations form a subclass of properly group-like equivalence relations (in particular, e.g.\ in Lemma~\ref{lem:main_abstract_grouplike}(2), in ``closed or properly group-like'', the ``closed'' part is redundant, and likewise in Lemma~\ref{lem:weakly_grouplike}(2).
	
	\begin{prop}
		\label{prop:closed_glike_is_properly_glike}
		Suppose $(G,X,x_0)$ is an ambit and $E$ is a closed group-like equivalence relation on $X$. Then $E$ is properly group-like.
	\end{prop}
	\begin{proof}
		The proof works via non-standard analysis.
		
		Consider the structure $M=(G,X,\cdot)$, where $G$ has its group structure, while $X$ has predicates for all open subsets of all its powers, and $\cdot\colon G\times X\to X$ is the group action.
		
		Now, given any $N=(G',X',\cdot )\equiv M$, we have a standard part function ${\st}\colon X'\to X$: $\st(x')=x$ when for every open $U\ni x$ we have $x'\in U^N$. By the Hausdorff condition, there is at most one such $x$, and by compactness, it always exists. Moreover, note that if $x'\in U^N$, then $\st(x')\in U$. Similarly, a finite tuple of elements of $X$ also has a standard part (which is the tuple of standard parts of its coordinates).
		
		Now let $M^*=(G^*,X^*,\cdot)\succeq M$ be a highly saturated elementary extension (to be precise, it is enough for it to be saturated in any cardinality greater than the local character of $X$, e.g.\ if $X$ is first-countable, we can take any ultrapower of $M$ with respect to a non-principal ultrafilter). For $\tilde g\in G^*$ let $[\tilde g]_\equiv=\st(\tilde g\cdot x_0)$. We will show that $\tilde G:=G^*$ witnesses proper group-likeness of $E$.
		
		(Note that the action of $G^*$ on $X^*$ generally does \emph{not} give us a well-defined action of $G^*$ on $X$, unless the action of $G$ on $X$ is equicontinuous, as we can have $\st(x_1^*)=\st(x_2^*)$ but $\st(\tilde g\cdot x_1^*)\neq \st(\tilde g\cdot x_2^*)$.)
		
		First, we need to show that $\tilde g\mapsto [\st(\tilde g\cdot x_0)]_E$ is a group homomorphism. Let $\mu\subseteq X^3$ be the set of triples $(x_1,x_2,x_3)$ such that $[x_1]_E\cdot [x_2]_E=[x_3]_E$. Then by group-likeness, for all $g_1,g_2\in G$ we have that $(g_1x_0,g_2x_0,g_1g_2x_0)\in \mu$, and therefore, for any open $U\supseteq \mu$ we have $(g_1x_0,g_2x_0,g_1g_2x_0)\in U$. By elementarity, for all $\tilde g_1,\tilde g_2\in G^*$ we have that $(\tilde g_1x_0,\tilde g_2x_0,\tilde g_1\tilde g_2 x_0)\in U^*$, where $U^*=U^{M^*}$. By the preceding remarks, we have also that $(\st(\tilde g_1x_0),\st(\tilde g_2x_0),\st(\tilde g_1\tilde g_2x_0))\in U$.
		
		Now, since $X/E$ is a Hausdorff topological group, the graph of its multiplication is closed, and it follows that $\mu$ is closed, and as such (because $X$ is compact Hausdorff), it is equal to the intersection of all the open sets contained it, so (because $U$ in the preceding paragraph was arbitrary) in fact, we have for any $\tilde g_1,\tilde g_2$ that $(\st(\tilde g_1x_0),\st(\tilde g_2x_0),\st(\tilde g_1\tilde g_2x_0))\in \mu$, i.e.\ $[\st(\tilde g_1x_0)]_E\cdot [\st(\tilde g_2x_0)]_E=[\st(\tilde g_1\tilde g_2 x_0)]_E$.
		
		To show that we have pseudocompleteness, note that if we have for some nets $g_ix_0\to x_1$ and $p_i\to x_2$ and $g_i\cdot p_i\to x_3$, for any open neighbourhoods $U_1,U_2,U_3$ of $x_1,x_2,x_3$ (respectively), there are $g',g''\in G$ such that $g'x_0\in U_1$, $g''x_0\in U_2$ and $g'g''x_0\in U_3$. Indeed, if we take any $i$ such that $g_ix_0\in U_1$, $p_i\in U_2$ and $g_ip_i\in U_3$, then we can take $g'=g_i$ and $g''$ such that $g''x_0\in U_2\cap g_i^{-1}[U_3]\ni p_i$ (which exist because $(X,x_0)$ is a $G$-ambit, and $U_2\cap g_i^{-1}[U_3]$ is a nonempty open set).
		
		It follows by compactness that we have $\tilde g_1,\tilde g_2\in G^*$ such that $\st(\tilde g_1\cdot x_0)=x_1$, $\st(\tilde g_2\cdot x_0)=x_2$ and $\st(\tilde g_1\tilde g_2\cdot x_0)=x_3$, which gives us pseudocompleteness.
		
		For the final part, note that $[\tilde g_1]_\equiv=[\tilde g_2]_\equiv$ means just that for every open $U\subseteq X$ we have that $\tilde g_1\cdot x_0\in U^*\leftrightarrow \tilde g_2\cdot x_0\in U^*$, which is a type-definable condition. Thus $F_0'=\{ \tilde g_1^{-1}\tilde g_2x_0\mid [\tilde g_1]_\equiv=[\tilde g_2]_\equiv \}$ is a type-definable set, so by compactness, $F_0=\{\st(x^*)\mid x^*\in F_0' \}$ is closed.
	\end{proof}
	
	Note that if $E$ is closed group-like, then $E$ itself is a symmetric closed set containing the diagonal and $E\circ E=E$, so $\mathcal E=\{E\}$ could conceivably witness \emph{uniformly} proper group-likeness of $E$ (according to Definition~\ref{dfn:unif_prop_glike}). Furthermore, closed group-like equivalence relations share many properties of uniformly properly group-like equivalence relations. This suggests the following question.
	
	\begin{qu}
		\label{qu:clsd_gplike_upglike}
		Are closed group-like equivalence relations uniformly properly group-like?
	\end{qu}

	The problem is that it is not clear how to choose $\tilde G$: we would need to have that for every $\tilde g\in \tilde G$ such that $[\tilde g]_{\equiv}\Er x_0$, for every other $\tilde g'\in \tilde G$, $[\tilde g']_{\equiv}\Er [\tilde g\tilde g']_{\equiv}$
	
	Note that for $\tilde G=G^*$ as in the proof of Proposition~\ref{prop:closed_glike_is_properly_glike}, there seems to be no obvious reason for this to be true.

	\microtypesetup{disable}
	\printbibliography
	\addcontentsline{toc}{chapter}{Bibliography}
	\microtypesetup{enable}
	\printindex
	\addcontentsline{toc}{chapter}{List of symbols and definitions}
\end{document}